\documentclass[10pt,oneside]{amsart}
\usepackage[a4paper, textwidth=15.6cm]{geometry}
\usepackage[T1]{fontenc}
\usepackage[english]{babel}
\usepackage{aliascnt}

\usepackage{xcolor} 
\usepackage{tikz}
\usepackage[draft=false]{hyperref} 
\usepackage{mathtools}
\usepackage{amsmath,amsthm,amssymb}
\usepackage[all,cmtip]{xy}

\usepackage[capitalise]{cleveref}

\newtheorem{theorem}{Theorem}[subsection]

\newaliascnt{lemma}{theorem}
\newtheorem{lemma}[lemma]{Lemma}
\aliascntresetthe{lemma}
\crefname{Lemma}{Lemma}{Lemmas}
\Crefname{Lemma}{Lemma}{Lemmas}

\newaliascnt{question}{theorem}

\aliascntresetthe{question}
\crefname{Question}{Question}{Questions}
\Crefname{Question}{Question}{Questions}

\newaliascnt{proposition}{theorem}
\newtheorem{proposition}[proposition]{Proposition}
\aliascntresetthe{proposition}
\crefname{Proposition}{Proposition}{Propositions}
\Crefname{Proposition}{Proposition}{Propositions}

\newaliascnt{corollary}{theorem}
\newtheorem{corollary}[corollary]{Corollary}
\aliascntresetthe{corollary}
\crefname{Corollary}{Corollary}{Corollaries}
\Crefname{Corollary}{Corollary}{Corollaries}

\theoremstyle{definition}
\newaliascnt{definition}{theorem}
\newtheorem{definition}[definition]{Definition}
\aliascntresetthe{definition}
\crefname{Definition}{Definition}{Definitions}
\Crefname{Definition}{Definition}{Definitions}

\newaliascnt{notation}{theorem}
\newtheorem{notation}[notation]{Notation}
\aliascntresetthe{notation}
\crefname{Notation}{notation}{Notations}
\Crefname{Notation}{notation}{Notations}

\newaliascnt{construction}{theorem}
\newtheorem{construction}[construction]{Construction}
\aliascntresetthe{construction}
\crefname{Construction}{construction}{Constructions}
\Crefname{Construction}{construction}{Constructions}

\newaliascnt{example}{theorem}
\newtheorem{example}[example]{Example}
\aliascntresetthe{example}
\crefname{Example}{example}{Examples}
\Crefname{Example}{example}{Examples}

\newaliascnt{remark}{theorem}
\newtheorem{remark}[remark]{Remark}
\aliascntresetthe{remark}
\crefname{Remark}{Remark}{Remarks}
\Crefname{Remark}{Remark}{Remarks}

\usepackage{amsmath,amsthm,amssymb,amsfonts,extarrows}
\usepackage{enumerate,amscd,amsxtra,MnSymbol}
\usepackage{mathrsfs}
\usepackage{color}
\usepackage[cmtip,all]{xy}
\usepackage{eucal}
\usepackage{bbold}
\usepackage{upgreek}
\usepackage{paralist}
\usepackage{tikz-cd}
\usepackage{dsfont}
\usepackage{bbm}
\usepackage[bbgreekl]{mathbbol}

\DeclareMathSymbol\bbDelta \mathord{bbold}{"01}
\DeclareMathSymbol\bDelta \mathord{bbold}{"01}

\newcommand{\bA}{{\mathbb A}}

\newcommand{\bD}{{\mathbb D}}
\newcommand{\bE}{{\mathbb E}}

\renewcommand{\P}{{\mathbb P}}

\newcommand{\bR}{{\mathbb R}}

\newcommand{\bN}{{\mathbb N}}

\newcommand{\mA}{{\mathcal A}}
\newcommand{\mB}{{\mathcal B}}
\newcommand{\mC}{{\mathcal C}}
\newcommand{\mD}{{\mathcal D}}
\newcommand{\mE}{{\mathcal E}}
\newcommand{\mF}{{\mathcal F}}

\newcommand{\mH}{{\mathcal H}}

\newcommand{\mJ}{{\mathcal J}}

\newcommand{\mM}{{\mathcal M}}
\newcommand{\mN}{{\mathcal N}}
\newcommand{\mO}{{\mathcal O}}
\newcommand{\mP}{{\mathcal P}}

\newcommand{\mU}{{\mathcal U}}
\newcommand{\mV}{{\mathcal V}}
\newcommand{\mW}{{\mathcal W}}
\newcommand{\mX}{{\mathcal X}}

\newcommand{\C}{{\mathrm C}}

\newcommand{\E}{{\mathrm E}}
\newcommand{\F}{{\mathrm F}}
\newcommand{\G}{{\mathrm G}}

\renewcommand{\L}{{\mathrm L}}

\renewcommand{\P}{{\mathrm P}}
\newcommand{\Q}{{\mathrm Q}}
\newcommand{\R}{{\mathrm R}}
\newcommand{\rS}{{\mathrm S}}

\newcommand{\W}{{\mathrm W}}

\newcommand{\Y}{{\mathrm Y}}

\newcommand{\bj}{\mathrm{j}}
\newcommand{\bi}{\mathrm{i}}
\newcommand{\m}{\mathrm{m}}
\newcommand{\bk}{\mathrm{k}}

\newcommand{\n}{\mathrm{n}}

\newcommand{\op}{\mathrm{op}}

\newcommand{\colim}{\mathrm{colim}}

\newcommand{\Mod}{{\mathrm{Mod}}}

\newcommand{\rev}{{\mathrm{rev}}}

\newcommand{\ot}{\otimes}

\newcommand{\co}{\mathrm{co}}

\newcommand{\Bi}{\mathrm{Bi}}

\newcommand{\Pre}{\mathrm{Pre}}

\newcommand{\id}{\mathrm{id}}

\newcommand{\Cat}{\mathrm{Cat}}

\newcommand{\Set}{\mathrm{Set}}

\renewcommand{\Pr}{\mathrm{Pr}}

\newcommand{\Alg}{\mathrm{Alg}}

\newcommand{\Mon}{\mathrm{Mon}}
\newcommand{\Fun}{\mathrm{Fun}}

\newcommand{\lax}{{\mathrm{lax}}}
\newcommand{\oplax}{{\mathrm{oplax}}}

\newcommand{\tu}{{\mathbb 1}}

\newcommand{\ev}{{\mathrm{ev}}}
\newcommand{\Ind}{{\mathrm{Ind}}}

\newcommand{\Map}{{\mathrm{Map}}}

\newcommand{\LinFun}{{\mathrm{LinFun}}} 

\newcommand{\bZ}{{\mathbb{Z}}}

\newcommand{\Rig}{{\mathrm{Rig}}} 
\newcommand{\Mor}{{\mathrm{Mor}}}

\newcommand{\fin}{\mathrm{fin}}
\newcommand{\Free}{\mathrm{Free}}

\newcommand{\cop}{\mathrm{cop}}

\newcommand{\Sp}{{\mathrm{Sp}}}

\newcommand{\Fil}{{\mathrm{Fil}}}

\newcommand{\Exc}{{\mathrm{Exc}}} 
\newcommand{\mon}{{\mathrm{mon}}}

\newcommand{\Grp}{{\mathrm{Grp}}} 
\newcommand{\Pic}{{\mathrm{Pic}}}

\newcommand{\Poset}{{\mathrm{Poset}}}

\newcommand{\cube}{{\,\vline\negmedspace\square}}

\begin{document}

\title{Stable homotopy theory of higher categories}


\author{Hadrian Heine, \\ Max Planck Institute Bonn, Germany, \\ heine@mpim-bonn.mpg.de}
\maketitle


\begin{abstract}

Stable homotopy theory is governed by the principle that inverting the operation of forming loop spaces produces representing objects for homology theories. 
We show that this principle is not specific to topology but reflects an intrinsic structural feature of higher category theory: inverting the operation of forming endomorphism $(\infty,\infty)$-categories leads to a stable homotopy theory of higher categories, in which higher categories play the role of spaces and categorical spectra represent homology theories of higher categories.

This theory necessarily requires enrichment in the Gray tensor product, reflecting its genuinely higher-categorical nature.
Classical stable homotopy theory is recovered by passing to classifying spaces. Its fundamental mechanisms arise as shadows of richer categorical phenomena: stabilization is governed by a categorical Freudenthal suspension theorem, whose classical counterpart arises by passing to classifying spaces.
Our main result is a categorical Brown representability theorem classifying homology theories of higher categories by categorical spectra. As a consequence, categorical homology theories give rise to categorical analogues of long exact sequences and to homological algebra of higher categories.

As guiding example we study categorical homology, the categorical homology theory whose coeffients are the natural numbers, and prove a Hurewicz theorem and Eilenberg-Zilber theorem.

\end{abstract}

\tableofcontents

\section{Introduction}

\subsubsection*{Motivation}

Grothendieck’s homotopy hypothesis \cite{grothendieck2022pursuing, quillen2006homotopical} suggests that higher category theory refines homotopy theory by encoding genuinely non-invertible phenomena.
This perspective raises a natural question: does stable homotopy theory itself admit an intrinsic higher-categorical form?

Stable homotopy theory is one of the fundamental organizing structures of homotopy theory. Its central objects, spectra, arise by stabilizing spaces, and this stabilization underlies some of the deepest structural and computational methods of algebraic topology. We show that this mechanism is not intrinsically tied to spaces. Stable homotopy theory itself has an intrinsic higher-categorical form in the world of $(\infty,\infty)$-categories.

The mechanism that produces spectra from spaces extends naturally to higher categories, where stabilization is obtained by inverting endomorphism $(\infty,\infty)$-categories rather than loop spaces. However, the resulting theory is not merely a formal categorification of classical stable homotopy theory. The presence of higher categorical dimensions fundamentally changes the mechanism of stabilization and forces the theory to be formulated using enrichment in the Gray tensor product of higher categories. Unlike the cartesian product, the Gray tensor product is sensitive to higher categorical dimension and is capable of encoding genuinely lax phenomena. This leads to stable phenomena absent from classical stable homotopy theory.
We reveal this intrinsic stable homotopy theory of higher categories, in which categorical spectra play the role of spectra and stability is obtained by inverting endomorphism $(\infty,\infty)$-categories.

The genuinely new features of this theory arise from categorical dimension, the smallest number $n$ for which an $(\infty,\infty)$-category is an $(\infty,n)$-category. 
A fundamental difference from the classical situation is that delooping increases categorical dimension. As a consequence, delooping is not enriched in the cartesian product of $\infty$-categories, which fixes categorical dimension, but in the Gray tensor product, which behaves additively with respect to categorical dimension.
This forces the stable homotopy theory of higher categories to take place in the setting of categories enriched in the Gray tensor product, which we call oriented categories \cite{gepner2025oriented}.

Classically, spectra admit several complementary characterizations: they provide the universal stabilization of spaces, encode stable homotopy controlled within a stable range, are the derived counterparts of abelian groups, and represent homology theories. 
We show that these features persist in higher category theory:

\begin{enumerate}

\item Categorical spectra provide the universal stabilization of higher categories (\cref{specunive}).

\item They are the higher-categorical counterparts of commutative monoids (\cref{symsp}).

\item They encode stabilized higher categorical information controlled within a stable range by a categorical Freudenthal suspension theorem (\cref{F}).

\item They represent homology theories of higher categories via a categorical Brown representability theorem (\cref{brown}).

\end{enumerate}

These properties characterize categorical spectra as the natural higher-categorical counterpart of spectra in topology. 
Conceptually, this realizes a fundamental principle of stable homotopy theory: stabilization produces representing objects for homology theories. We show that this principle is not specific to topology but reflects an intrinsic structural feature of higher category theory:
inverting the operation of forming endomorphism categories leads to spectra of higher categories that represent homology theories of higher categories.

A fundamental consequence of this theory is the emergence of homological invariants of higher categories.
Homological invariants are among the central organizing and computational tools of algebraic topology
and provide much of the computational power.
By contrast, higher category theory lacks an analogous intrinsic theory of homological invariants that extract information from higher categories while remaining in the world of higher categories, in the form of higher-categorical algebraic objects rather than their decategorifications.
We show that such categorical homological invariants arise naturally from the intrinsic stable homotopy theory of higher categories. 
In this way, the computational power of stable homotopy theory becomes available at the higher-categorical level, leading to a systematic theory of homological invariants of higher categories.

\subsubsection*{Stability in higher category theory}

Stable homotopy theory is governed by the principle that stabilization arises by inverting loop spaces.
In higher category theory, stabilization arises from inverting endomorphism categories, leading to the oriented category $\Cat\Sp$ of categorical spectra.  
Concretely, a categorical spectrum is a sequence of based $\infty$-categories $\{X_n\}_{n \geq 0}$ equipped with bonding equivalences $$ X_n \simeq \Omega(X_{n+1}).$$

The transition from higher categories to categorical spectra is governed by a categorical Freudenthal suspension theorem, which ensures that iterated suspension stabilizes in a controlled range. A key structural input is that endomorphisms are corepresented by a categorical analogue of spheres. These arise as iterated suspensions of the two-point space, in analogy with classical homotopy theory.
In topology, the $n$-sphere is the $n$-fold delooping of the free grouplike $\bE_n$-space on one generator. This principle persists in higher category theory: the categorical $n$-sphere is the $n$-fold delooping of the free $\bE_n$-space on one generator, which is the space of finite unordered configurations in $\mathbb{R}^n$.

These phenomena reflect a notion of stability governed jointly by suspension and endomorphisms. 
Stable oriented categories provide the axiomatization of this stability.
Oriented categories admit intrinsic notions of fiber and cofiber sequences, called oriented fiber and cofiber sequences, which play the role of homotopy fiber and cofiber sequences in topology. These lead to a natural notion of stability: following the principle that stabilization is obtained by inverting endomorphisms, an oriented category is stable if every morphism admits oriented fibers and cofibers and the operation of endomorphisms becomes invertible.

Stable oriented categories exhibit several characteristic features of classical stable categories in the higher-categorical setting.
For instance, categorical spectra carry a canonical monoidal structure \cite{masuda2026algebra} analogous to the classical smash product, which is the stabilization of the Gray tensor product (\cref{stten}),
and stable oriented categories are uniquely enriched in categorical spectra (\cref{spectral0}, \cref{specenr}).

\subsubsection*{Categorical homology theories}

Classically, reduced homology theories assign to each pointed homotopy type a graded abelian group and satisfy the Eilenberg–Steenrod axioms. From a modern derived perspective, they can instead be characterized as reduced functors from pointed homotopy types to pointed homotopy types that preserve filtered colimits and are excisive: they send homotopy pushout squares of based homotopy types
\[
\begin{tikzcd}
X \ar{r}{} \ar{d}{} & 0 \ar{d}{} \\
0 \ar{r}{} & Y
\end{tikzcd}
\]
to homotopy pullback squares.

We prove a Brown representability theorem characterizing the oriented endofunctors of based $\infty$-categories represented by categorical spectra. 
The resulting representable functors provide the natural higher-categorical counterparts of classical homology theories, which we call categorical homology theories.
These are precisely the reduced oriented functors from based $\infty$-categories to based $\infty$-categories that preserve filtered colimits and send oriented pushout squares of based $\infty$-categories 
\begin{equation}\label{equkb}
\begin{tikzcd}
X \ar{r}{} \ar{d}[swap]{} & 0 \ar[double]{dl}{} \ar{d}{} \\
0 \ar{r}[swap]{} & Y
\end{tikzcd}
\end{equation}
to oriented pullback squares.
In addition, they satisfy a compatibility condition with categorical spheres. 
This condition reflects a fundamental difference from classical homotopy theory: based $\infty$-categories are not generated under small colimits by the zero sphere.

A fundamental property of classical homology is that cofiber sequences give rise to long exact sequences on homology. This property is a key principle in computations: spaces are built by inductively attaching cells, and homology transforms these cell attachments into long exact sequences.
We establish an analogous mechanism for categorical homology theories, in which long exact sequences are replaced by a higher-categorical counterpart, called oriented exact sequences. This provides a systematic method for computing categorical homology theories through mechanisms analogous to those of classical stable homotopy theory (\cref{basiccomp}).

\vspace{1mm}

The present work establishes the structural foundations of a stable homotopy theory of higher categories and, in particular, its associated theory of homological invariants. It is part of a broader program \cite{heine2026categorification}, \cite{gepner2026homotopy}, based on oriented category theory \cite{gepner2025oriented}, \cite{gepner2026colimits}, \cite{gepner2026fibrations}, to develop an intrinsic homotopy theory of higher categories. The present paper supplies the stable homotopy-theoretic foundation of this program.

We apply this framework to the categorical homology theory represented by the categorical Eilenberg--MacLane spectrum of any rig $R$, which we call categorical $R$-homology. We show that categorical $R$-homology exhibits higher-categorical analogues of fundamental structural features of classical homology: it satisfies a Hurewicz theorem and lifts multiplicatively to $(R,R)$-bimodule categorical spectra.

In \cite{heine2026categorification} we provide a combinatorial presentation of categorical homology in the spirit of singular homology
and prove higher-categorical versions of the Dold--Kan and Dold--Thom theorems leading to explicit computations of categorical homology.
In particular, categorical homology distinguishes globes of different dimensions whose classifying spaces are contractible, thereby detecting genuinely higher-categorical information invisible to classical homology.

\section{Main results}

We prove the following higher-categorical Brown representability theorem characterizing the canonical homological invariants of higher categories by categorical spectra:

\begin{theorem}\label{3} (\cref{brown}) 
The assignment $$ E \mapsto \Omega^\infty(E \wedge(-))$$ 
induces an oriented equivalence between categorical spectra and
oriented functors $H: \infty\Cat_* \to \infty\Cat_* $ satisfying the following properties:

\begin{enumerate}
\item $H$ is reduced and preserves small filtered colimits.

\item $H$ sends oriented pushout squares of the form (\ref{equkb}) to oriented pullback squares.

\item For every $X \in \infty\Cat_*$ and even $\ell \geq 0$ the following canonical based functor is an equivalence:
$$ \Omega^\ell(H(S^\ell)\wedge X) \to \Omega^\ell(H(S^\ell \wedge X)) \simeq \Omega^\ell(H(X \wedge S^\ell)) \simeq H(X) $$

\end{enumerate}
\end{theorem}

The third condition expresses a compatibility with categorical spheres. 
It reflects a fundamental difference from classical homotopy theory: based $\infty$-categories are not generated under small colimits by categorical spheres.
Its formulation involves a non-formal canonical twist relating the smash product with even-dimensional categorical spheres.

\vspace{1mm}

The proof proceeds via a theory of categorical excisive approximations, which identifies categorical homology theories with their associated categorical spectra of coefficients.

\vspace{1mm}

A fundamental feature of homology is that cofiber sequences give rise to long exact sequences on homology.
The categorical Brown representability theorem implies 
a categorical analogue:

\begin{corollary}\label{coroui} (\cref{lifta}, \cref{homologic})
Every categorical homology theory $E$
admits a unique lift to categorical spectra.
Every oriented cofiber sequence $X \to Y \to Z$ of based $\infty$-categories gives rise to an oriented exact sequence of categorical spectra
$$ ... \to E(Z)[-1] \to E(X) \to E(Y) \to E(Z) \to E(X)[1] \to E(Y)[1] \to ... $$
on categorical homology.

\end{corollary}

An important feature of long exact sequences is that equivalences between long exact sequences can be detected inductively.
This is a key ingredient in inductive computations of homology.

\cref{coroui} and \cref{homolog} imply a higher-categorical analogue:

\begin{corollary}
Let $E, E'$ be categorical homology theories and $\alpha: E \to E'$ a morphism of categorical homology theories.
Let
$$	
\begin{tikzcd}
X \ar{r} \ar{d}{f}  & Y \ar{r} \ar{d}{g}  & Z \ar{d}{h} \\ 
X' \ar{r}  & Y' \ar{r} & Z'
\end{tikzcd}
$$	
be a commutative diagram of based $\infty$-categories, where the horizontal sequences are oriented cofiber sequences. 

We consider the following commutative diagram of categorical spectra on categorical homology:
$$	
\begin{tikzcd}
... \ar{r}  & E(X) \ar{r} \ar{d}{\alpha_f}  & E(Y) \ar{d}{\alpha_g}  \ar{r}  & E(Z) \ar{d}{\alpha_h} \ar{r}  & E(X)[1] \ar{d}{\alpha_f[1]}  \ar{r}  & E(Y)[1] \ar{d}{\alpha_g[1]}  \ar{r}  & ... \\ 
... \ar{r}  & E'(X') \ar{r}  & E'(Y') \ar{r}  & E'(Z') \ar{r}  & E'(X')[1] \ar{r}  & E'(Y')[1] \ar{r}  & ...
\end{tikzcd}
$$	

If any two of $$ \alpha_f[\bullet], \alpha_g[\bullet], \alpha_h[\bullet]$$ are equivalences, then so is the third.

\end{corollary}

This gives a systematic inductive method for computing categorical homology theories, analogous to the role played by long exact sequences in classical algebraic topology.

The representability theorem is the culmination of a broader stable homotopy theory of higher categories developed in Sections 4 and 5. In this theory, categorical spectra arise as the universal stabilization of higher categories (\ref{specunive}). The passage from unstable to stable
homotopy theory of higher categories is controlled by the following categorical Freudenthal suspension theorem:

\begin{theorem}(\cref{F}) Let $n \geq 0$ and $X \in \infty\Cat_*$ be $n$-connected.
The canonical functor $$ X \to \Omega\Sigma(X)$$
is $2 n+1$-connective.

\end{theorem}

Here, an $\infty$-category is $n$-connected if its maximal subspace is $n$-connected and all its morphism $\infty$-categories are $n-1$-connected. Similarly, a functor is $n$-connective if it is
essentially surjective and induces 
$n-1$-connective functors on morphism $\infty$-categories.

As a consequence, iterated suspension and endomorphisms converge to categorical stabilization in a controlled range:

\begin{corollary} Let $n \geq 0$ and $X \in \infty\Cat_*$.
The canonical functor $$\Omega^{\n+1} \Sigma^{\n+1}(X) \to \Omega^\infty\Sigma^\infty(X) $$ is $\n$-connective.

\end{corollary}

Finally, we apply the theory to categorical $R$-homology for any rig $R$, proving categorical Hurewicz and Eilenberg--Zilber theorems.

\subsubsection*{Relation to other work}

Christ-Dyckerhoff-Walde \cite{christ2023complexes} and   
Dyckerhoff \cite{dyckerhoff2021categorified} study chain complexes of stable $(\infty,1)$-categories, Lurie \cite{lurie2008classification} introduces categorical chain complexes and Kern \cite{kern2024categoricalspectrapointedinftymathbbzcategories}, Masuda \cite{masuda2026algebra}, Stefanich \cite[\S 13.3]{HigherQuasi}, and Scholze \cite{Scholze2025} study categorical spectra and related structures.
The present work identifies the ambient stable homotopy theory: categorical spectra arise as the universal stabilization of higher categories, satisfy a categorical Freudenthal suspension theorem, represent categorical homology theories, and give rise to oriented exact sequences.






\subsubsection*{Acknowledgements}

We thank Thorger Geiß, David Gepner, Markus Spitzweck and David White for helpful discussions.

\subsubsection*{Notation and terminology}

We fix a hierarchy of Grothendieck universes whose objects we call small, large, very large, etc.
We call a space small, large, etc. if its set of path components and its homotopy groups are small, large, etc. for any choice of base point. 

We refer to weak $(\infty,n)$-categories for $0 \leq \n \leq \infty$ as $\n$-categories and refer to weak $(\n,\n)$-categories as $(\n,\n)$-categories.
In particular, we refer to $(\infty,1)$-categories as 1-categories or simply categories.

\vspace{1mm}
We write 
\begin{itemize}
\item $\Set$ for the category of small sets.
\item $\Delta$ for (a skeleton of) the category of finite, non-empty, partially ordered sets and order preserving maps, whose objects we denote by $[n] = \{0 < ... < n\}$ for $\n \geq 0$.
\item $\infty\Grp$ for the $\infty$-category of small homotopy types.
\item $ \Cat$ for the $\infty$-category of small $\infty$-categories.	
\item $\Map_\mC(A,B)$ for the space of maps $A \to B$ in $\mC$ for any category $\mC$ containing objects $A,B.$
\item $\Fun(\mC,\mD)$ for the category of functors $\mC \to \mD$ between two categories $\mC,\mD$, the internal hom of $\Cat$ for the cartesian product.

\item We write $\ast$ for the final space.

\item We write $\bD^1$ for the walking arrow, the category with two objects and a unique non-identity arrow.

\item We write $\partial\bD^1$ and $S^0$ for the maximal subspace in $\bD^1$, the set with two elements.

\end{itemize}

\vspace{1mm}

We indicate $\infty$-categories of large objects by $\widehat{(-)}$, for example we write $\widehat{\infty\Grp}, \widehat{\Cat}$ for the $\infty$-categories of large spaces, large categories.

We often call a fully faithful functor $\mC \to \mD$ an embedding.
We call a functor $\mC \to \mD$ an inclusion if it induces an embedding on maximal subspaces and on all mapping spaces. The latter is equivalent to asking that for any category $\mB$ the induced map
$\Map_{\Cat}(\mB,\mC) \to \Map_{\Cat}(\mB,\mD)$ is an embedding.

\section{Higher categorical foundations}

In this section we recall the basic higher category theory needed for our work.


\subsection{Enriched categories}

We recall the theory of (homotopy coherent) enriched categories \cite{GEPNER2015575}, \cite{Heine2023AnEB}, \cite{heine2025equivalence}, \cite{HINICH2020107129}.


For every presentably monoidal category $\mV$ there is a presentable 2-category $ \mV\mathrm{-}\Cat$ of $\mV$-enriched categories and $\mV$-enriched functors and a forgetful functor $ \iota:\mV\mathrm{-}\Cat \to \Cat$ to the presentable 2-category $\Cat$ of categories.
For $\mV= \infty\Grp$ the category of homotopy types, the forgetful functor $ \iota:\infty\Grp\mathrm{-}\Cat \to \Cat$ is an equivalence by \cite[Corollary 3.23.]{heine2024bienrichedinftycategories}.

\begin{notation}Let $\mV$ be a presentably monoidal category.
A $\mV$-enriched category $\mC$ provides an underlying category $\iota(\mC)$, for all objects $X,Y \in \iota(\mC)$ a morphism object 
$$\Mor_\mC(X, Y) \in \mV$$
and for all objects $X,Y,Z \in \iota(\mC)$ a composition morphism in $\mV:$
$$\Mor_\mC(Y, Z) \ot \Mor_\mC(X, Y) \to \Mor_\mC(X, Z).$$

We write $ X \in \mC$ for $ X \in \iota(\mC)$ and usually notationally identify $\mC$ with $\iota(\mC).$





	
\end{notation}


\begin{notation}
Let $\mV$ be a presentably monoidal category and $\mC, \mD$ be $\mV$-enriched categories. Let $ \mV \mathrm{-}\Fun(\mC,\mD)$ denote the category of $\mV$-enriched functors $\mC \to \mD$ using that $\mV \mathrm{-}\Cat$ is a 2-category.

\end{notation}

\subsubsection{Opposite enrichment}

For every enriched category there is an opposite one:

The following is \cite[Proposition 4.60.]{heine2024bienrichedinftycategories}

\begin{proposition}
There is an involution $$(-)^\circ: \mV\mathrm{-}\Cat \simeq \mV^\rev\mathrm{-}\Cat$$ forming the opposite enriched category.

For every $\mC \in \mV\mathrm{-}\Cat$ and $X,Y \in \mC$
there are canonical equivalences
$$ \iota(\mC^\circ) \simeq \iota(\mC)^\op$$
and $$ \Mor_{\mC^\circ}(X,Y) \simeq \Mor_{\mC}(Y,X).$$
\end{proposition}

\subsubsection{Transfer of enrichment}

The following is \cite[Theorem 3.67., Corollary 3.71., Corollary 3.74.]{heine2024bienrichedinftycategories}:

\begin{proposition}

Let $\mV, \mW$ be presentably monoidal categories and $\phi: \mV \to \mW$ a lax monoidal functor.

\begin{enumerate}

\item There is a canonical functor $\phi_!: \mV\mathrm{-}\Cat \to \mW\mathrm{-}\Cat$
such that for every $\mV$-enriched category $\mC$ there is a canonical essentially surjective map $\kappa: \iota(\mC) \to \iota(\phi_!(\mC)) $ and for every $X, Y \in \mC$ there is a canonical equivalence
$$ \Mor_{\phi_!(\mC)}(\kappa(X), \kappa(Y)) \simeq \Mor_\mC(X,Y).$$

\item Let $\phi: \mV \to \mW$ be a monoidal functor that admits a right adjoint $\gamma.$ There is an adjunction $\phi_!: \mV\mathrm{-}\Cat \rightleftarrows \mW\mathrm{-}\Cat: \gamma_!.$
For every presentably left $\mW$-tensored category $\mD$ the restriction of the left $\mW$-action along $\phi$ identifies with $\gamma_!(\mD).$

\vspace{1mm}

\item If $\phi: \mV \to \mW$ is a lax monoidal embedding, the functor
$\phi_!: \mV\mathrm{-}\Cat \to \mW\mathrm{-}\Cat$ is fully faithful.

\end{enumerate}

\end{proposition}

\begin{remark}\label{indubi} Let $\mU, \mV, \mW $ be presentably monoidal categories and $F: \mV \rightleftarrows \mW :G $ a left adjoint monoidal functor.
The induced left adjoint monoidal functor
$\mU \ot \mV \to \mU \ot \mW $ identifies with the left adjoint
of the induced functor
\begin{equation}\label{homok}
\Fun^R(\mU^\op, \mW) \to \Fun^R(\mU^\op,\mV).\end{equation}

So if $G: \mW \to \mV$ is fully faithful, the right adjoint functor
(\ref{homok}) is fully faithful. 
If the functor $F: \mV \to \mW$ preserves small limits, the left adjoint of (\ref{homok}) identifies with the induced functor $ \Fun^R(\mU^\op, \mV) \to \Fun^R(\mU^\op,\mW)$, which preserves small limits.
The latter is fully faithful if $F$ is fully faithful.

\end{remark}

\subsubsection{Tensors and cotensors}

\begin{definition}Let $\mV$ be a presentably monoidal category, $\mC$ a $\mV$-enriched category and $X \in \mC, V \in \mV.$	 
	
\begin{enumerate}
	
\item The tensor of $V$ and $X$ in $\mC$ is the object $V \ot X \in \mC $ such that there is a morphism
$V \to \Mor_\mC(X, V \ot X) $ in $\mV$ that induces for every $Y \in \mC$ an equivalence
$$ \Mor_\mC(V \ot X,Y) \to \Mor_\mV(V, \Mor_\mC(X,Y)). $$ 

\item The cotensor of $V$ and $X$ in $\mC$ is the object ${^V X} \in \mC $ that is the tensor of $V $ and $X$ in the opposite $\mV^\rev$-enriched category $\mC^\circ.$

\end{enumerate}
	
\end{definition}

\subsubsection{Initial and final objects}

\begin{definition}\label{initi}
Let $\mV$ be a presentably monoidal category and $\mC$ a $\mV$-enriched category.

\begin{enumerate}
\item An object $X $ of $\mC$ is initial if for every object $Y \in \mC$
the morphism object $\Mor_\mC(X,Y)$ is a final object in $\mV.$

\vspace{1mm}

\item An object $X $ of $\mC$ is final if it is initial in $\mC^\circ$ or equivalently if for every object $Y \in \mC$
the morphism object $\Mor_\mC(Y,X)$ is a final object in $\mV.$

\vspace{1mm}

\item An object $X $ of $\mC$ is a zero object if it is initial and final.

\end{enumerate}

\end{definition}

\subsubsection{Enriched adjunctions}

\begin{definition}Let $\mV$ be a presentably monoidal category.
A $\mV$-enriched functor $\mC \to \mD$ admits a left (right) adjoint if
it admits a left (right) adjoint in the 2-category $\mV \mathrm{-}\Cat.$

\end{definition}

\begin{notation}
Let $\mV$ be a presentably monoidal category and $\mC, \mD$ be $\mV$-enriched categories. Let $$ \mV \mathrm{-}\L\Fun(\mC,\mD) \subset \mV \mathrm{-}\Fun(\mC,\mD)$$ be the full subcategory of $\mV$-enriched functors $\mC \to \mD$ that admit a $\mV$-enriched right adjoint.

\end{notation}

The following is \cite[Remark 2.75.]{heine2024bienrichedinftycategories}:

\begin{proposition}

Let $\mV$ be a presentably monoidal category.

\begin{enumerate}

\item A $\mV$-enriched functor $\phi: \mC \to \mD$ admits a $\mV$-enriched right adjoint if and only if for every $\Y \in \mD$ the $\mV^\rev$-enriched functor
$\Mor_\mD(\phi(-),\Y): \mC^\circ \to \mV$ is representable.

\vspace{1mm}

\item A $\mV$-enriched functor $\phi: \mC \to \mD$ admits a $\mV$-enriched left adjoint if and only if the opposite $\mV^\rev$-enriched functor $\phi^\circ: \mC^\circ \to \mD^\circ$ admits a $\mV^\rev$-enriched right adjoint. By (1) this holds if and only if for every $\Y \in \mD$ the $\mV$-enriched functor
$\Mor_\mD(\Y,\phi(-)): \mC \to \mV$ is representable.

\end{enumerate}
\end{proposition}

The following is \cite[Lemma 2.76.]{heine2024bienrichedinftycategories}:

\begin{proposition}\label{adj}\label{adj2}Let $\mV$ be a presentably monoidal category.

\begin{enumerate}
 
\item A $\mV$-enriched functor $\mC \to \mD$ admits a right adjoint if and only if it preserves tensors and the underlying functor admits a right adjoint.

\item A $\mV$-enriched functor $\mC \to \mD$ admits a left adjoint if and only if it preserves cotensors and the underlying functor admits a left adjoint.

\end{enumerate}

\end{proposition}


Next we introduce the enriched category of enriched presheaves:

\begin{notation}Let $\mV, \mW$ be presentably monoidal categories, $\mM$ a $\mV$-enriched category and $\mN$ a $\mW$-enriched category.
Let $\mM \boxtimes \mN$ be the pushforward of the product $\mM \times \mN$
along the universal monoidal functor $\mV \times \mW \to \mV \ot \mW$
preserving small colimits componentwise.
    
\end{notation}

The next is \cite[Proposition 4.11, Theorem 4.86]{heine2024bienrichedinftycategories}:

\begin{theorem}\label{psinho} Let $\mV, \mW$ be presentably monoidal categories and $\mN$ a $\mV \ot \mW$-enriched category. 

\begin{enumerate}[\normalfont(1)]\setlength{\itemsep}{-2pt}
\item Let $\mM$ be a small $\mV$-enriched category. The category 
$\mV \mathrm{-}\Fun(\mM, \mN)$
refines to a $\mW$-enriched category characterized by an 
equivalence
$$ \mW \mathrm{-}\Fun(\mO,\mV \mathrm{-} \Fun(\mM, \mN)) \to {\mV \ot \mW} \mathrm{-}\Fun(\mM \boxtimes \mO,\mN)$$
natural in any $\mW$-enriched category $\mO$.

\item Let $\mO$ be a small $\mW$-enriched category. 
The category $\mW \mathrm{-} \Fun(\mO, \mN)$ refines to a $\mV$-enriched category characterized by an equivalence
$$ \mV \mathrm{-} \Fun(\mM, \mW \mathrm{-}\Fun(\mO, \mN)) \to {\mV \ot \mW} \mathrm{-}\Fun(\mM \boxtimes \mO,\mN) $$
natural in any $\mV$-enriched category $\mM$.

\item If $\mN$ is a category presentably bitensored over $\mV, \mW^\rev$, then $\mV \mathrm{-} \Fun(\mM, \mN) $ is a category presentably right tensored over $\mW$ and $\mW \mathrm{-} \Fun(\mO, \mN)$ is a category presentably left tensored over $\mV$.

\end{enumerate}

\end{theorem}
    
\begin{definition}

Let $\mV$ be a presentably monoidal category and $\mC$ a small $\mV$-enriched category.
The presentably left $\mV$-tensored category of $\mV$-enriched presheaves on $\mC$ is
$$ \mP_\mV(\mC):= \mV^\rev \mathrm{-}\Fun(\mC^\op,\mV). $$
\end{definition}

The next theorem, which follows from \cite[Theorem 3.41, Theorem 4.70]{heine2024bienrichedinftycategories}, is an enriched version of the universal property of the category of presheaves as the free cocompletion under small colimits \cite[Theorem 5.1.5.6]{lurie.HTT}:

\begin{theorem}\label{Yonedaext}

Let $\mV$ be a presentably monoidal category, $\mC$ a small $\mV$-enriched category and $\mD$ a presentably left $\mV$-tensored category.

\begin{enumerate}[\normalfont(1)]\setlength{\itemsep}{-2pt}
\item There is a $\mV$-enriched embedding $\iota_\mC : \mC \to \mP_\mV(\mC)$ that sends $X$ to $\Mor_\mC(-,X)$ and induces
for every $\mV$-enriched category $\mD$ an equivalence
$$ {\mV\mathrm{-}\L\Fun}(\mP_\mV(\mC),\mD)\to {\mV\mathrm{-}\Fun}(\mC,\mD).$$


\item Let $F: \mC \to \mD$ be a $\mV$-enriched functor and
$\bar{F}: \mP_\mV(\mC) \to \mD $ the unique $\mV$-enriched left adjoint extension of $F$.
The $\mV$-enriched right adjoint $\mD \to \mP_\mV(\mC)$ of $\bar{F}$ sends $Y$ to $\Mor_\mD(-,Y) \circ F. $

\end{enumerate}

\end{theorem}

The following is the enriched Yoneda lemma proven in 
\cite[Corollary 4.44]{heine2024bienrichedinftycategories}:

\begin{lemma}\label{enryo}
Let $\mV$ be a presentably monoidal category and $\mC$ a small $\mV$- enriched $\infty$-category. For every object $X \in \mC$ and $F \in \mP_\mV(\mC)$ the induced morphism $$\Mor_{\mP_\mV(\mC)}(\Mor_\mC(-,X),F) \to F(X) $$ is an equivalence.

\end{lemma}

\subsubsection{Weighted colimits}

In the following we define weighted colimits following \cite{heine2024higheralgebraweightedcolimits}.

\begin{definition}
Let $\mV$ be a presentably monoidal category and $\mC$ a small $\mV$-enriched category.

\begin{itemize}
\item 
A $\mV$-enriched weight on $\mC$ is an object of $ \mP_\mV(\mC)$.

\item A small $\mV$-enriched weight is a $\mV$-enriched weight on some small $\mV$-enriched category.

\end{itemize}

\end{definition}

\begin{definition}
Let $\mV$ be a presentably monoidal category, $\mC$ a small $\mV$-enriched category, $\phi: \mC \to \mD$ a $\mV$-enriched functor, $X \in \mD$ and $H \in \mP_\mV(\mC).$

\begin{itemize}

\item A $H$-weighted cocone on $X$ is a morphism in $\mP_\mV(\mC):$
$$ H \to \Mor_\mD(-,X) \circ \phi. $$

\item A $\mV$-enriched cocone on $X$ is a $H$-weighted cocone on $X$ for some $H \in \mP_\mV(\mC).$

\item A $\mV$-enriched cocone in $\mD$ is a $\mV$-enriched cocone on $X$ for some $X \in \mD.$

\end{itemize}

\end{definition}

\begin{definition}
Let $\mV$ be a presentably monoidal category, $\mC$ a small $\mV$-enriched category and $H \in \mP_\mV(\mC).$ 
Let $\phi: \mC \to \mD$ be a $\mV$-enriched functor.
The $H$-weighted colimit of $\phi$ if it exists, is the object 
$\colim^H(\phi) \in \mD$ such that there is a $H$-weighted cocone on $\colim^H(\phi)$ that induces for every $Y \in \mD$ an equivalence
$$ \Mor_{\mD}(\colim^H(\phi),Y) \to \Mor_{\mP_\mV(\mC)}(H, \Mor_\mD(-,Y) \circ \phi).$$

\end{definition}

\begin{definition}
Let $\mV$ be a presentably monoidal category, $\mC$ a small $\mV$-enriched category and $H \in \mP_{\mV^\rev}(\mC^\circ).$ 
Let $\phi: \mC \to \mD $ be a $\mV$-enriched functor.
The $H$-weighted limit of $\phi$ is the $H$-weighted colimit of the opposite $\mV^\rev$-enriched functor $\phi^\circ : \mC^\circ \to \mD^\circ. $

\end{definition}

\begin{definition}
Let $\mV$ be a presentably monoidal category.

\begin{itemize}

\item Let $\mC$ be a small $\mV$-enriched category and $H \in \mP_\mV(\mC).$ A $\mV$-enriched category $\mD$ admits $H$-weighted colimits if it admits the $H$-weighted colimit of every $\mV$-enriched functor $\mC \to \mD.$

\item Let $\mH$ be a set of small $\mV$-enriched weights. 
A $\mV$-enriched category $\mD$ admits $\mH$-weighted colimits if it admits $H$-weighted colimits for every $H \in \mH.$ 
  
\end{itemize}

\end{definition}

We make a similar definition for weighted limits.

\vspace{1mm}

For presentably monoidal categories $\mV, \mW$ let 
$\alpha_\mV: \mV \to \mV \ot \mW, \alpha_\mW: \mW \to \mV \ot \mW$ be the canonical left adjoint functors.
The next is \cite[Theorem 3.102.]{heine2024higheralgebraweightedcolimits}:

\begin{proposition}\label{colimenrfun}
Let $\mV, \mW $ be presentably monoidal categories, $\mC$ a small $\mV$-enriched category, $\mD$ a $\mV \ot \mW$-enriched category
and $\mH$ a collection of $\mW$-enriched weights such that $\mD$ admits
$(\alpha_\mW)_!(\mH)$-weighted (co)limits.
The $\mW$-enriched category ${\mV\mathrm{-}\Fun}(\mC,\mD) $
admits $\mH$-weighted (co)limits and for every $X \in \mC$
the $\mW$-enriched functor ${\mV\mathrm{-}\Fun}(\mC,\mD) \to \alpha^*_\mW(\mD) $ evaluating at $X$ preserves $\mH$-weighted (co)limits.
    
\end{proposition}

\begin{notation}
    
Let $\mV$ be a presentably monoidal category and $\mH$ a set of small $\mV$-enriched weights. Let $\mV \mathrm{-}\Cat(\mH) \subset \mV \mathrm{-}\Cat$ be the subcategory of $\mV$-enriched categories that admit $\mH$-weighted colimits and $\mV$-enriched functors preserving $\mH$-weighted colimits.
    
\end{notation}

The following is \cite[Proposition 4.47.]{heine2024higheralgebraweightedcolimits}:

\begin{proposition}\label{weiadj} Let $\mV$ be a presentably monoidal category and $\mH$ a set of small $\mV$-enriched weights.
The inclusion $\mV \mathrm{-}\Cat(\mH) \subset \mV \mathrm{-}\Cat$ admits a left adjoint.

\end{proposition}

The following is \cite[Corollary 3.68.]{heine2024higheralgebraweightedcolimits}:

\begin{proposition}\label{weiadj} Let $\mV$ be a presentably monoidal category and $\mC$ a small $\mV$-enriched category.
The $\mV$-enriched Yoneda embedding $\mC \to \mP_\mV(\mC)$ preserves weighted limits.

\end{proposition}

\begin{notation}Let $\mV$ be a presentably monoidal category, $\mC$ a small $\mV$-enriched category and $\Lambda$ a set of small $\mV$-enriched cocones in $\mC$. Let $\mP^\Lambda_\mV(\mC) \subset \mP_\mV(\mC)$ be the full $\mV$-enriched subcategory of $\mV$-enriched presheaves sending $\mV$-enriched cocones in $\mC$ of $\Lambda$ to weighted colimits in $\mV^\circ$.
    
\end{notation}

The following is \cite[Lemma 3.80., Remark 3.85.]{heine2024higheralgebraweightedcolimits}:

\begin{proposition}\label{weiloc} Let $\mV$ be a presentably monoidal category, $\mC$ a small $\mV$-enriched category and $\Lambda$ a set of small $\mV$-enriched cocones in $\mC$.

\begin{enumerate}

\item The $\mV$-enriched embedding $\mP^\Lambda_\mV(\mC) \subset \mP_\mV(\mC)$ admits a $\mV$-enriched left adjoint.

\item The restricted $\mV$-enriched localization $\mC \to \mP_\mV(\mC) \to \mP^\Lambda_\mV(\mC)$ sends $\mV$-enriched cocones of $\Lambda$ to weighted colimits.

\item If $\Lambda$ consists of weighted colimit cocones, the 
$\mV$-enriched Yoneda embedding of $\mC$ lands in $\mP^\Lambda_\mV(\mC)$
and so the restricted $\mV$-enriched localization $\mC \to \mP_\mV(\mC) \to \mP^\Lambda_\mV(\mC)$ is a $\mV$-enriched embedding and preserves 
weighted limits.

\end{enumerate}

\end{proposition}

The following is \cite[Proposition 3.19.]{heine2024higheralgebraweightedcolimits}:

\begin{proposition}\label{weifunc}

Let $\mV, \mW$ be presentably monoidal categories, $\mC$ a small $\mV$-enriched category, $H$ a $\mV$-enriched weight on $\mC$ and $\mD$ a $\mV \ot \mW$-enriched category that admits $(\alpha_\mV)_!(H)$-weighted colimits.

There is a $\mW$-enriched functor 
$$ \mV\mathrm{-}\Fun(\mC,\mD) \to \alpha^*_\mW(\mD)$$
that sends a $\mV$-enriched functor to its $H$-weighted colimit.

\end{proposition}

\subsubsection{Presentable enriched categories}

\begin{notation}
Let $\Pr^L \subset \widehat{\Cat}$ be the subcategory of presentable categories and left adjoint functors.

\end{notation}

\begin{remark}
The category $\Pr^L$ carries a canonical closed symmetric monoidal structure such that the inclusion $$\Pr^L \subset \widehat{\Cat}$$ is lax symmetric monoidal, where $\widehat{\Cat}$ carries the cartesian structure \cite[Proposition 4.8.1.15.]{lurie.higheralgebra}.

We write $\ot $ for the tensor product of $\Pr^L.$

\end{remark}

\begin{definition}\label{present} Let $\mV$ be a presentably monoidal category.
A $\mV$-enriched category $\mC$ is presentable if it admits tensors
and the underlying category $\iota(\mC)$ is presentable.
\end{definition}




\begin{notation}Let $\mV$ be a presentably monoidal category.
Let $$\mV \mathrm{-}\Pr^L \subset \mV \mathrm{-}\Cat$$ 
be the subcategory of presentable $\mV$-enriched categories and left adjoint $\mV$-enriched functors.

\end{notation}

The following is \cite[Theorem 1.2.]{Heine2023AnEB}:

\begin{theorem}Let $\mV$ be a presentably monoidal category.
There is a canonical equivalence $${_\mV\Mod(\Pr^L)} \simeq \mV \mathrm{-}\Pr^L,$$ where the left hand side is the 2-category of left $\mV$-modules in $\Pr^L$ and $\mV$-linear left adjoint functors.

\end{theorem}

The following is \cite[Corollary 3.95.]{heine2024bienrichedinftycategories}:

\begin{proposition}\label{adjuiti}Let $\mV, \mW$ be presentably monoidal categories.
There is a canonical equivalence $${_\mV\Mod(\Pr^L)_\mW} \simeq {_{\mV \ot \mW^\rev}\Mod(\Pr^L)}.$$

\end{proposition}

\begin{corollary}

Let $\mV, \mW$ be presentably monoidal categories.
There is a canonical equivalence $${_\mV\Mod(\Pr^L)_\mW} \simeq {\mV \ot \mW^\rev} \mathrm{-}\Pr^L.$$

\end{corollary}

\vspace{1mm}

The following is \cite[Lemma 2.79.]{heine2024bienrichedinftycategories}:

\begin{proposition}\label{adjunc}Let $\mV$ be a presentably monoidal category.
There is a canonical equivalence $$\mV \mathrm{-}\Pr^L \simeq (\mV \mathrm{-}\Pr^R)^\op$$ sending left to right adjoints.

\end{proposition}

\subsubsection{Enriched slice categories}

Next we introduce enriched slice categories.


For the next notation we use that the presentable 2-category $\mV \mathrm{-}\Cat$ admits cotensors.

\begin{notation}\label{slices} Let $\mV$ be a presentably monoidal category
whose tensor unit is final.
Let $\mC$ be a $\mV$-enriched category and $X \in \mC.$

Let $\mC_{/X} $ be the fiber over $X$ of the $\mV$-enriched functor $\mC^{\bD^1} \to \mC^{\{1\}}$ evaluating at $1$
in the presentable 2-category $\mV \mathrm{-}\Cat.$
There is a forgetful $\mV$-enriched functor $\mC_{X/} \to \mC^{\bD^1} \to \mC^{\{0\}}$ evaluating at $0.$	
	
\end{notation}

The following is \cite[Remark 2.4.25.]{gepner2025oriented}:

\begin{remark}\label{bien} Let $\mV$ be a presentably monoidal category
whose tensor unit is final.
Let $\mC$ be a $\mV$-enriched category and $X \in \mC.$
Let $\alpha: X \to Y, \beta: X \to Z$ be morphisms in $\mC.$
There is a canonical equivalence 	
$$ \Mor_{\mC_{X/}}(Y,Z) \simeq \{\beta\} \times_{\Mor_{\mC}(X,Z)} \Mor_{\mC}(Y,Z).$$

\end{remark}

The following is \cite[Proposition 3.7.1.]{gepner2025oriented}:

\begin{lemma}\label{left adjoint slice} Let $\mV$ be a presentably monoidal category whose tensor unit is final.
Let $\mC$ be a $\mV$-enriched category and $X \in \mC.$
Every morphism $X \to Y$ in $\mC$ gives rise to a $\mV$-enriched functor
$ \mC_{Y/}  \to \mC_{X/} $ over $\mC.$
If $\mC$ admits conical pushouts, the latter admits a left adjoint.

\end{lemma}

\begin{lemma}\label{susk}Let $\mV, \mW$ be presentably monoidal categories, $K\in\mW$ and $\mC$ a $\mV \ot \mW$-enriched category that admits cotensors
with $\tu_\mV \ot K.$
Let $$\mC[K] \in \mV\mathrm{-}\Cat$$ be the pullback of evaluation at the target $\mC^{\bD^1}\to \mC$ along the
$\mV$-enriched functor $(-)^K: \mC \to \mC$ taking the cotensor with $\tu_\mV \ot K.$	
There is a $\mV$-enriched equivalence over $\mC \times \mC$
$$\mC[K] \simeq \mW\mathrm{-}\Fun(S(K),\mC)$$	
sending $(Y,X \to Y^K)$ to $(X,Y, K \to \Mor_{\alpha^*_\mW(\mC)}(X,Y)).$	
	
\end{lemma}

\begin{proof}
	
Let $\mB$ be a $\mV$-enriched category.
By the Yoneda lemma it suffices to see that equivalence classes of $\mV$-enriched functors $\mB \to \mC[K]$ naturally correspond to equivalence classes of $\mV$-enriched functors $$\mB \to \mW\mathrm{-}\Fun(S(K),\mC).$$
There are canonical equivalences of categories $$ \mV\mathrm{-}\Fun(\mB, \mC[K]) \simeq \mV\mathrm{-}\Fun(\mB, \mC)[K], $$
$$\mV\mathrm{-}\Fun(\mB, \mW\mathrm{-}\Fun(S(K),\mC)) \simeq \mW\mathrm{-}\Fun(S(K), \mV\mathrm{-}\Fun(\mB,\mC)).$$
So it is enough to show that there is a natural bijection of equivalence classes of objects of
$\mC[K] $ and $ \mW\mathrm{-}\Fun(S(K),\mC).$
An object of $\mC[K]$ is a pair $(Y,X \to Y^K)$ that corresponds to a triple $(X,Y,K \to \Mor_{\alpha^*_\mW(\mC)}(X,Y)),$ which is precisely an object of $\mW\mathrm{-}\Fun(S(K),\mC).$
	
\end{proof}

\subsection{$\infty$-categories}

Next we recall the basic theory of $\infty$-categories following \cite[\S 3]{gepner2025oriented}. 


	


\begin{definition}
For every $\n \geq 0$ we inductively define the presentable cartesian closed category $\n\Cat$ of small (univalent) $\n$-categories by setting:
$$0\Cat :=\infty\Grp,\ (\n+1)\Cat:= \n\Cat\mathrm{-}\Cat.$$
	
\end{definition}

\begin{notation}
For every $\n \geq 0$ we inductively define colocalizations
$\n\Cat \rightleftarrows (\n+1)\Cat: \iota_\n,$
where both adjoints  preserve finite products and filtered colimits.
Let
$$0\Cat= \infty\Grp \rightleftarrows 1\Cat = \infty\Grp\mathrm{-}\Cat : \iota_0 $$ be the unique colocalization whose left adjoint preserves finite products. Let $$(\n+1)\Cat= \n\Cat\mathrm{-}\Cat \rightleftarrows (\n+2)\Cat= (\n+1)\Cat \mathrm{-}\Cat:\iota_{\n+1}:= (\iota_\n)_!. $$

\end{notation}

\begin{definition}The presentable cartesian closed category $\infty\Cat$ of small (univalent) $\infty$-categories is the limit
$$\infty\Cat:= \lim(...\xrightarrow{\iota_{\n}} \n\Cat \xrightarrow{\iota_{\n-1}} ... \xrightarrow{\iota_0} 0 \Cat) $$
of presentable cartesian closed categories and right adjoint functors. 
	
\end{definition}

Since forming category objects preserves limits, we obtain the following:

\begin{remark}\label{fix}
	
There is a canonical equivalence
$ \infty\Cat \simeq \infty\Cat\mathrm{-}\Cat. $
	
\end{remark}

\begin{notation}\label{mormoro}

Let $\Mor: \infty\Cat_{\partial\bD^1/} \to \infty\Cat$
be the functor $$\infty\Cat_{\partial\bD^1/} \simeq \Cat_{\partial\bD^1/} \times_{\Cat} \infty\Cat\mathrm{-}\Cat \to \infty\Cat $$ sending $(\mC,X,Y)$ to $\Mor_\mC(X,Y).$

\end{notation}

\begin{remark}\label{homfil}
The functor $\Mor: \infty\Cat_{\partial\bD^1/} \to \infty\Cat$ preserves small limits and small filtered colimits.

\end{remark}

\begin{notation}
For every $0 \leq \n \leq \m$ the left adjoint embeddings $\n\Cat \leftrightarrows \m\Cat$ preserve small limits and thus induce a left adjoint embedding $\n\Cat \leftrightarrows \infty\Cat: \iota_\n$ that preserves small limits and so admits a left adjoint $$\tau_\n: \infty\Cat \to \n\Cat$$ by presentability.

\end{notation}

By the limit definition of $\infty\Cat$ we have the following filtration \cite[Lemma 2.3.10.]{gepner2025oriented}:

\begin{lemma}\label{decom}
Let $\mC$ be an $\infty$-category.
The sequential diagram $$\iota_0(\mC) \to ... \to \iota_\n(\mC) \to \iota_{\n+1}(\mC) \to ... \to \mC$$ exhibits $\mC$ as the colimit in $\infty\Cat$ of the diagram $\iota_0(\mC) \to ... \to \iota_\n(\mC) \to \iota_{\n+1}(\mC) \to ....$
\end{lemma}

\begin{definition}\label{ruik} 

Let $\n \geq 0.$
We inductively define involutions $(-)^\op_\n, (-)^\co_\n$ of $ \n\Cat$ by setting
$(-)^\co_0, (-)^\op_0$ are the identities and
$$(-)^\co_{\n+1}: {(\n+1)}\Cat \xrightarrow{((-)^\op_\n)_!}{(\n+1)}\Cat, $$$$
(-)^\op_{\n+1}: {(\n+1)}\Cat \xrightarrow{(-)^\circ} {(\n+1)}\Cat \xrightarrow{((-)^\co_\n)_!}{(\n+1)}\Cat. $$

There are commutative squares
$$\begin{xy}
\xymatrix{
{(\n+1)}\Cat \ar[d]^{\iota_\n}  \ar[rr]^{(-)_{\n+1}^\op} && {(\n+1)}\Cat  \ar[d]^{\iota_\n}
\\ 
 \n\Cat \ar[rr]^{(-)_{n}^\op} &&  \n\Cat
}
\end{xy}
\begin{xy}
\xymatrix{
{(\n+1)}\Cat \ar[d]^{\iota_\n}  \ar[rr]^{(-)_{\n+1}^\co} && {(\n+1)}\Cat  \ar[d]^{\iota_\n}
\\ 
 \n\Cat \ar[rr]^{(-)_{n}^\co} &&  \n\Cat
}
\end{xy}
$$
Hence the latter involutions give rise to involutions 
\begin{equation}\label{eqinvo}
(-)^\op, (-)^\co: \infty\Cat \to \infty\Cat
\end{equation}
on the limit.
	
\end{definition}

\begin{notation}
	
Let $\partial\bD^1:=S^0:=*\coprod*$ be the set with two elements.	
	
\end{notation}

\begin{remark}\label{oppo} By definition \ref{ruik} there are commutative squares, where $\sigma$ permutes the two objects:
	
$$\begin{xy}
\xymatrix{
\infty\Cat_{\partial\bD^1/} \ar[d]^{\Mor}  \ar[r]^{(-)^\co} & \infty\Cat_{\partial\bD^1/}  \ar[d]^{\Mor}
\\ 
\infty\Cat \ar[r]^{(-)^\op} & \infty\Cat
}
\end{xy}
\begin{xy}
\xymatrix{
\infty\Cat_{\partial\bD^1/} \ar[d]^{\Mor}  \ar[r]^{\sigma \circ (-)^\op} & \infty\Cat_{\partial\bD^1/}  \ar[d]^{\Mor}
\\ 
\infty\Cat \ar[r]^{(-)^\co} & \infty\Cat.
}
\end{xy}
$$
	
\end{remark}


	






Next we define $\n$-fullness for every $\n \geq -1.$

\begin{definition}Let $\n \geq 1.$
\begin{itemize}
		
\item A functor is $0$-full if it is essentially surjective.
		
\item A functor is $\n$-full if it is essentially surjective and induces $\n$-1-full functors on morphism $\infty$-categories.
		
\end{itemize}
	
\end{definition}

\begin{example}\label{exfull} For every $\n \geq 0$ and $\infty$-category $\mC$ the induced functor
$$\iota_{\n}(\mC) \to \iota_{\n+1}(\mC) $$ is $\n$-full.
This follows by induction on $\n \geq 0 $
since $\Mor_{\iota_{\n+1}(\mC)}(A,B) \simeq \iota_{\n}(\Mor_{\mC}(A,B))$ for any $A,B \in \mC.$
    
\end{example}

\begin{definition}Let $\n \geq 0.$

\begin{itemize}
\item A functor $\F: \mC \to \mD$ is a $0$-equivalence
if it induces a bijection on equivalence classes of objects.

\item A functor $\F: \mC \to \mD$ is an $\n$-equivalence
if it induces an $\n-1$-equivalence on morphism $\infty$-categories. 

\end{itemize}

\end{definition}



\begin{lemma}\label{homt} A functor $\F: \mC \to \mD$ is an equivalence if and only if it is an $\n$-equivalence for every $\n \geq 0.$
\end{lemma}

\begin{proof}
The only if-direction is clear. We prove the other direction.	
We first reduce to the case that $\mC,\mD$ are $\bk$-categories for some $\bk \geq 0.$
Let $\n \geq 0.$ 
If $\F$ is an $\n$-equivalence, the functor $\iota_\bk\F$ is an $\n$-equivalence for every $\bk \geq 0$ as one sees by induction on $k \geq 0.$
So if $\F$ is an $\n$-equivalence for every $\n \geq 0,$ then $\iota_\bk\F$ is an $\n$-equivalence for every $\n,\bk \geq 0.$
Thus $\iota_\bk\F$ is an equivalence for every $\bk \geq 0$ if we have proven the result in the case where $\mC,\mD$ are $\bk$-categories.
So we can assume that $\mC,\mD$ are $\bk$-categories.
We proceed by induction on $\bk \geq 0.$
For $\n=0$ the result follows from the fact that a map of spaces is an equivalence if it induces isomorphisms on all homotopy groups.
We assume the statement for $\bk$ and let $\F:\mC \to \mD$ be a functor of $\bk+1$-categories that is an $\n$-equivalence for every $\n \geq 0.$
Then $\F$ is a $0$-equivalence, i.e. induces an equivalence on the set of equivalence classes.
In particular, $\F$ is essentially surjective. So $\F$ is an equivalence if
$\F$ is fully faithful. For every $X,Y \in \mC$ the induced functor $$\Mor_\mC(X,Y)\to \Mor_\mD(\F(X),\F(Y))$$ of $\bk$-categories is an $\n$-equivalence for every $\n \geq 0.$ So by induction hypothesis the functor $\Mor_\mC(X,Y)\to \Mor_\mD(\F(X),\F(Y))$ is an equivalence. 	

\end{proof}	
 

Next we construct the Gray tensor product.


\begin{notation}

By \cite[Definition 3.5.19.]{gepner2025oriented} for every $\n \geq 0 $ there is an $\n$-category $\cube^\n$,
the oriented $\n$-cube, whose 1-truncation is the $\n$-fold product $(\bD^1)^{\times \n}.$
	
Let $$\cube \subset \infty\Cat $$ be the full subcategory of oriented cubes.
	
\end{notation}

\vspace{1mm}
By \cite[Remark 3.5.18.]{gepner2025oriented} the category $\cube$ carries a monoidal structure $\boxtimes$ whose tensor unit is $\cube^0$ and such that for every $\n,\m \geq 0$ there is an equivalence $\cube^\n \boxtimes \cube^\m \simeq \cube^{\n+\m}.$

\vspace{1mm}

The next is \cite[Theorem 3.14., Example 3.16.]{campion2023gray}
(see also \cite[Corollary 3.6.2.]{gepner2025oriented}):

\begin{theorem}\label{dense}
	
There is a unique presentably monoidal structure on $\infty\Cat$ such that the embedding
$$\cube \subset \infty\Cat$$ is monoidal.

\end{theorem}

We call the monoidal structure on $\infty\Cat$ of Theorem \ref{dense}
the Gray monoidal structure and write $\boxtimes$ for the Gray tensor product.

\begin{notation}

Since the Gray tensor product defines a presentably monoidal structure on $\infty\Cat$, it is biclosed: for every $\infty$-category $\mC$ the functor $$ \mC \boxtimes (-): \infty\Cat \to \infty\Cat$$ admits a right adjoint $\Fun^\lax(\mC,-): \infty\Cat \to \infty\Cat$ and the functor $$ (-) \boxtimes \mC : \infty\Cat \to \infty\Cat$$ admits a right adjoint $ \Fun^\oplax(\mC,-): \infty\Cat \to \infty\Cat.$
\end{notation}

\begin{definition}Let $\F,\G:\mC \to \mD $ be functors of $\infty$-categories.
An (op)lax natural transformation $\F \to \G$ is a morphism in $\Fun^{(\op)\lax}(\mC,\mD)$.

\end{definition}

\begin{remark}\label{lao}

If $\mC$ is an $\n$-category and $\mD$ an $\m$-category for $\n,\m \geq 0$,
then $\mC \boxtimes \mD$ is an $\n+\m$-category.
This holds since $\n\Cat$ is closed under small colimits in $\infty\Cat$ and $\n\Cat$ is generated under small colimits by the oriented $\ell$-cubes for $1 \leq \ell \leq \n$. 
\end{remark}

The following is \cite[Remark 3.12.7.]{gepner2025oriented}:

\begin{remark}\label{grayspace}

The full subcategory $\infty\Grp \subset \infty\Cat$ is closed under the Gray tensor product and the restricted Gray monoidal structure on $\infty\Grp$ is the cartesian structure.
The left adjoint $\tau_0: \infty\Cat \to \infty\Grp$ of the resulting monoidal embedding $\infty\Grp \subset \infty\Cat$ is monoidal.

\end{remark}


The next is \cite[Proposition 3.8.8.]{gepner2025oriented}:

\begin{proposition}\label{dua}
	
The involutions (\ref{eqinvo}) underlie monoidal involutions
$$(-)^\op, (-)^\co: (\infty\Cat, \boxtimes^\rev) \simeq (\infty\Cat, \boxtimes).$$
	
\end{proposition}

\begin{corollary}\label{grayhoms}
Let $\mC,\mD \in \infty\Cat$.
There are canonical equivalences $$ \Fun^\oplax(\mC,\mD)^\op \simeq \Fun^\lax(\mC^\op,\mD^\op), \ \Fun^\oplax(\mC,\mD)^\co \simeq \Fun^\lax(\mC^\co,\mD^\co).$$

\end{corollary}

Next we consider a variant of the Gray tensor product for $\infty$-categories with distinguished object that plays the role of the smash product in homotopy theory.

\begin{definition}
The category $\infty\Cat_*$ of small $\infty$-categories with distinguished object is the full subcategory of $\Fun(\bD^1, \infty\Cat)$ spanned by the functors whose source is final.
	
\end{definition}

\begin{definition}

The Gray smash monoidal structure on $\infty\Cat_{*}$,
denoted by $\wedge$, is the smash monoidal structure of \cite[Lemma 3.69.]{gepner2025oriented}
applied to $(\infty\Cat, \boxtimes).$
For every $X, Y \in \infty\Cat_*$ the Gray smash product $X \wedge Y$ is the cofiber of the canonical functor $X \vee Y \to X \boxtimes Y. $

\end{definition}

\begin{notation}
Let $\mC,\mD \in \infty\Cat_*.$
Let $ \Fun_*^{(\op)\lax}(\mC,\mD) $ be the fiber of the functor $$ \Fun^{(\op)\lax}(\mC,\mD) \to \Fun^{(\op)\lax}(\ast,\mD)\simeq \mD $$ over the base point.

\end{notation}

\cite[Lemma 2.6.11.]{gepner2025oriented} gives the following:

\begin{lemma}
Let $\mC \in \infty\Cat_*$. 
\begin{enumerate}
\item The functor $$ \mC \wedge (-): \infty\Cat_* \to \infty\Cat_*$$ is left adjoint to the functor $\Fun_*^\lax(\mC,-): \infty\Cat_* \to \infty\Cat_*$. 

\vspace{1mm}
\item The functor $$ (-) \wedge \mC : \infty\Cat_* \to \infty\Cat_*$$ is left adjoint to the functor $\Fun_*^\oplax(\mC,-): \infty\Cat_* \to \infty\Cat_*.$

\end{enumerate}

\end{lemma}

By functoriality of the smash monoidal structure, \cref{dua} gives the following:

\begin{corollary}\label{dua2}
	
The involutions $$(-)^\op, (-)^\co: \infty\Cat_* \to \infty\Cat_* $$
induced by the involutions (\ref{eqinvo}) underlie monoidal involutions
$$(-)^\op, (-)^\co: (\infty\Cat_*, \wedge^\rev) \simeq (\infty\Cat_*, \wedge).$$
	
\end{corollary}


\section{The geometry of higher categories}

In this section we define the basic unstable homotopy theory of higher categories needed to develop stable homotopy theory of higher categories
in the next section.
We introduce categorical suspension and its reduced variant, reduced categorical suspension,
to study categorical spheres, the categorical analogue of spheres,
which are deeply connected to the geometry of configuration spaces. We use this connection to show a Freudenthal suspension theorem in the setting of higher categories.

We further introduce oriented categories, a natural framework to define oriented pullbacks,
which will become essential to formulate stability in the next section.


\subsection{Categorical spheres}

In this subsection we define (reduced) categorical suspension and categorical spheres.

\vspace{1mm}

By \cite[Definition 3.4.2.]{gepner2025oriented} there is an $\infty$-category satisfying the following:  

\begin{notation}\label{suspi}

The functor $\Mor: \infty\Cat_{\partial\bD^1/} \to \infty\Cat$
of (\cref{mormoro}) admits a left adjoint $S: \infty\Cat \to \infty\Cat_{\partial\bD^1/}$, the categorical suspension.

Let $\mC$ be an $\infty$-category. The $\infty$-category $S(\mC)$ has two objects $0,1$ and morphism $\infty$-categories: $$\Mor_{S(\mC)}(1,0)\simeq \emptyset, \ \Mor_{S(\mC)}(0,1)\simeq \mC, \ \Mor_{S(\mC)}(0,0)\simeq \Mor_{S(\mC)}(1,1) \simeq *. $$ 

	
\end{notation}

	




\begin{definition}\label{susa}
For every $\infty$-category $\mC$ the categorical antisuspension of $\mC$ is $$\bar{S}(\mC):=S(\mC^\co)^\co.$$
	
\end{definition}

Remark \ref{oppo} implies the following:

\begin{lemma}\label{ahoitos}
	
Let $\mC\in \infty\Cat_*$.
There is a canonical equivalence
$\bar{S}(\mC) \simeq S(\mC^{\co\op}).$
	
\end{lemma}

The next proposition is \cite[Proposition 3.16.3.]{gepner2025oriented}:
\begin{proposition}\label{thas2}

For every $\infty$-category $\mC$ there is a canonical equivalence
$$ S(\mC) \simeq  \mC \boxtimes \bD^1 +_{ \mC \boxtimes \partial\bD^1} \partial\bD^1.$$

\end{proposition}

\begin{notation}
Let $\mC, \mD$ be $\infty$-categories and $X,Y \in \mC, X',Y' \in \mD.$
Let $$ \Fun_{\partial\bD^1/}^\oplax((\mC,X,Y), (\mD,X',Y'))$$ be the fiber over $(X',Y') \in \mD \times \mD$ of the following functor evaluating at $(X,Y) \in \mC \times \mC$:
$$ \Fun^\oplax(\mC,\mD) \to \mD \times \mD.$$ 

\end{notation}




We obtain the following immediate corollary:

\begin{corollary}\label{hom}
	
For every $\infty$-category $\mC$ and $X,Y \in \mC$ there is a canonical 
equivalence
$$ \Mor_\mC(X,Y) \simeq \{X \} \times_\mC \Fun^\oplax(\bD^1,\mC) \times_\mC \{Y\}.$$

Moreover for every $\infty$-category $\mB$ there is a canonical equivalence
$$ \Fun^\oplax(\mB, \Mor_\mC(X,Y)) \simeq \Fun^\oplax_{\partial\bD^1/}(S(\mB),\mC).$$
	
\end{corollary}




\begin{proof}

By Proposition \ref{thas2} for every $\infty$-category $\mB$ there is a canonical equivalence
$$ \Fun^\oplax(\mB, \ast \times_{(\mC \times \mC)} \Fun^\oplax(\bD^1,\mC)) \simeq $$$$
\ast \times_{\Fun^\oplax(\mB,\mC) \times \Fun^\oplax(\mB,\mC)} \Fun^\oplax(\mB,\Fun^\oplax(\bD^1,\mC)) \simeq $$$$
* \times_{(\mC \times \mC)}(\mC \times \mC) \times_{\Fun^\oplax(\mB \boxtimes \partial \bD^1, \mC)} \Fun^\oplax(\mB \boxtimes \bD^1,\mC) \simeq $$$$ * \times_{\Fun_{}^\oplax(\partial \bD^1,\mC)}\Fun^\oplax(\partial \bD^1 +_{\mB \boxtimes \partial\bD^1} \mB \boxtimes \bD^1,\mC)$$
$$ \simeq \Fun_{\partial\bD^1/}^\oplax(\partial \bD^1 +_{\mB \boxtimes \partial\bD^1} \mB \boxtimes \bD^1,\mC) \simeq \Fun_{\partial\bD^1}^\oplax(S(\mB),\mC).$$
Passing to maximal subspaces the result follows from the universal property of $S(\mC)$
of Definition \ref{suspi}.	
\end{proof}

\begin{definition}Let $\n \geq 0$.
The $\n$-disk (or walking $\n$-morphism) is $\bD^\n:= S^{n}(*).$
	
\end{definition}

\begin{definition}Let $\n \geq 0$.
The boundary of the $\n$-disk is $\partial\bD^\n:= S^{n}(\emptyset).$
	
\end{definition}

\begin{remark}
The functor $\emptyset \subset *$ induces inclusions $\partial\bD^\n \subset \bD^\n$ for every $\n \geq 0.$
	
\end{remark}

\begin{example}
Then $\partial\bD^0=\ast, \partial\bD^1=S(\emptyset)=*\coprod*$ is the set with two elements.	
\end{example}
	
\begin{definition}
Let $\infty\Cat^\fin \subset \infty\Cat$ be the full subcategory generated by the disks under finite colimits. We call objects of $\infty\Cat^\fin$ finite $\infty$-categories.
	
\end{definition}	

	
	

	








\begin{notation}Let $\mC\in \infty\Cat_*$.
	
\begin{enumerate}
\item The reduced categorical suspension of $\mC$ is
$$\Sigma(\mC):= S(\mC) +_{\bD^1} \bD^0.$$ 	
		
\item The reduced categorical antisuspension of $\mC$ is
$$\bar{\Sigma}(\mC):= \bar{S}(\mC) +_{\bD^1} \bD^0.$$ 		
\end{enumerate}
	
\end{notation}

We often call the reduced categorical (anti)suspension simply (anti)suspension.

\begin{remark}\label{tttj}
	
Let $\mC\in \infty\Cat_*$. By Definition \ref{susa} there is a canonical equivalence
$$\bar{\Sigma}(\mC) \simeq \Sigma(\mC^\co)^\co.$$ 			
	
\end{remark}

\cref{ahoitos} implies the following:

\begin{corollary}\label{ahoit}
	
Let $\mC\in \infty\Cat_*$.
There is a canonical equivalence
$$\bar{\Sigma}(\mC) \simeq \Sigma(\mC^{\co\op}).$$
	
\end{corollary}



\begin{definition}
Let $\Omega: \infty\Cat_{*} \to \infty\Cat$
be the composition $$ \infty\Cat_{*} \to \infty\Cat_{\partial\bD^1/} \xrightarrow{\Mor}\infty\Cat, \ (\mC,X) \mapsto \Mor_\mC(X,X), $$
where the first functor is induced by the functor $\partial\bD^1 \to *.$

The functor $\Omega$ preserves the final object and so lifts to a functor $\Omega: \infty\Cat_{*} \to \infty\Cat_*.$

\end{definition}

\begin{remark}

There is an adjunction $$ \Sigma : \infty\Cat_* \rightleftarrows \infty\Cat_*: \Omega.$$

\end{remark}

\cref{homfil} implies the following:

\begin{remark}\label{endfil}
The functor $\Omega: \infty\Cat_{*} \to \infty\Cat_*$
preserves small filtered colimits.

\end{remark}

\begin{definition}\label{conn} Let $\n \geq 0$.
An $\infty$-category $\mC$ is $\n$-connected if the underlying space of $\mC$ is $\n$-connected and all morphism $\infty$-categories are $\n$-1-connected.
Every $\infty$-category is -1-connected. 

\end{definition}

\begin{remark}
Let $\n \geq 0$. By induction an $\infty$-category is $\n$-connected if it is $\n$+1-connected.

\end{remark}

\cite[Theorem 6.3.2.]{GEPNER2015575} implies the following proposition:

\begin{proposition}\label{hhnl}
The functor $\Omega: \infty\Cat_* \to \infty\Cat_*$ lifts to a
functor $$\Omega: \infty\Cat_* \to \Mon(\infty\Cat),$$ 
which admits a fully faithful left adjoint $$B:\Mon(\infty\Cat) \to \infty\Cat_* $$ that preserves finite products.
The essential image of $B$ precisely consists of the connected $\infty$-categories with distinguished object.

\end{proposition}

\begin{notation}
For every $\n \geq 0$ let $ \Mon_{\bE_\n}(\infty\Cat)$
be the category of $\bE_\n$-monoids \cite[Definition 5.1.0.4.]{lurie.higheralgebra} in $\infty\Cat$.

\end{notation}

\begin{remark}

By \cite[Theorem 5.1.2.2]{lurie.higheralgebra} for every $\n \geq 1$ there is a canonical equivalence $$ \Mon_{\bE_\n}(\infty\Cat) \simeq \Mon_{\bE_{\n-1}}(\Mon(\infty\Cat)).$$	

\end{remark}

\begin{corollary}\label{deloop} Let $\n \geq 1.$ There is an adjunction
$$B^\n: \Mon_{\bE_\n}(\infty\Cat) \rightleftarrows \infty\Cat_*: \Omega^\n$$
whose right adjoint lifts $\Omega^\n: \infty\Cat_* \to \infty\Cat_*$ and whose left adjoint preserves finite products and induces an equivalence to the full subcategory of $\n-1$-connected $\infty$-categories with distinguished object.

\end{corollary}

\begin{proof}We proceed by induction on $\n \geq 1.$
The case $\n=1$ is \cref{hhnl}.
The statement for $n$ gives rise to an adjunction
$$B^\n: \Mon_{\bE_{\n+1}}(\infty\Cat)\simeq \Mon(\Mon_{\bE_\n}(\infty\Cat)) \rightleftarrows \Mon(\infty\Cat_*) \simeq \Mon(\infty\Cat): \Omega^\n$$
whose right adjoint lifts $\Omega^\n: \infty\Cat_* \to \infty\Cat_*$.
The left adjoint of the composed adjunction
$$B^{\n+1}: \Mon_{\bE_{\n+1}}(\infty\Cat)\simeq \Mon(\Mon_{\bE_\n}(\infty\Cat)) \rightleftarrows \Mon(\infty\Cat) \rightleftarrows \infty\Cat_*: \Omega^{\n+1}$$
preserves finite products and induces an equivalence to the full subcategory of $\n$-connected $\infty$-categories with distinguished object.

\end{proof}

\begin{corollary}\label{connel} Let $\n \geq 0.$
The full subcategory of $\infty\Cat_*$ of $\n$-1-connected $\infty$-categories with distinguished object is generated under small colimits by the essential image of $\Sigma^\n: \infty\Cat_* \to \infty\Cat_*.$
	
\end{corollary}

\begin{proof}
The category $ \Mon_{\bE_\n}(\infty\Cat)$ is generated under small colimits by the free functor
$\infty\Cat_* \to  \Mon_{\bE_\n}(\infty\Cat)$. So the result follows from \cref{deloop} and that the functor $\Sigma^\n: \infty\Cat_* \to \infty\Cat_*$ factors as the free functor followed by $B^\n.$	
	
\end{proof}


Next we define categorical spheres.

\begin{definition} 
Let $\n \geq 0$. The categorical $\n$-sphere $S^\n$ is the cofiber of the functor $\partial\bD^\n \subset \bD^\n$ in $\infty\Cat$.
\end{definition}

\begin{remark}By definition the categorical $\n$-sphere is an $\n$-category with distinguished object.

\end{remark}

\begin{notation}Let $\n \geq 0$. Let $$ \Free_{\bE_\n}: \infty\Cat \to \Mon_{\bE_\n}(\infty\Cat)$$ be the free functor.
 
\end{notation}

\begin{proposition}\label{freesp} Let $\n \geq 0$. There is an equivalence of $\n$-categories with distinguished object $$S^\n \simeq B^\n \Free_{\bE_\n}(*).$$ In particular, there is an equivalence $S^1 \simeq B\bN$ of categories with distinguished object. 

\end{proposition}

\begin{proof}

(1): Let $(\mC,X) \in \infty\Cat_*.$ There is a canonical equivalence $$
\zeta: \Map_{\infty\Cat_*}(S^\n,(\mC,X)) \simeq $$$$ \{X \}\times_{\iota_0(\mC)} \Map_{\infty\Cat}(S^\n,\mC) \simeq $$$$\{X\} \times_{\iota_0(\mC)} (\iota_0(\mC) \times_{ \Map_{\infty\Cat}(\partial\bD^\n,\mC)} \Map_{\infty\Cat}(\bD^\n,\mC)) \simeq$$$$
\{X\} \times_{ \Map_{\infty\Cat}(\partial\bD^\n,\mC)}\Map_{\infty\Cat}(\bD^\n,\mC)).$$

There is a canonical equivalence
$$\{X\} \times_{ \Map_{\infty\Cat}(\partial\bD^\n,\mC)}\Map_{\infty\Cat}(\bD^\n,\mC))\simeq $$$$ \{X\} \times_{ \Map_{\infty\Cat_{\partial\bD^1/}}(\partial\bD^\n,(\mC,X,X))} \Map_{\infty\Cat_{\partial\bD^1/}}(\bD^\n,(\mC,X,X)) \simeq $$
$$ \{X\} \times_{ \Map_{\infty\Cat}(\partial\bD^{\n-1},\Mor_\mC(X,X))} \Map_{\infty\Cat}(\bD^{\n-1}, \Mor_\mC(X,X)). $$

Equivalence $\zeta$ gives rise to a canonical equivalence
$$ \{X\} \times_{ \Map_{\infty\Cat}(\partial\bD^{\n-1},\Mor_\mC(X,X))} \Map_{\infty\Cat}(\bD^{\n-1}, \Mor_\mC(X,X)) \simeq$$$$ \Map_{\infty\Cat_*}(S^{\n-1},( \Mor_\mC(X,X),\id_X)). $$

We obtain an equivalence $$ \Map_{\infty\Cat_*}(S^\n,(\mC,X)) \simeq \Map_{\infty\Cat_*}(S^{\n-1},( \Mor_\mC(X,X),\id_X)). $$

There is a canonical equivalence $$ \Map_{\infty\Cat_*}(B^\n \Free_{\bE_\n}(*),(\mC,X)) \simeq
\Map_{\Mon_{\bE_\n}(\infty\Cat)}(\Free_{\bE_\n}(*),\Omega^\n(\mC,X)) \simeq \iota_0(\Omega^\n(\mC,X)).$$
This proves the case $\n=1.$ The latter equivalence specializes to the following one:
$$\Map_{\infty\Cat_*}(B^{\n-1} \Free_{\bE_{\n-1}}(*),(\Mor_\mC(X,X),\id_X)) \simeq \iota_0(\Omega^{\n-1}(\Mor_\mC(X,X),\id_X)) \simeq \iota_0(\Omega^\n(\mC,X)).$$
Hence there is the following equivalence and the statement follows by induction on $\n \geq 1: $ $$\Map_{\infty\Cat_*}(B^\n \Free_{\bE_\n}(*),(\mC,X)) \simeq
\Map_{\infty\Cat_*}(B^{\n-1} \Free_{\bE_{\n-1}}(*),(\Mor_\mC(X,X),\id_X)).$$

\end{proof}

\begin{corollary}\label{autoeq} Let $\n \geq \m \geq 0.$

\begin{enumerate}
\item There is a canonical isomorphism of monoids $$\pi_0(\Map_{\infty\Cat_*}(S^\n,S^\n)) \cong \bN.$$

\item There is a canonical isomorphism $$\pi_0(\Map_{\infty\Cat_*}(S^\n,S^\m)) \cong *.$$

\item There is a canonical isomorphism $$\pi_0(\Map_{\infty\Cat_*}(S^\m,S^\n)) \cong \pi_{\n-\m}(\Free_{\bE_\m}(*)).$$

\item There is a canonical isomorphism of monoids $$\pi_0(\Map_{\iota_0\infty\Cat_*}(S^\n,S^\n)) \cong \bN^\times \cong 0.$$
\end{enumerate}	

\end{corollary}

\begin{proof}

(1): By \cref{freesp} there is a canonical equivalence of $\bA_\infty$-spaces $$\Map_{\infty\Cat_*}(S^\n,S^\n) \simeq \Map_{\Mon_{\bE_\n}(\infty\Cat)}(\Free_{\bE_\n}(*),\Free_{\bE_\n}(*)) \simeq \Map_{\infty\Cat}(*,\Free_{\bE_\n}(*)) \simeq \Free_{\bE_\n}(*).$$
This equivalence induces an isomorphism of monoids $$\pi_0(\Map_{\infty\Cat_*}(S^\n,S^\n)) \cong \pi_0(\Free_{\bE_\n}(*)).$$
By the universal property there is a canonical isomorphism of monoids
$\pi_0(\Free_{\bE_\n}(*))\cong \bN.$

(2): By \cref{freesp} there is a canonical equivalence of spaces  $$\Map_{\infty\Cat_*}(S^\m,S^\n) \simeq \Map_{\Mon_{\bE_\m}(\infty\Cat)}(\Free_{\bE_\m}(*),B^{\n-\m}\Free_{\bE_\n}(*)) \simeq \Map_{\infty\Cat}(*,B^{\n-\m}\Free_{\bE_\n}(*)).$$
This equivalence induces an isomorphism $$ \pi_0(\Map_{\infty\Cat_*}(S^\m,S^\n)) \cong \pi_0(\Map_{\infty\Cat}(*,B^{\n-\m}\Free_{\bE_\n}(*))) \cong *.$$

(3): By \cref{freesp} there is a canonical equivalence of spaces  $$\Map_{\infty\Cat_*}(S^\n,S^\m) \simeq \Map_{\Mon_{\bE_\n}(\infty\Cat)}(\Free_{\bE_\n}(*),\Omega^{\n-\m}\Free_{\bE_\m}(*)) \simeq \Map_{\infty\Cat}(*,\Omega^{\n-\m}\Free_{\bE_\m}(*)).$$
This equivalence induces an isomorphism $$\pi_0(\Map_{\infty\Cat_*}(S^\n,S^\m)) \cong \pi_0(\Map_{\infty\Cat}(*,\Omega^{\n-\m}\Free_{\bE_\m}(*))) \cong \pi_{\n-\m}(\Free_{\bE_\m}(*)).$$
(4) follows immediately from (1).
\end{proof}

\begin{lemma}\label{smarem}\label{uuupo}
Let $\mC\in \infty\Cat_*$.
There are canonical equivalences:
$$\Sigma(\mC) \simeq \mC \wedge S^1,  \ \bar{\Sigma}(\mC) \simeq S^1\wedge \mC. $$ 		

\end{lemma}

\begin{proof}

The second equivalence follows from the first equivalence via \cref{tttj} and \cref{dua}.
By \cref{thas2} and the pasting law all squares in the next diagram are pushouts:
$$\begin{xy}
\xymatrix{
\mC \boxtimes \partial\bD^1 \ar[d] \ar[r] & \partial\bD^1 \ar[r] \ar[d]
& \ast \ar[d]
\\ 	
\mC \boxtimes \partial\bD^1+_{\partial\bD^1} \bD^1 \ar[d] \ar[r] & \partial\bD^1+_{\partial\bD^1} \bD^1=\bD^1 \ar[r] \ar[d]
& S^1 \ar[d] \ar[r] & \ast \ar[d]
\\ 
\mC \boxtimes \bD^1 \ar[r] & S(\mC) \ar[r] & S(\mC)/\partial\bD^1 \ar[r] & \mC \wedge S^1.
}
\end{xy}$$

\end{proof}

\begin{corollary}\label{freespar} Let $\n \geq 0$.	
There is a canonical equivalence of $\n$-categories with distinguished object $$(-)_+\wedge (S^1)^{\wedge n} \simeq B^\n \circ \Free_{\bE_\n}.$$

\end{corollary}

\begin{proof}
The functor $$\infty\Cat_* \xrightarrow{\Omega^\n} \Mon_{\bE_\n}(\infty\Cat) \to \infty\Cat$$
factors as $$\infty\Cat_* \xrightarrow{\Omega^\n} \infty\Cat_* \to \infty\Cat.$$
Hence by uniqueness of left adjoints the functor $\Sigma^\n \circ (-)_+$ factors as
$B^\n \circ \Free_{\bE_\n}.$
By \cref{uuupo} there is a canonical equivalence $$\Sigma^\n \simeq (-) \wedge (S^1)^{\wedge n}.$$

\end{proof}

\begin{corollary}\label{freespart} Let $\n \geq 0$.	
There is a canonical equivalence of $\n$-categories with distinguished object $$(S^1)^{\wedge n} \simeq B^\n \Free_{\bE_\n}(*).$$

\end{corollary}

\cref{freesp} implies the following:

\begin{corollary}Let $\n \geq 0$. There is an  equivalence of $\n$-categories $(S^1)^{\wedge \n} \simeq S^\n$.

\end{corollary}	

\begin{corollary}\label{conesto} Let $\n,\m \geq 0$ and $X \in \infty\Cat_*$ be $\n$-connected 
and $Y \in \infty\Cat_*$ be $\m$-connected.  Then $X \wedge Y$ is $\n+\m+1$-connected.
	
\end{corollary}

\begin{proof}
By 	\cref{connel} the full subcategory $\infty\Cat_*$ of $\n$-connected $\infty$-categories with distinguished object is generated under small colimits by the essential image of $\Sigma^{\n+1}.$
So we need to see that for any $\n,\m \geq 0$ and $X,Y \in \infty\Cat_*$ the $\infty$-category $\Sigma^{\n+1}(X) \wedge \Sigma^{\m+1}(Y) $ is in the essential image of the functor
$\Sigma^{\n+\m+2}.$
By \cref{smarem}, \cref{ahoit}, \cref{dua} there are canonical equivalences
$$ X \wedge \Sigma(Y) \simeq \Sigma(X \wedge Y)$$ and $$ \Sigma(X) \wedge Y \simeq \bar{\Sigma}(X^{\co\op}) \wedge Y \simeq \bar{\Sigma}(X^{\co\op} \wedge Y) \simeq \Sigma((X^{\co\op} \wedge Y)^{\co\op}) \simeq  \Sigma(X \wedge Y^{\co\op}).$$
	
\end{proof}

\subsection{Oriented categories}

In this section we consider categories enriched in the Gray tensor product, which are also studied in \cite{gepner2025oriented}. 









\begin{definition}\emph{}
\begin{enumerate}
\item An oriented category is a category enriched in $(\infty\Cat,\boxtimes^\rev).$

\item An oriented functor is a functor enriched in $(\infty\Cat,\boxtimes^\rev).$

\item An antioriented category is a category enriched in $(\infty\Cat,\boxtimes).$

\item An antioriented functor is a functor enriched in $(\infty\Cat,\boxtimes).$

\item A bioriented category is a category enriched in $(\infty\Cat,\boxtimes) \ot (\infty\Cat,\boxtimes^\rev).$

\item A bioriented functor is a functor enriched in $(\infty\Cat,\boxtimes) \ot (\infty\Cat,\boxtimes^\rev).$

\end{enumerate}

\end{definition}

\begin{example}

The Gray monoidal structure on $\infty\Cat$ is biclosed and so exhibits $\infty\Cat$ as enriched in $(\infty\Cat,\boxtimes) \ot (\infty\Cat,\boxtimes^\rev).$ This way we see
$\infty\Cat$ as a large bioriented category, which we denote by the same name.

\end{example}

\begin{notation}\emph{}

\begin{itemize}

\item Let $$\Cat\boxtimes:= (\infty\Cat,\boxtimes^\rev)\mathrm{-}\Cat $$
be the 2-category of oriented categories.

\item Let $$\boxtimes\Cat:= (\infty\Cat,\boxtimes)\mathrm{-}\Cat$$
be the 2-category of antioriented categories.

\item Let $$\boxtimes\Cat\boxtimes:= (\infty\Cat,\boxtimes) \ot (\infty\Cat,\boxtimes^\rev) \mathrm{-}\Cat$$
be the 2-category of bioriented categories.

\end{itemize}
\end{notation}

\begin{remark}
By \cref{indubi} the unique left adjoint monoidal embedding
$(\infty\Grp, \times) \to (\infty\Cat, \boxtimes), $ which preserves small limits, induces left adjoint monoidal embeddings
$$ (\infty\Cat, \boxtimes) \ot (\infty\Grp, \times) \to (\infty\Cat, \boxtimes) \ot (\infty\Cat, \boxtimes^\rev),$$$$ (\infty\Grp, \times) \ot (\infty\Cat, \boxtimes^\rev) \to (\infty\Cat, \boxtimes) \ot (\infty\Cat, \boxtimes^\rev),$$
which preserve small limits.
The latter induce left adjoint embeddings of 2-categories
$$ \Cat\boxtimes \subset {\boxtimes\Cat\boxtimes}, \boxtimes\Cat \subset {\boxtimes\Cat\boxtimes}, $$
which preserve small limits.

\end{remark}

\begin{notation}\emph{}

\begin{itemize}
\item Let $\mC$ be an oriented category and $X,Y \in \mC$.
We write $\R\Mor_\mC(X,Y)$ for the $\infty$-category of morphisms $X \to Y$ in $\mC$.

\item Let $\mC$ be an antioriented category and $X,Y \in \mC$.
We write $\L\Mor_\mC(X,Y)$ for the $\infty$-category of morphisms $X \to Y$ in $\mC$.

\item Let $\mC$ be a bioriented category and $X,Y \in \mC$.
We write $\L\Mor_\mC(X,Y)$ for the $\infty$-category of morphisms $X \to Y$ in the underlying antioriented category of $\mC$ and 
$\R\Mor_\mC(X,Y)$ for the $\infty$-category of morphisms $X \to Y$ in the underlying oriented category of $\mC$.
    
\end{itemize}
    
\end{notation}

We refer to adjunctions of oriented, antioriented, bioriented categories as oriented, antioriented, bioriented adjunctions.

\begin{notation}\emph{}
\begin{itemize}

\item Let $\mC,\mD \in \Cat\boxtimes$. Let $${\Fun\boxtimes}(\mC,\mD)$$ be the category of oriented functors $\mC \to \mD.$

\item Let $\mC,\mD \in \boxtimes\Cat$. Let $${\boxtimes\Fun}(\mC,\mD)$$ be the category of antioriented functors $\mC \to \mD.$	

\item Let $\mC,\mD \in \boxtimes\Cat\boxtimes $. Let $${\boxtimes\Fun\boxtimes}(\mC,\mD)$$ be the category of bioriented functors $\mC \to \mD$.	

\item Let $\mC,\mD \in \Cat\boxtimes$. Let $${\L\Fun\boxtimes}(\mC,\mD)$$ be the category of oriented functors $\mC \to \mD$ that admit an oriented right adjoint.

\item Let $\mC,\mD \in \boxtimes\Cat$. Let $${\boxtimes\L\Fun}(\mC,\mD)$$ be the category of antioriented functors $\mC \to \mD$ that admit an antioriented right adjoint.	

\item Let $\mC,\mD \in \boxtimes\Cat\boxtimes $. Let $${\boxtimes\L\Fun\boxtimes}(\mC,\mD)$$ be the category of bioriented functors $\mC \to \mD$ that admit a bioriented right adjoint.		

\end{itemize}
\end{notation}

We are mainly interested in the reduced version:

\begin{definition}\emph{}
\begin{enumerate}

\item A weakly reduced oriented category is a category enriched in $(\infty\Cat_*,\wedge^\rev).$

\item A weakly reduced oriented functor is a functor enriched in $(\infty\Cat_*,\wedge^\rev).$

\item A weakly reduced antioriented category is a category enriched in $(\infty\Cat_*,\wedge).$

\item A weakly reduced antioriented functor is a functor enriched in $(\infty\Cat_*,\wedge).$

\item A weakly reduced bioriented category is a category enriched in $(\infty\Cat_*,\wedge) \ot (\infty\Cat_*,\wedge^\rev).$

\item A weakly reduced bioriented functor is a functor enriched in $(\infty\Cat_*,\wedge) \ot (\infty\Cat_*,\wedge^\rev).$

\end{enumerate}

\end{definition}










		 

		
			

		
		
	

\begin{example}

The Gray smash monoidal structure on $\infty\Cat_*$ is biclosed and so exhibits $\infty\Cat_*$ as enriched in $(\infty\Cat_*,\wedge) \ot (\infty\Cat_*,\wedge^\rev).$ This way we see
$\infty\Cat_*$ as a large weakly reduced bioriented category, which we denote by the same name.

\end{example}	

\begin{notation}\emph{}

\begin{itemize}

\item Let $$\Cat\wedge:= (\infty\Cat,\wedge^\rev)\mathrm{-}\Cat$$
be the 2-category of weakly reduced oriented categories.	

\item Let $$\wedge\Cat:= (\infty\Cat,\wedge)\mathrm{-}\Cat$$
be the 2-category of weakly reduced antioriented categories.

\item Let $$\wedge\Cat\wedge:= (\infty\Cat,\smash) \ot (\infty\Cat,\wedge^\rev) \mathrm{-}\Cat$$
be the 2-category of weakly reduced bioriented categories.

\end{itemize}
\end{notation}

\begin{remark}\label{ember}

By \cref{indubi} the unique left adjoint monoidal embedding
$(\infty\Grp, \times) \to (\infty\Cat, \wedge), $ which preserves small limits, induces left adjoint monoidal embeddings
$$ (\infty\Cat, \wedge) \ot (\infty\Grp, \times) \to (\infty\Cat, \wedge) \ot (\infty\Cat, \wedge^\rev),$$$$ (\infty\Grp, \times) \ot (\infty\Cat, \wedge^\rev) \to (\infty\Cat, \wedge) \ot (\infty\Cat, \wedge^\rev),$$
which preserve small limits.
The latter induce left adjoint embeddings of 2-categories
$$ \Cat\wedge \subset {\wedge\Cat\wedge}, \wedge\Cat \subset {\wedge\Cat\wedge}, $$
which preserve small limits.

\end{remark}

We refer to adjunctions of weakly reduced oriented, antioriented, bioriented categories as weakly reduced oriented, antioriented, bioriented adjunctions.

\begin{notation} \emph{}
\begin{itemize}

\item Let $\mC,\mD \in \Cat\wedge$. Let $${\Fun\wedge}(\mC,\mD) $$ be the category of weakly reduced oriented functors $\mC \to \mD.$	

\item Let $\mC,\mD \in \wedge\Cat$. Let $${\wedge\Fun}(\mC,\mD)$$ be the category of weakly reduced antioriented functors $\mC \to \mD.$	

\item Let $\mC,\mD \in \wedge\Cat\wedge $. Let $${\wedge\Fun\wedge}(\mC,\mD)$$ be the category of weakly reduced bioriented functors $\mC \to \mD.$	

\item Let $\mC,\mD \in \Cat\wedge$. Let $${\L\Fun\wedge}(\mC,\mD) $$ be the category of weakly reduced oriented functors $\mC \to \mD$ that admit a weakly reduced oriented right adjoint.		

\item Let $\mC,\mD \in \wedge\Cat$. Let $${\wedge\L\Fun}(\mC,\mD)$$ be the category of weakly reduced antioriented functors $\mC \to \mD$
that admit a weakly reduced antioriented right adjoint.	

\item Let $\mC,\mD \in \wedge\Cat\wedge $. Let $${\wedge\L\Fun\wedge}(\mC,\mD)$$ be the category of weakly reduced bioriented functors $\mC \to \mD$ that admit a weakly reduced bioriented right adjoint.	

\end{itemize}
\end{notation}




\cref{psinho} specializes to the following applied to enrichment in the Gray smash product:

\begin{proposition}\label{funor} Let $\mD$ be a reduced bioriented category. 

\begin{enumerate}

\item Let $\mC$ be a weakly reduced oriented category.
The category ${\Fun\wedge}(\mC,\mD)$ of weakly reduced oriented functors is a weakly reduced antioriented category.

\item Let $\mC$ be a weakly reduced antioriented category.
The category ${\wedge\Fun}(\mC,\mD)$ of weakly reduced antioriented functors is a weakly reduced oriented category.

\end{enumerate}

\end{proposition}

\begin{notation}We fix the following notation:

\begin{itemize}
\item Let $\mC$ be an antioriented category, $X \in \mC$ and $K \in \infty\Cat.$
We refer to tensors of $K$ and $X$ by left tensors and write $K \ot X$
for the left tensor of $K$ and $X$.
We refer to cotensors of $K$ and $X$ by left cotensors and write ${^K X}$
for the left cotensor of $K$ and $X$.

\item Let $\mC$ be an oriented category, $X \in \mC$ and $K \in \infty\Cat.$
We refer to tensors of $K$ and $X$ by right tensors and write $X \ot K$
for the right tensor of $K$ and $X$.
We refer to cotensors of $K$ and $X$ by right cotensors and write ${X^K}$
for the right cotensor of $K$ and $X$.

\item Let $\mC$ be a bioriented category, $X \in \mC$ and $K \in \infty\Cat.$
We refer to tensors of $K \ot * $ and $X$ by left tensors and write $K \ot X$
for the left tensor of $K$ and $X$.
We refer to tensors of $* \ot K$ and $X$ by right tensors and write $X \ot K$
for the right tensor of $K$ and $X$.
We refer to cotensors of $K \ot *$ and $X$ by left cotensors and write ${^K X}$
for the left cotensor of $K$ and $X$.
We refer to cotensors of $* \ot K$ and $X$ by right cotensors and write ${X^K}$
for the right cotensor of $K$ and $X$.


\item Let $\mC$ be a weakly reduced antioriented category, $X \in \mC$ and $K \in \infty\Cat_*.$
We refer to tensors of $K$ and $X$ by left tensors and write $K \wedge X$
for the left tensor of $K$ and $X$.
We refer to cotensors of $K$ and $X$ by left cotensors and write ${^K X_*}$
for the left cotensor of $K$ and $X$.

\item Let $\mC$ be a weakly reduced oriented category, $X \in \mC$ and $K \in \infty\Cat_*.$
We refer to tensors of $K$ and $X$ by right tensors and write $X \wedge K$
for the right tensor of $K$ and $X$.
We refer to cotensors of $K$ and $X$ by right cotensors and write ${X^K_*}$ for the right cotensor of $K$ and $X$.

\item Let $\mC$ be a weakly reduced bioriented category, $X \in \mC$ and $K \in \infty\Cat_*.$
We refer to tensors of $K \ot * $ and $X$ by left tensors and write $K \wedge X$
for the left tensor of $K$ and $X$.
We refer to tensors of $* \ot K$ and $X$ by right tensors and write $X \wedge K$
for the right tensor of $K$ and $X$.
We refer to cotensors of $K \ot * $ and $X$ by left cotensors and write ${^K X_*}$
for the left cotensor of $K$ and $X$.
We refer to cotensors of $* \ot K$ and $X$ by right cotensors and write ${X^K_*}$ for the right cotensor of $K$ and $X$.




\end{itemize}
	
\end{notation}

An initial object, final object, zero object of a (weakly reduced) oriented, antioriented, bioriented category is an initial object, final object, zero object in the sense of \cref{initi}.

\begin{definition}
An oriented, antioriented, bioriented category is reduced if it admits a zero object.
An oriented, antioriented, bioriented functor is reduced if it preserves the zero object.
\end{definition}

\vspace{1mm}

The next is \cite[Proposition 3.8.]{gepner2025oriented}:

\begin{proposition}\label{reduu}
The forgetful functors $${\wedge\Cat} \to \boxtimes\Cat, \ {\Cat\wedge} \to \Cat\boxtimes, \ {\wedge\Cat \wedge} \to {\boxtimes\Cat \boxtimes}$$
restrict to equivalences between the full subcategories of weakly reduced oriented, antioriented, bioriented categories 
that admit an initial or final object and the subcategories of reduced oriented, antioriented, bioriented categories and reduced oriented, antioriented, bioriented functors, respectively.

\end{proposition}

In view of \cref{reduu} we identify reduced oriented, antioriented, bioriented categories with weakly reduced 
oriented, antioriented, bioriented categories that admit an initial or final object. 
Similarly, we identify reduced oriented, antioriented, bioriented functors between reduced oriented, antioriented, bioriented categories with weakly reduced oriented, antioriented, bioriented functors.

\begin{example}\label{Grayspaces} By \cref{grayspace} there is a monoidal localization
$\tau_0: \infty\Cat \rightleftarrows \infty\Grp$ that by \cref{adj} induces a bioriented localization
$$  \tau_0: \infty\Cat \rightleftarrows \tau_0^*(\infty\Grp).$$

\end{example}

\begin{definition}\emph{}
A weakly (reduced) oriented, antioriented, bioriented category is presentable if it is presentable in the sense of \cref{present}.

\end{definition}

\begin{notation}\label{presol}
Let $${\wedge\Pr^L}, \ {\wedge\Pr^R} \subset \wedge\widehat{\Cat},$$$$ {\Pr^L\wedge}, \ {\Pr^R\wedge} \subset \widehat{\Cat}\wedge, $$$$ {\wedge\Pr^L\wedge}, \ {\wedge\Pr^R\wedge} \subset \wedge\widehat{\Cat}\wedge$$ be the respective subcategories of presentable reduced antioriented, oriented, bioriented categories and left (right) adjoint reduced antioriented, oriented, bioriented functors.

\end{notation}

\cref{adjunc} gives the following:

\begin{proposition}

There are canonical equivalences sending left to right adjoints: $$	(\wedge\Pr^L)^\op \simeq \wedge\Pr^R, $$$$ (\Pr^L\wedge)^\op \simeq \Pr^R\wedge, $$$$ (\wedge\Pr^L\wedge)^\op \simeq {\wedge\Pr^R\wedge}.$$
\end{proposition}

We often use (weakly) (reduced) oriented, antioriented and bioriented slice categories, which we define as enriched slice categories in the sense of \cref{slices}.

	

 






\subsubsection{Opposite and conjugate oriented categories}

\begin{notation}
The equivalences of \cref{dua} give rise to the following equivalences
$$(-)^\co:= (-)^\op_!: {\boxtimes\Cat} \simeq {\Cat\boxtimes}, $$$$ (-)^\co:= (-)^\op_!: {\Cat\boxtimes} \simeq {\boxtimes\Cat}, $$
$$ (-)^\co:= ((-)^\op, (-)^\op)_! : {\boxtimes\Cat\boxtimes} \simeq {\boxtimes\Cat\boxtimes},$$
where the first two equivalences are inverse to each other and the third equivalence is an involution.
\end{notation}

\begin{notation}
Let $$(-)^\circ: {\Cat\boxtimes} \simeq {\boxtimes\Cat}, $$$$ (-)^\circ: {\boxtimes\Cat} \simeq {\Cat\boxtimes} $$$$ (-)^\circ:  {\boxtimes\Cat\boxtimes} \simeq {\boxtimes\Cat\boxtimes}$$ be the opposite enriched category involutions.	

\end{notation}

\begin{notation}
The equivalences of \cref{dua} give rise to the following equivalences
$$ (-)^\op:= (-)^\circ\circ (-)^\co_!: {\boxtimes\Cat} \simeq {\boxtimes\Cat}, $$$$ (-)^\op:= (-)^\circ\circ (-)^\co_!: {\Cat\boxtimes} \simeq {\Cat\boxtimes},$$
$$  (-)^\op :=(-)^\circ \circ ((-)^\co, (-)^\co)_!: {\boxtimes\Cat\boxtimes} \simeq {\boxtimes\Cat\boxtimes}, $$
where the first two equivalences are inverse to each other and the third equivalence is an involution.
\end{notation}

\begin{notation}

Moreover we set
$$ {(-)^{\co\op}}:=  {(-)^{\co}} \circ {(-)^{\op}} \simeq {(-)^{\op}} \circ {(-)^{\co}}:{\boxtimes\Cat} \simeq {\Cat\boxtimes}, {\Cat\boxtimes} \simeq {\boxtimes\Cat}, {\boxtimes\Cat\boxtimes} \simeq {\boxtimes\Cat\boxtimes},$$
$$ {(-)^{\cop}}:=  {(-)^{\co\op}} \circ {(-)^{\circ}} \simeq {(-)^{\circ}} \circ {(-)^{\co\op}} \simeq (-)^{\co\op}_!: {\boxtimes\Cat} \simeq {\boxtimes\Cat}, {\Cat\boxtimes} \simeq {\Cat\boxtimes},$$
$${^{\cop}(-)}:= ((-)^{\co\op}, \id)_!, \ (-)^\cop:= (\id, (-)^{\co\op})_!: {\boxtimes\Cat\boxtimes} \simeq {\boxtimes\Cat\boxtimes},$$
$${^{\cop}(-)}^\cop:=(-)^\cop \circ  {^{\cop}(-)} \simeq $$$$ {^{\cop}(-)}\circ (-)^\cop \simeq ((-)^{\co\op}, (-)^{\co\op})_! \simeq$$$$ {(-)^{\co\op}} \circ {(-)^{\circ}} \simeq {(-)^{\circ}} \circ {(-)^{\co\op}}: {\boxtimes\Cat\boxtimes} \simeq {\boxtimes\Cat\boxtimes}.$$

\end{notation}

The equivalences of \cref{dua} gives rise to the following equivalences of bioriented categories:
\begin{corollary}\label{duae}There are canonical equivalences of bioriented categories:
$$(-)^\op: \infty\Cat^\co \simeq \infty\Cat, $$$$ (-)^\co: \infty\Cat^\op \simeq \infty\Cat^\circ, $$$$ (-)^{\co\op} : {^\cop\infty\Cat^{\cop}} \simeq \infty\Cat. $$	

\end{corollary}	

We will often use the following notation:

\begin{notation}
Let $\mC$ be a presentable reduced antioriented category and $\mD$ a presentable reduced oriented category. The category $\mC \ot \mD$ is canonically a category presentably bitensored over
$\infty\Cat_*$ and so a presentable reduced bioriented category that we denote by $\mC \boxplus \mD.$

\end{notation}

\begin{definition}
Let $\mC, \mD$ be presentable reduced bioriented categories. The category $\mC \ot \mD$ is canonically a category presentably bitensored over $\infty\Cat_* \ot \infty\Cat_*$ and so $\mC \boxplus \mD$
carries extra structure that guarantees by \cite[Theorem 5.51., Theorem 4.89.]{heine2024bienrichedinftycategories} that for every small reduced bioriented category $\mB$ the category $ {\wedge\Fun\wedge}(\mB, \mC\boxplus \mD)$ refines to a presentable bioriented category.
\end{definition}

\begin{remark}\label{tencot} By \cite[Proposition 5.15.]{heine2024bienrichedinftycategories} the left tensor with $X\in\infty\Cat_*$ is the reduced oriented functor $$(\mC \boxplus (X \wedge (-)))_*:  {\wedge\Fun\wedge}(\mB, \mC\boxplus \mD) \to {\wedge\Fun\wedge}(\mB, \mC\boxplus \mD)$$
and the right tensor with $X\in\infty\Cat_*$ is the reduced antioriented functor $$(((-) \wedge X)\boxplus\mD)_*:  {\wedge\Fun\wedge}(\mB, \mC\boxplus \mD) \to {\wedge\Fun\wedge}(\mB, \mC\boxplus \mD).$$	

\end{remark}

We will also use sometimes use the following generalization of oriented, antioriented and bioriented categories, which behave formally similar.
For every monoidal category $\mC$ let $\widehat{\mP}(\mC)$ be the monoidal category of presheaves of large spaces on $\mC$ endowed with Day convolution.

\begin{definition}\emph{}
\begin{enumerate}
		 
\item A generalized oriented category is a category enriched in $(\widehat{\mP}(\infty\Cat),\boxtimes^\rev).$

\item A generalized antioriented category is a category enriched in $(\widehat{\mP}(\infty\Cat),\boxtimes).$
		
\item A generalized bioriented category is a category enriched in $(\widehat{\mP}(\infty\Cat),\boxtimes) \ot (\widehat{\mP}(\infty\Cat),\boxtimes^\rev).$
			
\item A weakly reduced generalized oriented category is a category enriched in $(\widehat{\mP}(\infty\Cat_*),\wedge^\rev).$

\item A weakly reduced generalized antioriented category is a category enriched in $(\widehat{\mP}(\infty\Cat_*),\wedge).$
		
\item A weakly reduced generalized bioriented category is a category enriched in $$(\widehat{\mP}(\infty\Cat_*),\wedge) \ot (\widehat{\mP}(\infty\Cat_*),\wedge^\rev).$$
		
\end{enumerate}
	
\end{definition}

The following is \cite[Notation 5.27., Proposition 5.107.]{heine2024bienrichedinftycategories} and \cite[Proposition 5.30.]{heine2024bienrichedinftycategories}:

\begin{proposition}Let $\mD$ be a weakly reduced bioriented category. 

\begin{enumerate}
\item Let $\mC$ be a weakly reduced bioriented category that admits left tensors.
The weakly reduced antioriented category ${\Fun\wedge}(\mC,\mD)$ of weakly reduced oriented functors of \cref{funor} underlies a weakly reduced generalized bioriented category that admits right tensors.
If moreover $\mD$ admits left tensors, also ${\Fun\wedge}(\mC,\mD)$ admits left tensors, which are formed objectwise.

\vspace{1mm}
\item Let $\mC$ be a weakly reduced bioriented category that admits right tensors. The weakly reduced oriented category ${\wedge\Fun}(\mC,\mD)$ of weakly reduced antioriented functors of \cref{funor} refines to a weakly reduced generalized bioriented category that admits left tensors.
If moreover $\mD$ admits right tensors, also ${\wedge\Fun}(\mC,\mD)$ admits right tensors, which are formed objectwise.

\end{enumerate}

\end{proposition}

\subsection{Oriented pullbacks}

Next we define oriented pushouts and oriented pullbacks following \cite[4.8.]{gepner2025oriented}.

\begin{definition}
Let $\mC$ be an oriented category.
The oriented pullback of a diagram $X \to Z\leftarrow Y$ in $\mC$ is a diagram
\[
\xymatrix{& W\ar[rd]\ar[ld] &\\
X\ar[rd] & \Longrightarrow & Y\ar[ld]\\
& Z &}
\]
in $\mC$ such that for all objects $T$ of $\mC$ the induced functor
\begin{equation}\label{indmapl}
\R\Mor_{\mC}(T,W)\to \R\Mor_{\mC}(T,X) \underset{\R\Mor_{\mC}(T,Z)}{\times}
\Fun^\oplax(\bD^1,\R\Mor_{\mC}(T,Z)) \underset{\R\Mor_{\mC}(T,Z)}{\times} \R\Mor_{\mC}(T,Y)
\end{equation}
is an equivalence in $\infty\Cat$.
We write $X\overset{\to}{\times}_Z Y$ or $ Y\overset{\leftarrow}{\times}_Z X$ for $W$ if it exists.

\end{definition}

\begin{definition}Let $\mC$ be an oriented category.
The oriented pushout of a diagram $X \leftarrow Z\to Y$ in $\mC$
is the oriented pullback in the oriented category $\mC^\op$ of the corresponding diagram.
We write $ X\overset{\to}{+}_Z Y $ or $Y\overset{\leftarrow}{+}_Z X$ for the oriented pushout.

\end{definition}

\begin{definition}Let $\mC$ be an antioriented category. 
\begin{enumerate}
\item The antioriented pullback of a diagram $X \to Z \leftarrow Y $ in $\mC$ denoted by $  X \overset{\bar{\to}}{\times}_Z Y $ or
$ Y \overset{\bar{\leftarrow}}{\times}_Z X $ is the oriented pullback of the corresponding diagram in the oriented category $\mC^{\co}$.	

\item The antioriented pushout of a diagram $X \leftarrow Z \to Y $ in $\mC$
denoted by $  X \overset{\bar{\to}}{+}_Z Y $ or $ Y \overset{\bar{\leftarrow}}{+}_Z X $ is the oriented pushout of the corresponding diagram in the oriented category $\mC^\co$.

\end{enumerate}	

\end{definition}

\begin{definition}
The (anti)oriented pushout and pullback in any bioriented category is the (anti)oriented pushout and pullback in the underlying (anti)oriented category.

\end{definition}

The next is \cite[Lemma 4.4.6.]{gepner2025oriented}:

\begin{lemma}\label{0desc}\label{adesc}
Let $\mC$ be an oriented category and $Z \to X, Z \to Y, X \to Z, Y \to Z$ morphisms in $\mC$.

\begin{enumerate}

\item If $\mC$ admits pushouts and right tensors with $\bD^1$, 
there is a canonical equivalence $$ X \overset{\to}{+}_Z Y \simeq X +_{\{0\}\otimes Z} (Z \ot \bD^1) +_{\{1\}\otimes Z} Y.$$	

\item If $\mC$ admits pullbacks and right cotensors with $\bD^1$, 
there is a canonical equivalence $$ X \overset{\to}{\times}_Z Y \simeq X \times_{Z^{\{0\}}} {Z^{\bD^1}} \times_{Z^{\{1\}}} Y.$$	
\end{enumerate}

\end{lemma}

\begin{corollary}\label{bdesc}
Let $\mC$ be an antioriented category and $Z \to X, Z \to Y, X \to Z, Y \to Z$ morphisms in $\mC$.

\begin{enumerate}

\item If $\mC$ admits pushouts and left tensors with $\bD^1$, there is a canonical equivalence $$ X \overset{\bar{\to}}{+}_Z Y \simeq X +_{\{0\}\ot Z} (\bD^1\ot Z) +_{ \{1\}\ot Z} Y.$$

\item If $\mC$ admits pullbacks and left cotensors with $\bD^1$, there is a canonical equivalence $$ X \overset{\bar{\to}}{\times}_Z Y \simeq X \times_{^{\{0\}}Z} {^{\bD^1}Z} \times_{^{\{1\}}Z} Y.$$

\end{enumerate}

\end{corollary}



\begin{definition}Let $\mC$ be a reduced oriented category.

\begin{enumerate}
\item Let $\phi: A \to B$ be a morphism in $\mC.$ The oriented left cofiber of $\phi$ is $ 0 \overset{\to}{+}_A B $.

\item Let $\phi: A \to B$ be a morphism in $\mC.$ The oriented right cofiber of $\phi$ is $ B \overset{\to}{+}_A 0 $.

\item Let $\phi: B \to A$ be a morphism in $\mC.$ 
The oriented left fiber of $\phi$ is $ 0 \overset{\to}{\times}_A B  $.

\item Let $\phi: B \to A$ be a morphism in $\mC.$ 
The oriented right fiber of $\phi$ is $ B \overset{\to}{\times}_A 0 $.

\end{enumerate} 
\end{definition}

\begin{definition}Let $\mC$ be a reduced antioriented category.

\begin{enumerate}
\item Let $\phi: A \to B$ be a morphism in $\mC.$ The antioriented left cofiber of $\phi$ is $ 0 \overset{\bar{\to}}{+}_A B $.

\item Let $\phi: A \to B$ be a morphism in $\mC.$ The antioriented right cofiber of $\phi$ is $B \overset{\bar{\to}}{+}_A 0$.

\item Let $\phi: B \to A$ be a morphism in $\mC.$ 
The antioriented left fiber of $\phi$ is $ 0 \overset{\bar{\to}}{\times}_A B $.

\item Let $\phi: B \to A$ be a morphism in $\mC.$ 
The antioriented right fiber of $\phi$ is $ B \overset{\bar{\to}}{\times}_A 0 $.

\end{enumerate} 
\end{definition}

\begin{definition}

A reduced (anti)oriented category admits (anti)oriented (co)fibers if it admits the (anti)oriented left and right (co)fiber of any morphism.

\end{definition}

\begin{notation}
Let $\bD^1$ be always equipped with the initial object.
Let $(\bD^1)^\op$ denote $\bD^1$ equipped with the final object.

\end{notation}

\begin{corollary}\label{redfib}

\begin{enumerate}

\item Let $\mC$ be a reduced oriented category that admits pushouts and reduced right tensors with $\bD^1$ and $C \to A$ a morphism in $\mC$.
There are canonical equivalences $$ 0 \overset{\to}{+}_C A  \simeq  C \wedge \bD^1 +_{C \wedge S^0}A, \ A \overset{\to}{+}_C 0 \simeq A +_{C \wedge S^0} C \wedge (\bD^1)^\op .$$	

\item Let $\mC$ be a reduced oriented category that admits pullbacks and reduced right cotensors with $\bD^1$ and $A \to C$ a morphism in $\mC$.
There are canonical equivalences $$0 \overset{\to}{\times}_C A \simeq C^{\bD^1}_*\times_{C^{S^0}_*} A, \ A \overset{\to}{\times}_C 0  \simeq A \times_{C^{S^0}_*} C^{(\bD^1)^\op}_*. $$	

\item Let $\mC$ be a reduced antioriented category that admits pushouts and reduced left tensors with $\bD^1$ and $C \to B$ a morphism in $\mC$.
There are canonical equivalences $$ 0 \overset{\bar{\to}}{+}_C B \simeq \bD^1 \wedge C   +_{S^0 \wedge C} B, \ A \overset{\bar{\to}}{+}_C 0 \simeq A  +_{S^0 \wedge C} (\bD^1)^\op \wedge C .$$	

\item Let $\mC$ be a reduced antioriented category that admits pullbacks and reduced left cotensors with $\bD^1$ and $B \to C$ a morphism in $\mC$.
There are canonical equivalences $$0 \overset{\bar{\to}}{\times}_C B \simeq {^{\bD^1}C}_* \times_{^{S^0}C_*} B, \ A \overset{\bar{\to}}{\times}_C 0 \simeq A \times_{^{S^0}C_*} {^{(\bD^1)^\op}C}_* .$$	

\end{enumerate}

\end{corollary}

\begin{proof}
We use \cref{adesc} and \cref{bdesc} and
that for every $X \in \infty\Cat_*$ the reduced left cotensor ${^XC}_*$ is the pullback $0 \times_{^*C}{^XC}$ and the reduced right cotensor $C^X_*$ is the pullback ${C^X}\times_{C^*}0.$

\end{proof}

\begin{corollary}Let $\mC$ be a reduced oriented category and $X \in \mC$.
The oriented pushout $ {\overset{\overset{\rightarrow}{}}{0 +_C} 0}$ and the right tensor $X \wedge S^1$
both satisfy the same universal property.

\end{corollary}

\begin{corollary}Let $\mC$ be a reduced antioriented category and $X \in \mC$.
The antioriented pushout $  {\overset{\overline{\rightarrow}}{0 +_C} 0}$ and the left tensor $S^1 \wedge X$ both satisfy the same universal property.
\end{corollary}

By \cite[Lemma 4.4.12.]{gepner2025oriented} there is the following pasting law:

\begin{lemma}\label{pasting}

Consider the following diagram in any oriented category $\mC$, where the left hand square is a commutative square:
\[
\begin{tikzcd}
\Q \ar{d} \ar{r} & \P \ar{r}{} \ar{d}[swap]{} & B  \ar{d}{} \\
\E \ar{r} & A \ar[double]{ur}{}  \ar{r}[swap]{} & C
\end{tikzcd}
\]
If the right hand square is an oriented pullback square, the left hand square is a pullback square if and only if the outer square is an oriented pullback square.

\end{lemma}

Next we study the relationship between oriented pullbacks and pushouts
in any reduced oriented category.

\begin{lemma}\label{spli}
Let $\mC$ be a reduced oriented category that admits oriented pushouts.
There is a commutative square 
\begin{equation}\label{sqlo}
\begin{xy}
\xymatrix{
A \overset{\to}{+}_C B \ar[d] \ar[r]
& 0 \overset{\to}{+}_C B
\ar[d]
\\ 
A \overset{\to}{+}_C 0
\ar[r] & 0 \overset{\to}{+}_C 0 \simeq \Sigma(C) 
}
\end{xy}
\end{equation}
natural in $ (A \leftarrow C \to B) \in \mC^{\Lambda_0^2}$ determining a functor $\sigma: \mC^{\Lambda_0^2} \to \mC^{\bD^1 \times \bD^1}$.

If $\mC$ is a reduced bioriented category that admits oriented pullbacks, the functor $\sigma: \mC^{\Lambda_0^2} \to \mC^{\bD^1 \times \bD^1}$ refines to a reduced antioriented functor.

\end{lemma}

\begin{proof}
The reduced bioriented functors $$\rho: \mC^{\Lambda_0^2} \to \mC^{\bD^1}, (A \leftarrow C \to B) \mapsto (A \leftarrow C), $$$$ \rho': \mC^{\Lambda_0^2} \to \mC^{\bD^1}, (A \leftarrow C \to B) \mapsto (C \to B), $$$$ \rho'': \mC^{\bD^1} \to \mC, (A \to B) \mapsto A$$ admit fully faithful right adjoints sending
$A \leftarrow C$ to $ A \leftarrow C \to 0$, $C \to B$ to $0 \leftarrow C \to B $ and 
$A \to B$ to $A \to 0$, respectively.
The units give rise to a commutative square of reduced (right) bioriented functors $\mC^{\Lambda_0^2} \to \mC^{\Lambda_0^2}:$
$$\begin{xy}
\xymatrix{
\id  \ar[d] \ar[r]
&\rho' \ar[d]
\\ 
\rho \ar[r] & \rho''\rho \simeq \rho'' \rho' \simeq \ev_0.
}
\end{xy}
$$
The latter is a pullback square by the description of adjoints and induces a commutative square 
$$\begin{xy}
\xymatrix{
\overset{\to}{+} \ar[d] \ar[r]
& \overset{\to}{+}\circ \rho \simeq (\overset{\to}{+}0) \circ \rho \ar[d]
\\ 
\overset{\to}{+} \circ \rho' \simeq (0\overset{\to}{+}) \circ \rho' \ar[r] & \overset{\to}{+} \circ \ev_0 \simeq \Sigma  \ev_0
}
\end{xy}
$$
of (reduced antioriented) functors $\mC^{\Lambda_0^2} \to \mC$ that yields at any $ (A \leftarrow C \to B) \in \mC^{\Lambda_0^2}$ square \ref{sqlo}.
\end{proof}

\begin{lemma}\label{compa}
Let $\mC$ be a reduced oriented category that admits oriented left cofibers and oriented left fibers.	
There is a map $$\xi: 0 \overset{\to}{+} \to (0 \overset{\to}{\times}) \circ \Sigma$$ of functors $ \mC^{\bD^1} \to \mC$.
If $\mC$ is a reduced bioriented category that admits oriented left cofibers and oriented left fibers, then $\xi$ refines to a map of reduced antioriented functors $ \mC^{\bD^1} \to \mC$.

\end{lemma}

\begin{proof}By \cref{redfib} the oriented left fiber is the limit weighted with respect to the weight $ S^0 \to \bD^1$ and so by \cref{weifunc} there is an adjunction $$ L:=(-)\wedge (S^0 \to \bD^1) : \mC \rightleftarrows \mC^{\bD^1} : 0 \overset{\to}{\times}.$$
Moreover if $\mC$ is a reduced bioriented category, the latter adjunction is a reduced antioriented adjunction.
By \cite[Proposition 4.40.(3)]{heine2024bienrichedinftycategories} the map $$\Map_{\wedge\Fun}(\infty\Cat_*^{\bD^1},\infty\Cat_*^{\bD^1})(L \circ (0 \overset{\to}{+}), \Sigma) \to \Map_{\infty\Cat_*^{\bD^1\times \bD^1}}((S^0 \to \bD^1) \to (\bD^1 \to \bD^2), (0 \to S^1) \to \id_{S^1})$$
restricting along the bioriented Yoneda embedding $$\infty\Cat_*^{\bD^1\times \bD^1} \simeq {\Fun}(\bD^1,\infty\Cat_*^{\bD^1})  \simeq {\wedge\Fun}(\bD^1,\infty\Cat_*^{\bD^1}) \hookrightarrow {\wedge\Fun}(\infty\Cat_*^{\bD^1},\infty\Cat_*^{\bD^1})$$ is an equivalence. 
Let $\kappa: L \circ (0 \overset{\to}{+}) \to \Sigma$ be the map of left adjoint reduced antioriented functors $ \infty\Cat_*^{\bD^1} \to \infty\Cat_*^{\bD^1}$ corresponding to the following commutative square in $\infty\Cat_*^{\bD^1}:$
$$\begin{xy}
\xymatrix{
S^0 \wedge S^0 \to S^0 \wedge \bD^1 \ar[d] \ar[rr]
&& 0 \to S^1 \ar[d]
\\ 
\bD^1 \wedge S^0 \simeq \bD^1 \to \bD^1 \wedge \bD^1 \simeq \bD^2 \ar[rr] && \id_{S^1},
}
\end{xy}
$$	
where the functor $\bD^2 \to S^1$ in the square is the composition of 1-truncation $\bD^2 \to \bD^1$ and the quotient functor $\bD^1 \to S^1$.
Let $\kappa_\mC:=\mC {\ot_{\infty\Cat_*} \kappa} :  L \circ  (0 \overset{\to}{+}) \to \Sigma$.
Then $\kappa_\mC$ is a map of left adjoint functors $ \mC^{\bD^1} \to \mC^{\bD^1}$,
which is a map of reduced antioriented functors if $\mC$ is a reduced bioriented category.
The map $\kappa_\mC$ gives rise to the desired map $\xi:  0 \overset{\to}{+} \to (0 \overset{\to}{\times}) \circ L \circ (0 \overset{\to}{+}) \to (0 \overset{\to}{\times}) \circ \Sigma$. 
\end{proof}

Replacing $\mC$ by $^\cop\mC^\cop$ we obtain the following dual version:

\begin{corollary}\label{compacor}
Let $\mC$ be a reduced oriented category that admits oriented right cofibers and oriented right fibers.	
There is a map $$\xi: \overset{\to}{+} 0 \to (\overset{\to}{\times} 0) \circ \Sigma$$ of functors $ \mC^{\bD^1} \to \mC$.
If $\mC$ is a reduced bioriented category that admits oriented right cofibers and oriented right fibers, then $\xi$ refines to a map of reduced antioriented functors $ \mC^{\bD^1} \to \mC$.

\end{corollary}

\begin{notation}
Let $\mC$ be a reduced oriented category that admits oriented pushouts, oriented fibers and pullbacks.
By \cref{spli} and \cref{compa} there is a map of functors $\mC^{\Lambda_0^2} \to \mC$:
\begin{equation}\label{fact}
\kappa:\overset{\to}{+} \to (\overset{\to}{+}0)
\times_{\Sigma\circ \ev_0} (0\overset{\to}{+}) \to (\overset{\to}{\times}0) \circ \Sigma \times_{\Sigma\circ \ev_0} (0\overset{\to}{\times})\circ \Sigma \simeq (\overset{\to}{\times}0) \times_{\ev_0} (0 \overset{\to}{\times})) \circ \Sigma. 
\end{equation}
If $\mC$ is a reduced bioriented category that admits oriented pushouts, oriented fibers and pullbacks, then by 
\cref{spli} and \cref{compa} the map $\kappa$ refines to a map of reduced antioriented functors $\mC^{\Lambda_0^2} \to \mC$.
 
\end{notation}

\begin{lemma}\label{repp}
Let $\mC$ be a reduced oriented category that admits oriented pushouts, oriented fibers and pullbacks.
The canonical map $$\theta:= \Sigma \circ \kappa: \Sigma \circ\overset{\to}{+}\to \Sigma \circ ((\overset{\to}{\times}0) \times_{\ev_0} (0\overset{\to}{\times})) \circ \Sigma$$ admits a left inverse.
If $\mC$ is a reduced bioriented category that admits oriented pushouts, oriented fibers and pullbacks
(so that $\theta$ is a map of reduced antioriented functors $\mC^{\Lambda_0^2} \to \mC$),
the left inverse of $\theta$ refines to a map of reduced antioriented functors $\mC^{\Lambda_0^2} \to \mC$.

\end{lemma}

\begin{proof}

By \cref{weiloc} there is an oriented (bioriented) embedding of $\mC$ into a presentable reduced oriented (bioriented) category that preserves oriented pushouts, oriented fibers and pullbacks. Consequently, we can assume that $\mC$ is presentable.

We construct a map $\Phi: \Sigma \circ ((\overset{\to}{\times}0) \times_{\ev_0} (0\overset{\to}{\times})) \circ \Sigma \to \Sigma \circ \overset{\to}{+} $ of reduced (antioriented) functors $\mC^{\Lambda_0^2} \to \mC$ and prove that $\Phi \circ \theta: \Sigma \circ \overset{\to}{+} \to\Sigma \circ\overset{\to}{+}$ is an equivalence. 
This proves the result: if $\gamma$ is an inverse of $ \Phi \circ \theta$, then $\gamma \circ \Phi \circ \theta$ is the identity so that $\theta$ admits a left inverse. 
To see that $\Phi \circ \theta $ is an equivalence, it is enough to see that
$\Phi_\varphi \circ \theta_\varphi: \Sigma \circ \overset{\to}{+} \varphi \to \Sigma \circ \overset{\to}{+} \varphi $ is an equivalence for
$\varphi$ the spans $ X \leftarrow 0 \to 0, 0 \leftarrow 0 \to X, X \leftarrow X \to X$ for any $X \in \mC$. This holds since $\mC^{\Lambda_0^2}$ is generated under small colimits by these three spans and $\Sigma \circ \overset{\rightarrow}{+}$ preserves small colimits.
By adjointness a map $\Phi: \Sigma \circ ((\overset{\to}{\times}0) \times_{\ev_0} (0\overset{\to}{\times})) \circ \Sigma \to \Sigma \circ \overset{\to}{+} $ corresponds to a map $\Psi: ((\overset{\to}{\times}0) \times_{\ev_0} (0\overset{\to}{\times})) \circ \Sigma \to \Omega \circ \Sigma \circ \overset{\to}{+} $ and for every $\varphi \in \mC^{\Lambda_0^2}$ the morphism $\Phi_\varphi \circ \theta_\varphi$ is an equivalence
if and only if $\Psi_\varphi \circ \kappa_\varphi$ is a unit $\overset{\to}{+}\varphi \to \Omega\Sigma\overset{\to}{+}\varphi$. 

We will use the following notation:
for every category $K$ and presentable reduced oriented category $\mC$ the right action of $\infty\Cat_*$
gives rise to a left adjoint functor $$\mC \otimes \infty\Cat_*^K \simeq \mC \otimes \infty\Cat_* \ot \infty\Grp^K \to \mC^K \simeq \mC \ot \infty\Grp^K, (X,\beta) \mapsto (X \ot (-)) \circ \beta. $$
Thus for every $\beta \in \infty\Cat_*^K$ the functor 
$\mC \to \mC^K, X  \mapsto (X \ot (-)) \circ \beta $ admits a right adjoint,
which we denote by $(-)^\beta: \mC^K \to \mC.$
If $\mC$ is a presentable reduced bioriented category, the functor $$\mC \otimes \infty\Cat_*^K \to \mC^K, (X,\beta) \mapsto (X \ot (-)) \circ \beta $$ is a left adjoint reduced antioriented functor. Thus for every $\beta \in \infty\Cat_*^K$ the functor 
$$\mC \to \mC^K, X  \mapsto (X \ot (-)) \circ \beta $$ is a left adjoint reduced antioriented functor.

Let $\alpha$ be the span $(\bD^1)^\op \leftarrow S^0 \to \bD^1$ in $\infty\Cat_*.$
There is an equivalence of reduced (antioriented) functors $\mC^{\Lambda_0^2} \to \mC$:
$$ \Sigma \circ ((\overset{\to}{\times}0) \times_{\ev_0} (0\overset{\to}{\times})) \circ \Sigma \simeq ((-)^{S^0 \to (\bD^1)^\op} \times_{(-)^{S^0}} (-)^{S^0 \to \bD^1}) \circ \Sigma \simeq (-)^\alpha \circ \Sigma.$$

By \cref{0desc} there is an equivalence  
$$ \overset{\leftarrow}{+}\alpha \simeq \mO:=\bD^1 \vee(\bD^1)^\op +_{S^0 \vee S^0} \bD^1_+ $$ in $\infty\Cat_*.$ 

The composition of the morphism $ \bD^1 \simeq \{0 < 2\} \subset [2]\simeq \bD^1 \vee (\bD^1)^\op \to\mO$ followed by the morphism $\bD^1 \to \mO $ is an endomorphism in $\mO  $ corresponding to a functor 
$S^1 \to \mO \simeq  \overset{\leftarrow}{+}\alpha .$

Let $\Psi^\mC$ be the natural transformation 
$$(-)^\alpha \circ \Sigma  \to (\overset{\leftarrow}{+})^{\overset{\leftarrow}{+}\alpha} \circ \Sigma \to (\overset{\leftarrow}{+})^{S^1}\circ \Sigma \simeq \Omega \circ \overset{\leftarrow}{+} \circ \Sigma \simeq\Omega\circ \Sigma\circ \overset{\to}{+}$$
of reduced (antioriented) functors $\mC^{\Lambda_0^2} \to \mC$.
For every $\varphi \in \infty\Cat_*^{\Lambda_0^2}$ and $X \in \mC$ the morphism $$(((\overset{\to}{\times}0) \times_{\ev_0} (0\overset{\to}{\times})) \Sigma\varphi) \wedge X \xrightarrow{\Psi_{\varphi}^{\infty\Cat_*} \wedge X}  (\Omega \Sigma \overset{\to}{+}\varphi) \wedge X \xrightarrow{\gamma}  \Omega \Sigma (\overset{\to}{+} \varphi \wedge X)$$ factors as $$ (((\overset{\to}{\times}0) \times_{\ev_0} (0\overset{\to}{\times})) \Sigma\varphi) \wedge X \xrightarrow{\beta} ((\overset{\to}{\times}0) \times_{\ev_0} (0\overset{\to}{\times})) \Sigma(\varphi \wedge X) \xrightarrow{\Psi_{\varphi \wedge X}^\mC} \Omega \Sigma(\overset{\to}{+}\varphi \wedge X).$$
Moreover the morphism 
$\kappa_{\varphi \wedge X}^\mC: \overset{\to}{+}(\varphi \wedge X) \xrightarrow{} ((\overset{\to}{\times}0) \times_{\ev_0} (0\overset{\to}{\times})) \Sigma(\varphi \wedge X)$ factors as $$ \overset{\to}{+}(\varphi \wedge X) \simeq (\overset{\to}{+}\varphi) \wedge X \xrightarrow{\theta_\varphi^{\infty\Cat_*} \wedge X} (((\overset{\to}{\times}0) \times_{\ev_0} (0\overset{\to}{\times})) \Sigma\varphi) \wedge X \xrightarrow{\beta} ((\overset{\to}{\times}0) \times_{\ev_0} (0\overset{\to}{\times})) \Sigma(\varphi \wedge X).$$

Hence we obtain a canonical equivalence 
$$ \Psi_{\varphi \wedge X}^\mC \circ \kappa_{\varphi \wedge X}^\mC \simeq \gamma \circ ((\Psi_\varphi^{\infty\Cat_*} \circ \kappa_\varphi^{\infty\Cat_*}) \wedge X).$$ Thus $\Psi_{\varphi \wedge X}^\mC \circ \kappa_{\varphi \wedge X}^\mC$ is the unit if 
$\Psi_\varphi^{\infty\Cat_*} \circ \kappa_\varphi^{\infty\Cat_*}$ is the unit.
So we can assume that $\mC=\infty\Cat_*$ and $\varphi$ is one of the spans $$ S^0 \leftarrow 0 \to 0, 0 \leftarrow 0 \to S^0, S^0 \leftarrow S^0 \to S^0. $$ We write $\Psi$ for $\Psi^{\infty\Cat_*}.$
The component of $\Psi \circ \kappa$ at $S^0 \leftarrow 0 \to 0$ factors as
$$S^0 \xrightarrow{\sigma}\L\Mor_{\infty\Cat_*^{\Lambda_0^2}}(\alpha ,S^1 \leftarrow 0 \to 0)  \xrightarrow{\rho} \L\Mor_{\infty\Cat_*}(\mO, S^1) \xrightarrow{\tau} \L\Mor_{\infty\Cat_*}(S^1, S^1) .$$
The map $\sigma$ classifies the cofiber sequence $ S^0 \to (\bD^1)^\op \to S^1.$
Hence $\rho \circ \sigma$ classifies the canonical map 
$$\mO= \bD^1 \vee (\bD^1 )^\op +_{S^0 \vee S^0} \bD^1_+ \to  0 \vee S^1 +_{0} 0  \simeq S^1. $$

Thus $\tau \circ \rho \circ \sigma$ classifies the map $S^1 \to \mO \to S^1$
preserving the generator and so is the identity.
Hence $\tau \circ \rho \circ \sigma$ is the identity. 
The case of the component of $\Psi \circ \kappa$ at $0 \leftarrow 0 \to S^0$ is similar.

The component of $\Psi\circ \kappa$ at the span $S^0 \leftarrow S^0 \to S^0$
factors as $$ \hspace{4mm}\bD^1_+ \xrightarrow{\sigma} \L\Mor_{\infty\Cat_*^{\Lambda^2_0}}(\alpha, S^1 \leftarrow S^1 \to S^1) \xrightarrow{\rho}  \L\Mor_{\infty\Cat_*}(\mO, \Sigma(\bD^1_+)) \xrightarrow{\tau} \L\Mor_{\infty\Cat_*}(S^1, \Sigma(\bD^1_+)) \simeq \Omega\Sigma(\bD^1_+). $$

For $\bi=0,1$ let $\alpha_\bi$ be the inclusion $\{\bi\}_+ \subset \bD^1_+$.
The functors $\rho \sigma \alpha_0, \rho \sigma \alpha_1$ classify the respective functors $$\mO=  \bD^1 \vee (\bD^1)^\op+_{S^0 \vee S^0}    
S^0 \wedge \bD^1_+ \to  S^1 \vee S^1 +_{S^1 \vee S^1}    
S^1 \wedge \bD^1_+ \simeq S^1 \wedge \bD^1_+ \simeq \bD^1_+ \wedge S^1 \simeq \Sigma(\bD^1_+). $$
Here $\rho \sigma \alpha_0$ is induced by the zero functor $\bD^1 \to S^1$ and the quotient functor $(\bD^1)^\op \to (S^1)^\op \simeq S^1$  and $\rho \sigma \alpha_1$ is induced by the 
quotient functor $\bD^1 \to S^1$ and zero functor $(\bD^1)^\op \to S^1$. Moreover the last equivalence
$S^1 \wedge \bD^1_+ \simeq \bD^1_+ \wedge S^1 $ is an equivalence under the flip equivalence
$ S^1 \vee S^1 \simeq S^1 \vee S^1. $

Thus the functor $\rho \sigma \alpha_0$ sends the functor $(\bD^1)^\op \to \mO$ to the functor $$(\bD^1)^\op \to S^1 \simeq (\{0\}_+) \wedge S^1 \to \bD^1_+ \wedge S^1$$ and the functor $\rho \sigma \alpha_1$ sends the functor $\bD^1 \to \mO$ to the functor $$\bD^1 \to S^1 \simeq (\{1\}_+) \wedge S^1 \to \bD^1_+ \wedge S^1.$$ Consequently, $\tau \rho \sigma \alpha_\bi$ classifies 
$$\Sigma(\alpha_\bi): S^1\simeq \Sigma(\{\bi \}_+) \to \Sigma(\bD_+^1).$$
Hence $\tau \rho\sigma \alpha_\bi$ is  $\alpha_\bi$ followed by the unit $\bD^1_+ \to \Omega\Sigma(\bD^1_+).$
So $\tau \rho \sigma: \bD^1_+ \to \Omega\Sigma(\bD^1_+)$ is the unit since by \cref{freespar} 
the unit $\bD^1_+ \to \Omega\Sigma(\bD^1_+)$ is the canonical functor
$\coprod_{0 \leq \n \leq 1} (\bD^1)^{\times \n} \to \coprod_{\n \geq 0} (\bD^1)^{\times \n}$ into the free monoidal $\infty$-category, which is fully faithful.

\end{proof}

\subsection{Oriented suspension}

Next we refine (reduced) categorical suspension and (reduced) categorical antisuspension to a bioriented functor (\cref{reduc}) and prove that this refinement is functorial (\cref{impol}).
This is crucial to construct a bioriented category of categorical spectra in the next section.

The next is \cite[Theorem 4.2.8.]{gepner2025oriented}:

\begin{theorem}\label{alfas}
	
\begin{enumerate}
\item The functor $$S: \infty\Cat^\cop \to \infty\Cat_{\partial\bD^1/}$$ refines to a bioriented functor. 
		
\item The functor $$\bar{S}: {^\cop}\infty\Cat \to \infty\Cat_{\partial\bD^1/}$$ refines to a bioriented functor. 
		
\item There is an equivalence of bioriented functors $$S \simeq \bar{S} \circ (-)^{\co\op}: \infty\Cat^{\cop} \to \infty\Cat_{\partial\bD^1/}.$$
		
\item There is an equivalence of bioriented functors $$(-)^{\co\op} \circ S \simeq S \circ (-)^{\co\op}: {^{\cop}\infty\Cat} \to \infty\Cat_{\partial\bD^1/}.$$
		
\end{enumerate}
	
\end{theorem}

\begin{corollary}

The functors $$S \circ S,  \bar{S} \circ \bar{S}: \infty\Cat \to \infty\Cat_{\partial\bD^2/}$$ refine to bioriented functors and there is a canonical equivalence of bioriented functors $$S \circ S \simeq \bar{S} \circ \bar{S}: \infty\Cat \to \infty\Cat_{\partial\bD^2/}.$$
\end{corollary}

\begin{corollary}\label{Morenrfun}
	
The underlying antioriented functor $S: \infty\Cat \to \infty\Cat_{\partial \bD^1 /} $ admits a right adjoint antioriented functor $$ \Mor: \infty\Cat_{\partial \bD^1/}  \to \infty\Cat$$ sending $(\mC,X,Y)$ to $\Mor_\mC(X,Y).$
	
\end{corollary}

\begin{proof}
	
By \cref{hom} the induced functor
$$\Fun^\oplax(\mB, \Mor_\mC(X,Y)) \to \Fun_{\partial\bD^1/}^\oplax(S(\mB),\mC) $$
is an equivalence.
This proves the result by \cite[Remark 2.72.]{heine2024bienrichedinftycategories} and \cref{bien}.
	
\end{proof}

\begin{theorem}\label{reduc} Let $\mC$ be a reduced bioriented category.

\begin{enumerate}\item If $\mC$ admits suspensions, the antioriented functor $\Sigma: \mC \to \mC$ refines to a bioriented functor $$\mC^\cop \to \mC.$$

\item If $\mC$ admits endomorphisms, the antioriented functor $\Omega: \mC \to \mC$ refines to a bioriented functor $$\mC \to \mC^\cop.$$

\item If $\mC$ admits suspensions and endomorphisms, there is a bioriented adjunction
$$\Sigma: \mC^\cop \rightleftarrows \mC: \Omega.$$ 

\item If $\mC$ admits antisuspensions, the oriented functor $\bar{\Sigma}: \mC \to \mC$ refines to a bioriented functor $${^\cop\mC} \to \mC.$$

\item If $\mC$ has antiendomorphisms, the oriented functor $\bar{\Omega}: \mC \to \mC$ underlies a bioriented functor $$ \mC \to {^\cop\mC}.$$

\item If $\mC$ admits antisuspensions and antiendomorphisms, there is a bioriented adjunction
$$\bar{\Sigma}: {^\cop\mC} \rightleftarrows \mC: \bar{\Omega}.$$ 

\item If $\mC$ admits small colimits and left and right tensors, then $\Sigma: \mC \to \mC$ factors as
$$\mC \ot_{\infty\Cat_*} \Sigma: \mC \simeq \mC {\ot_{\infty\Cat_*} \infty\Cat_*}  \to \mC {\ot_{\infty\Cat_*} \infty\Cat_*} \simeq \mC$$ and $\bar{\Sigma}: \mC \to \mC$ factors as
$$\bar{\Sigma} \ot_{\infty\Cat_*} \mC : \mC \simeq {\infty\Cat_* \ot_{\infty\Cat_*}} \mC \to \infty\Cat_* {\ot_{\infty\Cat_*} \mC} \simeq \mC.$$

\item There is an equivalence of reduced bioriented functors $$\Sigma \simeq \bar{\Sigma} \circ (-)^{\co\op}: \infty\Cat_*^\cop\to \infty\Cat_*.$$

\item There is an equivalence of reduced bioriented functors $$\Sigma \circ (-)^{\co\op}\simeq (-)^{\co\op} \circ \Sigma: {^\cop\infty\Cat_*} \to \infty\Cat_*.$$

\end{enumerate}	
\end{theorem}

\begin{proof}

(1) is dual to (4). (2) is dual to (5). (3) is dual to (6). We prove (1), (2), (3), (7), (8), (9).

(1): First let $\mC:= \infty\Cat_*$. The functor $\Sigma$ factors as functors
$$ \infty\Cat_{*} \xrightarrow{S} \infty\Cat_{\bD^1/} \to \infty\Cat_*$$
preserving small colimits, where the latter functor takes the pushout along the functor $\bD^1 \to *$.
Hence the functor $\Sigma$ preserves the zero object.
By \cref{alfas} (1) the first functor in the composition refines to a bioriented functor
$ \infty\Cat_{*}^\cop \to \infty\Cat_{\bD^1/}.$ By \cref{left adjoint slice} the second functor in the composition refines to a bioriented functor $\infty\Cat_{\bD^1/} \to \infty\Cat_*.$  So the composition underlies a reduced bioriented functor $$\infty\Cat_{*}^\cop \to \infty\Cat_{*}.$$

We prove linearity. We first prove left linearity.
We would like to see that for every $\mC,\mD \in \infty\Cat_*$
the canonical functor $$\mC \wedge \Sigma(\mD)\to \Sigma(\mC\wedge\mD)$$
is an equivalence.
The latter identifies with the equivalence $$\mC \wedge (\mD \wedge S^1) \simeq (\mC \wedge \mD) \wedge S^1.$$
We continue with right linearity. We would like to see that for every $\mC,\mD \in \infty\Cat_*$ the canonical functor $$ \Sigma(\mC) \wedge \mD \to \Sigma(\mC\wedge\mD^{\co\op}) $$ is an equivalence.
Since $\infty\Cat_*$ is generated under left tensors by $S^0$, and $\Sigma$ is left linear as we proved, we can assume that $\mC=S^0.$
In this case the latter functor identifies with the equivalence $$ S^1 \wedge \mD \simeq \bar{\Sigma}(\mD) \simeq \Sigma(\mD^{\co\op})$$ of \cref{smarem} and \cref{ahoit}.

We assume next that $\mC$ is a reduced bioriented category that admits
left and right (co)tensors. 
The reduced left adjoint bioriented functor $ \Sigma: \infty\Cat^\cop_{*} \to \infty\Cat_*$ 
gives rise to a reduced linear bioriented functor $$
\xi: \mC^\cop \simeq \mC {\ot_{\infty\Cat_*} \infty\Cat^\cop_{*} } \to \mC {\ot_{\infty\Cat_*} \infty\Cat_*} \simeq \mC $$
whose underlying antioriented functor factors as
$$ \mC \simeq \mC {\ot_{\infty\Cat_*} \infty\Cat_*} \xrightarrow{\mC \ot {((-)\wedge S^1)}} \mC {\ot_{\infty\Cat_*} \infty\Cat_*} \simeq \mC$$ and so identifies with $\Sigma \simeq (-)\wedge S^1 : \mC \to \mC.$
Since the latter admits a right adjoint, by \cref{adj} the bioriented functor $\xi$ admits a right adjoint $\Omega: \mC^\cop \to \mC$ refining the right adjoint of the reduced antioriented functor
$\Sigma : \mC \to \mC.$ 

Let $\mC$ be a reduced bioriented category that admits endomorphisms.
The bioriented Yoneda embedding gives a reduced bioriented embedding $$\mC \hookrightarrow \mD:={\wedge\Fun\wedge}(\mC^\circ, \infty\Cat_*\boxplus \infty\Cat_*)$$ preserving endomorphisms
into a reduced bioriented category that admits left and right (co)tensors. Since the embedding preserves endomorphisms, 
the reduced bioriented functor $\Omega_\mD: \mD \to \mD^\cop$ restricts to a reduced bioriented functor $\Omega_\mC: \mC \to \mC^\cop$. This proves (2). 

Let $\mC$ be a reduced bioriented category that admits suspensions
and $$\mE \subset {\wedge\Fun\wedge}(\mC^\circ, \infty\Cat_*\boxplus \infty\Cat_*)$$ the full subcategory of reduced bioriented functors preserving antiendomorphisms.
By \cref{weiloc} the bioriented embedding $\mE \subset \mD$ admits a left adjoint $L$. So $\mE$ is a reduced bioriented category that admits left and right (co)tensors.
By \cref{weiloc} the reduced bioriented embedding $\mC \subset \mD$ induces a reduced bioriented embedding $\mC \subset \mE$ preserving suspensions. So the bioriented functor $\Sigma_\mE: \mE^\cop \to \mE$ restricts to a bioriented functor $\Sigma_\mC: \mC^\cop \to \mC$. This proves (1).

The bioriented left adjoint $L: \mD \to \mE$ gives rise to an equivalence $\Sigma_\mE \circ L \simeq L \circ \Sigma_\mD$ 
of reduced bioriented functors $\mD^\cop \to \mE$. By adjointness the bioriented functor 
$\Omega_\mE$ is the restriction of the bioriented functor $\Omega_\mD$.
So if $\mC$ admits suspensions and endomorphisms, the bioriented functor 
$\Omega_\mC$ is the restriction of the bioriented functor $\Omega_\mE$.
The bioriented adjunction $$\Sigma_\mE: \mE^\cop \rightleftarrows \mE: \Omega_\mE$$ restricts to a bioriented adjunction $ \Sigma_\mC: \mC \rightleftarrows \mC^\cop: \Omega_\mC.$
This proves (3).

(7): Let $\mC$ be a reduced bioriented category that admits small colimits and left and right tensors. Let $$\mE \subset \mD:={\wedge\Fun\wedge}(\mC^\circ, \infty\widehat{\Cat}_*\boxplus \infty\widehat{\Cat}_*)$$ be the full subcategory of reduced bioriented functors preserving antiendomorphisms. Let $\mB \subset \mE$ be the full subcategory of reduced bioriented functors preserving small limits and left and right cotensors.
By \cref{weiloc} the bioriented embedding $\bj: \mB \subset \mE$ admits a left adjoint $L$. Thus $\mB$ is a reduced bioriented category that admits left and right (co)tensors.
So there is an equivalence of reduced bioriented functors $L \circ \Sigma_\mE \simeq \Sigma_\mB \circ L$ and so a map of reduced bioriented functors $ \Sigma_\mE \circ \bj \to \bj \circ \Sigma_\mB. $
By \cref{weiloc} the reduced bioriented embedding $\mC \subset \mD$ induces a reduced bioriented embedding $\mC \subset \mE$ preserving suspensions and a reduced bioriented embedding $\bi: \mC \subset \mB$
preserving small colimits and left and right tensors.
So the bioriented functor $\Sigma_\mB:  \mB^\cop \to \mB$ restricts to $\Sigma_\mC: \mC^\cop \to \mC$. Thus the map $ \Sigma_\mE \circ \bj \to \bj \circ \Sigma_\mB $
induces a map $$ \Sigma_\mE \circ \bj \circ \bi \to \bj \circ \Sigma_\mB \circ \bi \simeq \bj \circ \bi \circ \Sigma_\mC$$ of reduced bioriented functors, which induces on underlying antioriented functors an equivalence since $\bj \circ \bi$ preserves suspensions.
In other words $\Sigma_\mE$ restricts to $\Sigma_\mC= \mC \ot_{\infty\Cat_*} \Sigma $. This proves (7). The second part of (7) is proven similar.

Statements (8) and (9) follow immediately from \cref{alfas} (3), (4), where 
we use that there is a canonical commutative square of bioriented categories and bioriented functors, where
$\tau$ is induced by the non-identity duality on $\bD^1$ and the vertical functors take the pushout along the functor $\bD^1 \to *$:
\begin{equation*}
\begin{xy}
\xymatrix{\infty\Cat_{\bD^1/}   \ar[d] \ar[r]^{\tau \circ (-)^{\co\op}} 
& \infty\Cat_{\bD^1/}   \ar[d]
\\ 
\infty\Cat_\ast \ar[r]^{(-)^{\co\op}} & \infty\Cat_\ast.}
\end{xy}
\end{equation*}
This follows from the existence of a commutative square of left adjoints provided by \cref{left adjoint slice}.

\end{proof}

\begin{proposition}\label{impol} Let $\mC, \mD$ be reduced bioriented categories.

\begin{enumerate}
\item If $\mC,\mD$ admit endomorphisms, there is a transformation of endofunctors of ${\wedge\Fun\wedge}(\mC,\mD):$
$$\F \mapsto \F \circ \Omega_\mC \to \Omega_\mD \circ \F.$$

\item  If $\mC,\mD$ admit antiendomorphisms, there is a transformation of endofunctors of ${\wedge\Fun\wedge}(\mC,\mD):$
$$\F \mapsto \F \circ \bar{\Omega}_\mC \to \bar{\Omega}_\mD \circ \F.$$

\item If $\mC,\mD$ admit suspensions, there is a transformation of endofunctors of ${\wedge\Fun\wedge}(\mC,\mD):$
$$\F \mapsto \Sigma_\mD \circ \F \to \F \circ \Sigma_\mC. $$

\item If $\mC,\mD$ admit antisuspensions, there is a transformation of endofunctors of ${\wedge\Fun\wedge}(\mC,\mD):$
$$\F \mapsto \bar{\Sigma}_\mD \circ \F \to \F \circ \bar{\Sigma}_\mC. $$

\item If $\mC$ admits suspensions and $\mD$ admits endomorphisms,
there is a transformation of endofunctors of ${\wedge\Fun\wedge}(\mC,\mD):$
$$\F \mapsto \F \to\Omega_\mD \circ \F \circ \Sigma_\mC. $$

\item If $\mC$ admits cosuspensions and $\mD$ admits antiendomorphisms,
there is a transformation of endofunctors of ${\wedge\Fun\wedge}(\mC,\mD):$
$$\F \mapsto \F \to\bar{\Omega}_\mD \circ \F \circ \bar{\Sigma}_\mC. $$

\end{enumerate}

\end{proposition}

\begin{proof}
(1) and (2) are similar, (3) and (4) are similar, (5) and (6) are similar.
We prove (1), (3), (5).

We first assume that $\mC, \mD$ are reduced bioriented categories that have suspensions, antisuspensions, endomorphisms and antiendomorphisms.
There is an equivalence $\lambda$ of endofunctors of $\wedge\LinFun{\wedge}(\mC,\mD):$
$$\F \mapsto \Sigma_\mD \circ \F= (\mD \ot_{\infty\Cat_*} \Sigma) \circ \F \simeq (\mD \ot_{\infty\Cat_*} \Sigma) \circ (\F \ot_{\infty\Cat_*} \infty\Cat_*) \simeq $$$$
\F \ot_{\infty\Cat_*} \Sigma \simeq (\F \ot_{\infty\Cat_*} \infty\Cat_*) \circ (\mC \ot_{\infty\Cat_*} \Sigma)\simeq \F \circ (\mC \ot_{\infty\Cat_*} \Sigma)=\F \circ \Sigma_\mC. $$
We obtain a natural transformation $\rho$ of functors $\wedge\LinFun{\wedge}(\mC,\mD) \to \wedge\Fun{\wedge}(\mC,\mD)$
$$\F \mapsto \F \circ \Omega _\mC \to \Omega_\mD \circ \Sigma_\mD \circ \F \circ \Omega _\mC \simeq \Omega_\mD \circ \F \circ  \Sigma_\mC \circ \Omega _\mC \to \Omega_\mD \circ \F.$$

(1): Let $\mC, \mD$ be reduced bioriented categories that admit endomorphisms.
By the universal property of $\gamma(\mC):=\boxtimes\Fun\boxtimes(\mC^\circ, \infty\Cat_* \boxplus \infty\Cat_*)$ of \cite[Proposition 4.40.]{heine2024bienrichedinftycategories} there is a functor $\gamma: {\wedge\Cat \wedge} \to \wedge\widehat{\Cat} \wedge$ and a natural transformation $\id \to \gamma$ of functors $ {\wedge\Cat \wedge} \to {\wedge\widehat{\Cat} \wedge}$
whose component at $\mC$ is the bioriented Yoneda embedding of $\mC$ that preserves endomorphisms.
By definition the reduced bioriented functor $\Omega_\mC$ is the restriction of the reduced bioriented functor
$\Omega_{\gamma(\mC)}$ right adjoint to the reduced bioriented functor
$\Sigma_{\gamma(\mC)}= \gamma(\mC) \ot_{\infty\Cat_*} \Sigma.$

Let $\kappa$ be the induced embedding $${\wedge\Fun\wedge}(\mC, \mD) \subset {\wedge\Fun\wedge}(\mC,\gamma(\mC)) \simeq {{\wedge\L\Fun\wedge}(\gamma(\mC),\gamma(\mD)) }
\subset {{\wedge\Fun\wedge}(\gamma(\mC),\gamma(\mD))}.$$

The transformation $\rho:\F \mapsto \F \circ \Omega_ {\gamma(\mC)} \to  \Omega_ {\gamma(\mD)}  \circ \F$ of functors $$ {\wedge\LinFun\wedge}(\gamma(\mC),\gamma(\mD)) \to {\wedge\Fun\wedge}(\gamma(\mC),\gamma(\mD))$$ gives rise to a transformation $\iota_\mC^* \circ \rho \circ \kappa$ of functors ${\wedge\Fun\wedge}(\mC, \mD) \to {{\wedge\Fun\wedge}(\mC, \gamma(\mD))}:$
$$\F \mapsto \kappa(\F) \circ  \Omega_ {\gamma(\mC)} \circ \iota_\mC \simeq
\kappa(\F) \circ \iota_\mC \circ \Omega_\mC \simeq \iota_\mD \circ \F \circ \Omega_\mC \to \Omega_ {\gamma(\mD)} \circ \kappa(\F) \circ \iota_\mC \simeq  \Omega_ {\gamma(\mD)} \circ \iota_\mD \circ \F \simeq \iota_\mD \circ \Omega_\mD \circ \F $$
that restricts to a natural transformation $\F \mapsto \F \circ \Omega_\mC \to \Omega_\mD \circ \F$ of endofunctors of ${\wedge\Fun\wedge}(\mC, \mD).$

\vspace{1mm}
(3): Let $\mC, \mD$ be reduced bioriented categories that admit suspensions.
Let $\beta(\mC) \subset \gamma(\mC)$ be the full subcategory of bioriented functors $\mC^\circ \to {\infty\Cat_* \boxplus \infty\Cat_*}$ preserving antiendomorphisms.
By \cref{weiloc} the bioriented embedding $\bj_\mC: \beta(\mC) \subset \gamma(\mC)$ admits a left adjoint $L$.
By \cref{weiloc} the bioriented Yoneda embedding $\iota_\mC: \mC \hookrightarrow \gamma(\mC)$ induces a bioriented embedding
$\mC \hookrightarrow \beta(\mC)$ that preserves suspensions and identifies with $L \circ \iota_\mC.$
By definition the reduced bioriented functor $\Sigma_\mC$ is the restriction of the reduced bioriented functor
$\Sigma_{\beta(\mC)}= \beta(\mC) \ot_{\infty\Cat_*} \Sigma$
so that $$ \Sigma_{\beta(\mC)} \circ L \circ \iota_\mC \simeq L \circ \iota_\mC \circ \Sigma_{\mC}.$$
Since $L: \gamma(\mC) \to \beta(\mC)$ is a left adjoint reduced bioriented functor, there is a canonical equivalence
$$ \Sigma_{\beta(\mC)} \circ L \simeq L \circ \Sigma_{\gamma(\mC)}.$$
We obtain a natural transformation $$ \Sigma_{\gamma(\mC)}\circ \bj_\mC \to \bj_\mC \circ L \circ \Sigma_{\gamma(\mC)}\circ \bj_\mC \simeq \bj_\mC \circ \Sigma_{\beta(\mC)} \circ L \circ \bj_\mC \simeq \bj_\mC \circ \Sigma_{\beta(\mC)}$$ of bioriented functors $\beta(\mC)  \to \gamma(\mC)$ 
that gives rise to a natural transformation \begin{equation}\label{huh}
\Sigma_{\gamma(\mC)}\circ \iota_\mC \simeq \Sigma_{\gamma(\mC)}\circ \bj_\mC \circ L \circ \iota_\mC \to \bj_\mC \circ \Sigma_{\beta(\mC)}\circ L \circ \iota_\mC  \simeq \bj_\mC \circ L \circ \iota_\mC \circ \Sigma_{\mC}  \simeq \iota_\mC \circ \Sigma_{\mC} \end{equation} of bioriented functors $\mC  \to \gamma(\mC).$ 
The equivalence $\lambda:\F \mapsto \Sigma_{\gamma(\mD)} \circ \F \simeq \F \circ \Sigma_{\gamma(\mC)} $ of functors $$ {\wedge\LinFun\wedge}(\gamma(\mC),\gamma(\mD)) \to {\wedge\LinFun\wedge}(\gamma(\mC),\gamma(\mD))$$ 
and the transformation \ref{huh} gives a transformation $\zeta $
of functors ${\wedge\Fun\wedge}(\mC, \mD) \to {{\wedge\Fun\wedge}(\mC, \gamma(\mD))}:$
$$\F \mapsto \Sigma_ {\gamma(\mD)} \circ \kappa(\F) \circ \iota_\mC \simeq
\Sigma_ {\gamma(\mD)} \circ \iota_\mD \circ \F \simeq \iota_{\mD} \circ \Sigma_ {\mD} \circ \F
\to \kappa(\F) \circ \Sigma_ {\gamma(\mC)} \circ \iota_\mC \to \kappa(\F) \circ \iota_\mC \circ  \Sigma_ {\mC} \simeq  \iota_\mD \circ \F \circ \Sigma_ {\mC} $$
that restricts to a natural transformation $ \F \mapsto \Sigma_\mD \circ \F \to \F \circ \Sigma_\mC $ of endofunctors of ${\wedge\Fun\wedge}(\mC, \mD).$

\vspace{1mm}
(5): The natural transformation $$\zeta: \F \mapsto \Sigma_ {\gamma(\mD)} \circ \iota_\mD \circ \F \to \iota_\mD \circ \F \circ \Sigma_ {\mC} $$ of functors ${\wedge\Fun\wedge}(\mC, \mD) \to {{\wedge\Fun\wedge}(\mC, \gamma(\mD))}$ gives rise to a natural transformation 
$$\F \mapsto \iota_\mD \circ \F \to \Omega_{\gamma(\mD)} \circ \Sigma_ {\gamma(\mD)} \circ \iota_\mD \circ \F \to \Omega_{\gamma(\mD)} \circ \iota_\mD \circ \F \circ \Sigma_ {\mC}  \simeq \iota_\mD  \circ \Omega_\mD \circ \F \circ \Sigma_ {\mC}$$
of functors ${\wedge\Fun\wedge}(\mC, \mD) \to {{\wedge\Fun\wedge}(\mC, \gamma(\mD))}$
that induces a natural transformation
$$\F \mapsto \F  \to \Omega_\mD \circ \F \circ \Sigma_\mC$$
of functors ${\wedge\Fun\wedge}(\mC, \mD) \to {\wedge\Fun\wedge}(\mC,\mD).$

\end{proof}

\begin{remark}

\cref{reduc} and \cref{impol} generalize with the same proof to generalized reduced bioriented categories.

\end{remark}


\section{Stability of higher categories}

In this section we develop a stable homotopy theory of higher categories.
We introduce stable oriented categories, a higher-categorical analogue of classical stable categories, which provide the natural framework for stability phenomena in higher category theory. Stable oriented categories are dual to antistable antioriented categories.

The universal example is provided by categorical spectra (\cref{catspe}), which form a bistable bioriented category, a compatible stable oriented category and antistable antioriented category. We show that categorical spectra arise as the stabilization of the bioriented category of $\infty$-categories (\cref{qqq}). 
Categorical spectra are an example of spectrum objects in oriented categories and their variants, which we exhibit as stabilizations in this framework (\cref{specunive}).

We construct a monoidal structure on the category of presentable bistable bioriented categories (\cref{stten}), which is a higher-categorical analogue of Lurie’s tensor product of stable presentable categories. The tensor unit is the presentable bistable bioriented category of categorical spectra, which therefore inherits a monoidal structure, the higher-categorical analogue of the smash product, which was also constructed by Masuda \cite[Theorem 4.2.1.]{masuda2026algebra}.

As the main result of this section (\cref{F}), we prove a categorical version of the Freudenthal suspension theorem, which governs the passage from unstable to stable homotopy theory of higher categories.


\subsection{Categorical stability}

In the following we introduce stable oriented categories and their variants, which are higher-categorical analogues of classical stable categories.

\begin{definition}
An oriented category $\mC$ is quasi-stable if the following conditions are satisfied:
	
\begin{enumerate}
\item $\mC$ admits a zero object.
		
\item $\mC$ admits endomorphisms.
 
\item The endomorphisms functor $\Omega: \mC \to \mC $ is an equivalence. 

\end{enumerate}
	
An oriented category is stable if it is quasi-stable and admits oriented fibers and oriented cofibers.

\end{definition}

\begin{remark}
Every quasi-stable oriented category $\mC$ also admits suspensions since $\Omega$ is an equivalence and $\Sigma(X)\simeq \Omega^{-1}(X)$ for any $X \in \mC.$
Hence an oriented category $\mC$ is quasi-stable if and only if the following dual conditions hold:
	
\begin{enumerate}
\item $\mC$ admits a zero object.
	
\item $\mC$ admits suspensions.
	
\item The suspensions functor $\Sigma: \mC \to \mC $ is an equivalence. 
	
\end{enumerate}	
	
\end{remark}

\begin{definition}

An antioriented category $\mC$ is 

\begin{enumerate}
\item quasi-antistable if the oriented category $\mC^\co$ is quasi-stable.

\item antistable if the oriented category $\mC^\co$ is stable.

\end{enumerate}

\end{definition}

\begin{remark}

An antioriented category $\mC$ is quasi-antistable if and only if
it admits a zero object, antiendomorphisms and the antiendomorphisms functor of $\mC$ is an equivalence.
An antioriented category $\mC$ is antistable if and only if it is 
quasi-antistable and admits antioriented fibers and antioriented cofibers.

\end{remark}

\begin{definition}
A bioriented category is
\begin{enumerate}

\item quasi-(anti)stable if it is reduced and its underlying (anti)oriented category is quasi-(anti)stable.	

\item (anti)stable if it is reduced and its underlying (anti)oriented category is (anti)stable.	
	
\item quasi-bistable if it is quasi-stable and quasi-antistable.

\item bistable if it is stable and antistable.

\end{enumerate}
	
\end{definition}

\begin{remark}\label{agos}

\begin{itemize}

\item An (anti)oriented category $\mC$ is quasi-(anti)stable if and only if the (anti)oriented category $\mC^\op$ is quasi-(anti)stable since $\Sigma_{\mC^\op} \simeq \Omega_\mC^\op.$
\item Consequently, a bioriented category $\mC$ is quasi-bistable if and only if $\mC^\op$ is quasi-bistable if and only if $\mC^\co$ is quasi-bistable if and only if $\mC^{\co\op}$ is quasi-bistable.

\item An (anti)oriented category $\mC$ is (anti)stable if and only if $\mC^\op$ is (anti)stable since (anti)oriented cofibers in $\mC$ are (anti)oriented fibers in $\mC^\op$.
\item Consequently, a bioriented category $\mC$ is bistable if and only if $\mC^\op$ is bistable if and only if $\mC^\co$ is bistable if and only if $\mC^{\co\op}$ is bistable.
\end{itemize}
	    
\end{remark}

\begin{example}

Every stable category viewed as a bioriented category is bistable.

\end{example}

For the following notation we use \cref{presol}:

\begin{notation}\emph{}
	
\begin{itemize}

\item Let $$\Pr^L_{\mathrm{st}}\wedge \subset {\Pr^L\wedge}$$ be the full subcategory of stable presentable oriented categories.

\item Let $${\wedge}\Pr^L_\mathrm{anst} \subset \wedge\Pr^L$$ be the full subcategory of antistable presentable antioriented categories.

\item Let $${\wedge\Pr_\mathrm{st}^L\wedge} \subset \wedge\Pr^L\wedge$$ be the full subcategory of stable presentable bioriented categories.

\item Let $$\wedge\Pr^L_{\mathrm{anst}}\wedge \subset \wedge\Pr^L\wedge$$ be the full subcategory of antistable presentable bioriented categories.

\item Let $${\wedge}\Pr^L_{\mathrm{bist}}\wedge \subset \wedge\Pr^L\wedge$$ be the full subcategory of bistable presentable bioriented categories.

\end{itemize}

\end{notation}

We prove next that stability implies preadditivity.

\begin{definition}
An oriented, antioriented, bioriented category is preadditive if it is reduced, admits finite coproducts and finite products and for every $A,B \in \mC$ the canonical morphism $A \coprod B \to A \times B$ induced by the identity of $A,B$ and the zero morphisms, is an equivalence.	
	
\end{definition}

\begin{remark}

An oriented, antioriented, bioriented category is preadditive if and only if it is reduced, admits finite coproducts and finite products and the underlying category is preadditive.
    
\end{remark}


\begin{proposition}\label{preadd}
Every quasi-stable oriented category is preadditive.	
\end{proposition}

\begin{proof}

By quasi-stability for every $X, Y \in \mC$ there is a canonical equivalence of $\infty$-categories
$$\R\Mor_{\mC}(X,Y) \simeq \R\Mor_{\mC}(X,\Omega(\Sigma(Y))) \simeq \Omega(\R\Mor_{\mC}(X,\Sigma(Y)))$$ 
that endows $\R\Mor_{\mC}(X,Y)$ with the structure of a monoidal $\infty$-category
whose tensor unit is the zero morphism.
The tensor product functor $$\R\Mor_{\mC}(X,Y \times Y) \simeq \R\Mor_{\mC}(X,Y) \times \R\Mor_{\mC}(X,Y) \to \R\Mor_{\mC}(X,Y) $$ sends the identity of $Y \times Y$ to a morphism $\mu_Y: Y \times Y \to Y.$
By construction of $\mu_Y$ for every morphism $\alpha: Y \to Z$ the morphism $\alpha \circ \mu_Y$
is equivalent to $\mu_Z \circ (\alpha \times \alpha)$ and for every $X,Y \in \mC$
the morphism $\mu_{X\times Y}$ factors as $$ (X\times Y) \times (X\times Y) \simeq (X\times X) \times (Y\times Y) \xrightarrow{\mu_X \times \mu_Y} X \times Y.$$
Moreover by unitality of the monoidal structure on $\R\Mor_{\mC}(X,Y)$
the composition $$Y \simeq Y \times * \to Y \times X\xrightarrow{\mu_Y} Y$$ is the identity.
By \cite[Proposition 2.4.3.19.]{lurie.higheralgebra} this implies that $\mC$ is preadditive.

\end{proof}

\begin{corollary}
Every quasi-antistable antioriented category is preadditive.	
\end{corollary}

\begin{notation}
Let $\Pr^L \subset \widehat{\Cat}$ be the subcategory of presentable categories and left adjoint functors.	

\end{notation}
\begin{notation}
Let $\Pr^L_{\mathrm{preadd}} \subset \Pr^L$ be the full subcategory of preadditive presentable categories.

\end{notation}

By \cite[Theorem 4.6.]{gepner2016universality} there is a symmetric monoidal localization \begin{equation}\label{loo}\Mon_{\bE_\infty}:\Pr^L \rightleftarrows \Pr^L_{\mathrm{preadd}},\end{equation} where $\Mon_{\bE_\infty}$ assigns to any presentable category $\mD$ the category of commutative monoid objects in $\mD$ and the unit is the free functor.
The localization (\ref{loo}) gives rise to a localization on associative algebras whose  local objects are precisely the preadditive presentably monoidal categories:  \begin{equation}\label{adjass}
\Mon_{\bE_\infty}:\Alg(\Pr^L) \rightleftarrows \Alg(\Pr^L_{\mathrm{preadd}}).\end{equation}
Moreover it gives rise to localizations, where the subscript preadd refers to the full subcategories of preadditive presentable reduced oriented, antioriented, bioriented categories:
$$ \Mon_{\bE_\infty}: \wedge \Pr^L\rightleftarrows \wedge \Pr^L_{\mathrm{preadd}}, \  \Mon_{\bE_\infty}: \Pr^L\wedge \rightleftarrows \Pr^L_{\mathrm{preadd}}\wedge, \  \Mon_{\bE_\infty}: \wedge \Pr^L\wedge \rightleftarrows \wedge {\Pr^L_{\mathrm{preadd}}\wedge} .$$

\begin{corollary}\label{hhoo}\emph{ }
\begin{enumerate}

\item Let $\mC$ be a presentable stable oriented category.
The right adjoint oriented forgetful functor $ \Mon_{\bE_\infty}(\mC) \to \mC $ is an equivalence.
\vspace{1mm}

\item Let $\mC$ be a presentable antistable antioriented category.
The right adjoint antioriented forgetful functor $\Mon_{\bE_\infty}(\mC) \to \mC$ is an equivalence.

\vspace{1mm}
\item Let $\mC$ be a presentable antistable or stable bioriented category.
The right adjoint bioriented forgetful functor $ \Mon_{\bE_\infty}(\mC) \to \mC$ is an equivalence.
	
\end{enumerate}
\end{corollary}


\subsection{Categorical spectra}

Next we define spectrum objects in bioriented categories and show that such provide the stabilization (\cref{specunive}). The prime example are spectrum objects in the bioriented category $\infty\Cat$, which are known as categorical spectra (\cref{catspe}).


\begin{definition}\label{spectraa} Let $\mC$ be a reduced bioriented category that admits endomorphisms.
The bioriented category $\Sp(\mC)$ of spectrum objects in $\mC$ is the limit of the following  diagram of reduced bioriented categories: 
$$...\xrightarrow{\Omega} \mC^\cop \xrightarrow{\Omega} \mC  \xrightarrow{\Omega} \mC^\cop \xrightarrow{\Omega} \mC.$$

\end{definition}

\begin{definition}\label{antispectraa}
	
Let $\mC$ be a reduced bioriented category that admits antiendomorphisms.
The bioriented category of antispectrum objects in $\mC$ is 
$$\overline{\Sp}(\mC):= \Sp(\mC^\co)^\co.$$
\end{definition}

\begin{remark}\label{iwowo} 
By definition there is a canonical equivalence $\bar{\Omega}_\mC \simeq \Omega^\co_{\mC^\co}$ of bioriented functors. Thus the bioriented category $\overline{\Sp}(\mC)$ is the sequential limit of the following  diagram of reduced bioriented categories:
$$...\xrightarrow{\bar{\Omega}} {^\cop\mC} \xrightarrow{\bar{\Omega}} \mC  \xrightarrow{\bar{\Omega}} {^\cop\mC} \xrightarrow{\bar{\Omega}} \mC.$$
\end{remark}

\begin{example}
	
The bioriented categories $\Sp(\infty\Grp), \overline{\Sp}(\infty\Grp)$ are the usual category of spectra viewed as a bioriented category.	

\end{example}

\begin{remark}

By \cref{ember} every weakly reduced oriented category gives rise to a weakly reduced bioriented category, which is reduced and 
admits endomorphisms if the original oriented category is reduced and admits endomorphisms (\cref{indubi}).
Hence we can apply \cref{spectraa} to every reduced oriented category that admits endomorphisms to obtain an oriented category of spectrum objects.

Similarly, we can apply \cref{antispectraa} to every reduced antioriented category that admits antiendomorphisms to obtain an antioriented category of antispectrum objects.

\end{remark}

\begin{definition}\label{catspe}
The bioriented category of categorical spectra is $$\Cat\Sp := \Sp(\infty\Cat_*).$$	
\end{definition}
	
\begin{proposition}\label{spantisp}
There is a canonical equivalence of bioriented categories $$ \Cat\Sp=\Sp(\infty\Cat_*) \simeq \overline{\Sp}(\infty\Cat_*).$$

\end{proposition}

\begin{proof}
By \cref{reduc} there is a canonical equivalence of bioriented functors $$ \Omega^2 \simeq \bar{\Omega}^2: \infty\Cat_* \to \infty\Cat_*.$$ Consequently, there is a canonical equivalence of bioriented categories $$ \Sp(\infty\Cat_*) \simeq \lim(... \xrightarrow{\Omega^2} \mC \xrightarrow{\Omega^2}\mC)  \simeq \lim(... \xrightarrow{\bar{\Omega}^2} \mC \xrightarrow{\bar{\Omega}^2}\mC) \simeq \overline{\Sp}(\infty\Cat_*),$$
where the first and third equivalence are by cofinality.

\end{proof}

\begin{remark}\label{nnm}
Let $\mC$ be a reduced presentable bioriented category. Then the bioriented functor $\Omega$ admits a left adjoint.
This implies that the bioriented category $\Sp(\mC)$ is presentable.
This follows from the fact that the subcategory $\wedge\Pr^R\wedge $ of presentable bioriented categories and right adjoint bioriented functors admits small limits preserved by the inclusion $\wedge\Pr^R\wedge \subset {\wedge\widehat{\Cat}\wedge}.$
\end{remark}

\begin{remark}\label{nnmt}
Via the equivalence $\wedge\Pr^L\wedge \simeq (\wedge\Pr^R\wedge)^\op$ and \cref{nnm} the bioriented category $\Sp(\mC)$ is also the colimit in $\wedge\Pr^L\wedge$ of the sequential diagram $$\mC \xrightarrow{\Sigma} \mC^\cop \xrightarrow{\Sigma} \mC \xrightarrow{\Sigma}....$$
\end{remark}

\begin{corollary}\label{sptens}
Let $\mC$ be a presentable reduced bioriented category.
There are canonical equivalences of reduced bioriented categories $$\Sp(\mC) \simeq \mC \ot_{\infty\Cat_*} \Cat\Sp, \ \overline{\Sp}(\mC)\simeq \Cat\Sp \ot_{\infty\Cat_*} \mC.$$

\end{corollary}

\begin{proof}
By \cref{reduc} the reduced bioriented functor $\Sigma: \mC \to \mC$ factors as
$$\mC \ot_{\infty\Cat_*} \Sigma: \mC \simeq \mC \ot_{\infty\Cat_*} \infty\Cat_* \to \mC \ot_{\infty\Cat_*} \infty\Cat_* \simeq \mC$$ and $\bar{\Sigma}: \mC \to \mC$ factors as
$$\bar{\Sigma} \ot_{\infty\Cat_*} \mC : \mC \simeq \infty\Cat_* \ot_{\infty\Cat_*} \mC \to \infty\Cat_* \ot_{\infty\Cat_*} \mC\simeq \mC.$$
Hence there are canonical equivalences of reduced bioriented categories $$\Sp(\mC) \simeq \mC \ot_{\infty\Cat_*} \Sp(\infty\Cat_*), \ \overline{\Sp}(\mC)\simeq \overline{\Sp}(\infty\Cat_*) \ot_{\infty\Cat_*} \mC.$$
We apply \cref{spantisp}.
\end{proof}

\begin{notation}Let $\mC$ be a reduced bioriented category that admits endomorphisms.
The reduced bioriented functor of infinite endomorphisms $$\Omega^\infty: \Sp(\mC) \to \mC$$ is the canonical bioriented functor projecting to the value of the final object of the diagram.

\end{notation}

\begin{notation}\label{infsuspi}
Let $\mC$ be a reduced presentable bioriented category. The bioriented category $\Sp(\mC)$ is presentable by \cref{nnm} and the bioriented functor $\Omega^\infty: \Sp(\mC) \to \mC$ is accessible, preserves small limits, and left and right cotensors. Thus $\Omega^\infty: \Sp(\mC) \to \mC$ admits a bioriented left adjoint $\Sigma^\infty: \mC \to \Sp(\mC)$
by \cref{adj}, which we call infinite suspension.

\end{notation}

\begin{remark}Let $\mC$ be a reduced presentable bioriented category such that endomorphisms preserve small filtered colimits.
The bioriented functor $\Omega^\infty: \Sp(\mC) \to \mC $ preserves small filtered colimits so that the bioriented left adjoint $\Sigma^\infty: \mC \to  \Sp(\mC) $ preserves compact objects. 	
	
\end{remark}

\begin{lemma}Let $\mC$ be a bioriented category that admits a final object and the coproduct of any object with the final object.
The bioriented forgetful functor $\mC_* \to \mC$ admits a bioriented left adjoint $(-)_+$ that sends $X$ to $ X \coprod *.$ 
	
\end{lemma}

\begin{proof}
By \cref{weiloc} there is a bioriented embedding of $\mC$ into a presentable bioriented category preserving the final object and the coproduct of any object with the final object. So we can assume that $\mC$ is a presentable bioriented category. In this case the statement follows from \cref{left adjoint slice}.
	
\end{proof}

\begin{notation}Let $\mC$ be a presentable bioriented category. Then by \cref{infsuspi} the forgetful bioriented functor $\Sp(\mC) \to \mC_*$ admits a bioriented left adjoint $\Sigma^\infty.$
We define the bioriented adjunction $$\Sigma^\infty_+:= \Sigma^\infty \circ (-)_+: \mC \rightleftarrows  \mC_* \rightleftarrows \Sp(\mC):\Omega^\infty.$$
	
\end{notation}

\begin{remark}\label{comprestri}
Let $\mC$ be a reduced compactly generated bioriented category.
Since the bioriented functor $\Omega: \infty\Cat^\cop \to \infty\Cat $ preserves small filtered colimits by \cref{endfil}, the bioriented functor $\Omega: \mC^\cop \to \mC $ preserves small filtered colimits. Thus the bioriented left adjoint $\Sigma: \mC \to \mC^\cop $ preserves compact objects and so restricts to a functor $\Sigma: \mC^\omega \to (\mC^\omega)^\cop $ on compact objects. 

\end{remark}

\subsubsection{The categorical Spanier-Whitehead category}

We have the following categorical version of Spanier-Whitehead category, where we use \cref{comprestri}:

\begin{definition}

Let $\mC$ be a reduced compactly generated bioriented category.
The bioriented Spanier-Whitehead category $ \rS\W(\mC) $ of $\mC$ is the colimit of the following diagram of small reduced bioriented categories:
$$\mC^\omega \xrightarrow{\Sigma} (\mC^\omega)^\cop \xrightarrow{\Sigma} .... $$ 

\end{definition}
	
\begin{proposition}\label{presi} Let $\mC$ be a reduced compactly generated bioriented category.
There is a canonical equivalence of bioriented categories:
$$ \Sp(\mC) \simeq \Ind(\rS\W(\mC)). $$ 

In particular, the bioriented category $\Sp(\mC)$ is compactly generated.


\end{proposition}

\begin{proof}\label{pres}


Since $ \mC^\omega$ admits finite colimits and $\Sigma: \mC^\omega \to \mC^\omega$ preserves finite colimits, the filtered colimit $\rS\W(\mC)$ admits finite colimits and is the filtered colimit in the subcategory $\wedge\Cat^\mathrm{rex}\wedge \subset \wedge\Cat\wedge$ of bioriented categories that admit finite colimits and bioriented functors preserving finite colimits.

Since the functor $\Ind: {\wedge \Cat^\mathrm{rex}\wedge} \to \wedge\Pr^L\wedge$ preserves small filtered colimits, the bioriented category $\Ind(\rS\W(\mC))$ is the sequential colimit of the diagram of bioriented categories
$$\mC \simeq \Ind(\mC^\omega) \xrightarrow{\Sigma} \Ind(\mC^\omega)^\cop \simeq \mC^\cop \xrightarrow{\Sigma} ... $$ in $\wedge\Pr^L\wedge$, and so  by \cref{nnmt} agrees with the bioriented category $\Sp(\mC)$.


\end{proof}

\subsubsection{Stability of categorical spectra}

Next we prove that the bioriented category of categorical spectra is bistable.

\begin{proposition}\label{rstab}
Let $\mC$ be a reduced bioriented category that admits endomorphisms.
The functor $\Omega: \Sp(\mC) \to \Sp(\mC)$ is equivalent to the evident induced functor on limits $$ \lim(... \xrightarrow{\Omega}\mC \xrightarrow{\Omega} \mC) \to \lim(... \xrightarrow{\Omega}\mC \xrightarrow{\Omega} \mC), \ (X_0, X_1,X_2, ...) \mapsto (X_{-1}, X_0,X_1, ...),$$
which is inverse to the shift functor $$ \lim(... \xrightarrow{\Omega}\mC \xrightarrow{\Omega} \mC) \to \lim(... \xrightarrow{\Omega}\mC \xrightarrow{\Omega} \mC), \ (X_0, X_1,X_2, ...) \mapsto (X_1, X_2,X_3, ...).$$
In particular, the bioriented category $\Sp(\mC)$ is quasi-stable.
	
\end{proposition}

\begin{proof}
The functor $\Omega :\Sp(\mC) \to \Sp(\mC)$ is the functor on limits induced by the map of diagrams:
$$\begin{xy}
\xymatrix{
 \ar[d]^{\Omega}
& \ar[d]^{\Omega}
\\ 	
\mC \ar[d]^{\Omega} \ar[r]^{\Omega}
& \mC \ar[d]^{\Omega}
\\ 
\mC \ar[d]^{\Omega} \ar[r]^{\Omega}
& \mC \ar[d]^{\Omega}
\\ 
\mC \ar[r]^{\Omega} & \mC,
}
\end{xy}$$
where each square commutes via the equivalence ${\Omega} \circ {\Omega} \simeq {\Omega} \circ {\Omega}$ 
induced by the equivalence $$\tau: S^1 \wedge S^1 \simeq \Sigma(S^1) \simeq \bar{\Sigma}((S^1)^{\co\op})
\simeq \bar{\Sigma}(S^1) \simeq S^1 \wedge S^1 $$ of \cref{ahoit}.
By \cref{autoeq} every auto-equivalence of $S^1 \wedge S^1$ is the identity so that $\tau$ is the identity.
	
\end{proof}

\begin{corollary}\label{crstab} Let $\mC$ be a reduced bioriented category that admits antiendomorphisms.
The bioriented category $\overline{\Sp}(\mC)$ is quasi-antistable.
	
\end{corollary}

\cref{spantisp} and \cref{crstab} give the following:

\begin{corollary}\label{stab}
The presentable bioriented category $\Cat\Sp$ is bistable.

\end{corollary}

\begin{proposition}\label{infsuso}

Let $\mC$ be a compactly generated bioriented category.

For every $Y \in \mC$ there is a canonical equivalence
$$\colim_{n \geq 0} \Omega^n(\Sigma^n(Y)) \to \Omega^\infty(\Sigma^\infty(Y)). $$

\end{proposition}

\begin{proof}

The canonical morphism $$ \Omega^n(\Sigma^n(Y)) \to \Omega^n(\Omega^\infty(\Sigma^\infty(\Sigma^n(Y)))) \simeq \Omega^\infty(\Omega^n(\Sigma^n(\Sigma^\infty(Y)))) \simeq \Omega^\infty(\Sigma^\infty(Y)),$$
which uses \cref{rstab},
induces a canonical morphism $$\colim_{n \geq 0} \Omega^n(\Sigma^n(Y)) \to \Omega^\infty(\Sigma^\infty(Y)). $$

By Yoneda it suffices to prove that for every $X, Y \in \mC$
the canonical map
$$\Map_{\mC}(X,\colim_{n \geq 0} \Omega^n(\Sigma^n(Y))) \to \Map_{\mC}(X,\Omega^\infty(\Sigma^\infty(Y)))  $$ 
is an equivalence.
Since the functors $\Omega^n, \Omega^\infty$ both preserve small filtered colimits and $\mC$ is generated under small filtered colimits by compact objects, we can assume that $X, Y$ are compact.


The canonical map $$\Map_{\mC}(X,\colim_{n \geq 0} \Omega^n(\Sigma^n(Y))) \to \Map_{\mC}(X,\Omega^\infty(\Sigma^\infty(Y)))  $$ identifies with the following canonical equivalence, which is by \cref{presi}:
$$\Map_{\mC}(X,\Omega^\infty(\Sigma^\infty(Y))) \simeq \Map_{\Sp(\mC)}(\Sigma^\infty(X),\Sigma^\infty(Y)) \simeq $$$$ \lim_{n \geq 0} \Map_{\mC}(\Sigma^n(X),\Sigma^n(Y)) \simeq \lim_{n \geq 0} \Map_{\mC}(X,\Omega^n(\Sigma^n(Y))).$$ 

\end{proof}

\subsubsection{Connective categorical spectra}

Next we consider connective categorical spectra.

\begin{definition}

A categorical spectrum $\mC$ is connective if for every $\n \geq 0$ the $\infty$-category
$\Omega^\infty(\Sigma^\n(\mC)) $ is $\n$-1-connected in the sense of \cref{conn}.

\end{definition}

\begin{notation}
Let $\Cat\Sp_{\geq0} \subset \Cat\Sp $ be the full subcategory of connective categorical spectra.	

\end{notation}

\begin{example}\label{connex}
	
Let $\mC$ be a symmetric monoidal $\infty$-category. There is a  categorical spectrum $$B^\infty(\mC):= (\mC, B\mC, B^2\mC, ...)$$ with the natural bonding maps.
The categorical spectrum $ B^\infty(\mC)$ is connective.
	
\end{example}

The next result is a categorical version of May's recognition principle identifying infinite loop spaces with grouplike $\bE_\infty$-spaces \cite{may1974}:

\begin{proposition}\label{symsp0}

The induced functor \begin{equation}\label{eqojj}
\Cat\Sp \simeq \Mon_{\bE_\infty}(\Cat\Sp) \xrightarrow{\Mon_{\bE_\infty}(\Omega^\infty)} \Mon_{\bE_\infty}(\infty\Cat) \end{equation}
admits a left adjoint that sends a symmetric monoidal $\infty$-category $\mC$ to $B^\infty(\mC)$ and induces an equivalence
$$ \Mon_{\bE_\infty}(\infty\Cat) \simeq \Cat\Sp_{\geq0}.$$

\end{proposition}	

\begin{proof}
By preadditivity of $\Cat\Sp$ of \cref{preadd} there is a canonical equivalence: $$ \Cat\Sp \simeq \Mon_{\bE_\infty}(\Cat\Sp) \simeq \lim(... \xrightarrow{\Omega} \Mon_{\bE_\infty}(\infty\Cat) \xrightarrow{\Omega} \Mon_{\bE_\infty}(\infty\Cat)).$$

Hence the functor (\ref{eqojj}) identifies with the projection $$\Cat\Sp \simeq \Mon_{\bE_\infty}(\Cat\Sp) \simeq \lim(... \xrightarrow{\Omega} \Mon_{\bE_\infty}(\infty\Cat) \xrightarrow{\Omega} \Mon_{\bE_\infty}(\infty\Cat)) \to \Mon_{\bE_\infty}(\infty\Cat).$$


For every $\mC \in \Mon_{\bE_\infty}(\infty\Cat)$ and $(\mD_0,\mD_1,...) \in \Cat\Sp$ the induced functor
$$ \Map_{\Cat\Sp}(B^\infty(\mC),\mD) \to \Map_{\Mon_{\bE_\infty}(\infty\Cat)}(\mC,\mD_0)$$
factors as $$\lim(... \to \Map_{\Mon_{\bE_\infty}(\infty\Cat)}(B\mC,\mD_1) \to \Map_{\Mon_{\bE_\infty}(\infty\Cat)}(\mC,\mD_0)) \to \Map_{\Mon_{\bE_\infty}(\infty\Cat)}(\mC,\mD_0),$$
which by stability (\cref{stab}) identifies with the canonical equivalence 
$$\lim(... \to \Map_{\Mon_{\bE_\infty}(\infty\Cat)}(\mC,\mD_0) \to \Map_{\Mon_{\bE_\infty}(\infty\Cat)}(\mC,\mD_0)) \to \Map_{\Mon_{\bE_\infty}(\infty\Cat)}(\mC,\mD_0)$$
projecting from the limit of the constant tower. 
By \cref{connex} the resulting adjunction $$ \Mon_{\bE_\infty}(\infty\Cat) \rightleftarrows \Mon_{\bE_\infty}(\Cat\Sp) \simeq \Cat\Sp: \Mon_{\bE_\infty}(\Omega^\infty)$$ restricts to an adjunction
$\Mon_{\bE_\infty}(\infty\Cat) \rightleftarrows \Cat\Sp_{\geq0}.$
The left adjoint of the latter adjunction is an equivalence since the right adjoint is conservative.

\end{proof}

\subsubsection{Examples of categorical spectra}

Next we consider examples of categorical spectra.
The next example is \cite[Example 13.3.12.]{HigherQuasi}:

\begin{example}

Let $\mC$ be a symmetric monoidal category compatible with
geometric realizations.

Haugseng \cite{haugseng2017} constructs for every $\n \geq 0$ an $(\n+ 2)$-category $\mathrm{Morita}^{n}(\mC)$ of $\bE_{\n+1}$-algebras in $\mC$ whose objects are $\bE_{\n+1}$-algebras in $\mC$ and whose morphisms between
$\bE_{\n+1}$-algebras $A,B$ in $\mC$ are $\bE_{\n}$-algebras in $(A,B)$-bimodules in $\mC$, and so on.
The latter are non-unital after taking $\n+1$-fold morphism objects.

So
$\mathrm{Morita}^0(\mC)$
is the usual Morita 2-category of $\mC$.
We view $\mathrm{Morita}^\n(\mC)$ as $\infty$-category with distinguished object, where the distinguished object is the tensor unit of $\mC.$

By \cite[Corollary 5.51]{haugseng2017} 
for every $\n \geq 0$ there is a canonical equivalence $$ \Omega(\mathrm{Morita}^\n(\mC)) \simeq \mathrm{Morita}^{\n-1}(\mC)$$
of $\infty$-categories with distinguished object.
Hence there is a categorical spectrum
$$ \mathrm{morita}(\mC) := \{\mathrm{Morita}^\n(\mC)\}, $$
which we call the Morita spectrum of $\mC$. 

\end{example}

	

\begin{definition}\label{Brau} Let $R$ be an $\bE_\infty$-ring spectrum.
The Morita spectrum $\mathrm{morita}(R) $ of $R$ is the Morita spectrum of $\Mod_R.$	
	
\end{definition}

Every spectrum of $\infty$-categories gives rise to a spectrum of homotopy types.

\begin{notation}For every category $\mD$ that admits finite products let $$\Grp_{\bE_\infty}(\mD) \subset \Mon_{\bE_\infty}(\mD)$$ be the full subcategory spanned by the group objects.
	
\end{notation}

\begin{remark}\label{ahoo}
	
For every additive category $\mD$ the forgetful functor $\Grp_{\bE_\infty}(\mD)\to \mD$ is an equivalence \cite[Proposition 2.8.]{gepner2016universality}.
\end{remark}

\begin{definition}
	
The embedding $$ \Grp_{\bE_\infty}(\infty\Grp)  \subset  \Mon_{\bE_\infty}(\infty\Grp) $$ admits a right adjoint that sends an $\bE_\infty$-space to the full $\bE_\infty$-subspace of invertible objects.
The Picard space functor $\Pic$ is the following composition of the functor induced by $\iota_0$ and the right adjoint: $$\Pic: \Mon_{\bE_\infty}(\infty\Cat) \xrightarrow{\Mon_{\bE_\infty}(\iota_0)} \Mon_{\bE_\infty}(\infty\Grp) \to \Grp_{\bE_\infty}(\infty\Grp).$$	
	
\end{definition}

For the next definition we use \cref{ahoo} and \cref{hhoo}:
\begin{definition}
The Picard spectrum functor $\mathrm{pic}$ is the following composition
$$ \Cat\Sp \simeq \Mon_{\bE_\infty}(\Cat\Sp) \simeq  \Sp(\Mon_{\bE_\infty}(\infty\Cat)) \xrightarrow{\Sp(\Pic)}$$$$ \Sp(\Grp_{\bE_\infty}(\infty\Grp)) \simeq \Grp_{\bE_\infty}(\Sp(\infty\Grp)) \simeq \Sp(\infty\Grp).$$
	
\end{definition}

\begin{lemma}
	
There is a colocalization $ \Sp(\infty\Grp)\rightleftarrows \Cat\Sp: \mathrm{pic},$
where the left adjoint is the embedding.
	
\end{lemma}

\begin{proof}
The colocalization $\infty\Grp \rightleftarrows \infty\Cat : \iota_0 $ gives rise to a colocalization $$\Mon_{\bE_\infty}(\infty\Grp) \rightleftarrows \Mon_{\bE_\infty}(\infty\Cat) : \iota_0.$$
By the limit definition of spectra the composition of colocalizations $$\Grp_{\bE_\infty}(\infty\Grp) \rightleftarrows \Mon_{\bE_\infty}(\infty\Grp) \rightleftarrows \Mon_{\bE_\infty}(\infty\Cat): \Pic\circ \iota_0$$ 
gives rise to a colocalization $$ \Sp(\Grp_{\bE_\infty}(\infty\Grp)) \simeq \Grp_{\bE_\infty}(\Sp(\infty\Grp)) \simeq \Mon_{\bE_\infty}(\Sp(\infty\Grp)) \simeq \Sp(\infty\Grp)\rightleftarrows $$$$ \Sp(\Mon_{\bE_\infty}(\infty\Cat)) \simeq \Mon_{\bE_\infty}(\Cat\Sp) \simeq \Cat\Sp: \mathrm{pic}.$$

\end{proof}

\begin{example}
The Picard spectrum of the Morita spectrum of $R$ of \cref{Brau} is the Brauer spectrum of \cite[§3]{antieau2014brauer}.
In particular, $\pi_0(\mathrm{pic}(\mathrm{morita}(R)))$ is the Brauer group of $R.$

Concretely, $\pi_0(\mathrm{pic}(\mathrm{morita}(R)))$ is the
set of equivalence classes of tensor-invertible objects in the Morita 2-category of $\Mod_R,$ which is precisely the set of
Morita equivalence classes of Azumaya algebras over $R.$

	
\end{example}







\subsection{Stabilization}

Next we exhibit spectrum objects as the universal stabilization.
We first exhibit the bioriented category of spectrum objects as the stabilization in the setting of bioriented categories.

\begin{proposition}\label{qqq}\label{qqccq} Let $\mC$ be a reduced presentable bioriented category.

\begin{enumerate}
\item The left adjoint bioriented functor $$\Sigma^\infty: \mC \to \Sp(\mC)$$ induces for every stable presentable bioriented category $\mD$ an equivalence:
$${\wedge\L\Fun\wedge}(\Sp(\mC), \mD) \to {\wedge\L\Fun\wedge}(\mC, \mD).$$ 
\item The left adjoint bioriented functor $$\Sigma^\infty: \mC \to \overline{\Sp}(\mC)$$ induces for every antistable presentable bioriented category $\mD$ an equivalence:$${\wedge\L\Fun\wedge}(\overline{\Sp}(\mC), \mD) \to {\wedge\L\Fun\wedge}(\mC, \mD).$$ 
	
\item The left adjoint bioriented functor $$ \mC \to \overline{\Sp}(\Sp(\mC)) \simeq \Sp(\overline{\Sp}(\mC))$$ induces for every bistable presentable bioriented category $\mD$ an equivalence:
$${\wedge\L\Fun\wedge}(\overline{\Sp}(\Sp(\mC)), \mD) \to {\wedge\L\Fun\wedge}(\mC, \mD).$$ 	
\end{enumerate}

\end{proposition}

\begin{proof}(1): The functor of the statement factors as $${\wedge\L\Fun\wedge}(\Sp(\mC), \mD) \simeq \lim (... \xrightarrow{\Sigma^2_*} {\wedge\L\Fun\wedge}(\mC, \mD) \xrightarrow{\Sigma^2_*} {\wedge\L\Fun\wedge}(\mC, \mD)) \xrightarrow{\gamma} {\wedge\L\Fun\wedge}(\mC, \mD),$$
where $\gamma$ is the projection to the rightmost object of the diagram.
If $\mD$ is stable, the functor $$\Sigma^2_*: {\wedge\L\Fun\wedge}(\mC, \mD) \to {\wedge\L\Fun\wedge}(\mC, \mD)$$ is an equivalence so that $\gamma$ is an equivalence.	

(2) follows from (1) by replacing $\mC$ by $\mC^\co.$
(3): The functor of the statement factors as equivalences $${\wedge\L\Fun\wedge}(\overline{\Sp}(\Sp(\mC)), \mD) \to {\wedge\L\Fun\wedge}(\Sp(\mC), \mD) \to {\wedge\L\Fun\wedge}(\mC, \mD).$$

\end{proof}

Next we use the stabilization in the setting of bioriented categories
to construct a tensor product of stable presentable bioriented categories (Theorem \ref{stten}).
This tensor product gives rise to a monoidal structure on the bioriented category of categorical spectra
(\cref{unicity}).

\begin{definition}\label{stmon}

A presentably monoidal reduced bioriented category is a presentably monoidal category $X$ equipped with a left adjoint monoidal functor $\phi: \infty\Cat_* \to \mC.$

A left adjoint monoidal reduced bioriented functor $\mC \to \mD$ between presentably monoidal reduced bioriented categories is a left adjoint monoidal functor $\mC \to \mD$ under $ \infty\Cat_*.$
	
\end{definition}

\begin{remark}
Any presentably monoidal reduced bioriented category $(\mC, \phi)$ has an underlying reduced presentable bioriented category by restricting the canonical biaction of $\mC$ over $\mC, \mC$ along $\phi.$

Similarly, every left adjoint monoidal reduced bioriented functor between presentably monoidal reduced bioriented categories has an underlying left adjoint reduced bioriented functor.

\end{remark}

\begin{definition}
	
A presentably monoidal reduced bioriented category is stable, antistable, bistable, respectively, if its underlying presentable reduced bioriented category is
stable, antistable, bistable.
\end{definition}

The category $ \wedge \Pr^L\wedge$ carries a monoidal structure via the relative tensor product of $\infty\Cat_*, \infty\Cat_*$-bimodules in $\Pr^L$ \cite[\S 4.4.]{lurie.higheralgebra}.
By \cite[Corollary 3.4.1.7.]{lurie.higheralgebra} the category of associative algebras in $ \wedge \Pr^L\wedge$ is equivalent to the category of associative algebras in $\Pr^L$ under $\infty\Cat_*$, which is the category of presentably monoidal categories under $\infty\Cat_*$,
i.e. presentably monoidal reduced bioriented categories.

\begin{theorem}\label{stten}

There is a monoidal structure on the $\infty$-category ${\wedge}\Pr^L_{\mathrm{st}}\wedge$ 
of presentable bistable bioriented categories whose tensor unit is $\Sp$
such that the embedding ${\wedge}\Pr^L_{\mathrm{st}}\wedge \subset \wedge \Pr^L\wedge$ refines to a lax monoidal embedding.
\end{theorem}

\begin{proof}

By \cref{qqq} (3) there is a localization $L$ of $\wedge\Pr^L\wedge$ whose local objects are the bistable bioriented categories, that sends $\mC $ to $$\overline{\Sp}(\Sp(\mC))=\Cat\Sp \ot_{\infty\Cat_*} \mC \ot_{\infty\Cat_*} \Cat\Sp.$$
This implies by \cite[Lemma 3.13.]{heine2024local} that the full subcategory ${\wedge}\Pr^L_{\mathrm{st}}\wedge$ of local objects inherits an oplax monoidal structure such that the embedding ${{\wedge}\Pr^L_{\mathrm{st}}\wedge} \subset \wedge\Pr^L\wedge$ is a lax monoidal embedding and such that the local tensor product of $\mC_1,...,\mC_\n \in {\wedge}\Pr^L_{\mathrm{st}}\wedge$ for $\n \geq 0$ is $L(\mC_1 \ot_{\infty\Cat_*} ... \ot_{\infty\Cat_*} \mC_\n).$
For $\n \geq 1$ we have that $$ L(\mC_1 \ot_{\infty\Cat_*} ... \ot_{\infty\Cat_*} \mC_\n) \simeq \Cat\Sp \ot_{\infty\Cat_*} \mC_1 \ot_{\infty\Cat_*} ... \ot_{\infty\Cat_*} \mC_\n \ot_{\infty\Cat_*} \Cat\Sp \simeq \mC_1 \ot_{\infty\Cat_*} ... \ot_{\infty\Cat_*} \mC_\n.$$
For $\n=0$ we have that $$L(\infty\Cat_*)\simeq \Cat\Sp \ot_{\infty\Cat_*} \Cat\Sp \simeq \Sp(\Cat\Sp)\simeq \Sp.$$
So to see that this oplax monoidal structure on ${\wedge}\Pr^L_{\mathrm{st}}\wedge$ is a monoidal structure it is enough to see that for every $\mC^1_1,...,\mC^1_{\n_1}, \mC^2_1,...,\mC^2_{\n_2},..., \mC^\ell_1,...,\mC^\ell_{\n_\ell} \in {\wedge}\Pr^L_{\mathrm{st}}\wedge$ for $\ell \geq 2, \n_1, ...,\n_\ell \geq 0$ the canonical morphism
$$\theta: L(\mC^1_1 \ot_{\infty\Cat_*} ... \ot_{\infty\Cat_*} \mC^1_{\n_1} \ot_{\infty\Cat_*} ... \ot_{\infty\Cat_*} \mC^\ell_1 \ot_{\infty\Cat_*} ... \ot_{\infty\Cat_*} \mC^\ell_{\n_\ell}) \to $$$$ L(L(\mC^1_1 \ot_{\infty\Cat_*} ... \ot_{\infty\Cat_*} \mC^1_{\n_1}) \ot_{\infty\Cat_*} ... \ot_{\infty\Cat_*} L(\mC^\ell_1 \ot_{\infty\Cat_*} ... \ot_{\infty\Cat_*} \mC^\ell_{\n_\ell}))$$
is an equivalence. 
Assume first that $ \n_\bj=0$ for every $0 \leq \bj \leq \ell.$
In this case $\theta$ factors as $$\Cat\Sp \simeq \Cat\Sp \ot_{\infty\Cat_*} \infty\Cat_*^{\ot_{\infty\Cat_*}\ell} \ot_{\infty\Cat_*} \Cat\Sp \to \Cat\Sp \ot_{\infty\Cat_*} \Cat\Sp^{\ot_{\infty\Cat_*} \ell} \ot_{\infty\Cat_*} \Cat\Sp \simeq \Cat\Sp$$ and so is the identity.
So we assume that $ \n_\bj \geq 1 $ for some $0 \leq \bj \leq \ell.$
In this case $\theta$ identifies with
$$\mC^1_1 \ot_{\infty\Cat_*} ... \ot_{\infty\Cat_*} \mC^1_{\n_1} \ot_{\infty\Cat_*} ... \ot_{\infty\Cat_*} \mC^\ell_1 \ot_{\infty\Cat_*} ... \ot_{\infty\Cat_*} \mC^\ell_{\n_\ell} \to $$$$ L(\mC^1_1 \ot_{\infty\Cat_*} ... \ot_{\infty\Cat_*} \mC^1_{\n_1}) \ot_{\infty\Cat_*} ... \ot_{\infty\Cat_*} L(\mC^\ell_1 \ot_{\infty\Cat_*} ... \ot_{\infty\Cat_*} \mC^\ell_{\n_\ell}).$$
For every $1 \leq \bj \leq \ell$ the reduced bioriented functor $\mC^\bj_1 \ot_{\infty\Cat_*} ... \ot_{\infty\Cat_*} \mC^\bj_{\n_\bj} \to L(\mC^\bj_1 \ot_{\infty\Cat_*} ... \ot_{\infty\Cat_*} \mC^\bj_{\n_\bj})$ is an equivalence if $\n_\bj \geq 1$ as in this case $\mC^\bj_1 \ot_{\infty\Cat_*} ... \ot_{\infty\Cat_*} \mC^\bj_{\n_\bj}$ is local.
If $\n_\bj=0$, the reduced bioriented functor $$\mC^\bj_1 \ot_{\infty\Cat_*} ... \ot_{\infty\Cat_*} \mC^\bj_{\n_\bj} \to L(\mC^\bj_1 \ot_{\infty\Cat_*} ... \ot_{\infty\Cat_*} \mC^\bj_{\n_\bj})$$ is the canonical reduced bioriented functor $\infty\Cat_* \to \Cat\Sp$.
So by induction on $\ell$ it is enough to see that for every bistable presentable bioriented category $\mC$
the canonical reduced bioriented functors $$\mC \simeq \infty\Cat_* \ot_{\infty\Cat_*} \mC \to \Cat\Sp \ot_{\infty\Cat_*} \mC, \mC \simeq \mC \ot_{\infty\Cat_*} \infty\Cat_* \to \mC \ot_{\infty\Cat_*} \Cat\Sp$$ are equivalences. The latter identify with the stabilization functors and so are equivalences.
\end{proof}

\begin{corollary}\label{unicity}
	
The category of bistable presentably monoidal bioriented categories admits an initial object whose underlying bistable bioriented category is $\Cat\Sp$, and the bioriented functor $$\Sigma^\infty: \infty\Cat \to \Cat\Sp$$ is monoidal.

\end{corollary}

\begin{proof}

For every monoidal category the category of associative algebras admits an initial object that lies over the tensor unit \cite[Proposition 3.2.1.8., Lemma 3.2.1.10.]{lurie.higheralgebra}. We apply this to the monoidal category ${\wedge}\Pr^L_{\mathrm{st}}\wedge$ of \cref{stten}
and use that the embedding ${\wedge}\Pr^L_{\mathrm{st}}\wedge \subset {\wedge}\Pr^L\wedge$ is lax monoidal.

\end{proof}






Next we exhibit the oriented category of spectrum objects as the stabilization in the setting of oriented categories (\cref{specunive}).
For that we introduce spectral oriented categories.


\begin{definition}\emph{}

\begin{itemize}

\item A spectral oriented category is a category enriched in $(\Cat\Sp,\wedge^\rev).$

\item A spectral oriented functor is a functor enriched in $(\Cat\Sp,\wedge^\rev).$

\item An antispectral antioriented category is a category enriched in $(\Cat\Sp,\wedge).$

\item An antispectral antioriented functor is a functor enriched in $(\Cat\Sp,\wedge).$

\item A spectral bioriented category is a category enriched in $(\infty\Cat_*,\wedge) \ot (\Cat\Sp,\wedge^\rev)$.

\item A spectral bioriented functor is a functor enriched in $(\infty\Cat_*,\wedge), (\Cat\Sp,\wedge^\rev)$.
	
\item An antispectral bioriented category is a category enriched in $(\Cat\Sp,\wedge), (\infty\Cat_*,\wedge^\rev)$.

\item An antispectral bioriented functor is a functor enriched in $(\Cat\Sp,\wedge), (\infty\Cat_*,\wedge^\rev)$.

\item A bispectral bioriented category is a category enriched in $(\Cat\Sp,\wedge), (\Cat\Sp,\wedge^\rev)$.

\item A bispectral bioriented functor is a functor enriched in $(\Cat\Sp,\wedge), (\Cat\Sp,\wedge^\rev)$.
		
\end{itemize}
	
\end{definition}


	

        



\begin{notation} \emph{}

\begin{itemize}
\item Let $\Cat\barwedge$ 
be the 2-category of spectral oriented categories.	
	
\item Let ${\barwedge}\Cat$
be the 2-category of antispectral antioriented categories.

\item Let ${\wedge}\Cat\barwedge$
be the 2-category of spectral bioriented categories.
        
\item Let ${\barwedge}\Cat\wedge$
be the 2-category of antispectral bioriented categories.

\item Let ${\barwedge}\Cat\barwedge$
be the 2-category of bispectral bioriented categories.

\end{itemize}
\end{notation}



A spectral oriented, antispectral antioriented, spectral, antispectral, bispectral bioriented category is presentable if the respective enriched category is presentable. It is reduced if it admits a zero object,
or equivalently if it admits an initial object or final object \cite[Proposition 2.5.2.]{gepner2025oriented}. 
Moreover any spectral oriented functor, antispectral antioriented functor, spectral, antispectral, bispectral
bioriented functor between reduced spectral oriented, antispectral antioriented, spectral, antispectral, bispectral bioriented categories, respectively, preserves the zero object \cite[Proposition 2.5.3.]{gepner2025oriented}.

An oriented pullback in a spectral oriented category is an oriented pullback in the underlying oriented category such that the induced map on categorical morphism spectra (\ref{indmapl}) is an equivalence.
An oriented pushout in a spectral oriented category is an oriented pullback
in the opposite spectral oriented category.
An antioriented pushout (antioriented pullback) in an antispectral antioriented category is an oriented pushout (oriented pullback) in the conjugate spectral oriented category.

We refer to adjunctions between spectral oriented, antispectral antioriented, spectral, antispectral, bispectral bioriented categories as spectral oriented, 
antispectral antioriented, spectral, antispectral, bispectral bioriented adjunctions. There is a similar remark like \cref{adj} for spectral adjunctions.

\begin{notation}
Let $$\Pr^L\barwedge \subset \widehat{\Cat}\barwedge, \ {\barwedge}\Pr^L \subset {\barwedge}\widehat{\Cat} $$
be the subcategories of presentable spectral oriented categories, antispectral antioriented categories, respectively, and left adjoint spectral oriented functors, antispectral antioriented functors, respectively.

Let $$ {\wedge}\Pr^L\barwedge \subset
{\wedge}\widehat{\Cat}\barwedge, \ {\barwedge}\Pr^L\wedge \subset
{\barwedge}\widehat{\Cat}\wedge, \
{\barwedge}\Pr^L\barwedge  \subset
{\barwedge}\widehat{\Cat}\barwedge $$ be the subcategories of presentable spectral, antispectral, bispectral bioriented categories and left adjoint spectral, antispectral, bispectral bioriented functors, respectively.

\end{notation}

\begin{example}
Let $\mC$ be a presentable reduced bioriented category.	
\cref{sptens} implies that $$\Sp(\mC)\simeq \mC {\ot_{\infty\Cat_*} \Cat\Sp}$$ is a presentable spectral bioriented category, $$\overline{\Sp}(\mC)\simeq \Cat\Sp {\ot_{\infty\Cat_*} \mC} $$ is a presentable antispectral bioriented category and $$\Sp(\overline{\Sp}(\mC)) \simeq {\Cat\Sp \ot_{\infty\Cat_*}} \mC  {\ot_{\infty\Cat_*} \Cat\Sp} $$ is a presentable bispectral bioriented category.
	
\end{example}

\begin{lemma}\label{specstab}
	
Let $\mC$ be a reduced (anti)spectral (anti)oriented category that admits (anti)endomorphisms. The underlying reduced (anti)oriented category of $\mC$ is (anti)stable.

\end{lemma}

\begin{proof}

The (anti)endomorphisms of any $X \in \mC$ is the (anti)oriented fiber of $ X \to 0$, which is the (left) right cotensor of $X$ with $S^1.$ 
Since $S^1$ is tensor-invertible in $\Cat\Sp$, the claim follows.
	
\end{proof}

\cref{repp} and \cref{specstab} imply the following corollary:

\begin{corollary}\label{ttut}
Let $\mC$ be a reduced spectral oriented category that admits pullbacks and oriented pushouts.
The commutative square \ref{sqlo} is a pullback square. In other words, the canonical map
$$ (-)\overset{\to}{+}_{(-)}(-) \to 0\overset{\to}{+}_{(-)}(-)\times_{0 \overset{\to}{+}_{(-)} 0} (-)\overset{\to}{+}_{(-)} 0 $$ of functors $\mC^{\Lambda_0^2} \to \mC$  is an equivalence. 

\end{corollary}

\begin{notation} \emph{}

\begin{itemize}
		
\item Let $\mC,\mD \in \Cat\barwedge$. Let ${\Fun\barwedge}(\mC,\mD)$ be the category of spectral oriented functors $\mC \to \mD.$	

\item Let $\mC,\mD \in \barwedge\Cat$. Let ${\barwedge\Fun}(\mC,\mD) $ be the category of antispectral antioriented functors $\mC \to \mD.$	
		
\item Let $\mC,\mD \in {\wedge}\Cat\barwedge $. Let ${\wedge}\Fun{\barwedge}(\mC,\mD) $ be the category of spectral bioriented functors $\mC \to \mD.$			

\item Let $\mC,\mD \in {\barwedge}\Cat\wedge $. Let ${\barwedge}\Fun{\wedge}(\mC,\mD)$ be the category of  antispectral bioriented functors $\mC \to \mD.$	

\item Let $\mC,\mD \in {\barwedge}\Cat\barwedge $. Let ${\barwedge}\Fun{\barwedge}(\mC,\mD) $ be the category of bispectral bioriented functors $\mC \to \mD.$			
		
\end{itemize}
\end{notation}

\begin{notation}\label{raras}
	
The monoidal equivalences $(-)^\op, (-)^\co: \infty\Cat^\rev \simeq \infty\Cat$ with respect to the Gray tensor product give rise to monoidal equivalences $(-)^\op, (-)^\co: \infty\Cat_*^\rev \simeq \infty\Cat_*$ with respect to the Gray smash product, which by \cref{unicity} induce unique monoidal equivalences
$$ (-)^\op, (-)^\co: \Cat\Sp^\rev \simeq \Cat\Sp$$ such that $(-)^\co \circ \Sigma^\infty \simeq \Sigma^\infty  \circ (-)^\co$ and $ (-)^\op  \circ \Sigma^\infty \simeq \Sigma^\infty  \circ (-)^\op.$ 
\end{notation}

\begin{notation}
Let $$(-)^\circ: {\Cat\barwedge} \simeq \barwedge\Cat, $$$$ (-)^\circ: \barwedge\Cat \simeq {\Cat\barwedge}, $$$$ (-)^\circ:  \barwedge\Cat\wedge \simeq {\wedge}\Cat\barwedge, $$$$ (-)^\circ:  {\wedge}\Cat\barwedge \simeq \barwedge\Cat\wedge, $$$$ (-)^\circ:  \barwedge\Cat\barwedge \simeq \barwedge\Cat\barwedge$$ be the opposite enriched category involutions.	
	
\end{notation}

\begin{notation} The equivalences of \cref{raras} give rise to the following equivalences
$$(-)^\co:= (-)^\op_!: \barwedge\Cat \simeq {\Cat\barwedge}, $$$$ (-)^\co:= (-)^\op_!: {\Cat\barwedge} \simeq \barwedge\Cat$$
$$(-)^\co:= ((-)^\op, (-)^\op)_!: \barwedge\Cat\barwedge \simeq \barwedge\Cat\barwedge,$$
$$(-)^\co:= ((-)^\op, (-)^\op)_!: \barwedge\Cat\wedge \simeq {\wedge}\Cat\barwedge, $$$$ (-)^\co:= ((-)^\op, (-)^\op)_!: {\wedge}\Cat\barwedge \simeq \barwedge\Cat\wedge,$$

$$(-)^\op:= (-)^\circ\circ (-)^\co_!: \barwedge\Cat \simeq \barwedge\Cat , $$$$ (-)^\op:= (-)^\circ\circ (-)^\co_!: {\Cat\barwedge}  \simeq {\Cat\barwedge} ,$$
$$ (-)^\op :=(-)^\circ \circ ((-)^\co, (-)^\co)_!: \barwedge\Cat\barwedge \simeq \barwedge\Cat\barwedge, $$
$$ (-)^\op :=(-)^\circ \circ ((-)^\co, (-)^\co)_!: \barwedge\Cat\wedge \simeq \barwedge\Cat\wedge, $$
$$ (-)^\op :=(-)^\circ \circ ((-)^\co, (-)^\co)_!: {\wedge}\Cat\barwedge \simeq {\wedge}\Cat\barwedge. $$

\end{notation}

The equivalences of \cref{dua} gives rise to the following equivalences of bioriented categories:
\begin{corollary}\label{duae2} There are canonical equivalences of bispectral bioriented categories:
$$(-)^\op: \Cat\Sp^\co \simeq \Cat\Sp, \ (-)^\co: \Cat\Sp^\op \simeq \Sp^\circ. $$	
\end{corollary}

\cref{stten} gives the following corollary:

\begin{corollary}\label{unistab}
\begin{enumerate}
\item The forgetful functor $${\barwedge}\Pr^L\barwedge= {_{\Cat\Sp} \Mod_{\Cat\Sp}(\Pr^L)}  \to {_{\infty\Cat_*} \Mod_{\infty\Cat_*}(\Pr^L)} = {\wedge \Pr^L \wedge}$$ induces an equivalence of 2-categories
$${\barwedge}\Pr^L\barwedge \to {\wedge}\Pr^L_{\mathrm{bist}}\wedge.$$

\item The forgetful functor $$ {\barwedge}\Pr^L\wedge= {_{\Cat\Sp} \Mod_{\infty\Cat_*}(\Pr^L)}  \to {_{\infty\Cat_*} \Mod_{\infty\Cat_*}(\Pr^L)} = {\wedge \Pr^L \wedge}$$ induces an equivalence of 2-categories
$$ {\barwedge}\Pr^L\wedge \to {\wedge}\Pr^L_{\mathrm{anst}}\wedge.$$

\item The forgetful functor $$\wedge\Pr^L\barwedge= {_{\infty\Cat_*}\Mod_{\Cat\Sp}(\mathrm{Pr}^{L})}  \to {_{\infty\Cat_*} \Mod_{\infty\Cat_*}(\Pr^L)} = {\wedge \Pr^L \wedge}$$ induces an equivalence of 2-categories
$$ {\wedge\Pr^L}\barwedge  \to {\wedge}\Pr^L_{\mathrm{st}}\wedge.$$

\item The forgetful functor $$ {\barwedge}\Pr^L= {_{\Cat\Sp} \Mod(\Pr^L)}  \to {_{\infty\Cat_*} \Mod(\Pr^L)} = {\wedge \Pr^L}$$ induces an equivalence of 2-categories
$$ {\barwedge}\Pr^L\to {\wedge}\Pr_{\mathrm{anst}}^L.$$

\item The forgetful functor $$ \Pr^L\barwedge= { \Mod_{\Cat\Sp}(\Pr^L)}  \to {\Mod_{\infty\Cat_*}(\Pr^L)} = { \Pr^L \wedge}$$ induces an equivalence of 2-categories
$$ \Pr^L\barwedge \to {\Pr^L_{\mathrm{st}} \wedge}.$$

\end{enumerate}
\end{corollary}

\begin{proof}
(1) : The forgetful functor $ {_{\Cat\Sp} \Mod_\Sp(\Pr^L)}  \to {\wedge \Pr^L \wedge} $ induces a functor
$$ {_{\Cat\Sp} \Mod_{\Cat\Sp}(\Pr^L)}  \to{\wedge}\Pr^L_{\mathrm{bist}}\wedge$$ that is right adjoint to the restricted free
functor $ \mC \mapsto {\Cat\Sp} \ot_{\infty\Cat_*} \mC \ot_{\infty\Cat_*} {\Cat\Sp} $. The unit at $\mC \in {\wedge}\Pr^L_{\mathrm{bist}}\wedge$
is the reduced bioriented functor $$\eta: \mC \simeq  \infty\Cat \ot_{\infty\Cat_*} \mC \ot_{\infty\Cat_*} \infty\Cat \to {\Cat\Sp} \ot_{\infty\Cat_*} \mC \ot_{\infty\Cat_*} {\Cat\Sp},$$
which is an equivalence by bistability of $\mC.$ 
The proofs of (2), (3), (4), (5) are similar.

\end{proof}

We immediately obtain the following universal characterization of spectrum objects:

\begin{theorem}\label{specunive}
\begin{enumerate}

\item The embeddings of 2-categories $$ {\Pr^L_{\mathrm{st}} \wedge} \subset {\Pr^L \wedge},  \ {\wedge\Pr^L_{\mathrm{st}} \wedge} \subset {\wedge\Pr^L \wedge}$$ admit a left adjoint sending $\mC$ to $\Sp(\mC).$

\item The embeddings of 2-categories
$$ {\wedge}\Pr_\mathrm{anst}^L \subset {\wedge \Pr^L }, \   {\wedge}\Pr_\mathrm{anst}^L\wedge \subset {\wedge \Pr^L\wedge} $$
admit a left adjoint sending $\mC$ to $\overline{\Sp}(\mC).$

\item The embedding of 2-categories $${\wedge}\Pr^L_{\mathrm{bist}}\wedge \subset {\wedge\Pr^L \wedge}$$ admits a left adjoint sending $\mC$ to $\overline{\Sp}(\Sp(\mC)). $ 

\end{enumerate}
\end{theorem}

\begin{corollary}\label{corp} Let $\mC$ be a presentable reduced bioriented category.
There is a monoidal functor $$ {{\wedge \L\Fun\wedge}(\mC, \mC)} \to {{\wedge\L\Fun\wedge}(\Sp(\overline{\Sp}(\mC)), \Sp(\overline{\Sp}(\mC)))}.$$

\end{corollary}

\cref{corp} specializes to the following corollary:

\begin{corollary}
There is a monoidal functor $$ {{\wedge \L\Fun\wedge}(\infty\Cat_*, \infty\Cat_*)} \to {{\wedge \L\Fun\wedge}(\Cat\Sp, \Cat\Sp)}.$$

\end{corollary}

\vspace{1mm}

The following is a higher-categorical analogue of the key property of spectral functors to preserve pushouts.

\begin{theorem}\label{preserv}\emph{}

\begin{enumerate}

\item Let $\mC, \mD$ be spectral oriented categories and $\phi: \mC \to \mD$ a spectral oriented functor. Then $\phi$ preserves oriented pushouts and oriented pullbacks.

\item Let $\mC, \mD$ be antispectral antioriented categories and $\phi: \mC \to \mD$ an antispectral antioriented functor. Then $\phi$ preserves antioriented pushouts and antioriented pullbacks.

\item Let $\mC, \mD$ be spectral bioriented categories and $\phi: \mC \to \mD$ a spectral bioriented functor. Then $\phi$ preserves oriented pushouts and oriented pullbacks. 

\item Let $\mC, \mD$ be antispectral bioriented categories and $\phi: \mC \to \mD$ an antispectral bioriented functor. Then $\phi$ preserves antioriented pushouts and antioriented pullbacks.

\item Let $\mC, \mD$ be bispectral bioriented categories and $\phi: \mC \to \mD$ a bispectral bioriented functor. Then $\phi$ preserves antioriented pushouts, antioriented pullbacks, oriented pushouts and oriented pullbacks. 

\end{enumerate}
\end{theorem}

\begin{proof}(5) follows from (3) and (4). (3) follows from (1). (4) follows from (2). (1) is dual to (2) by replacing $\phi$ by $\phi^\co.$

The second half of (2) is dual to the first half of (2): it follows by replacing $\phi$ by $\phi^\op.$
So it suffices to show that every antispectral antioriented functor $\phi:\mC \to \mD$ preserves antioriented pullbacks.

By \cref{repp} and the enriched adjoint functor theorem \cite[Theorem 3.73.]{heine2024higheralgebraweightedcolimits} the antispectral antioriented functor $\overset{\to}{+} : \Cat\Sp^{\Lambda_0^2} \to \Cat\Sp$ admits an antioriented left adjoint $\alpha$
since $\overset{\to}{\times} \circ \Sigma$ admits an antioriented left adjoint.
For every spectral oriented category $\mA$ that admits small weighted colimits the antioriented adjunction $$\alpha: \Cat\Sp \rightleftarrows \Cat\Sp^{\Lambda_0^2}:\overset{\to}{+}$$ induces an adjunction $$\alpha_\mA:= \mA \ot_{\Cat\Sp} \alpha: \mA \simeq \mA \ot_{\Cat\Sp} \Cat\Sp \rightleftarrows \mA \ot_{\Cat\Sp} \Cat\Sp^{\Lambda_0^2} \simeq \mA^{\Lambda_0^2}: \overset{\to}{+}.$$
For every left adjoint spectral oriented functor $\F: \mA \rightleftarrows \mB: \G$ between spectral oriented categories having small weighted colimits there is a canonical equivalence of functors $$ \alpha_\mB \circ \F \simeq \mB \ot_{\Cat\Sp} \alpha \circ \F \ot_{\Cat\Sp} \Cat\Sp \simeq  \F \ot_{\Cat\Sp} \alpha \simeq \F \ot_{\Cat\Sp} \Cat\Sp \circ \mA \ot_{\Cat\Sp} \alpha \simeq 
\F \circ \alpha_\mA$$
providing an equivalence of right adjoints $ \G \circ\overset{\to}{+} \simeq\overset{\to}{+} \circ \G.$

Consequently, every right adjoint spectral oriented functor between spectral oriented categories 
having small weighted colimits preserves oriented pushouts.
So dually, every left adjoint antispectral antioriented functor between antispectral antioriented categories 
having small weighted limits preserves antioriented pullbacks. 

By antispectral antioriented Yoneda extension (\cref{Yonedaext}) every antispectral antioriented functor $\phi:\mC \to \mD$ extends to a left adjoint antispectral antioriented functor $\phi':\mC' \to \mD'$ between presentable antispectral antioriented categories, which in particular admit small weighted limits. Moreover this extension is along the antispectral antioriented Yoneda embeddings, which preserve weighted limits by \cref{weiadj} and in particular antioriented pullbacks. Thus $\phi:\mC \to \mD$ preserves antioriented pullbacks.

\end{proof}

\begin{remark}
In his thesis \cite{masuda2026algebra} Masuda constructs a tensor product of categorical spectra and proves that the categorical spectrum $\Sigma^\infty(\bD^1) $ is dualizable \cite[Proposition 5.2.1.]{masuda2026algebra}.
Masuda uses this result to give an alternative proof \cite[Theorem 5.2.7., Corollary 5.2.10.]{masuda2026algebra} of \cref{preserv}
for $\Cat\Sp.$
	
\end{remark}




\begin{proposition}\label{laxfib}
Let $\mC$ be a reduced spectral oriented category that admits oriented left fibers and suspensions.
Then $\mC$ admits oriented left cofibers and the map $$\xi: 0\overset{\to}{+}  \to (0 \overset{\to}{\times}) \circ \Sigma $$
of functors $\mC^{\bD^1} \to \mC$ is an equivalence.

\end{proposition}

\begin{proof}
By \cref{Yonedaext} there is a spectral oriented embedding of $\mC$ into a presentable spectral oriented category that preserves oriented left fibers and suspensions. So we can assume that $\mC$ is presentable.

Restricting $\kappa$ along the spectral oriented functor $$\mC^{\bD^1} \to \mC^{\Lambda_0^2}, \ Y \leftarrow X  \mapsto Y \leftarrow X \to 0 $$ gives $\xi$ so that $\xi$ admits a left inverse by \cref{repp}.
By the adjoint functor theorem 
the functor $0\overset{\to}{+}$ admits a left adjoint since $  (0 \overset{\to}{\times}) \circ \Sigma$ does. Replacing $\mC$ by $\mC^\op$ 
the functor $ 0 \overset{\to}{\times}$ admits a right adjoint.
Hence in the adjunction $$ L: \Cat\Sp \rightleftarrows \Cat\Sp^{\bD^1}:  0 \overset{\to}{\times}$$ both adjoints admit a right adjoint so that the latter adjunction induces an adjunction $$L: \mC \simeq \mC \otimes_{\Cat\Sp}\Cat\Sp \rightleftarrows \mC^{\bD^1} \simeq \mC \otimes_{\Cat\Sp}\Cat\Sp^{\bD^1}:  0 \overset{\to}{\times}.$$
This implies that $$ \mC \ot_{\Cat\Sp} \xi_\Cat\Sp: \mC \ot_{\Cat\Sp} 0\overset{\to}{+} \to \mC \ot_{\Cat\Sp} ( (0 \overset{\to}{\times}) \circ \Sigma)$$ is $\xi_\mC.$ 
Hence we can assume that $\mC =\Cat\Sp$. 
In this case $\xi$ is a map of antispectral antioriented functors $\Cat\Sp^{\bD^1}\to \Cat\Sp$.
Since $\Cat\Sp^{\bD^1} $ is generated under small colimits and left tensors by
the morphisms $ 0 \to S^0, \id: S^0 \to S^0$, and source and target of $\xi$ preserve small colimits and left tensors, it suffices to show that the components of $\xi$ at $ 0 \to S^0, \id: S^0 \to S^0$
are equivalences.
The component of $\xi$ at $0 \to S^0$ is the canonical morphism
$$ 0 \overset{\to}{+}_0 S^0 \simeq S^0 \to 0 \overset{\to}{+}_{S^1} 0 \simeq \Omega (S^1),$$
which is an equivalence by stability. 
It remains to see that the component of $\xi$ at the identity of $S^0$ is an equivalence.
This component is the morphism
$$ \alpha: 0 \overset{\to}{+}_{S^0} S^0 \simeq \Sigma^\infty(\bD^1) \to 0 \overset{\to}{+}_{S^1} S^1 \simeq \L\Mor_{\Cat\Sp}(\Sigma^\infty(\bD^1), S^1)$$ corresponding to the map
$ \Sigma^\infty(\beta): \Sigma^\infty(\bD^1) \wedge \Sigma^\infty(\bD^1) \to S^1$, where
$\beta$ is the composition $$\bD^1 \wedge \bD^1 \simeq \bD^2 \to \bD^1 \to S^1$$ of 1-truncation and the quotient functor.
By \cref{repp} the morphism $\alpha$ admits a left inverse.  So it suffices to show that $\alpha$ admits a section. We rewrite $\alpha.$

For any $A, B \in {\Cat\Sp}$ if $(-) \wedge A : {\Cat\Sp} \to {\Cat\Sp}$ is left adjoint to $(-)\wedge B$, then 
$B \wedge (-)$ is left adjoint to $A \wedge (-) $ and so by adjointness there is an equivalence $$ A \simeq \R\Mor_{\Cat\Sp}(B,S^0).$$

Moreover the bioriented functor $\ev_1: {\Cat\Sp}^{\bD^1} \to {\Cat\Sp}$ evaluating at the target admits a right adjoints sending $X $ to the identity of $X.$
Hence the map $(\Sigma\ev_1(\xi^L)(S^0))^\co$ is induced by $$ \xi(\id_{S^0}): \Sigma^\infty(\bD^1) \to \L\Mor_{\Cat\Sp}(\Sigma^\infty(\bD^1), S^1)$$ and factors as 
$$\Sigma^\infty(\bD^1)\simeq \Sigma^\infty(\bD^1)^\co \to \R\Mor_{\Cat\Sp}(\L\Mor_{\Cat\Sp}(\Sigma^\infty(\bD^1), S^1), S^1)^\co  \to \R\Mor_{\Cat\Sp}(\Sigma^\infty(\bD^1), S^1)^\co $$$$ \simeq \L\Mor_{\Cat\Sp}(\Sigma^\infty(\bD^1), S^1),$$ which corresponds to $\Sigma^\infty(\beta),$ and so identifies with $\alpha.$
Let $\lambda$ be the left inverse of the map $$ \xi: 0\overset{\to}{+}\to  (0 \overset{\to}{\times})\circ \Sigma$$ of right adjoint functors ${\Cat\Sp}^{\bD^1} \to {\Cat\Sp}$.
We obtain an equivalence $$\id \simeq (\Sigma \ev_1((\lambda \xi)^L(S^0)))^\co \simeq (\Sigma\ev_1(\xi^L(S^0)))^\co \circ (\Sigma\ev_1(\lambda^L(S^0)))^\co $$ that exhibits $(\Sigma\ev_1(\lambda^L(S^0)))^\co $ as a section of $$\alpha\simeq (\Sigma\ev_1(\xi^L(S^0)))^\co: \Sigma^\infty(\bD^1)^\co \simeq \Sigma^\infty(\bD^1) \to \R\Mor_{\Cat\Sp}(\Sigma^\infty(\bD^1),S^1)^\co \simeq \L\Mor_{\Cat\Sp}(\Sigma^\infty(\bD^1),S^1).$$

\end{proof}



\subsection{Categorical rig spectra}

Next we define categorical rig spectra.

\begin{definition}\emph{}
Let $\mC$ be a presentably monoidal category.
A rig in $\mC$ is an associative algebra in $ \Mon_{\bE_\infty}(\mC).$

\end{definition}

\begin{definition}

A rig is a rig in $\Set.$
A rig space is a rig in $\infty\Grp.$

\end{definition}

For the next definition we use that $\Cat\Sp$ carries a monoidal structure (\cref{unicity}).

\begin{definition}A categorical rig spectrum is an associative algebra in $\Cat\Sp.$

\end{definition}




\begin{notation}Let $\mC$ be a presentably monoidal category.
Let $$ \Rig(\mC):= \Alg(\Mon_{\bE_\infty}(\mC)).$$ 

\end{notation}

\begin{example}\emph{}
\begin{itemize}

\item The natural numbers $(\bN,+,\bullet)$ form a rig.	

\item The free symmetric monoidal category $\coprod_{\n \geq 0} B\Sigma_\n$ generated by one object is a rig space.

\end{itemize}

\end{example}

\begin{example}
	
By adjunction (\ref{adjass}) the canonical symmetric monoidal localizations $$ \infty\Cat \rightleftarrows \infty\Grp, \ \pi_0: \infty\Grp \rightleftarrows \Set$$ give rise to monoidal localizations $$ \Mon_{\bE_\infty}(\infty\Cat) \rightleftarrows \Mon_{\bE_\infty}(\infty\Grp), \ \Mon_{\bE_\infty}(\infty\Grp) \rightleftarrows \Mon_{\bE_\infty}(\Set), $$
which give rise to localizations $$ \Rig(\infty\Cat) \rightleftarrows \Rig(\infty\Grp), \ \Rig(\infty\Grp) \rightleftarrows \Rig(\Set). $$

\end{example}

\cref{symsp0} implies the following:

\begin{corollary}\label{symsp}

The left adjoint monoidal functor $$ \Mon_{\bE_\infty}(\Sigma^\infty): \Mon_{\bE_\infty}(\infty\Cat) \to \Mon_{\bE_\infty}(\Cat\Sp) \simeq \Cat\Sp$$
induced via (\ref{loo}) by the left adjoint monoidal functor $\Sigma^\infty: \infty\Cat \to \Cat\Sp$ of \cref{unicity}
sends a symmetric monoidal $\infty$-category $\mC$ to $B^\infty(\mC)$ and induces a monoidal equivalence
$$ \Mon_{\bE_\infty}(\infty\Cat) \simeq \Cat\Sp_{\geq0}.$$

\end{corollary}	

\cref{symsp} implies the following corollary:

\begin{corollary}

For every $\infty$-category $X$ the unit $$X \to  \Omega^\infty(\Sigma^\infty_+(X))$$
exhibits the symmetric monoidal $\infty$-category $\Omega^\infty(\Sigma^\infty_+(X))$ as the free
symmetric monoidal $\infty$-category generated by $X.$
In particular, there is a canonical equivalence
$$ \Omega^\infty(\Sigma^\infty_+(X)) \simeq \coprod_{\n \geq 0} (X^{\times \n})_{\Sigma_\n}.$$


\end{corollary}

We obtain the following pre-group complete version of Barratt-Priddy-Quillen's theorem \cite{barratt1972homology}: 

\begin{corollary}
	
There is a canonical symmetric monoidal equivalence $$ \Omega^\infty(\Sigma^\infty(S^0))= \Omega^\infty(\Sigma_+^\infty(*)) \simeq \coprod_{\n \geq 0} B\Sigma_\n.$$
\end{corollary}

\begin{remark}

The left adjoint monoidal embedding $$B^\infty: \Mon_{\bE_\infty}(\infty\Cat) \simeq \Mon_{\bE_\infty}(\infty\Cat_*) \to \Mon_{\bE_\infty}(\Cat\Sp) \simeq \Cat\Sp$$
of \cref{symsp} induces a left adjoint embedding
$$B^\infty: \Rig(\infty\Cat) \to \Rig(\Cat\Sp).$$

\end{remark}

	
	


\begin{definition}\label{rigit}
The categorical Eilenberg-MacLane spectrum functor $H$ is the following composition of lax monoidal embeddings:
$$\Mon_{\bE_\infty}(\Set) \subset \Mon_{\bE_\infty}(\infty\Grp) \xrightarrow{B^\infty} \Cat\Sp,$$
which gives rise to an embedding
$$\Rig(\Set) \subset \Rig(\Cat\Sp).$$

\end{definition}

\begin{definition}Let $R$ be a rig.

The derived oriented category of $R$ is
$$\mD_R :={_{H(R)} \Mod(\Cat\Sp)}.$$

The derived antioriented category of $R$ is
$$\bar{\mD}_{R} :={\Mod_{H(R)}(\Cat\Sp)}.$$
    
\end{definition}

\begin{remark}

The bioriented category $\Cat\Sp$ is presentable and bistable \cref{stab}. 
\begin{itemize}
\item Hence
$\mD_R $ is a presentable stable oriented category and the forgetful functor $\mD_R \to \Cat\Sp$ is a left and right adjoint oriented functor. 
\item Hence
$\bar{\mD}_R $ is a presentable antistable antioriented category and the forgetful functor $\bar{\mD}_R \to \Cat\Sp$ is a left and right adjoint antioriented functor. 

\end{itemize}

\end{remark}

\subsection{A categorical Freudenthal suspension theorem}

In the following we prove a categorical version of the Freudenthal suspension theorem (\cref{F}) that governs the passage from unstable to stable homotopy theory of higher categories.

\begin{definition} Let $ \n \geq 0.$

\begin{itemize}
\item A functor is $-1$-connected if it is essentially surjective.

\item A functor is $\n$-connected if it is $-1$-connected and induces $\n-1$-connected functors on morphism $\infty$-categories. 

\end{itemize}

\end{definition}

\begin{definition} Let $ \n \geq 0.$
A functor is $ \n$-connective if it is $\n-1$-connected.

\end{definition}

\begin{lemma}\label{cateqst}

Let $\n \geq -1$ and $X$ a space.
A functor $\phi: A \to B$ over $X$ is $\n$-connected
if it induces fiberwise $\n$-connected functors.

\end{lemma}

\begin{proof}

The case of $n=-1$ is clear.
The general case follows by induction on $n \geq -1$. The induction step follows from the fact that for every
$U,V \in A $ lying over $U', V' \in X$
the induced functor $\Mor_A(U,V) \to \Mor_B(\phi(U), \phi(V))$
induces on the fiber over every $\alpha \in \Mor_X(U',V')$
the induced functor $$ \Mor_{A_{V'}}(\alpha_!(U),V) \to \Mor_{B_{V'}}(\phi(\alpha_!(U)), \phi(V)).$$

\end{proof}

\begin{notation}
	
Let $ \n, \bk \geq 0.$
The $\bk$-th ordered configuration space $ \mathrm{Conf}_\bk(\bR^\n)$ in $\bR^\n$ 
is the subspace $$\{  \{1,...,\bk\}  \hookrightarrow \bR^\n \}  \subset (\bR^\n)^{\times \bk}$$
of injective maps $  \{1,...,\bk\}  \to \bR^\n$ endowed with the subspace topology of the product topology of the Euclidean topology on $\bR^\n.$
		
\end{notation}

\begin{lemma}\label{confi}
	
Let $ \n, \bk \geq 0.$
The space $ \mathrm{Conf}_\bk(\bR^\n)$ is $\n$-2 connected.
	
\end{lemma}

\begin{proof}
We proceed by induction on $\bk \geq 0.$	
For $\bk=0$ there is nothing to prove. For $\bk=1$ we find that 
$ \mathrm{Conf}_\bk(\bR^\n)= \bR^\n$, which is contractible.
We prove the induction step.
We consider the map $$\xi:  \mathrm{Conf}_{\bk+1}(\bR^\n) \to  \mathrm{Conf}_\bk(\bR^\n)$$
that sends a $\bk$+1-tuple $(Y_1,...,Y_{\bk+1})$ to $(Y_1,..., Y_\bk).$
Then $\xi$  is a fibration and the (homotopy) fiber of $\xi$ over any point
$(X_1,..., X_\bk) \in \mathrm{Conf}_\bk(\bR^\n)$ is the space
$\bR^\n \setminus \{X_1,...,X_\bk\}.$ 
The latter space is equivalent to the space $ \bigvee_{i=1}^\bk B^{\n-1}(\bZ)$.
Since $B^{\n-1}(\bZ)$ is $\n$-2-connected, the space  $ \bigvee_{i=1}^{\bk} B^{\n-1}(\bZ)$ is  $\n$-2-connected. Thus all homotopy fibers of $\xi$ are $\n$-2-connected.
Via the long exact sequence on homotopy groups this implies that 
$\mathrm{Conf}_{\bk+1}(\bR^\n)$ is  $\n$-2-connected if $\mathrm{Conf}_\bk(\bR^\n)$ is  $\n$-2-connected. This proves the induction step.

\end{proof}

\begin{theorem}\label{F} Let $\n \geq 0$ and $X \in \infty\Cat_*$ be $n$-connected.
The canonical functor $$ X \to \Omega\Sigma(X)$$
is $2 \n+1$-connective.

\end{theorem}

\begin{proof}
	
For every $\n \geq 0$ let $ \infty\Cat^{\leq \n} \subset \infty\Cat $ be the full subcategory of $\n$-connected $\infty$-categories.
By \cref{deloop} there is an adjunction
\begin{equation}\label{adki} B^{\n+1}: \Mon_{\bE_{\n+1}}(\infty\Cat)  \rightleftarrows \infty\Cat_*: \Omega^{\n+1}\end{equation}
whose left adjoint preserves finite products and induces an equivalence to $ \infty\Cat_*^{\leq \n} \subset \infty\Cat_*.$
The latter induces an embedding $$\Mon( \infty\Cat^{\leq \n})  \simeq \Mon( \infty\Cat_*^{\leq \n}) \subset \Mon(\infty\Cat_* )\simeq \Mon(\infty\Cat),$$ which by adjointness preserves the free reduced associative monoid functors.
By \cref{hhnl} the functor $\Sigma$ factors as the free reduced associative monoid $\infty\Cat_* \to \Mon(\infty\Cat) $ followed by $$B: \Mon(\infty\Cat) \to  \infty\Cat_*.$$
Thus the adjunction $$\Sigma: \infty\Cat_* \rightleftarrows  \infty\Cat_* : \Omega$$
restricts to an adjunction $$\Sigma: \infty\Cat^{\leq \n}_* \rightleftarrows  \infty\Cat^{\leq \n+1}_* : \Omega.$$

Hence for every $ \n \geq 0$ there is a commutative square 
$$\begin{xy}
\xymatrix{
\Mon_{\bE_{\n+2}}(\infty\Cat)  \ar[d] \ar[rr]^\simeq
&& \infty\Cat ^{\leq \n+1}_* \ar[d]^{\Omega}
\\ 
\Mon_{\bE_{\n+1}}(\infty\Cat)  \ar[rr]^\simeq && \infty\Cat^{\leq \n}_*.
}
\end{xy}$$	

So for every $\n$-connected $\infty$-category $X$ the canonical functor $X \to \Omega\Sigma(X)$ identifies with the canonical functor
$$ B^{\n+1}(\Omega^{\n+1}(X)) \to B^{\n+1}(\Free_{\bE_{\n+1}, \bE_{n+2}}\Omega^{\n+1}(X)),$$
where $ \Free_{\bE_{\n+1}, \bE_{n+2}}$ is the left adjoint of the left vertical functor in the latter commutative square.


Consequently, it suffices to see that for every $\infty$-category $X$ the canonical functor 
\begin{equation}\label{tgrrr}
X \to \Free_{\bE_{\n+1}, \bE_{n+2}}(X) 
\end{equation}
is $\n$-connective, i.e. $\n$-1-connected.
By \cite[Theorem 3.3.19.]{gepner2026homotopy} the class of $\n$-1-connected functors is the left class of a factorization system on $\infty\Cat$. Thus the full subcategory of the arrow $\infty$-category of $\infty\Cat$ spanned by the $\n-1$-connected functors is closed under small colimits.
This guarantees that the full subcategory of 
$\Mon_{\bE_{\n+1}}(\infty\Cat)$ spanned by those $X$ such that the canonical functor (\ref{tgrrr}) is $\n$-1-connected, is closed under small sifted colimits.
Since $\Mon_{\bE_{\n+1}}(\infty\Cat)$ is generated under small sifted colimits by the free $\bE_{\n+1}$-monoids, we can assume that $X$ is free.
Consequently, it suffices to see that for every $\infty$-category $X$ the canonical functor $\Free_{\bE_{\n+1}}(X) \to \Free_{\bE_{\n+2}}(X) $ is $\n$-1-connected.

For every $\bk \geq 0$ let $\bE^\bk_\n$ be the space of $\bk$-ary operations of the $\bE_\n$-operad,
which is equivalent to the space $\mathrm{Conf}_\bk(\bR^\n)$ and so $\n-2$-connected by \cref{confi}.



For every $\infty$-category $X$ the canonical functor 
$ \Free_{\bE_{\n+1}}(X) \to \Free_{\bE_{\n+2}}(X) $ 
identifies with the induced functor
$$ \coprod_{\bk \geq 0}(\bE^\bk_{\n+1} \times X^{\times \bk})_{\Sigma_\bk} \to \coprod_{\bk \geq 0}(\bE^\bk_{\n+2} \times X^{\times \bk})_{\Sigma_\bk}. $$
The latter functor is $\n$-1-connected if for every $\bk \geq 0$ the canonical functor
$$ (\bE^\bk_{\n+1}  \times X^{\times \bk})_{\Sigma_\bk} \to (\bE^\bk_{\n+2} \times X^{\times \bk})_{\Sigma_\bk} $$
is $\n$-1-connected.
To see this, by \cref{cateqst} it is enough to verify that the canonical functor
$$ \bE^\bk_{\n+1}  \times X^{\times \bk} \to \bE^\bk_{\n+2} \times X^{\times \bk} $$
is $\n$-1-connected. Since $\n$-1-connected functors are stable under product, it suffices to show that the canonical map $ \bE^\bk_{\n+1} \to \bE^\bk_{\n+2} $
is $\n$-1-connected. 
Since it is a map of spaces, this means that this map induces an isomorphism on homotopy groups in degrees lower than $n$ and a surjection in degree $n.$ By \cref{confi} the space
$ \bE^\bk_{\n+1} $ is $n$-1-connected and the space 
$ \bE^\bk_{\n+2} $ is $n$-connected. Hence the result follows.

\end{proof}






\begin{corollary} Let $\n \geq 0$ and $X$ an $\infty$-category with distinguished object. 
The canonical functor $$\Omega^{\n+1} \Sigma^{\n+1}(X) \to \Omega^\infty\Sigma^\infty(X) $$ is $\n$-connective.

\end{corollary}

\begin{proof} Let $\bk \geq 0.$
By \cref{F} the canonical functor $$ \Omega^{\n+1+\bk} \Sigma^{\n+1+\bk}(X) \to \Omega^{\n+\bk+2} \Sigma^{\n+\bk+2}(X) $$
is $\n+\bk$-1-connected and so in particular $\n$-1-connected.

Since $\n-1$-connected functors are stable under filtered colimits, the canonical functor $$\Omega^{\n+1} \Sigma^{\n+1}(X) \to \Omega^\infty\Sigma^\infty(X) \simeq \colim_{\bk \geq 0}\Omega^{\n+1+\bk} \Sigma^{\n+1+\bk}(X) $$ is $\n$-1-connected, which is $n$-connective.

\end{proof}

We also recover the classical Freudenthal suspension theorem:

\begin{corollary}

Let $\n \geq 0$ and $X \in \infty\Grp_*$ be $n$-connected.
The canonical map $ X \to \Omega\Sigma(X)$
is $2 \n+1$-connective.

\end{corollary}

\begin{proof}

For every space $X$ the functor of \cref{F} is a map of spaces,
which identifies with the canonical map $X \to \Free_{\bE_1}(X)$
to the free $\bE_1$-space. The latter is $2 \n+1$-connective by \cref{F}.
The canonical map $ X \to \Omega\Sigma(X)$
is the image of the map $X \to \Free_{\bE_1}(X)$ under the left adjoint $\Mon_{\bE_1}(\infty\Grp) \to \Grp_{\bE_1}(\infty\Grp) $
of the embedding $\Grp_{\bE_1}(\infty\Grp)  \subset \Mon_{\bE_1}(\infty\Grp).$
So it suffices to see that the left adjoint $\Mon_{\bE_1}(\infty\Grp) \to \Grp_{\bE_1}(\infty\Grp) $ preserves $m$-connected maps for every $m \geq -1.$

For every $m \geq -1$ there is a factorization system on $\infty\Grp$, where the left class is the class of $m$-connected maps and the right class is the class of $m$-truncated maps \cite[Example 5.2.8.16.]{lurie.HTT}. This implies by \cite[Lemma 5.2.8.19.]{lurie.HTT} that the full subcategory of $m$-types is a localization of $\infty\Grp$ where the local equivalences are precisely the $m$-connected maps. Thus the localization functor preserves finite products and so induces localizations
on $\Mon_{\bE_1}(\infty\Grp), \Grp_{\bE_1}(\infty\Grp)$ whose local objects are the
(grouplike) $\bE_1$-spaces whose underlying space is an $m$-type, and whose local equivalences are precisely the maps of 
(grouplike) $\bE_1$-spaces whose underlying map of spaces is $m$-connected. 
So the embedding $\Grp_{\bE_1}(\infty\Grp)  \subset \Mon_{\bE_1}(\infty\Grp)$
restricts to an embedding between $m$-types. This implies that the left adjoint $\Mon_{\bE_1}(\infty\Grp) \to \Grp_{\bE_1}(\infty\Grp) $ preserves local equivalences, which are $m$-connected maps.

\end{proof}

\section{Categorical Brown representability}

In this section we prove a categorical Brown representability theorem (\cref{brown}), which represents every categorical homology theory by a unique categorical spectrum, the categorical coefficient spectrum of the categorical homology theory.



We reconstruct a categorical homology theory from its coefficients via the following two methods:
\begin{enumerate}
\item a theory of categorical prespectra which give rise to categorical spectra via spectrification. 

\item a theory of excisive approximation of oriented functors and of bioriented functors.

\end{enumerate}

The following subsections are devoted to the study of categorical prespectra, categorical homology theories and excisive approximations of oriented and bioriented functors.





\subsection{Spectrification}

In this subsection we introduce prespectrum objects in any reduced bioriented category. We show a formula for the spectrification (\cref{spectri}) and use this formula to prove that negative dimensional spheres are dense in categorical spectra (\cref{spgen}).

\begin{notation}

Let $\mC$ be a reduced bioriented category that admits endomorphisms.
Let $$\mC^\Omega:= \mC^{\bD^1}\times_{\mC^{\{1\}}}\mC^\cop$$ be the pullback in $\wedge\Cat\wedge$ of evaluation at the target along $\Omega: \mC^\cop \to \mC.$

\end{notation}

\begin{remark}

There are canonical bioriented functors $\mC^\Omega \to \mC^{\bD^1} \to \mC^{\{0\}}$ and $ \mC^\Omega \to \mC^\cop$.

\end{remark}

\begin{definition}\label{prespectraa}
The bioriented category of prespectrum objects in $\mC$ is the following limit in $\wedge\Cat\wedge$:
$$\Pre\Sp(\mC):= ... \times_{\mC^\cop} \mC^\Omega \times_{\mC}(\mC^\cop)^\Omega \times_{\mC^\cop} \mC^\Omega. $$

\end{definition}

\begin{remark}
A prespectrum in $\mC$ consists of a sequence $(X_0,X_1,X_2,...)$ of objects in $\mC$ and bonding morphisms $X_\n \to \Omega(X_{\n+1})$ in $\mC$ for every $\n \geq 0$.

\end{remark}

\begin{remark}

By \cref{ember} every weakly reduced oriented category gives rise to a weakly reduced bioriented category, which is reduced and 
admits endomorphisms if the original oriented category is reduced and admits endomorphisms (\cref{indubi}).
Hence we can apply \cref{prespectraa} to every reduced oriented category that admits endomorphisms to obtain an oriented category of prespectrum objects.

\end{remark}

\begin{example}
	
The bioriented category $\Pre\Sp(\infty\Grp)$ is the usual category of prespectra viewed as a bioriented category.	
	
\end{example}

\begin{remark}
	
Since $\mC$ admits a zero object and endomorphisms, by \cref{colimenrfun} and \cref{weiadj} also $\mC^\Omega$ admits a zero object and endomorphisms. Thus by \cref{weiadj} also the limit $\Pre\Sp(\mC)$ admits a zero object and endomorphisms.	
	
\end{remark}

\begin{definition}
The bioriented category of categorical prespectra is $$\Cat\Pre\Sp:=\Pre\Sp(\infty\Cat_*).$$

\end{definition}

\begin{lemma}\label{comp}
Let $\mC$ be a reduced bioriented category that admits endomorphisms
and $X,Y \in \Pre\Sp(\mC).$ There is a canonical equivalence of $\infty$-categories:	
$$\Mor_{\Pre\Sp(\mC)}(X,Y) \simeq \Mor_{\mC}(X_0,Y_0)\times_{\Mor_{\mC}(X_{0},\Omega(Y_1))} 
\Mor_{\mC}(X_1,Y_1)\times_{\Mor_{\mC}(X_1,\Omega(Y_2))} ... $$ 
\end{lemma}

\begin{proof}

By the limit definition of $\Pre\Sp(\mC)$ the $\infty$-category 
$\Mor_{\Pre\Sp(\mC)}(X,Y)$ is the limit:
$$ \Mor_{\mC^\Omega}(X_0 \to \Omega(X_1), Y_0 \to \Omega(Y_1)) \times_{\Mor_\mC(X_1,Y_1)} \Mor_{\mC^\Omega}(X_1 \to \Omega(X_2), Y_2 \to \Omega(Y_2))  \times_{\Mor_\mC(X_2,Y_2)} ....$$
For every $\n \geq 0$ there is a canonical equivalence 
$$ \Mor_{\mC^\Omega}(X_\n \to \Omega(X_{\n+1}), Y_\n \to \Omega(Y_{\n+1})) \simeq \Mor_\mC(X_\n,Y_\n)\times_{\Mor_\mC(X_{n},\Omega(Y_{\n+1}))} \Mor_\mC(X_{\n+1},Y_{\n+1}).$$
Hence by the pasting law there is a canonical equivalence
$$\Mor_{\Pre\Sp(\mC)}(X,Y) \simeq \Mor_\mC(X_0,Y_0)\times_{\Mor_\mC(X_{0},\Omega(Y_{1}))} 
\Mor_\mC(X_1,Y_1)\times_{\Mor_\mC(X_{1},\Omega(Y_{2}))}...$$

\end{proof}

\begin{corollary}\label{compo}
Let $\mC$ be a reduced bioriented category that admits endomorphisms
and $X \in \Pre\Sp(\mC), Y \in \Sp(\mC).$ There is a canonical equivalence of $\infty$-categories:	
$$\Mor_{\Pre\Sp(\mC)}(X,Y) \simeq \lim(... \to \Mor_{\mC}(X_2,Y_2) \rightarrow
\Mor_{\mC}(X_1,Y_1) \rightarrow \Mor_{\mC}(X_0,Y_0)).$$	

\end{corollary}

\begin{notation}Let $\mC$ be a reduced bioriented category that admits endomorphisms.
Let $$\gamma: \Pre\Sp(\mC) \to \mC$$
be the composition of projection to the final object of the diagram $\Pre\Sp(\mC) \to \mC^\Omega $
followed by the composition $$\mC^\Omega \to \mC^{\bD^1} \to \mC^{\{0\}}.$$	
\end{notation}

\begin{example}Let $\mC$ be a reduced bioriented category that admits endomorphisms and suspensions. Let $X \in \mC.$ The infinite suspension prespectrum of $X$ is the prespectrum $$\Sigma_{\mathrm{pre}}^\infty(X):= (X, \Sigma(X),\Sigma^2(X),...)$$ in $\mC$, where the bonding morphisms are the unit $ \Sigma^\n(X) \to \Omega(\Sigma(\Sigma^\n(X))).$
Hence $\gamma(\Sigma_{\mathrm{pre}}^\infty(X)) \simeq X$.
\end{example}

\begin{lemma}\label{embes}
Let $\mC$ be a reduced bioriented category that admits endomorphisms.
There is an embedding of bioriented categories 
$$\Sp(\mC) \subset \Pre\Sp(\mC)$$
whose essential image precisely consists of the prespectra whose bonding morphisms are equivalences.

\end{lemma}

\begin{proof}

Let $(\mC^\Omega)' \subset \mC^\Omega$ be the full subcategory of objects
$(X,Y, \alpha: X \to \Omega(Y))$ such that $\alpha$ is an equivalence.
The bioriented functor $\mC^{\bD^1}\to \mC^{\{1\}}$ restricts to an equivalence to the full subcategory of equivalences. So the restricted projection $(\mC^\Omega)' \subset \mC^\Omega \to \mC^\cop$ is an equivalence.
The resulting embedding $\mC^\cop\simeq (\mC^\Omega)' \subset \mC^\Omega$ is a bioriented functor over $\mC \times \mC^\cop$ when we view $\mC^\cop$ over $\mC \times \mC^\cop$ via the bioriented functors $\Omega, \id$, and so induces an embedding of bioriented categories 
$$\Sp(\mC) \simeq \lim(... \xrightarrow{\Omega} \mC \xrightarrow{\Omega} \mC^\cop \xrightarrow{\Omega}\mC) \simeq ... \times_{\mC} \mC \times_{\mC^\cop} \mC^\cop \subset \Pre\Sp(\mC)=... \times_{\mC^\cop} \mC^\Omega \times_{\mC}(\mC^\cop)^\Omega \times_{\mC^\cop} \mC^\Omega.$$

\end{proof}

\begin{lemma}\label{infs} Let $\mC$ be a reduced bioriented category that admits endomorphisms and suspensions. Let $X \in \mC, Y \in \Pre\Sp(\mC).$ The bioriented functor $\gamma: \Pre\Sp(\mC) \to \mC $ induces an equivalence of $\infty$-categories:
$$\Mor_{\Pre\Sp(\mC)}(\Sigma_{\mathrm{pre}}^\infty(X),Y) \to \Mor_{\mC}(\gamma(\Sigma_{\mathrm{pre}}^\infty(X)),\gamma(Y)) \simeq \Mor_{\mC}(X,\gamma(Y)). $$

So the bioriented functor $\gamma: \Pre\Sp(\mC) \to \mC $ admits a fully faithful left adjoint $\Sigma_{\mathrm{pre}}^\infty$ that sends $X$ to $\Sigma_{\mathrm{pre}}^\infty(X).$ 

\end{lemma}

\begin{proof}
Using \cref{comp} the bioriented functor $\gamma: \Pre\Sp(\mC) \to \mC $ induces the following functor:
$$\Mor_{\Pre\Sp(\mC)}(\Sigma_{\mathrm{pre}}^\infty(X),Y) \simeq \Mor_{\mC}(X,Y_0)\times_{\Mor_{\mC}(X,\Omega(Y_1))} 
\Mor_{\mC}(X,\Omega(Y_1))\times_{\Mor_{\mC}(X,\Omega^2(Y_2))} $$$$ \Mor_{\mC}(X,\Omega^2(Y_2))\times_{\Mor_{\mC}(X,\Omega^3(Y_{3}))} ... \simeq 
\Mor_{\mC}(X,Y_0). $$	

\end{proof}

\begin{definition}\label{hiul}
Let $\mC$ be a reduced bioriented category that admits endomorphisms and sequential colimits and $X \in \Pre\Sp(\mC)$. The associated spectrum of $X$ is the following spectrum $X'$:
for every $\n \geq 0$ let $$X'_\n := \colim(X_\n \to \Omega(X_{\n+1}) \to \Omega^2(X_{\n+2}) \to ...).$$
The bonding morphism $ X'_\n \to \Omega(X'_{\n+1})$ is the canonical equivalence
$$  \colim(X_\n \to \Omega(X_{\n+1}) \to \Omega^2(X_{\n+2}) \to ...) \simeq \colim(\Omega(X_{\n+1}) \to \Omega^2(X_{\n+2}) \to \Omega^3(X_{\n+3}) \to ...) $$$$\simeq \Omega(\colim(X_{\n+1} \to \Omega(X_{\n+2}) \to \Omega^2(X_{\n+3}) \to ...)).$$

\end{definition}

\begin{remark}

There is a morphism $X \to X'$ in $\Pre\Sp(\mC)$
such that for any $\n \geq 0$ the morphism $X_\n \to X'_\n$ is the morphism to the colimit using the canonical factorization $$X_\n \to \Omega^2(X_{\n+2}) \to \Omega^2(X'_{\n+2}) \simeq X'_\n $$ of the morphism $X_\n \to X'_\n.$

\end{remark}

\begin{proposition}\label{spectri}

Let $\mC$ be a reduced bioriented category that admits endomorphisms and sequential colimits and let $X \in \Pre\Sp(\mC)$ and $Y \in \Sp(\mC).$
The following induced functor is an equivalence: $$\Mor_{\Pre\Sp(\mC)}(X',Y) \to \Mor_{\Pre\Sp(\mC)}(X,Y).$$

The bioriented embedding $\Sp(\mC)\subset \Pre\Sp(\mC)$ admits a left adjoint sending $X \in \Pre\Sp(\mC)$ to $X'$.

\end{proposition}

\begin{proof}
Using \cref{comp} the induced functor $$\Mor_{\Pre\Sp(\mC)}(X',Y) \to \Mor_{\Pre\Sp(\mC)}(X,Y)$$
factors as 
$$\theta: \lim(\Mor_{\mC}(X'_0,Y_0) \leftarrow
\Mor_{\mC}(X'_1,Y_1) \leftarrow \Mor_{\mC}(X'_2,Y_2)\leftarrow ...) ... \to $$$$ \lim(\Mor_{\mC}(X_0,Y_0) \leftarrow
\Mor_{\mC}(X_1,Y_1) \leftarrow \Mor_{\mC}(X_2,Y_2)\leftarrow ...). $$
For every $\n \geq 0$ there is a canonical equivalence
$$ \Mor_{\mC}(X'_\n,Y_\n) \simeq \lim(\Mor_{\mC}(X_\n,Y_\n) \leftarrow
\Mor_{\mC}(\Omega(X_{\n+1}),Y_\n) \leftarrow \Mor_{\mC}(\Omega^2(X_{\n+2}),Y_\n)\leftarrow ...)
\simeq $$$$\lim(\Mor_{\mC}(X_\n,Y_\n) \leftarrow
\Mor_{\mC}(\Omega(X_{\n+1}),\Omega(Y_{\n+1})) \leftarrow \Mor_{\mC}(\Omega^2(X_{\n+2}),\Omega^2(Y_{\n+2})) \leftarrow ...) \simeq $$
$$ \lim(\Mor_{\mC}(X_0,Y_0) \leftarrow ... \leftarrow \Mor_{\mC}(X_\n,Y_\n) \leftarrow
\Mor_{\mC}(\Omega(X_{\n+1}),\Omega(Y_{\n+1})) \leftarrow \Mor_{\mC}(\Omega^2(X_{\n+2}),\Omega^2(Y_{\n+2})) \leftarrow ...).$$
Commuting limits with limits we see that $\theta$ is an equivalence.

The second part follows from \cref{adj} and that $\Sp(\mC)$ is closed in $\Pre\Sp(\mC)$ under left and right cotensors:
the left cotensor is formed objectwise. The right cotensor of any $X \in \infty\Cat_* $ and $\E \in \Pre\Sp(\mC)$ is in even degree $\n \geq 0$ the right cotensor $\E_\n^X$ and in odd degree
$\n \geq 0$ the right cotensor $\E_\n^{X^\op}$ and for even $\n \geq 0$ the bonding morphism
$$\E_\n^X \to \Omega(\E_{\n+1}^{X^\op}) \simeq \Omega(\E_{\n+1})^X$$ is induced by the bonding morphism $\E_\n \to \Omega(\E_{\n+1})$ and  for odd $\n \geq 0$ the bonding morphism
$$\E_\n^{X^\op} \to \Omega(\E_{\n+1}^{X}) \simeq \Omega(\E_{\n+1})^{X^\op}$$ is induced by the bonding morphism $\E_\n \to \Omega(\E_{\n+1})$.

\end{proof}
	
\begin{corollary}\label{spectrii}
	
Let $\mC$ be a reduced compactly generated bioriented category.
The left adjoint of the bioriented embedding $\Sp(\mC)\subset \Pre\Sp(\mC)$ preserves finite limits and left and right cotensors with finite $\infty$-categories with distinguished object.
	
\end{corollary}	
	
\begin{proof}
	
The statement follows from the description of the left adjoint \ref{hiul}, \cref{spectri} and that small filtered colimits in $\mC$ commute with finite limits and cotensors with finite $\infty$-categories with distinguished object since $\mC$ is compactly generated.
	
\end{proof}

\begin{corollary}\label{infsus} Let $\mC$ be a reduced bioriented category that admits endomorphisms, suspensions and sequential colimits such that endomorphisms preserve sequential colimits.
There is a bioriented adjunction $\Sigma^\infty: \mC \rightleftarrows \Sp(\mC) : \Omega^\infty$.
\end{corollary}

\begin{remark}

The bioriented localization $\tau_0: \infty\Cat \rightleftarrows \infty\Grp$ of \cref{Grayspaces}
gives rise to a bioriented localization $$ \Cat\Pre\Sp \rightleftarrows \Pre\Sp(\infty\Grp)$$
that descends to a bioriented localization $\tau_0: \Cat\Sp \rightleftarrows \Sp(\infty\Grp)$ using \cref{spectri},
where the left adjoint takes the spectrification of the degreewise classifying space.

\end{remark}


Next we prove that negative categorical spheres are dense in the antioriented category of categorical spectra (\cref{spgen}).

\begin{lemma}\label{Geh} Let $\mC$ be a reduced bioriented category that admits endomorphisms. There is a canonical equivalence of antioriented categories
$$\Pre\Sp(\mC) \simeq {\Fun\wedge}(S(S^1)+_{\bD^0} S(S^1) +_{\bD^0} ...,\mC).$$
\end{lemma}

\begin{proof}

By \cref{susk} there is an equivalence $$ \mC^\Omega \simeq {\Fun\wedge}(S(S^1),\mC)$$ of (underlying) antioriented categories over $\mC \times \mC$
sending $(Y,X \to \Omega(Y))$ to $(X,Y, S^1 \to \R\Mor_\mC(X,Y)).$
Hence there is an equivalence of (underlying) antioriented categories
$$\Pre\Sp(\mC) \simeq ... \times_\mC \mC^\Omega \times_\mC \mC^\Omega \simeq ... \times_\mC {\Fun\wedge}(S(S^1),\mC) \times_{\mC} {\Fun\wedge}(S(S^1),\mC) \simeq {\Fun\wedge}(S(S^1)+_{\bD^0} S(S^1) +_{\bD^0} ..., \mC).$$
	
\end{proof}

\begin{proposition}\label{spgen} Let $\mC$ be a reduced presentable bioriented category.
The bioriented category $\Sp(\mC)$ is generated under small colimits and left tensors by the following spectra in $\mC$ for $\n \geq 0$ and $Y \in \mC$:
$$(Y \wedge (-))_!( \Omega^\n(S^0)).$$

The full subcategory $ \{\Omega^\n(S^0) \mid \n \geq 0 \} \subset \Cat\Sp$ of negative categorical spheres
is dense in the antioriented category $\Cat\Sp$.

\end{proposition}

\begin{proof}Let $$\Xi:= S(S^1)+_{\bD^0} S(S^1) +_{\bD^0} ....$$
By \cref{Geh} there is an equivalence $\Pre\Sp(\mC) \simeq {\Fun\wedge}(\Xi, \mC)$ of antioriented categories.
The right hand side is generated under small colimits and left tensors by the objects 
$$(Y \wedge (-)) \circ \R\Mor_\Xi(S^\n,-)$$ for $\n \geq 0$ and $Y \in \mC$.
Thus $\Pre\Sp(\mC)$ is generated under small colimits and left tensors by the corresponding prespectra $(Y \wedge (-))_!(T(\n))$ in $\mC$ for $\n \geq 0$ and $Y \in \mC$,
where $ T(\n)$ is the prespectrum corresponding to 
$\R\Mor_\Xi(S^\n,-) \in  {\Fun\wedge}(\Xi, \infty\Cat_*).$ 
Then $ T(\n)$ is the prespectrum $ (\emptyset, ..., \emptyset, S^0, S^1, S^2,...) $ in $\mC$ for $\n \geq 0$ whose $\n$-th term is $S^0$. 
By \cref{spectri} the full subcategory $\Sp(\mC) \subset \Pre\Sp(\mC)$ is a localization.
Moreover the associated spectrum of $ T(\n) $ is $\Omega^\n(S^0)$ since for every $\bk \geq 0$ there is a canonical equivalence respecting the bonding maps: $$ T(\n)'_\bk=  \colim_{\ell \geq 0}(\Omega^\ell(T(\n)_{\bk+\ell})) \simeq \colim_{\ell \geq n}(\Omega^\ell(T(\n)_{\bk+\ell})) \simeq $$$$\colim_{\ell \geq 0}(\Omega^{\ell+n}(T(\n)_{\bk+\ell+n})) \simeq \Omega^\n(\colim_{\ell \geq 0}(\Omega^{\ell}(S^{\bk+\ell}))) \simeq \Omega^\n(S^0)_\bk.$$
		
\end{proof}

\subsection{Categorical homology theories}


We introduce a categorical analogue of excision.


\begin{definition}Let $\mC, \mD$ be reduced oriented categories. 
A reduced oriented functor $\F: \mC \to \mD$ is excisive if for every oriented pushout square
\[
\begin{tikzcd}
X \ar{r}{} \ar{d}[swap]{} & 0 \ar[double]{dl}{} \ar{d}{} \\
0 \ar{r}[swap]{} & Y
\end{tikzcd}
\]
in $\mC$ the induced oriented square 
\[
\begin{tikzcd}
F(X) \ar{r}{} \ar{d}[swap]{} & 0 \ar[double]{dl}{} \ar{d}{} \\
0 \ar{r}[swap]{} & F(Y)
\end{tikzcd}
\]
in $\mD$ is an oriented pullback square.

\end{definition}

\begin{definition}
Let $\mC, \mD$ be reduced antioriented categories. A reduced antioriented functor $\F: \mC \to \mD$ is antiexcisive if the reduced oriented functor $\F^\co: \mC^\co \to \mD^\co$ is excisive.
	
\end{definition}

\begin{definition}
Let $\mC, \mD$ be reduced bioriented categories.  

\begin{enumerate}
\item A reduced bioriented functor $\F: \mC \to \mD$ is excisive if its underlying oriented functor is excisive.

\item A reduced bioriented functor $\F: \mC \to \mD$ is antiexcisive if its underlying antioriented functor is antiexcisive.

\item A reduced bioriented functor $\F: \mC \to \mD$ is biexcisive if it is excisive and antiexcisive.	
	
\end{enumerate}
	
\end{definition}

\begin{notation}
	
Let $\mC, \mD$ be reduced oriented categories. 
Let $${\Exc\wedge}(\mC,\mD) \subset {\Fun\wedge}(\mC,\mD)$$
be the full subcategory of excisive oriented functors.

Let $\mC, \mD$ be reduced antioriented categories. 
Let $${\wedge \overline{\Exc}}(\mC,\mD) \subset {\wedge\Fun}(\mC,\mD)$$
be the full subcategory of antiexcisive antioriented functors.

\end{notation}
\begin{notation}

Let $\mC, \mD$ be reduced bioriented categories. 
Let $${\wedge\Exc\wedge}(\mC,\mD),\ {\wedge\overline{\Exc}\wedge}(\mC,\mD), \  {\wedge\Bi\Exc\wedge}(\mC,\mD) \subset {\wedge\Fun\wedge}(\mC,\mD)$$
be the full subcategories of excisive, antiexcisive, biexcisive reduced bioriented functors.

\end{notation}

\begin{lemma}\label{coext}Let $\mD$ be a reduced bioriented category and $\E\in \overline{\Sp}(\mD).$  The oriented functor $$\L\Mor_{\overline{\Sp}(\mD)}(-, \E): \overline{\Sp}(\mD)^\circ \to \infty\Cat_*$$ is excisive. 
\end{lemma}
\begin{proof}
	
The right adjoint oriented functor $$\L\Mor_{\overline{\Sp}(\mD)}(-, \E): \overline{\Sp}(\mD)^\circ \xrightarrow{} \infty\Cat_*$$ preserves endomorphisms. So for every $Z \in \overline{\Sp}(\mD) $ the canonical functor $$ \L\Mor_{\overline{\Sp}(\mD)}(\bar{\Sigma}(Z), \E) \to \Omega(\L\Mor_{\overline{\Sp}(\mD)}(Z, \E))$$
is an equivalence. Since $\overline{\Sp}(\mD)$ is quasi-antistable by \cref{crstab}, the canonical functor $$ \L\Mor_{\overline{\Sp}(\mD)}(Z, \E) \to \Omega(\L\Mor_{\overline{\Sp}(\mD)}(\bar{\Omega}(Z), \E))$$ is an equivalence. Note that the antiendomorphisms $\bar{\Omega}$ of $\overline{\Sp}(\mD) $ are the suspension of $\overline{\Sp}(\mD)^\circ.$

\end{proof}

\begin{remark}\label{Homt}
Let $\mB, \mC, \mD, \mE$ be reduced oriented categories, $\F: \mC \to \mD$ a reduced and excisive oriented functor, $\phi: \mB \to \mC$ a reduced oriented functor that preserves suspensions and $\rho: \mD \to \mE$ a reduced oriented functor that preserves endomorphisms. The reduced oriented functor $\rho \circ \F\circ \phi: \mB \to \mE$ is excisive.
	
\end{remark}

\begin{remark}
Let  $\mC$ be a reduced oriented category that admits suspensions and $\mD$ a reduced oriented category that admits small filtered colimits and endomorphisms such that both commute with each other. By the definition of excision the full subcategory ${\Exc\wedge}(\mC,\mD)$
is  closed in ${\Fun\wedge}(\mC, \mD) $ under small filtered colimits.
	
\end{remark}

Now we are ready to state the categorical Eilenberg-Steenrod axioms.

\begin{construction}
	
Let $\mD$ be a reduced bioriented category that admits endomorphisms, $F: \infty\Cat_* \to \mD$ a reduced excisive oriented functor and $X \in \infty\Cat_*.$ 
For every even $\ell \geq 0$ the canonical morphism
$$ F(S^\ell)\wedge X \to F(S^\ell \wedge X)$$ gives rise to a canonical morphism
$$ \Omega^\ell(F(S^\ell)\wedge X) \to \Omega^\ell(F(S^\ell \wedge X)) \simeq \Omega^\ell(F(X \wedge S^\ell)) \simeq F(X).$$
The family of such morphisms for even $\ell \geq 0$ induces a canonical morphism:
\begin{equation}\label{eqzz}
\Omega^\infty(F(S^\bullet)\wedge X) \simeq \colim(F(S^0) \wedge X \to \Omega(F(S^1)\wedge X) \to \Omega^2(F(S^2)\wedge X) \to...)  \simeq \end{equation}	$$ \colim(F(S^0) \wedge X \to \Omega^2(F(S^2)\wedge X) \to \Omega^4(F(S^4)\wedge X) \to ...) \to F(X),$$
where the first equivalence is by \cref{spectri} and the second equivalence is by cofinality.
	
\end{construction}

\begin{definition}Let $\mD$ be a reduced bioriented category that admits endomorphisms.
A categorical homology theory is a reduced oriented functor $E: \infty\Cat_* \to \mD$ 
that satisfies the following axioms:
\begin{enumerate}
\item $E$ preserves small filtered colimits.

\vspace{1mm}

\item $E$ is excisive.

\vspace{1mm}

\item $E$ is spherical, i.e. for every $X \in \infty\Cat_*$ the morphism (\ref{eqzz}) is an equivalence.

\end{enumerate}
\end{definition}

	



\begin{example}\label{homol}Let $\mD$ be a reduced presentable bioriented category and $\E$ a spectrum in $\mD$. The oriented functor $$\infty\Cat_* \to \mD, X \mapsto \Omega^\infty(\E \wedge X)$$
is a categorical homology theory since it is reduced and excisive by stability (\cref{stab}), preserves small filtered colimits, and the morphism (\ref{eqzz}) is an equivalence, where we use that $\Omega^\infty(\E \wedge S^\n)\simeq \E_\n$ compatible with the bonding maps. We call $ \Omega^\infty(\E \wedge (-))$ the categorical homology theory represented by $\E.$

\end{example}









For the next definition we use \cref{rigit}.

\begin{definition} Let $R$ be a commutative monoid.
Categorical $R$-homology, denoted by $ H(-; R)$, is the categorical homology theory associated to the categorical Eilenberg-MacLane spectrum $H(R).$	
For $R=\bN$ we skip $R.$
	
\end{definition}

\begin{remark}

By definition, categorical $R$-homology factors as the oriented functor
$H(R) \wedge (-): \infty\Cat_* \to D_{R}= {_{H(R)}\Mod(\Cat\Sp)} $
followed by the oriented forgetful functor $D_{R} \to \Cat\Sp $
followed by the oriented functor $\Omega^\infty: \Cat\Sp \to \infty\Cat_*.$

\end{remark}

\begin{example}\label{homosph} Let $R$ be a commutative monoid and $\n \geq 0.$ The categorical $R$-homology of $S^\n$ is 
$$ \Omega^\infty(H(R)\wedge  S^\n) \simeq \Omega^\infty(H(R)\wedge \Sigma^\n(S^0)) \simeq  \Omega^\infty(\Sigma^\n(H(R)\wedge S^0)) \simeq \Omega^\infty(\Sigma^\n(H(R))) \simeq B^\n(R).$$

\end{example}




\subsection{Excisive approximation of bioriented functors}

In the following we study excisive approximations of bioriented functors, which is the easier case, and after that turn to excisive approximations of oriented functors.

We construct the excisive approximation of bioriented functors (\cref{{excista}}).
We characterize spectrum objects by a universal property in terms of excisive bioriented functors (\cref{lifti}), which is a crucial ingredient in the proof of the categorical Brown representability theorem. 


\begin{definition}

A bioriented category is differentiable if it is reduced,
admits endomorphisms and sequential colimits and forming endomorphisms preserves sequential colimits.

\end{definition}

\begin{notation}

Let $\mC,\mD$ be reduced bioriented categories such that $\mC$ admits suspensions and $\mD$ is differentiable.
Let $\F:\mC \to \mD$ be a reduced bioriented functor.
For every $\ell \geq 0$ the unit $\id \to \Omega\circ \Sigma$ and the canonical map $\Sigma \circ \F \to \F\circ \Sigma$ of reduced bioriented functors $\mC \to \mD$ give rise to a map of reduced bioriented functors $\mC \to \mD$:
$$ \Omega^\ell \circ \F \circ \Sigma^\ell \to \Omega^\ell \circ \Omega \circ \Sigma \circ \F \circ \Sigma^\ell \to
\Omega^{\ell} \circ \Omega \circ \F \circ \Sigma \circ \Sigma^{\ell} = \Omega^{\ell+1} \circ \F \circ \Sigma^{\ell+1}.$$	

Let $$ L(\F):= \colim(\F \to \Omega \circ \F \circ \Sigma \to ... \to \Omega^\ell \circ \F \circ \Sigma^\ell \to \Omega^{\ell+1} \circ \F \circ \Sigma^{\ell+1}\to ...)$$ be the sequential colimit of bioriented functors $\mC \to \mD$
and $\alpha_\F: \F \to L(\F)$ the map to the zero-th term.

\end{notation}

\begin{remark}
The canonical map $ L(\F) \to \Omega \circ L(\F) \circ \Sigma$ identifies with the canonical equivalence
$$ \colim_{\ell \geq 0}(\Omega^\ell \circ \F \circ \Sigma^\ell) \to \colim_{\ell \geq 0}(\Omega^{\ell+1} \circ \F \circ \Sigma^{\ell+1})\simeq \colim_{\ell \geq 1}(\Omega^\ell \circ \F \circ \Sigma^\ell).$$
Hence $L(\F)$ is reduced and excisive.	
So sending $\F$ to $L(\F)$ defines a functor $${\wedge\Fun\wedge}(\mC,\mD) \to {\wedge\Exc\wedge}(\mC,\mD)$$ and the canonical map $\F \to L(\F)$ defines a natural transformation $\alpha: \id \to L$.

If $\F:\mC \to \mD$ is antiexcisive, $ \Omega^\ell \circ \F \circ \Sigma^\ell$ is antiexcisive for every $\ell \geq 0$ by \cref{Homt} . So if sequential colimits in $\mD$ commute with antiendomorphisms, the sequential colimit $L(\F)$ is antiexcisive and so biexcisive.

\end{remark}

\begin{proposition}\label{excista}
Let $\mC,\mD$ be reduced bioriented categories such that $\mC$ admits suspensions and $\mD$ is differentiable.
Let $\F:\mC \to \mD$ be a reduced bioriented functor.
The map $\alpha_\F: \F \to L(\F)$ induces for every reduced excisive bioriented functor $\G:\mC \to \mD$ an equivalence $$ {\wedge\Fun\wedge}(\mC,\mD)(L(\F),\G) \to {\wedge\Fun\wedge}(\mC,\mD)(\F,\G).$$	

In other words, the natural transformation $\alpha: \id \to L$ exhibits $L$ as left adjoint of the embedding $$ {\wedge\Exc\wedge}(\mC,\mD) \subset {\wedge\Fun\wedge}(\mC,\mD).$$

If sequential colimits in $\mD$ commute with antiendomorphisms, the localization
$$L:  {\wedge\Fun\wedge}(\mC,\mD) \rightleftarrows  {\wedge\Exc\wedge}(\mC,\mD)$$
restricts to a localization
$L:  {\wedge\Exc\wedge}(\mC,\mD) \rightleftarrows  {\wedge\Bi\Exc\wedge}(\mC,\mD)$.

\end{proposition}

\begin{proof}

If $\F$ is excisive, the canonical map $\F \to \Omega^\ell \circ \F \circ \Sigma^\ell$
is an equivalence so that all maps under $\F$ in the sequential diagram for $L(\F)$ are equivalences. Hence the map $\alpha_\F: \F \to L(\F)$ is an equivalence. 
By \cite[Proposition 5.2.7.4]{lurie.HTT} it suffices to show that the map $L(\alpha_\F): L(\F) \to L(L(\F))$ is an equivalence. It identifies with the following identity:
$$ \colim_{\ell \geq 0}(\Omega^\ell \circ \F \circ \Sigma^\ell) \to \colim_{\ell \geq 0}(\Omega^\ell \circ \colim_{\bk \geq 0}(\Omega^\bk \circ \F \circ \Sigma^\bk) \circ \Sigma^\ell) \simeq$$$$ \colim_{\ell,\bk \geq 0}(\Omega^{\ell+\bk} \circ \F \circ \Sigma^{\ell+\bk}) \simeq \colim_{\ell \geq 0}(\Omega^{2\ell} \circ \F \circ \Sigma^{2\ell}) \simeq \colim_{\ell \geq 0}(\Omega^\ell \circ \F \circ \Sigma^\ell).$$
The second equivalence holds since sequential diagrams are sifted. The last equivalence is by cofinality.

\end{proof}	

\begin{corollary}\label{excis}
Let $\mC$ be a bioriented category that admits suspensions and endomorphisms and $\n \geq 0$. 
The unit $\Sigma^\n \to \Omega^\infty\Sigma^\infty\Sigma^\n$
induces for every reduced excisive bioriented functor $\G:\mC \to \mD$ an equivalence $$ {\wedge\Fun\wedge}(\mC,\mD)(\Omega^\infty\Sigma^\infty\Sigma^\n,\G) \to {\wedge\Fun\wedge}(\mC,\mD)(\Sigma^\n,\G).$$	
The unit $\Omega^\n \to \Omega^\n\Omega^\infty\Sigma^\infty$
induces for any reduced excisive bioriented functor $\G:\mC \to \mD$ an equivalence $$ {\wedge\Fun\wedge}(\mC,\mD)(\Omega^\n\Omega^\infty\Sigma^\infty,\G) \to {\wedge\Fun\wedge}(\mC,\mD)(\Omega^\n,\G).$$		

\end{corollary}

\begin{proposition} Let $\mC,\mD$ be reduced bioriented categories such that $\mD$ admits left and right tensors.
	
\begin{enumerate}

\item A reduced bioriented functor $\mC \to \mD$ is excisive if and only if it is local with respect to the following canonical maps in $ {\wedge\Fun\wedge}(\mC,\mD)$ for $X \in \mC$ and $Y \in \mD:$ $$((-)\wedge Y\wedge(-)) \circ \Sigma \circ \Omega \circ \Mor_\mC(X,-) \to ((-)\wedge Y\wedge(-)) \circ \Mor_\mC(X,-). $$

\item A reduced bioriented functor $\mC \to \mD$ is antiexcisive if and only if it is local with respect to the following canonical maps in $ {\wedge\Fun\wedge}(\mC,\mD)$ for $X \in \mC$ and $Y \in \mD: $$$((-)\wedge Y\wedge(-)) \circ \bar{\Sigma} \circ \bar{\Omega} \circ \Mor_\mC(X,-) \to ((-)\wedge Y\wedge(-)) \circ \Mor_\mC(X,-). $$
\end{enumerate}
	
\end{proposition}

\begin{proof}
	
(1) is dual to (2). We prove (1). There is a canonical equivalence $$\Omega \circ \Mor_\mC(X,-) \simeq 
\Mor_\mC(X,-) \circ \Omega \simeq \Mor_\mC(\Sigma(X),-)$$ of reduced bioriented functors 
$\mC \to \infty\Cat_* \boxplus \infty\Cat_*$.
The first equivalence holds since by \cref{impol} right adjoint reduced bioriented functors commute with $\Omega.$ The second equivalence holds since $\Sigma$ is left adjoint to $\Omega$ as reduced bioriented functors. 
For every reduced bioriented functor $\F:\mC \to \mD$ the induced map $$ \Map_{{\wedge\Fun\wedge}(\mC,\mD)}(((-)\wedge Y\wedge(-)) \circ \Mor_\mC(X,-) ,\F) \to \Map_{{\wedge\Fun\wedge}(\mC,\mD)}(((-)\wedge Y\wedge(-)) \circ \Sigma \circ \Mor_\mC(\Sigma(X),-) ,\F) $$ identifies by adjointness with the map
$$ \Map_{{\wedge\Fun\wedge}(\mC,\infty\Cat_*\boxplus \infty\Cat_*)}(\Mor_{\mC}(X,-), \Mor_{\mD}(Y,-) \circ \F)\to$$$$ \Map_{{\wedge\Fun\wedge}(\mC,\infty\Cat_*\boxplus \infty\Cat_*)}(\Mor_{\mC}(\Sigma(X),-), \Omega \circ \Mor_{\mD}(Y,-) \circ \F)$$
which identifies by the enriched Yoneda lemma with the map
$$ \Map_\mD(Y, \F(X)) \to \Map_\mD(Y, \Omega(\F(\Sigma(X))))$$  induced by the canonical morphism
$\F(X) \to \Omega(\F(\Sigma(X))).$
	
\end{proof}

\begin{corollary}\label{laxoplax2}
	
Let $\mC$ be a small reduced bioriented category and $\mE$ a reduced bioriented category
that admits left and right tensors and small colimits.
The following induced functors are equivalences:
\begin{enumerate}
\item $${\wedge\L\Fun\wedge}({\wedge\Exc\wedge}(\mC^\circ,\infty\Cat_*\boxplus \infty\Cat_*) ,\mE) \to {\wedge\Exc\wedge}(\mC,\mE), $$ 
\item $${\wedge\L\Fun\wedge}({\wedge\overline{\Exc}\wedge}(\mC^\circ,\infty\Cat_*\boxplus \infty\Cat_*) ,\mE) \to {\wedge\overline{\Exc}\wedge}(\mC,\mE), $$ 
\item $${\wedge\L\Fun\wedge}({\wedge\Bi\Exc\wedge}(\mC^\circ,\infty\Cat_*\boxplus \infty\Cat_*) ,\mE) \to {\wedge\Bi\Exc\wedge}(\mC,\mE).$$ 
	
\end{enumerate}	
\end{corollary}

\begin{theorem}\label{lifti} Let $\mC,\mD$ be reduced generalized bioriented categories.

\begin{enumerate}\item 
Assume that $\mC$ admits endomorphisms and
$\mD$ admits suspensions.
The bioriented functor $\Omega^\infty: \Sp(\mC) \to \mC $ induces an equivalence:
\begin{equation*}\label{ravelo}
{\wedge\Exc\wedge}(\mD, \Sp(\mC)) \to {\wedge\Exc\wedge}(\mD, \mC).	
\end{equation*}
The inverse sends $\F$ to $\{\F \circ \Sigma^\n, \F \circ \Sigma^\n \simeq \Omega \circ \F \circ \Sigma^{\n+1} \}.$
	
\vspace{1mm}

\item Assume that $\mC$ admits antiendomorphisms and
$\mD$ admits antisuspensions.
The bioriented functor $\Omega^\infty: \overline{\Sp}(\mC) \to \mC $ induces an equivalence:
\begin{equation*}\label{ravelo2}
{\wedge\overline{\Exc}\wedge}(\mD, \overline{\Sp}(\mC)) \to {\wedge\overline{\Exc}\wedge}(\mD, \mC).	
\end{equation*}
The inverse sends $\F$ to $\{\F \circ \bar{\Sigma}^\n, \F \circ \bar{\Sigma}^\n \simeq \bar{\Omega} \circ \F \circ \bar{\Sigma}^{\n+1} \}.$

\vspace{1mm}
	
\item Assume that $\mC$ admits endomorphisms, antiendomorphisms and $\mD$ has suspensions, antisuspensions. The bioriented functor $\Omega^\infty: \overline{\Sp}(\Sp(\mC)) \simeq \Sp(\overline{\Sp}(\mC)) \to \mC $ induces an equivalence:
\begin{equation*}\label{ravelo3}
{\wedge\Bi\Exc\wedge}(\mD, \overline{\Sp}(\Sp(\mC))) \to {\wedge\Bi\Exc\wedge}(\mD, \mC).\end{equation*}

\end{enumerate}

\end{theorem}
\begin{proof}
	
(1): By \cref{Homt} the bioriented functors $\Omega_*, \Sigma^*: {\wedge\Fun\wedge}(\mD, \mC) \to {\wedge\Fun\wedge}(\mD, \mC)$ postcomposing with $\Omega:\mC \to \mC$ and precomposing with $\Sigma: \mD \to \mD $ restrict to functors $$\Omega_*, \Sigma^*: {\wedge\Exc\wedge}(\mD, \mC) \to {\wedge\Exc\wedge}(\mD, \mC).$$
By \cref{impol} there is a natural transformation $ \sigma: \id \to \Omega_* \circ \Sigma^* \simeq \Sigma* \circ \Omega_*$ of endofunctors of ${\wedge\Fun\wedge}(\mD, \mC)$
whose component at any $\F \in {\wedge\Fun\wedge}(\mD, \mC)$ is the canonical morphism $\F \to \Omega \circ \F \circ \Sigma.$ The latter is an equivalence if $\F \in {\wedge\Exc\wedge}(\mD, \mC)$. Hence $\sigma$ restricts to an equivalence $\id \to \Omega_* \circ \Sigma^* \simeq \Sigma^* \circ \Omega_*$ of endofunctors of ${\wedge\Exc\wedge}(\mD, \mC)$.
So the functors $$\Omega_*, \Sigma^*: {\wedge\Exc\wedge}(\mD, \mC) \to {\wedge\Exc\wedge}(\mD, \mC)$$ are inverse to each other.
The functor of the statement factors as
$${\wedge\Exc\wedge}(\mD, \Sp(\mC)) \simeq \lim(... \xrightarrow{\Omega_*} {\wedge\Exc\wedge}(\mD, \mC)) \xrightarrow{\gamma} {\wedge\Exc\wedge}(\mD, \mC),$$
where $\gamma$ is the projection to the rightmost object.
Since $\Omega_*: {\wedge\Exc\wedge}(\mD, \mC) \to {\wedge\Exc\wedge}(\mD, \mC)$ is an equivalence, $\gamma$ is an equivalence.
(2) follows from (1) by replacing $\mC$ by $\mC^\co.$

(3): By the description of the inverse the equivalence of (1) restricts to an equivalence $${\wedge\Bi\Exc\wedge}(\mD, \Sp(\mC)) \to {\wedge\Bi\Exc\wedge}(\mD, \mC)$$
and the equivalence of (2) restricts to an equivalence $${\wedge\Bi\Exc\wedge}(\mD, \overline{\Sp}(\mC)) \to {\wedge\Bi\Exc\wedge}(\mD, \mC).$$ Hence the latter equivalence gives rise to an equivalence $${\wedge\Bi\Exc\wedge}(\mD, \overline{\Sp}(\Sp(\mC))) \to {\wedge\Bi\Exc\wedge}(\mD, \Sp(\mC))$$ since $\Sp(\mC)$ admits antiendomorphisms if $\mC$ does. 
The functor of (3) factors as $${\wedge\Bi\Exc\wedge}(\mD, \overline{\Sp}(\Sp(\mC))) \to {\wedge\Bi\Exc\wedge}(\mD, \Sp(\mC)) \to {\wedge\Bi\Exc\wedge}(\mD, \mC).$$

\end{proof}

\begin{lemma}\label{collif}
Let $\mC$ be a reduced generalized bioriented category that admits endomorphisms and
$\mD$ a reduced generalized bioriented category that admits suspensions.
An excisive reduced bioriented functor $\F:\mD \to \Sp(\mC)$ preserves small colimits and left and right tensors if the excisive reduced bioriented functor $\Omega^\infty \circ \F: \mD \to \mC$ 
preserves small colimits and left and right tensors.	
	
\end{lemma}

\begin{proof}
	
Let $K \in \infty\Cat_*, X \in \mC$. The canonical morphism $$\Omega^\infty\Sigma^\n(K \wedge \F(X)) \to \Omega^\infty\Sigma^\n(\F(K \wedge X)) $$ 
identifies with the morphism
$$   \colim_{\bk \geq 0} \Omega^\bk(K \wedge \Omega^\infty(\F(\Sigma^{n+\bk}(X)))) \to 
\colim_{\bk \geq 0} \Omega^\bk(\Omega^\infty(\F(K \wedge\Sigma^{n+\bk}(X)))) \simeq $$$$ \colim_{\bk \geq 0} \Omega^\bk(\Omega^\infty(\F(\Sigma^{n+\bk}(K \wedge X)))) \simeq \Omega^\infty(\F(\Sigma^\n(K \wedge X))).$$ 
The case of right tensors is similar.
Let $\mJ$ be a category and $\phi: \mJ \to \mD$ a functor.
The canonical morphism $$\Omega^\infty\Sigma^\n(\colim(\F \circ \phi)) \to \Omega^\infty\Sigma^\n(\F(\colim(\phi)) $$ 
identifies with the morphism
$$ \colim_{\bk \geq 0} \Omega^\bk(\colim(\Omega^\infty \circ \F \circ \Sigma^{n+\bk} \circ \phi))  \to 
\colim_{\bk \geq 0} \Omega^\bk(\Omega^\infty(\F(\colim(\Sigma^{n+\bk} \circ \phi)))) $$$$ \simeq \colim_{\bk \geq 0} \Omega^\bk(\Omega^\infty(\F(\Sigma^{n+\bk}(\colim(\phi))))) \simeq \Omega^\infty(\F(\Sigma^\n(\colim(\phi)))).$$ 
	
\end{proof}

We will see next that the theory of excisive approximations implies that every stable bioriented category bispectral.

\begin{notation}Let $\mH$ be a set of reduced bioriented weights and $\mC, \mD$ reduced bioriented categories.
\begin{enumerate}
\item Let $${\wedge\Exc^\mH\wedge}(\mC, \mD) \subset {\wedge\Fun\wedge}(\mC,\mD)$$ be the full subcategory of excisive bioriented functors preserving $\mH$-weighted limits.

\item Let $${\wedge\overline{\Exc}^\mH\wedge}(\mC, \mD) \subset {\wedge\Fun\wedge}(\mC,\mD)$$ be the full subcategory of antiexcisive bioriented functors preserving $\mH$-weighted limits.

\item Let $${\wedge\Bi\Exc^\mH\wedge}(\mC, \mD) \subset {\wedge\Fun\wedge}(\mC,\mD)$$ be the full subcategory of biexcisive bioriented functors preserving $\mH$-weighted limits.

\end{enumerate}

\end{notation}

\begin{proposition}Let $\mH$ be a set of reduced bioriented weights and $\mD, \mE$ presentable reduced bioriented categories.

\begin{enumerate}
\item Let $\mC$ be a reduced bioriented category that admits $\mH$-weighted colimits and is quasi-stable if $\mH$ is not empty.
The reduced bioriented category $ {\wedge\Exc^\mH \wedge}(\mC, \mD\boxplus \mE)$ is presentable and antistable.	
	
\item Let $\mC$ be a reduced bioriented category that admits $\mH$-weighted colimits and is quasi-antistable if $\mH$ is not empty.
The reduced bioriented category $ {\wedge\overline{\Exc}^\mH \wedge}(\mC, \mD\boxplus \mE)$ is presentable and stable.

\item Let $\mC$ be a reduced bioriented category that admits $\mH$-weighted colimits and is quasi-bistable if $\mH$ is not empty.
The reduced bioriented category $ {\wedge\Bi\Exc^\mH \wedge}(\mC, \mD\boxplus \mE)$ is presentable and bistable.

\end{enumerate}	

\end{proposition}

\begin{proof}
By \cref{weiloc} the full subcategory of ${\wedge\Fun\wedge}(\mC, \mD  \boxplus \mE)$ of reduced bioriented functors preserving $\mH$-weighted limits is an accessible localization.
By \cref{excista} the respective full subcategories of ${\wedge\Fun\wedge}(\mC, \mD  \boxplus \mE)$ of reduced excisive, antiexcisive, biexcisive bioriented functors are accessible localizations.
Thus the intersections $$ {\wedge\Exc^\mH\wedge}(\mC,\mD  \boxplus \mE), {\wedge\overline{\Exc}^\mH\wedge}(\mC,\mD  \boxplus \mE), {\wedge\Bi\Exc^\mH\wedge}(\mC,\mD  \boxplus \mE) \subset {\wedge\Fun\wedge}(\mC, \mD  \boxplus \mE)$$ are accessible localizations.
Moreover these full subcategories are closed in $ {\wedge\Fun\wedge}(\mC, \mD  \boxplus \mE)$ under left and right cotensors by the description of cotensors \ref{tencot} and so by \cref{adj} accessible bioriented localizations and presentable reduced bioriented categories.

The respective reduced bioriented adjunctions $$\hspace{9mm} \Sigma^\infty : \mD \boxplus \mE \rightleftarrows \overline{\Sp}(\mD \boxplus \mE) \simeq \Cat\Sp \ot_{\infty\Cat_*} \mD \ot \mE \simeq \overline{\Sp}(\mD) \boxplus \mE: \Omega^\infty, $$
$$\Sigma^\infty : \mD \boxplus \mE \rightleftarrows \Sp(\mD \boxplus \mE) \simeq \mD \ot \mE \ot_{\infty\Cat_*} \Cat\Sp \simeq \mD \boxplus \Sp(\mE): \Omega^\infty, $$
$$\hspace{31mm}\Sigma^\infty : \mD \boxplus \mE \rightleftarrows \overline{\Sp}(\Sp(\mD \boxplus \mE)) \simeq \Cat\Sp \ot_{\infty\Cat_*} \mD \ot \mE \ot_{\infty\Cat_*} \Cat\Sp \simeq \overline{\Sp}(\mD) \boxplus \Sp(\mE): \Omega^\infty $$ give rise to reduced bioriented adjunctions
$$  {\wedge\Exc\wedge}(\mC, \mD \boxplus\mE) \rightleftarrows  {\wedge\Exc\wedge}(\mC, \overline{\Sp}(\mD) \boxplus \mE), $$	
$$  {\wedge\overline{\Exc}\wedge}(\mC, \mD \boxplus\mE) \rightleftarrows  {\wedge\overline{\Exc}\wedge}(\mC, \mD \boxplus \Sp(\mE)), $$	
$$  {\wedge\Bi\Exc\wedge}(\mC, \mD \boxplus\mE) \rightleftarrows  {\wedge\Bi\Exc\wedge}(\mC, \overline{\Sp}(\mD) \boxplus \Sp(\mE)), $$	
whose right adjoints induce an equivalence on underlying categories by \cref{lifti}
and so are equivalences of bioriented categories.
If $\mC$ is quasi-stable, quasi-antistable, quasi-bistable, respectively, by the description of the inverse of \cref{lifti} the latter equivalences restrict to equivalences
$$  {\wedge\Exc^\mH\wedge}(\mC, \mD \boxplus\mE) \simeq  {\wedge\Exc^\mH\wedge}(\mC, \overline{\Sp}(\mD) \boxplus \mE), $$	
$$  {\wedge\overline{\Exc}^\mH\wedge}(\mC, \mD \boxplus\mE) \simeq  {\wedge\overline{\Exc}^\mH\wedge}(\mC, \mD \boxplus \Sp(\mE)), $$	
$$  {\wedge\Bi\Exc^\mH\wedge}(\mC, \mD \boxplus\mE) \simeq  {\wedge\Bi\Exc^\mH\wedge}(\mC, \overline{\Sp}(\mD) \boxplus \Sp(\mE)). $$	

By \cref{stab} the presentable bioriented category $ \mD \boxplus \Sp(\mE) \simeq \Sp(\mD \boxplus \mE)$ is stable, $ \overline{\Sp}(\mD) \boxplus \mE \simeq \overline{\Sp}(\mD \boxplus \mE)$ is antistable and $\overline{\Sp}(\mD) \boxplus \Sp(\mE)\simeq \overline{\Sp}(\Sp(\mD \boxplus \mE)) $ is bistable.
The description of left and right tensors (\cref{tencot}) guarantees that the presentable bioriented category ${\wedge\Fun\wedge}(\mC, \mD \boxplus \Sp(\mE))$ is stable, ${\wedge\Fun\wedge}(\mC, \overline{\Sp}(\mD) \boxplus \mE)$ is antistable and ${\wedge\Fun\wedge}(\mC, \overline{\Sp}(\mD) \boxplus \Sp(\mE))$ is bistable and the full subcategories $${\wedge\Exc^\mH\wedge}(\mC,\overline{\Sp}(\mD) \boxplus \mE) \subset {\wedge\Fun\wedge}(\mC, \overline{\Sp}(\mD) \boxplus \mE), $$$$
{\wedge\overline{\Exc}^\mH\wedge}(\mC, \mD \boxplus \Sp(\mE)) \subset {\wedge\Fun\wedge}(\mC, \mD \boxplus \Sp(\mE)), $$$$ {\wedge\Bi\Exc^\mH\wedge}(\mC,\overline{\Sp}(\mD) \boxplus \Sp(\mE)) \subset {\wedge\Fun\wedge}(\mC, \overline{\Sp}(\mD) \boxplus \Sp(\mE))$$ are closed under endomorphisms, antiendomorphisms, suspensions and antisuspensions, and so are stable, antistable, bistable, respectively.

\end{proof}

\begin{notation}
By \cite[Notation 5.36., Theorem 4.89.]{heine2024bienrichedinftycategories} for every small reduced bioriented category $\mC$ there is a reduced bioriented embedding $$\mC \hookrightarrow {\wedge\Fun\wedge}(\mC^\circ, \infty\Cat_*\boxplus \infty\Cat_*),$$
which we call bioriented Yoneda embedding, that preserves weighted limits by \cref{weiadj}.

\end{notation}

\begin{remark}Let $\mC$ be a small reduced bioriented category.

\begin{enumerate}

\item If $\mC$ is quasi-stable, the full subcategory $ {\wedge\Exc^\mH\wedge}(\mC, \mD) \subset {\wedge\Fun\wedge}(\mC, \mD) $ precisely consists of the reduced bioriented functors preserving endomorphisms and $\mH$-weighted limits. 

Thus if $\mC$ is quasi-antistable (so that $\mC^\circ$ is quasi-stable), the bioriented Yoneda embedding induces a bioriented embedding 
$$\mC \hookrightarrow {\wedge\Exc^\mH \wedge}(\mC^\circ, \infty\Cat_*\boxplus \infty\Cat_*) $$
into an antistable presentable bioriented category
that preserves weighted limits and $\mH$-weighted colimits by \cref{weiloc}.

\item If $\mC$ is quasi-antistable, the full subcategory $ {\wedge\overline{\Exc}^\mH\wedge}(\mC, \mD) \subset {\wedge\Fun\wedge}(\mC, \mD) $ precisely consists of the reduced bioriented functors preserving antiendomorphisms and $\mH$-weighted limits. 

Thus if $\mC$ is quasi-stable (so that $\mC^\circ$ is quasi-antistable), the bioriented Yoneda embedding induces a bioriented embedding 
$$\mC \hookrightarrow {\wedge\overline{\Exc}^\mH \wedge}(\mC^\circ, \infty\Cat_*\boxplus \infty\Cat_*) $$
into a stable presentable bioriented category
that preserves weighted limits and $\mH$-weighted colimits.

\item If $\mC$ is quasi-bistable, the full subcategory $ {\wedge\Bi\Exc^\mH\wedge}(\mC, \mD) \subset {\wedge\Fun\wedge}(\mC, \mD) $ precisely consists of the reduced bioriented functors preserving endomorphisms, antiendomorphisms and $\mH$-weighted limits. 
Thus if $\mC$ is quasi-bistable (so that $\mC^\circ$ is quasi-bistable), the bioriented Yoneda embedding induces a bioriented embedding
$$\mC \hookrightarrow {\wedge\Bi\Exc^\mH \wedge}(\mC^\circ, \infty\Cat_*\boxplus \infty\Cat_*) $$ 
into a bistable presentable bioriented category
that preserves weighted limits and $\mH$-weighted colimits.

\end{enumerate}

\end{remark}

\begin{corollary}
Let $\mH$ be a set of reduced bioriented weights and $\mC$ a reduced bioriented category that admits $\mH$-weighted colimits.
If $\mC$ is quasi-stable, quasi-antistable, quasi-bistable, respectively, there is a canonical functorial bioriented embedding $\mC \hookrightarrow \mD$ preserving $\mH$-weighted colimits and weighted limits into a stable, antistable, bistable presentable bioriented category, respectively.

\end{corollary}

We obtain the following corollary:

\begin{corollary}\label{spectral}
\begin{enumerate}
	
\item Every quasi-stable bioriented category $\mC$ refines to a spectral bioriented category. 

\item Every quasi-antistable bioriented category $\mC$ refines to an antispectral bioriented category. 

\item Every quasi-bistable bioriented category $\mC$ refines to a bispectral bioriented category. 
\end{enumerate}		
\end{corollary}

We obtain the following important corollaries:

\begin{corollary}\label{Stab}
Let $\mC$ be a quasi-stable bioriented category. 
The following are equivalent:
	
\begin{enumerate}
\item $\mC$ is stable.	
		
\item $\mC$ admits oriented fibers.
		
\item $\mC$ admits oriented cofibers.
		
\end{enumerate}
	
\end{corollary}

\begin{proof}

This follows immediately from \cref{spectral0} combined with
\cref{laxfib} and its dual.

\end{proof}

\begin{corollary}\label{sstab}
Let $\mC$ be reduced bioriented category that has oriented fibers.
Then $\Sp(\mC)$ is stable.
	
\end{corollary}
\begin{proof}
If $\mC$ has oriented fibers, $\Sp(\mC)$ has oriented fibers.
We apply \cref{rstab} and \cref{Stab}.	
	
\end{proof}

\cref{preserv} implies the following corollary:

\begin{corollary}\label{preserv2}

Let $\mC, \mD$ be bioriented categories and $\phi: \mC \to \mD$ a bioriented functor.

\begin{enumerate}
\item If $\mC,\mD$ are stable and $ \phi$ preserves endomorphisms, $\phi$ preserves oriented pushouts and oriented pullbacks. 

\item If $\mC,\mD$ are antistable and $\phi$ preserves antiendomorphisms, $\phi$ preserves antioriented pushouts and antioriented pullbacks.


\end{enumerate}

\end{corollary}

\subsection{Excisive approximation of oriented functors}

In the following we study excisive approximations of oriented functors.

We prove existence of excisive approximation of oriented functors (\cref{excloc}) and prove a formula for the excisive approximation (\cref{appro}).

\begin{lemma}\label{cloten}
	
\begin{enumerate}
\item Let $\mC$ be a reduced bioriented category that admits left tensors such that the left tensor with any small $\infty$-category with distinguished object preserves suspensions, and $ \mD$ a reduced oriented category.
The full subcategory $${\Exc\wedge}(\mC,\mD) \subset {\Fun\wedge}(\mC,\mD) $$ is closed under right tensors.

\vspace{2mm}

\item Let  $\mC$ be a reduced oriented category and $\mD$ a reduced bioriented category that admits left cotensors.
The full subcategory $$ {\Exc\wedge}(\mC, \mD) \subset {\Fun\wedge}(\mC, \mD) $$ is closed under left cotensors.
\end{enumerate}

\end{lemma}

\begin{proof}
(1): For every $Y \in \infty\Cat_*$ the right tensor with $Y$ is the functor $$(Y \wedge (-))^*: {\Fun\wedge}(\mC, \mD) \to {\Fun\wedge}(\mC, \mD).$$
So (1) follows from \cref{Homt} .
(2): For every $Y \in \infty\Cat_*$ the left cotensor with $Y$ is the functor $$ {^Y(-)}_*: {\Fun\wedge}(\mC, \mD) \to {\Fun\wedge}(\mC, \mD). $$ 
So (2) follows from \cref{Homt} .

\end{proof}

\begin{theorem}\label{excloc} Let $\mC$ be a small reduced oriented category and $\mD$ a reduced presentable oriented category. The embedding $${\Exc\wedge}(\mC,\mD) \subset {\Fun\wedge}(\mC,\mD)$$ is accessible and admits a left adjoint.
\end{theorem}

\begin{proof}

The unit $ \Sigma \circ \Omega \to \id$ of the bioriented adjunction $ \Sigma: \infty\Cat^\cop_* \rightleftarrows \infty\Cat_* :\Omega$ is a map of bioriented functors $\infty\Cat_* \to \infty\Cat_*$
and so for every $X \in \mC$ gives rise to a map in ${\Fun\wedge}(\mC,\infty\Cat_*)$:
$$ \zeta_{X} : \Sigma \circ \Omega \circ \R\Mor_{\mC}(X,-) \to \R\Mor_{\mC}(X,-).$$

Since the oriented functor $\R\Mor_{\mC}(X,-): \mC \to \infty\Cat_*$ preserves endomorphisms, there is a canonical equivalence $$ \R\Mor_{\mC}(X,\Omega(\Sigma(X))) \simeq \Omega(\R\Mor_{\mC}(X,\Sigma(X))).$$ Under this equivalence 
the unit $X \to \Omega(\Sigma(X))$ corresponds to an object of $  \Omega(\R\Mor_{\mC}(X,\Sigma(X))) $ that corresponds by the enriched Yoneda lemma to an oriented natural transformation
$$ \R\Mor_{\mC}(\Sigma(X),-) \to \Omega(\R\Mor_{\mC}(X,-))$$ of oriented functors
$\mC \to \infty\Cat_*$. The latter induces at any $Y \in \mC$ the canonical equivalence
$$ \R\Mor_{\mC}(\Sigma(X),Y) \simeq \R\Mor_{\mC}(X,\Omega(Y)) \simeq \Omega(\R\Mor_{\mC}(X,Y))$$ and so is an equivalence.
So we can view $\zeta_X$ likewise as a map
$ \Sigma \circ \R\Mor_{\mC}(\Sigma(X),-) \to \R\Mor_{\mC}(X,-).$

Let $\kappa$ be a strongly inaccessible cardinal such that the underlying category of $\mD$ is $\kappa$-compactly generated. Since $\mD^\kappa$ is small, the following is a set:
$$\Theta:= \{ ( Y \wedge (-)) \circ \zeta_{X} \mid X \in \mC, Y \in \mD^\kappa\}.$$

For every reduced oriented functor $\F: \mC  \to \mD$ and $Y \in \mD^\kappa$ and $X \in \mC$ 
the induced map $$\Map_{{\Fun\wedge}(\mC,\mD)}(( Y \wedge (-)) \circ \R\Mor_{\mC}(X,-) ,\F) \to \Map_{{\Fun\wedge}(\mC,\mD)}(( Y \wedge (-)) \circ \Sigma \circ \R\Mor_{\mC}(\Sigma(X),-) ,\F) $$ identifies by adjointness with the map
$$ \Map_{{\Fun\wedge}(\mC,\infty\Cat_*)}(\R\Mor_{\mC}(X,-), \R\Mor_\mC(Y,-) \circ  \F) \to $$$$ \Map_{{\Fun\wedge}(\mC,\infty\Cat_*)}(\Sigma \circ \R\Mor_{\mC}(\Sigma(X),-),  \R\Mor_\mC(Y,-)\circ \F) \simeq $$$$ \Map_{{\Fun\wedge}(\mC,\infty\Cat_*)}(\R\Mor_{\mC}(\Sigma(X),-),  \Omega \circ \R\Mor_\mC(Y,-)\circ \F) \simeq $$$$ \Map_{{\Fun\wedge}(\mC,\infty\Cat_*)}(\R\Mor_{\mC}(\Sigma(X),-),  \R\Mor_\mC(Y,-)\circ \Omega \circ \F),$$
which identifies by the enriched Yoneda lemma with the map
$$ \Map_\mC(Y, \F(X)) \to \Map_\mC(Y, \Omega(\F(\Sigma(X)))).$$ 
Consequently, by the enriched Yoneda lemma a reduced oriented functor $\mC  \to \mD$ is excisive if and only if it is $\Theta$-local.

\end{proof}

\begin{corollary}\label{excloci}\label{excprese}
Let $\mC$ be a small reduced oriented category and $\mD$ a reduced right presentable bioriented category that admits left cotensors.
The antioriented embedding ${\Exc\wedge}(\mC,\mD) \subset {\Fun\wedge}(\mC,\mD)$ is accessible and admits a left adjoint.
In particular, the reduced antioriented category ${\Exc\wedge}(\mC,\mD)$ is presentable.

\end{corollary}

The description of local equivalences in the proof of \cref{excloc} gives the following corollary:
\begin{corollary}\label{laxoplax}
	
Let $\mC$ be a small reduced antioriented category and $\mE$ a reduced antioriented category
that admits left tensors and small colimits.
The following induced functor is an equivalence: $${\wedge\L\Fun({\Exc\wedge}(\mC^\circ,\infty\Cat_*) ,\mE)} \to {\wedge\overline{\Exc}}(\mC,\mE) .$$ 
	
\end{corollary}

\begin{construction}Let $\mC$ be a reduced oriented category that admits suspensions.
For every $\n \geq 0$ and $X \in \mC$ the canonical equivalence $X \wedge S^\n \wedge S^1 \simeq X \wedge S^{\n+1} $
corresponds to a functor $$S^1 \to \R\Mor_{\mC}(X \wedge S^\n, X \wedge S^{\n+1})$$ that
corresponds to an oriented functor $\alpha^X_\n: S(S^1) \to \mC$.
The family of oriented functors $\alpha^X_\n: S(S^1) \to \mC$ for $\n \geq 0$
determines an oriented functor $$\alpha^X: S(S^1)+_{\bD^0} S(S^1) +_{\bD^0} ... \to \mC .$$

\end{construction}

\begin{notation}\label{forgetu} Let $\mD$ be a reduced bioriented category.
Let $$\nu^X: {\Fun\wedge}(\mC,\mD) \to {\Fun\wedge}(S(S^1)+_{\bD^0} S(S^1) +_{\bD^0} ..., \mD) \simeq \Pre\Sp(\mD)$$ be the reduced antioriented functor
precomposing along $\alpha^X,$ where the last equivalence of antioriented categories is by \cref{Geh}. 
If $\mC=\infty\Cat_*$ and $X=S^0$, we remove $X$ from the notation.

\end{notation}

\begin{remark}\label{restr}
The antioriented functor $\nu^X$ sends a reduced oriented functor $\F: \mC \to \mD $ to $$\{\F(X \wedge S^\n), \F(X \wedge S^\n) \to \Omega(\F(X \wedge S^{\n+1}))\}. $$  In particular, $\nu^X $ 
restricts to an antioriented functor 
$ \nu^X : {\Exc\wedge}(\mC,\mD) \to \Sp(\mD). $
\end{remark}

\begin{remark}\label{lefffa}If $\mC$ is small and $\mD$ admits small colimits and left tensors, 
the reduced antioriented functor $$\nu^X: {\Fun\wedge}(\mC,\mD) \to \Pre\Sp(\mD)$$ 
has an antioriented right adjoint \cite[Proposition 5.20.]{heine2024bienrichedinftycategories} since $\nu^X$ is precomposition. 
If $\mC$ is small and $\mD$ admits small limits and left cotensors, $\nu^X$ admits an antioriented left adjoint \cite[Proposition 5.20.]{heine2024bienrichedinftycategories}.

\end{remark}

\begin{lemma}\label{locpre} Let $\mC$ be a small reduced oriented category that admits suspensions, $\mD$ a reduced right presentable bioriented category, $X \in \mC$.
The following antioriented functor preserves local equivalences: $${\nu_X: \Fun\wedge}(\mC,\mD) \to \Pre\Sp(\mD). $$
	
\end{lemma}

\begin{proof}
	
We prove that $\nu_X$ sends generating local equivalences to local equivalences.
Let $\kappa$ be a strongly inaccessible cardinal such that the underlying $\infty$-category of $\mD$ is $\kappa$-compactly generated. We use the notation of the proof of \cref{excloc}.
We show that $\nu_X$ sends the generating local equivalence $(Y \wedge (-)) \circ \zeta_{Z} $ for $Z \in \mC , Y \in \mD^\kappa$ to a local equivalence.
It sends $(Y \wedge (-)) \circ \zeta_{Z} $ to $(Y \wedge (-)) \circ \nu_X(\zeta_{Z}) $.

The oriented adjunction $$ Y \wedge (-) : \infty\Cat_* \rightleftarrows \mD: \R\Mor_\mD(Y,-)$$ gives rise to an oriented adjunction $$(Y \wedge (-))_!: \Pre\Cat\Sp \rightleftarrows \Pre\Sp(\mD): \R\Mor_\mD(Y,-)_!.$$ The right adjoint of the second adjunction preserves spectra since the right adjoint of the first adjunction preserves endomorphisms. Thus the left adjoint of the
second adjunction preserves local equivalences.
So it suffices to show that $\nu_X$ sends $\zeta_{Z} $ to a local equivalence. The morphism $\nu_X(\zeta_{Z})$ identifies with the morphism
$$ \Sigma(\Omega( \R\Mor_{\mC}(Z,X \wedge S^\bullet))) \to \R\Mor_{\mC}(Z, X \wedge S^\bullet).$$

Let $T:= \R\Mor_{\mC}(Z, X \wedge S^\bullet)\in \Pre\Sp(\mD)$ and $L$ the left adjoint of the embedding $\Sp(\mD) \subset \Pre\Sp(\mD).$
We would like to see that the following canonical morphism in $\Sp(\mC)$ 
is an equivalence:
$$\Sigma(L(\Omega(T))) \simeq L(\Sigma(\Omega(T))) \to L(T). $$
The latter factors as $$\Sigma(L(\Omega(T))) \to \Sigma(\Omega(L(T)))  \to L(T).$$
The first morphism is an equivalence since $L$ is an oriented left adjoint and so preserves suspensions and moreover endomorphisms by \cref{spectri}. The second one is an equivalence since $\Sp(\mC)$ is stable by \cref{rstab}.

\end{proof}

\begin{corollary}\label{locpre2} Let $\mC$ be a small reduced oriented category that admits suspensions, $\mD$ a reduced presentable bioriented category and $X \in \mC$.
The antioriented functor $${\nu_X: \Exc\wedge}(\mC,\mD) \to \Sp(\mD)$$
admits an antioriented left and right adjoint.
	
\end{corollary}

\begin{proposition}\label{appro}
Let $\mC$ be a small reduced oriented category that admits suspensions,	
$\mD$ a reduced right presentable bioriented category and $X \in \infty\Cat_*.$
Let $\F \in {\Fun\wedge}(\mC,\mD) $ and $\F'$ the excisive approximation of
$\F$.
The following canonical map is an equivalence: $$ \colim_{\n \geq 0} \Omega^\n(\F(X \wedge S^\n))\to\colim_{\n \geq 0} \Omega^\n(\F'(X \wedge S^\n))\simeq \F'(X).$$

\end{proposition}

\begin{proof}
	
The reduced antioriented functor $\nu_X: {\Fun\wedge}(\mC,\mD) \to \Pre\Sp(\mD)$
preserves local objects and local equivalences by \cref{locpre}, and so sends the universal local equivalence $\F \to \F'$ to a local equivalence $\F(X \wedge S^\bullet) \to \F'(X \wedge S^\bullet)$ that corresponds to an equivalence $ L(\F(X \wedge S^\bullet))\simeq \F'(X \wedge S^\bullet) $ of spectra in $\mD,$ where $L$ is the spectrification.
Hence by \cref{spectri} evaluating at the zeroth term of a prespectrum in $\mD$ we obtain the following equivalence
in $\mD:$ 
$$\F'(X) \simeq \Omega^\infty(L(\F(X \wedge S^\bullet))) \simeq \colim_{n\geq 0} \Omega^\n(\F(X \wedge S^\n)).$$	
	
\end{proof}

\begin{corollary}\label{excend} Let $\mB$ be a reduced oriented category that admits suspensions,
$\mC,\mD$ reduced presentable bioriented categories and $\F: \mC \to \mD$ a reduced bioriented functor.
Assume that in $\mD$ sequential colimits commute with finite limits and left cotensors with finite $\infty$-categories with distinguished object.
The reduced antioriented functor $$\F_!: {\Exc\wedge}(\mB,\mC) \to {\Exc\wedge}(\mB,\mD)$$
preserves finite limits and left cotensors with finite $\infty$-categories with distinguished object.
	
\end{corollary}

\begin{proof}

By definition of excisiveness the full subcategory ${\Exc\wedge}(\mB,\mD) \subset {\Fun\wedge}(\mB,\mD)$
is closed under left cotensors and small limits. 
Hence for any $X \in \mB$ the antioriented functor $$\nu^X: {\Exc\wedge}(\mB,\mD) \to \Sp(\mD) , \F \mapsto \F(X \wedge S^\bullet)$$ preserves small limits and left cotensors. 
So it is enough to see that for every $ X \in \mB$
the reduced antioriented functor $$ {\Exc\wedge}(\mB, \mC) \xrightarrow{\F_!} {\Exc\wedge}(\mB, \mD) \xrightarrow{\nu^X} \Sp(\mD)$$ preserves finite limits and left cotensors with finite $\infty$-categories with distinguished object.

By \cref{appro} the latter factors as reduced antioriented functors
$$ 	{\Exc\wedge}(\mB,\mC) \xrightarrow{\rho} 
\Sp(\mD)^{\bN} \xrightarrow{\colim} \Sp(\mD),$$
where $\rho$ induces in degree $\n \geq 0$ the reduced antioriented functor $ \Omega^\n \circ \nu^{\Sigma^\n(X)}$ and $\colim$ assigns the sequential colimit.
So the result follows from the fact that in $\Sp(\mD)$ sequential colimits commute with finite limits and left cotensors with finite $\infty$-categories if this holds in $\mD$.

\end{proof}












We will see next that the theory of excisive approximations implies that every stable oriented category is spectral.

\begin{proposition}

\begin{enumerate}
\item Let $\mC$ be a quasi-stable oriented category. The reduced presentable antioriented category $$ {\Exc\wedge}(\mC, \infty\Cat_*)$$ is antistable.
	
\item Let $\mC$ be a quasi-antistable antioriented category. The reduced presentable oriented category $$ {\wedge\overline{\Exc}}(\mC, \infty\Cat_*) $$ is stable.

\end{enumerate}	

\end{proposition}

\begin{proof}The proof of (2) is similar to that of (1).
We prove (1). 
By \cref{forgetu} for every $X \in \mC$ there is a reduced antioriented functor $\nu_X: {\Exc\wedge}(\mC, \infty\Cat_*) \to \Cat\Sp$ sending $\F$ to $\F(X \wedge S^\bullet)$ that preserves small limits, left cotensors, small colimits and left tensors by \cref{locpre2}. In particular, $\nu_X$ preserves antisuspensions and antiendomorphisms.
Since the family $(\nu_X)_{X \in \mC}$ is jointly conservative, ${\Exc\wedge}(\mC, \infty\Cat_*)$ is antistable since $\Cat\Sp$ is antistable by \cref{stab}.

\end{proof}

\begin{remark}
Let $\mC$ be a small quasi-stable oriented category. By quasi-stability the full subcategory $$ {\Exc\wedge}(\mC, \infty\Cat_*) \subset {\Fun\wedge}(\mC, \infty\Cat_*) $$ precisely consists of the reduced oriented functors preserving endomorphisms.

In particular, for every quasi-antistable antioriented category $\mC$ the antioriented Yoneda embedding induces an antioriented embedding $$\mC \hookrightarrow {\Exc \wedge}(\mC^\circ, \infty\Cat_*)$$
into a presentable antistable antioriented category (\ref{excprese}) that preserves weighted limits by \cref{weiadj}.

Dually, let $\mC$ be a small quasi-antistable antioriented category. By quasi-antistability the full subcategory 
$$ {\wedge\overline{\Exc}}(\mC, \infty\Cat_*) \subset {\wedge\Fun}(\mC, \infty\Cat_*) $$ precisely consists of the reduced antioriented functors preserving antiendomorphisms.

In particular, for every small quasi-stable oriented category $\mC$ the oriented Yoneda embedding induces an oriented embedding $$\mC \hookrightarrow {\wedge\overline{\Exc}}(\mC^\circ, \infty\Cat_*)$$ into a presentable stable oriented category that preserves weighted limits.

\end{remark}

\begin{corollary}
\begin{enumerate}
\item Let $\mC$ be a quasi-stable oriented category. There is a canonical functorial oriented embedding $\mC \hookrightarrow \mD$ preserving weighted limits into a stable presentable oriented category.
	
\item Let $\mC$ be a quasi-antistable antioriented category.  There is a canonical functorial antioriented embedding $\mC \hookrightarrow \mD$ preserving weighted limits into an antistable presentable antioriented category.
	
\end{enumerate}

\end{corollary}

\begin{corollary}\label{spectral0}
\begin{enumerate}
	
\item Every quasi-stable oriented category $\mC$ refines to a spectral oriented category. 
	
\item Every quasi-antistable antioriented category $\mC$ refines to an antispectral antioriented category. 
	
\end{enumerate}		

\end{corollary}

We obtain the following important corollaries:

\begin{corollary}\label{Stab2}
Let $\mC$ be a quasi-stable oriented category. 
The following are equivalent:
	
\begin{enumerate}
\item $\mC$ is stable.	
		
\item $\mC$ admits oriented fibers.
		
\item $\mC$ admits oriented cofibers.
		
\end{enumerate}
	
\end{corollary}

\begin{proof}

This follows immediately from \cref{spectral0} combined with
\cref{laxfib} and its dual.

\end{proof}

\begin{corollary}\label{sstab}
Let $\mC$ be reduced oriented category that has oriented fibers.
Then $\Sp(\mC)$ is stable.
	
\end{corollary}
\begin{proof}
If $\mC$ admits oriented fibers, $\Sp(\mC)$ admits oriented fibers.
We apply \cref{rstab} and \cref{Stab2}.	
	
\end{proof}

\cref{preserv} implies the following corollary:

\begin{corollary}\label{preserv1}

\begin{enumerate}

\item Let $\mC, \mD$ be stable oriented categories and $\phi: \mC \to \mD$ a spectral oriented functor preserving endomorphisms. Then $\phi$ preserves oriented pushouts and oriented pullbacks.

\item Let $\mC, \mD$ be antistable antioriented categories and $\phi: \mC \to \mD$ an antioriented functor preserving antiendomorphisms. Then $\phi$ preserves antioriented pushouts and antioriented pullbacks.

\end{enumerate}
\end{corollary}








\subsection{The representability theorem}

\begin{notation}
Let $\mC, \mD$ be reduced oriented categories. 

\begin{itemize}
\item Let $$ {\mathrm{Fil}\Fun\wedge}(\mC,\mD) \subset {\Fun\wedge}(\mC,\mD)$$ 	be the full subcategory of reduced oriented functors preserving small filtered colimits.

\item Let $$ {\mathrm{Fil}\Exc\wedge}(\mC,\mD) \subset {\Fun\wedge}(\mC,\mD)$$ 	be the full subcategory of reduced excisive oriented functors preserving small filtered colimits.

\item Let $$\mH(\mD) \subset {\Fil\Exc\wedge}(\infty\Cat_*,\mD)$$
be the full subcategory of categorical homology theories.

\end{itemize}
\end{notation}

\begin{lemma}\label{rioko}

Let $\mD$ be a reduced bioriented category whose underlying reduced antioriented category is presentable.
The full subcategory $$\mH(\mD) \subset {\Fil\Exc\wedge}(\infty\Cat_*,\mD)$$
is closed under small colimits and left tensors. 

\end{lemma}

\begin{proof}

Let $E: \infty\Cat_* \to \mD$ be a reduced and excisive oriented functor and $X \in \infty\Cat_*.$
\cref{spectri} implies that the right tensor $E(S^\bullet) \wedge X$ in $\Sp(\mD)$ is the spectrum associated to the prespectrum $((-)\wedge X)_! (E(S^\bullet)).$
The canonical map of prespectra \begin{equation}\label{ijjmn}
((-)\wedge X)_! (E(S^\bullet)) \to E(S^\bullet \wedge X) \end{equation} 
in $\mD$ induces a map of spectra 
\begin{equation}\label{ijjm}
E(S^\bullet) \wedge X \to E(S^\bullet \wedge X).\end{equation} 
The image under $\Omega^\infty$ of \ref{ijjm} is the functor \ref{eqzz}. Since $E$ is excisive, for any $\ell \geq 0 $ the map $$ E(S^\bullet) \wedge \Sigma^\ell(X) \to E(S^\bullet \wedge \Sigma^\ell(X)) \simeq \Sigma^\ell(\Omega^\ell(E(S^\bullet \wedge \Sigma^\ell(X))))$$
of spectra identifies with the following map of spectra
$$ \Sigma^\ell(E(S^\bullet) \wedge X) \to \Sigma^\ell(E(S^\bullet \wedge X)).$$
So the functor \ref{eqzz} is an equivalence for every $X \in \infty\Cat_*$ if and only if the map of spectra \ref{ijjm} is an equivalence for every $X \in \infty\Cat_*$.
So a reduced excisive oriented functor $E:\infty\Cat_* \to \mD$ is a 
homology theory if and only if for every $X \in \infty\Cat_*$ the map of spectra \ref{ijjm} is an equivalence.

For every $X \in \infty\Cat_* $ there is a canonical natural transformation $\rho: ((-) \wedge X) \circ \nu \to \nu^X $
of reduced antioriented functors $ {\Fun\wedge}(\infty\Cat_*,\mD) \to \Pre\Sp(\mD)$ whose component at any
$E \in {\Fun\wedge}(\infty\Cat_*,\mD)$ is the morphism \ref{ijjmn}.

Let $\iota: {\Fil\Exc\wedge}(\infty\Cat_*,\mD) \subset {\Fun\wedge}(\infty\Cat_*,\mD)$ be the canonical embedding.
The natural transformation $\rho$ of reduced antioriented functors $ {\Fun\wedge}(\infty\Cat_*,\mD) \to \Pre\Sp(\mD)$ gives rise to a natural transformation $$L \circ \rho \circ \iota:  ((-) \wedge X) \circ \nu  \to \nu^X $$
of functors $ {\Fil\Exc\wedge}(\infty\Cat_*,\mD) \to \Sp(\mD)$
whose component at any $E \in {\Fil\Exc\wedge}(\infty\Cat_*,\mD) $ is the morphism \ref{ijjm}. 
In the first part of the proof we showed that $\mH(\mD) \subset {\Fil\Exc\wedge}(\infty\Cat_*,\mD)$ is the full subcategory spanned by those $E$ such that $\rho_E$ is an equivalence.
Source and target of $\rho$ preserve small colimits and left tensors.
By \cref{locpre} source and target of $L \circ \rho \circ \iota$ preserve local equivalences and therefore also preserve small colimits and left tensors. 
Consequently, the full subcategory $\mH(\mD) \subset {\Fil\Exc\wedge}(\infty\Cat_*,\mD)$
is closed under small colimits and left tensors.

\end{proof}

\begin{theorem}\label{brown}
Let $\mD$ be a presentable differentiable bioriented category.
The antioriented functor $$\nu: {\Fil\Exc\wedge}(\infty\Cat_*,\mD) \to \Sp(\mD)$$
refines to a bioriented functor that admits a fully faithful bioriented left adjoint sending $\E \in \Sp(\mD)$ to $ \Omega^\infty(\E \wedge (-)).$
The left adjoint induces an equivalence of bioriented categories $$\Sp(\mD) \simeq \mH(\mD). $$


\end{theorem}

\begin{proof} 

By \cite[Proposition 5.7.]{heine2024higheralgebraweightedcolimits} evaluation at $S^0$ induces an equivalence of antioriented categories
$${\L\Fun\wedge}(\infty\Cat_*,\Sp(\mD)) \to \Sp(\mD),$$
whose inverse sends $\E$ to $\E \wedge (-).$
The bioriented functor $\Omega^\infty:\Sp(\mD) \to \mD$ induces an antioriented functor
$$ {\Fil\Fun\wedge}(\infty\Cat_*,\Sp(\mD)) \to {\Fil\Fun\wedge}(\infty\Cat_*,\mD).$$
We obtain an antioriented functor $$\Psi: \Sp(\mD) \simeq {\L\Fun\wedge}(\infty\Cat_*,\Sp(\mD)) \subset {\Fun\wedge}(\infty\Cat_*,\Sp(\mD))\to {\Fun\wedge}(\infty\Cat_*,\mD)$$
that sends $\E$ to $\Omega^\infty(\E \wedge(-))$ and so by \cref{homol} lands in $\mH(\mD) \subset {\Fil\Exc\wedge}(\infty\Cat_*,\mD).$

We prove next that $\Psi: \Sp(\mD) \to {\Fil\Exc\wedge}(\infty\Cat_*,\mD)$
preserves small colimits and left tensors. By \cref{rioko} the embedding $\mH(\mD) \subset {\Fil\Exc\wedge}(\infty\Cat_*,\mD)$ preserves small colimits and left tensors.
So it is enough to see that $\Psi: \Sp(\mD) \to \mH(\mD)$ preserves small colimits and left tensors.
The antioriented functor $\nu: {\Fil\Fun\wedge}(\infty\Cat_*,\mD) \to \Pre\Sp(\mD)$ preserves small colimits and left tensors, preserves local objects and by \cref{locpre} also preserves local equivalences.
Thus the restriction $$\nu: {\Fil\Exc\wedge}(\infty\Cat_*,\mD) \to \Sp(\mD) $$
preserves small colimits and left tensors.
The full subcategory $\mH(\mD) \subset {\Fil\Exc\wedge}(\infty\Cat_*,\mD)$ is closed under small colimits and left tensors by \cref{rioko}. So the restriction $\nu: \mH(\mD) \to \Sp(\mD)$ is an antioriented functor that preserves small colimits and left tensors between antioriented categories that admit small colimits and left tensors.
By the definition of homology theory $\nu: \mH(\mD) \to \Sp(\mD)$ is conservative and so also detects small colimits and left tensors. Thus $\Psi: \Sp(\mD) \to \mH(\mD)$ preserves small colimits
and left tensors if and only if the composition $\nu \circ \Psi$ does.
There is a canonical equivalence of antioriented functors $\Sp(\mD) \to \Sp(\mD)$: $$ \nu \circ \Psi \simeq \Omega^\infty((-)\wedge S^\bullet) \simeq (-) \simeq S^0 \simeq \id. $$
Hence $\Psi: \Sp(\mD) \to \mH(\mD)$ and so also $\Psi: \Sp(\mD) \to \mH(\mD) \subset  {\Fil\Exc\wedge}(\infty\Cat_*,\mD) $ preserve small colimits and left tensors.

We prove next that the equivalence $\lambda: \id \simeq \nu \circ \Psi$ exhibits $\Psi$ as left adjoint to $\nu.$
Since $\lambda$ is an equivalence, this will imply that $\Phi$ is fully faithful.
So we verify that for every $\E \in \Sp(\mD)$ and $H \in  {\Fil\Exc\wedge}(\infty\Cat_*,\mD)$ the induced functor
$$\theta: \L\Mor_{{\Fil\Exc\wedge}(\infty\Cat_*,\mD)}(\Psi(\E), H) \to \L\Mor_{\Sp(\mD)} (\E,\nu(H)) $$ is an equivalence.
By \cref{spgen} the antioriented category $\Sp(\mD) $ is generated under small colimits by the images of $\Omega^\n(S)$ for even $\n \geq 0$ under the antioriented functors $(Y \wedge (-))_!: \Cat\Sp \to \Sp(\mD)$
induced by the left adjoint antioriented functors $Y \wedge (-): \infty\Cat_* \to \mD$ for every $Y \in \mD$.
Since $\Psi: \Sp(\mD) \to {\Fil\Exc\wedge}(\infty\Cat_*,\mD) $ preserves small colimits, 
we can assume that $\E= (Y \wedge (-))_!(\Omega^\n(S))$ for $\n \geq 0$ even and $Y \in \mD.$
In this case the functor  $\theta$ factors as
$$ \L\Mor_{{\Fil\Exc\wedge}(\infty\Cat_*,\mD)}(\Psi(\E), H) \to $$$$  \L\Mor_{{\Fil\wedge}(\infty\Cat_*,\mD)}((Y \wedge (-)) \circ \Omega^\n, H) \simeq $$
$$\L\Mor_{{\Fil\wedge}(\infty\Cat_*,\infty\Cat_*)}(\Omega^\n, \R\Mor_\mD(Y,-) \circ H)\simeq $$$$
\L\Mor_{{\Fil\wedge}(\infty\Cat_*,\infty\Cat_*)}(\L\Mor_{\infty\Cat_*}(S^\n,-), \R\Mor_\mD(Y,-) \circ H)
\simeq$$$$ \R\Mor_\mD(Y,H(S^\n)) \simeq$$$$ \nu(\R\Mor_\mD(Y,-) \circ H)_\n \simeq $$$$
\Omega^\infty(\Sigma^\n(\nu(\R\Mor_\mD(Y,-) \circ H))) \simeq $$
$$\L\Mor_{\Cat\Sp}(S, \Sigma^\n(\nu(\R\Mor_\mD(Y,-)\circ H))) 
\simeq $$$$ \L\Mor_{\Cat\Sp}(\Omega^\n(S), \nu(\R\Mor_\mD(Y,-) \circ H)) \simeq $$$$ \simeq \L\Mor_{\Cat\Sp} (\Omega^\n(S),\R\Mor_\mD(Y,-) \circ \nu(H)) \simeq $$$$ \L\Mor_{\Sp(\mD)} (\E,\nu(H)).$$ 

Consequently, we need to see that the canonical morphism $$ (Y \wedge (-)) \circ \Omega^\n \to \Psi(\E)= \Omega^\infty(  (Y \wedge (-))_!(\Omega^\n(S)) \wedge(-) )  \simeq $$$$ \Omega^\infty( \Omega^\n( (Y \wedge (-))_!(S)) \wedge(-) )  \simeq $$$$  \Omega^\n(\Omega^\infty(\Sigma^\infty(Y) \wedge(-) ))  \simeq \Omega^\n \circ  \Omega^\infty \circ \Sigma^\infty \circ (Y \wedge(-) )  $$ exhibits the target as excisive approximation.
By \cref{appro} the canonical morphism $ \Omega^\n \to \Omega^\n \circ \Omega^\infty \circ \Sigma^\infty $ exhibits the target as excisive approximation.
The reduced antioriented functor $$(Y \wedge (-))_!: {\Fil\wedge}(\infty\Cat_*,\infty\Cat_*) \to {\Fil\wedge}(\infty\Cat_*,\mD)$$ preserves local equivalences since its right adjoint
$$\R\Mor_\mD(Y,-)_!: {\Fil\wedge}(\infty\Cat_*,\mD) \to {\Fil\wedge}(\infty\Cat_*,\infty\Cat_*)$$ preserves local objects. Hence the canonical morphism \begin{equation}\label{lal}
 (Y \wedge (-)) \circ \Omega^\n \to (Y \wedge (-)) \circ  \Omega^\n \circ \Omega^\infty \circ \Sigma^\infty \end{equation} is a local equivalence. 
So we need to see that the canonical morphism
$$ (Y \wedge (-)) \circ  \Omega^\n \circ \Omega^\infty \circ \Sigma^\infty  \to \Omega^\n \circ  \Omega^\infty \circ \Sigma^\infty \circ (Y \wedge(-) )$$ 
exhibits the target as excisive approximation, or equivalently that
the induced morphism $$ (Y \wedge (-)) _! ( \Omega^\n \circ \Omega^\infty \circ \Sigma^\infty)  \to \Omega^\n \circ  \Omega^\infty \circ \Sigma^\infty \circ (Y \wedge(-) )$$  
is an equivalence.
The latter factors as $$ (Y \wedge (-)) _! ( \Omega^\n \circ \Omega^\infty \circ \Sigma^\infty)  \simeq \Omega^\n \circ (Y \wedge (-)) _! (\Omega^\infty \circ \Sigma^\infty)  \to \Omega^\n \circ  \Omega^\infty \circ \Sigma^\infty \circ (Y \wedge(-)), $$ 
where the first equivalence is by \cref{excend}.
So it suffices to show that the canonical morphism $$ (Y \wedge (-)) _! (\Omega^\infty \circ \Sigma^\infty)  \to \Omega^\infty \circ \Sigma^\infty \circ (Y \wedge(-) )$$ is an equivalence.
The latter is equivalent to saying that the canonical morphism $$ (Y \wedge (-)) \circ (\Omega^\infty \circ \Sigma^\infty)  \to \Omega^\infty \circ \Sigma^\infty \circ (Y \wedge(-) )$$ is a local equivalence.
Since the canonical morphism $(Y \wedge (-)) \to (Y \wedge (-)) \circ (\Omega^\infty \circ \Sigma^\infty) $ of (\ref{lal}) for $\n=0$ is a local equivalence, it is enough to see that the following composition is a local equivalence:
$$(Y \wedge (-)) \to (Y \wedge (-)) \circ (\Omega^\infty \circ \Sigma^\infty)  \to \Omega^\infty \circ \Sigma^\infty \circ (Y \wedge(-) ). $$
This follows from \cref{appro}. Thus we obtain an antioriented adjunction \begin{equation}\label{lokyz} 
\Psi : \Sp(\mD) \rightleftarrows {\Fil\Exc\wedge}(\infty\Cat_*,\mD) :\nu,\end{equation} where $\Psi$ is an embedding that lands in $\mH(\mD).$
Hence (\ref{lokyz}) restricts to an antioriented adjunction $$ \Psi : \Sp(\mD) \rightleftarrows \mH(\mD) :\nu,$$ where the left adjoint is an embedding and the right adjoint is conservative.
Consequently, $$\Psi: \Sp(\mD) \to \mH(\mD)$$ is an equivalence inverse to
$\nu: \mH(\mD) \to \Sp(\mD).$ In particular, the underlying antioriented functor of the embedding $$\mH(\mD) \subset \Fil{\Exc\wedge}(\infty\Cat_*,\mD)$$ of bioriented categories admits a right adjoint.

We complete the proof by enhancing the antioriented adjunction (\ref{lokyz}) to a bioriented adjunction.
The full subcategory $\mH(\mD)$ is closed in ${\Fun\wedge}(\infty\Cat_*,\mD)$ under right tensors: for every $Y \in \infty\Cat_*$ and spectrum $\E$
there is a canonical equivalence $$ \Omega^\infty(\E\wedge(-)) \wedge Y = \Omega^\infty(\E\wedge(-)) \circ (Y \wedge(-)) \simeq \Omega^\infty((\E\wedge Y) \wedge(-)).$$

By \cref{cloten} the full subcategories $$ \Fil{\Exc\wedge}(\infty\Cat_*,\mD), \ {\Exc\wedge}(\infty\Cat_*,\mD)$$
are closed in ${\Fun\wedge}(\infty\Cat_*,\mD) $ under right tensors. 
Consequently, the full subcategory $\mH(\mD)$ is closed under right tensors in $\Fil{\Exc\wedge}(\infty\Cat_*,\mD)$ and ${\Exc\wedge}(\infty\Cat_*,\mD)$. Since the underlying antioriented functor of the embedding $$\mH(\mD) \subset \Fil{\Exc\wedge}(\infty\Cat_*,\mD)$$ of bioriented categories admits a right adjoint, the first statement implies by \cref{adj}
that the embedding $\mH(\mD) \subset \Fil{\Exc\wedge}(\infty\Cat_*,\mD)$ of bioriented categories admits a right adjoint $$ R: \Fil{\Exc\wedge}(\infty\Cat_*,\mD) \to \mH(\mD).$$

By \cite[Proposition 5.108.]{heine2024bienrichedinftycategories} there is a left linear reduced bioriented functor ${\Fun\wedge}(\infty\Cat_*,\mD) \to \mD $ evaluating at $S^0.$
The restriction $$\beta: {\Exc\wedge}(\infty\Cat_*,\mD) \subset {\Fun\wedge}(\infty\Cat_*,\mD) \to \mD $$
is excisive because for every $\F \in {\Exc\wedge}(\infty\Cat_*,\mD)$ the canonical morphism $$ \beta(\F)=\F(S^0) \to  \Omega^2(\beta(\Sigma^2(\F))) \simeq \Omega^2(\beta(\F \circ \Sigma^2)) =  \Omega^2(\F(\Sigma^2(S^0)))$$ is an equivalence.
By \cref{lifti} the excisive bioriented functor $\beta: {\Exc\wedge}(\infty\Cat_*,\mD) \to \mD$ uniquely lifts to an excisive bioriented functor $\bar{\beta}: {\Exc\wedge}(\infty\Cat_*,\mD)\to \Sp(\mD).$

Since $\mH(\mD)$ is closed in ${\Exc\wedge}(\infty\Cat_*,\mD)$ under right tensors,
the restriction $$\gamma: \mH(\mD) \subset {\Fun\wedge}(\infty\Cat_*,\mD) \to \mD$$ is excisive
and so uniquely lifts to an excisive bioriented functor $\bar{\gamma}: \mH(\mD)\to \Sp(\mD),$
which by uniqueness is the restriction of $\bar{\beta}.$
We prove next that $\bar{\gamma}$ induces $\nu$ on underlying antioriented categories.

For that we prove that $\mH(\mD)$ is stable. We show first that $\mH(\mD)$ is a presentable bioriented category.
The equivalence of antioriented categories $\nu: \mH(\mD) \to \Sp(\mD)$ guarantees that the underlying antioriented category of $\mH(\mD)$ is presentable.
Since $\mH(\mD)$ is closed in ${\Fun\wedge}(\infty\Cat_*,\mD)$ under right tensors,
the bioriented category $\mH(\mD)$ has left and right tensors.
We prove next that the resulting biaction preserves small colimits componentwise.
This holds for the left action since the underlying antioriented category of $\mH(\mD)$ is presentable.
The right tensor $$\mH(\mD)\times \infty\Cat_* \to \mH(\mD)$$ identifies with the functor
\begin{equation}\label{kolp}
\Sp(\mD) \times \infty\Cat_* \simeq \mH(\mD)\times \infty\Cat_* \to \mH(\mD) \simeq \Sp(\mD), \end{equation}$$ (\E, Y) \mapsto  \nu(\Omega^\infty(\E \wedge (-)) \circ (Y \wedge (-))) \simeq \nu(\Omega^\infty(\E \wedge Y \wedge (-))) \simeq  \E \wedge Y  $$
and so preserves small colimits componentwise. This proves that $\mH(\mD)$ is a presentable bioriented category. To see that $\mH(\mD)$ is stable, it remains to see that the right tensor 
$\Sigma \simeq (-) \wedge S^1 :  \mH(\mD) \to \mH(\mD)$ is an equivalence.
By equivalence (\ref{kolp}) the latter identifies with the functor $$(-) \wedge S^1:  \mH(\mD) \simeq \Sp(\mD) \xrightarrow{ } \Sp(\mD)\simeq \mH(\mD)$$ and so is an equivalence since $\Sp(\mD)$ is stable by \cref{rstab}.

Since $\mH(\mD) $ is stable, the excisive bioriented functor $\gamma: \mH(\mD) \to \mD$ preserves endomorphisms and so 
induces a bioriented functor $$\mH(\mD) \simeq \Sp(\mH(\mD)) \to \Sp(\mD)$$ that preserves endomorphisms and so is excisive and lifts $\gamma$. By the uniqueness statement in \cref{lifti}, it identifies with the bioriented functor $\bar{\gamma}.$
The underlying antioriented functor of $\gamma$ evaluating at $S^0$ factors as antioriented functors
$$\mH(\mD) \xrightarrow{\nu}\Sp(\mD) \xrightarrow{\Omega^\infty} \mD.$$
Thus the underlying antioriented functor of $\bar{\gamma}$,
which  factors as $$\mH(\mD) \simeq \Sp(\mH(\mD)) \to \Sp(\mD),$$ 
factors as $$ \mH(\mD) \simeq \Sp(\mH(\mD)) \to \Sp(\Sp(\mD)) \simeq \Sp(\mD)$$ and so identifies with $\nu: \mH(\mD) \to \Sp(\mD).$
Hence $\bar{\gamma}$ induces the equivalence $\nu$ on underlying antioriented categories and 
so by \cref{adj}  is an equivalence of bioriented categories if and only if it preserves right tensors.
For that we need to verify that for every $Y \in \infty\Cat_*$ and $\F \in \mH(\mD) $ the canonical morphism $$\theta: \nu(\F) \wedge Y \to \nu(\F \wedge Y)$$ is an equivalence.
Since $\F \in \mH(\mD)$, there is an $\E \in \Sp(\mD)$ such that 
$\F \simeq \Omega^\infty(\E \wedge (-))$ and $\theta$ identifies with the canonical equivalence
$$\nu(\F) \wedge Y \simeq \E \wedge Y \simeq \nu(\Omega^\infty(\E \wedge Y \wedge (-))) \simeq \nu(\Omega^\infty(\E \wedge (-)) \circ (Y\wedge(-))) \simeq  \nu(\F \wedge Y).$$ 

\vspace{1mm}

We prove next that the bioriented functor $\bar{\beta}: \Fil{\Exc\wedge}(\infty\Cat_*,\mD) \to \Sp(\mD)$ has a left adjoint.
Let $$\zeta:=\bar{\gamma} \circ R: \Fil{\Exc\wedge}(\infty\Cat_*,\mD) \to \mH(\mD)\to \Sp(\mD)$$ be the composition of right adjoint bioriented functors.
By definition the underlying antioriented functor of $\zeta$ is right adjoint to $\Phi$
and so identifies with $\nu.$
By definition the bioriented functor $\zeta$ factors as $$ \Fil{\Exc\wedge}(\infty\Cat_*,\mD) \xrightarrow{R} \mH(\mD)\subset  \Fil{\Exc\wedge}(\infty\Cat_*,\mD) \xrightarrow{\bar{\beta}} \Sp(\mD).$$
So the counit $\epsilon: R \to \id$ of the bioriented adjunction $$\mH(\mD) \rightleftarrows \Fil{\Exc\wedge}(\infty\Cat_*,\mD):R$$ yields a map of bioriented functors
$\bar{\beta} \epsilon:  \zeta \to \bar{\beta}$. 
By the triangle identities the counit $\epsilon: R \to \id$ is inverted by $R$ and so inverted by $\nu$,
and so inverted by $\bar{\beta}$ that induces $\nu$ on underlying antioriented categories.
Hence the map $ \bar{\beta} \epsilon$ is an equivalence. So the bioriented functor $$\bar{\beta}: \Fil{\Exc\wedge}(\infty\Cat_*,\mD) \to \Sp(\mD)$$ admits a left adjoint.

\end{proof}

\begin{corollary}\label{lifta} Let $\mD$ be a presentable differentiable bioriented category.
Let $H: \infty\Cat_* \to \mD$ be a categorical homology theory.
There is a unique categorical homology theory  $\bar{H}: \infty\Cat_* \to \Sp(\mD)$ lifting $H.$

More precisely, the functor $\Omega^\infty: \Sp(\mD) \to \mD$ induces an equivalence $\mH(\Sp(\mD)) \to \mH(\mD).$
	
\end{corollary}

\begin{proof}
By \cref{brown} the functor $\mH(\Sp(\mD)) \to \mH(\mD)$
identifies with the functor $\Sp(\Sp(\mD)) \to \Sp(\mD),$
which is an equivalence since $\Sp(\mD)$ is stable by \cref{stab}.
	
\end{proof}

We also obtain the following relationship between reduced and unreduced categorical homology:

\begin{corollary} Let $\mD$ be a presentable differentiable bioriented category. 
Let $\nu:  \infty\Cat_* \to \infty\Cat $ be the forgetful oriented functor.
Sending a categorical homology theory $ \widetilde{H}: \infty\Cat_* \to \Sp(\mD) $ to $ \widetilde{H} \circ (-)_+: \infty\Cat \to \Sp(\mD) $
gives an equivalence between the following:

\begin{itemize}
\item Categorical homology theories $\infty\Cat_* \to \Sp(\mD)$.

\item Oriented functors $H: \infty\Cat \to \Sp(\mD)$
that preserve finite coproducts and such that the cofiber of the map $H(0) \to H \circ \nu $
in the category of oriented functors $ \infty\Cat_* \to \Sp(\mD)$, which is a reduced oriented functor, is a categorical homology theory. 
\end{itemize}

The inverse sends $H: \infty\Cat \to \Sp(\mD) $ to the cofiber of the map $H(0) \to H \circ \nu $.




\end{corollary}

\begin{proof}

Let $ \widetilde{H}: \infty\Cat_* \to \Sp(\mD) $ be a categorical homology theory. By \cref{brown} the oriented functor $\widetilde{H}$ is an oriented left adjoint. Let $H:= \widetilde{H} \circ (-)_+ : \infty\Cat \to \Sp(\mD) $ which is an oriented left adjoint, too. 
The cofiber sequence
$ 0_+ \to (-)_+ \circ \nu \to \id $ in the category of oriented functors $\infty\Cat_* \to \infty\Cat_*$ is sent by $\widetilde{H}$ to a cofiber sequence $H(0) \to H \circ \nu \to \widetilde{H} $
in the category of oriented functors $ \infty\Cat_* \to \Sp(\mD).$ 

Conversely, let $H : \infty\Cat \to \Sp(\mD) $ be an oriented functor that preserves finite coproducts. Let $ \widetilde{H}: \infty\Cat_* \to \Sp(\mD) $ be the cofiber of the map $H(0) \to H \circ \nu $
in the category of oriented functors $ \infty\Cat_* \to \Sp(\mD).$ 
The cofiber sequence
$$ H(0) \to H \circ \nu \circ (-)_+ \simeq H \coprod H(0) \to H $$ in the category of oriented functors $ \infty\Cat_* \to \Sp(\mD)$ 
identifies $H $ with $\widetilde{H} \circ (-)_+.$

\end{proof}






\section{Oriented exact sequences}

In the following we show structural consequences of the categorical Brown representability theorem. We introduce oriented exact sequences, a higher-categorical analogue of long exact sequences.

We prove that oriented cofiber sequences
give rise to oriented exact sequences in any spectral oriented category that admits oriented cofibers and oriented fibers (\cref{oplaxexacto}).
We prove that oriented fibers detect equivalences in any spectral oriented category (\cref{Fibb}). This makes oriented exact sequences a useful tool to inductively detect equivalences (\cref{homolog}).

We use the categorical Brown representability theorem to prove that oriented cofiber sequences give rise to oriented exact sequences on categorical homology theories (\cref{homologic}).


\begin{definition}\label{oplaxexact} Let $\mC$ be a reduced oriented category.
A sequence
$$... \xrightarrow{\gamma_{\n-1}} A_\n \xrightarrow{\alpha_\n} B_\n \xrightarrow{\beta_\n} C_\n \xrightarrow{\gamma_\n} A_{\n+1} \xrightarrow{\alpha_{\n+1}} B_{\n+1}  \xrightarrow{\beta_{\n+1}}  ... $$
in $\mC$ is oriented exact if 
\begin{itemize}
\item for every even $n \in \bZ$ the morphism $\beta_\n $ is the oriented left cofiber of $\alpha_\n$ and fiber of $\gamma_\n$,
\item for every even $n \in \bZ$ the morphism $ \gamma_\n$ is the cofiber of $ \beta_\n$ and oriented left fiber of $ \alpha_{\n+1}$,
\item for every odd $n \in \bZ$ the morphism $\beta_\n $ is the oriented right cofiber of $\alpha_\n$ and fiber of $\gamma_\n$,
\item for every odd $n \in \bZ$ the morphism $ \gamma_\n$ is the cofiber of $ \beta_\n$ and oriented right fiber of $ \alpha_{\n+1}.$
\end{itemize}

Dually, we define antioriented exact sequences in reduced antioriented categories.

\end{definition}

\begin{corollary}\label{oplaxexacto} Let $\mC$ be a reduced spectral oriented category that admits oriented cofibers and oriented fibers.
	
\begin{enumerate}
\item There is an oriented exact sequence in $\mC:$
$$... \xrightarrow{\gamma_{\n-1}} A_\n \xrightarrow{\alpha_\n} B_\n \xrightarrow{\beta_\n} C_\n \xrightarrow{\gamma_\n} A_{\n+1} \xrightarrow{\alpha_{\n+1}} B_{\n+1}  \xrightarrow{\beta_{\n+1}}  ...$$

\item Let $\mD$ be a reduced spectral oriented category that admits oriented cofibers and oriented fibers.
The oriented exact sequence of (1) is preserved by every spectral oriented functor $\phi :\mC \to \mD$.
\end{enumerate}	
\end{corollary}	

\begin{proof}
	
Let $\phi: A \to B$ be a morphism in $\mC.$ By \cref{laxfib} there is the following diagram in $\mC:$
\[
\begin{tikzcd}
A \ar{d} \ar{r} & B \ar{r}{} \ar{d}[swap]{} & 0  \ar{d}{} \\
0 \ar[double]{ur}{} \ar{r} & 0 \overset{\to}{+}_A B \simeq {0 \overset{\to}{\times}_{\Sigma(B)} \Sigma(A)} \ar[d]  \ar{r}[swap]{} & {0 \overset{\to}{+}_A 0} \simeq \Sigma(A) \ar[d] \\
& 0 \ar[double]{ur}{} \ar{r}[swap]{} & \Sigma(B),
\end{tikzcd}
\]
The upper-left square is an oriented pushout square and lower-right square is an oriented pullback square. By \cref{specstab} (2) the underlying reduced oriented category of $\mC$ is stable.
So by the pasting law (\cref{pasting}) and its dual version the upper-right square is a pushout and pullback square.  By \cref{preserv} and the pasting law the oriented pushout, pushout, pullback and oriented pullback in the latter diagram are preserved by any spectral oriented functor.
The reduced oriented functor $\Sigma: \mC^\cop \to \mC$ of \cref{reduc} sends the latter diagram to a diagram 
\[
\begin{tikzcd}
\Sigma(A) \ar{d} \ar{r} & \Sigma(B)  \ar[double]{ld}{}  \ar{r}{} \ar{d}[swap]{} & 0  \ar{d}{} \\
0 \ar{r} &  \Sigma(B) \overset{\to}{+}_{\Sigma(A)} 0 \simeq {\Sigma^2(A) \overset{\to}{\times}_{\Sigma^2(B)} 0} \ar[d]  \ar{r}[swap]{} & \Sigma^2(A)  \ar[double]{ld}{} \ar[d] \\
& 0 \ar{r}[swap]{} & \Sigma^2(B),
\end{tikzcd}
\]
in which the upper-left square is an oriented pushout square and lower-right square is an oriented pullback square and the upper-right square is a pushout and pullback square.
\end{proof}

	



\cref{brown} implies the following corollary:

\begin{corollary}\label{homologic} Let $\mD$ be a presentable differentiable bioriented category and $H: \infty\Cat_* \to \mD$ a categorical homology theory.
For every oriented cofiber sequence $A \to B \to C$ in $\infty\Cat_*$ there is an induced oriented exact sequence in $\Sp(\mD)$:
$$... \to \bar{H}(A)[-1] \to \bar{H}(B)[-1] \to \bar{H}(C)[-1] \to \bar{H}(A) \to \bar{H}(B) \to \bar{H}(C) \to \bar{H}(A)[1] \to .... $$
	
\end{corollary}

\begin{proof}

By \cref{brown} there is a unique $E \in \Sp(\mD)$ such that $H \simeq \Omega^\infty(E \wedge(-))$ as reduced oriented functors. Moreover by the uniqueness statement in \cref{brown} there is a canonical equivalence $\bar{H} \simeq E \wedge (-) $ of reduced oriented functors 
$\infty\Cat_* \to \Sp(\mD).$
Hence $\bar{H}$ is an oriented left adjoint and so preserves oriented cofiber sequences, which by \cref{oplaxexacto} give rise to oriented exact sequences.

\end{proof}

\begin{lemma}\label{redcofi} Let $\mC$ be an oriented category that admits oriented pushouts and a final object, and $X \to Y $ a morphism in $\mC$.
There is a canonical oriented pushout square in $\mC_*:$
\[
\begin{tikzcd}
X_+ \ar{r}{} \ar{d} & 0 \ar[double]{dl}{} \ar{d}{} \\
Y_+ \ar{r} & Y/X
\end{tikzcd}
\]

\end{lemma}

\begin{proof}

There is a canonical oriented pushout square
\[
\begin{tikzcd}
X \ar{r}{} \ar{d} & * \ar[double]{dl}{} \ar{d}{} \\
Y \ar{r} & Y/X
\end{tikzcd}
\]
in $\mC,$ which is sent by the left adjoint oriented functor $(-)_+: \mC \to \mC_*$ to an oriented pushout square
\[
\begin{tikzcd}
X_+ \ar{r}{} \ar{d} & 0_+ \ar[double]{dl}{} \ar{d}{} \\
Y_+ \ar{r} & (Y/X)_+
\end{tikzcd}
\]
in $\mC_*$.
For every $Z \in \mC_*$ the canonical commutative square
\[
\begin{tikzcd}
0_+ \ar{r}{} \ar{d} & 0 \ar{d}{} \\
Z_+ \ar{r} & Z
\end{tikzcd}
\]
in $\mC_*$
is a pushout square. 
Hence the canonical commutative square
\[
\begin{tikzcd}
0_+ \ar{r}{} \ar{d} & 0 \ar{d}{} \\
(Y/X)_+ \ar{r} & Y/X
\end{tikzcd}
\]
in $\mC_*$
is a pushout square. 
The pasting law (\cref{pasting}) implies the result.

\end{proof}

\cref{redcofi} and \cref{homologic} imply the following:

\begin{corollary}\label{orienex} Let $\mD$ be a presentable differentiable bioriented category, $E: \infty\Cat_* \to \mD$ a categorical homology theory and $X \to Y$ a functor.
There is an induced oriented exact sequence in $\Sp(\mD)$:
$$... \to  \bar{E}(X_+)[-1] \to  \bar{E}(Y_+)[-1] \to \bar{E}(Y/X)[-1] \to  \bar{E}(X_+) \to  \bar{E}(Y_+) \to \bar{E}(Y/X) \to  \bar{E}(X_+)[1] \to .... $$

\end{corollary}

We consider a basic computation.

    

 

    







\begin{example}\label{basiccomp}

Let $E: \infty\Cat_* \to \infty\Cat_*$ be categorical homology.
By \cref{homosph} we have $$E(\bD^{0}_+) = E(S^0) \simeq H(\bN).$$

By \cref{orienex} there is an induced oriented exact sequence of categorical spectra:
$$ ... \to H(\bN) \to \bar{E}(\bD^1) \to H(\bN)[1] \to H(\bN)[1] \to .... $$

In particular, there is an oriented left fiber sequence of categorical spectra:
$$ \bar{E}(\bD^1) \to H(\bN)[1] \to H(\bN)[1] $$
which gives rise to an oriented left fiber sequence of symmetric monoidal $\infty$-categories:
$$ E(\bD^1) \to B \bN \to B\bN.$$


Let $(\bN,\leq)$ denote $\bN$ with its natural total order, viewed as a poset and so as 1-category.
The commutative monoid structure of $\bN$ is compatible with the natural order and endows $(\bN,\leq)$ with the structure of a symmetric monoidal category.

There is a canonical oriented left fiber sequence of symmetric monoidal $\infty$-categories:
$$ (\bN, \leq) \to B\bN \to B\bN.$$

Indeed, the oriented left fiber of the identity of the symmetric monoidal category $B\bN$
is the symmetric monoidal category $B\bN_{\tu/}$ of objects under the
tensor unit,
where the symmetric monoidal structure is lifted from $B\bN$ along the forgetful functor.
The symmetric monoidal category $B\bN_{\tu/}$ is precisely $(\bN, \leq).$


We obtain a symmetric monoidal equivalence $H(\bD^1; \bN) = E(\bD^1) \simeq (\bN, \leq).$


\end{example}

\begin{remark}

In \cite{heine2026categorification} we develop further computational techniques, which we use to compute the categorical homology of globes, and to describe models for the categorical homology of oriented simplices.

\end{remark}

\vspace{2mm}

Next we prove a higher-categorical analogue of the property
that equivalences in stable categories are detected by fibers:
in spectral oriented categories equivalences are detected by oriented fibers.

\begin{theorem}\label{Fibb}

Let $\mC$ be a reduced spectral oriented category and
$$\begin{xy}
\xymatrix{
A  \ar[rd]_\alpha \ar[rr]^\phi
&& B \ar[ld]^\beta
\\ 
& C
}
\end{xy}$$
a commutative triangle in $\mC$ such that $\alpha, \beta$ admit oriented left fibers. Then $\phi$ is an equivalence if and only if the following induced morphism on oriented left fibers is an equivalence: 
$$0 \overset{\to}{\times}_\mE \phi: 0 \overset{\to}{\times}_\mE \mC \to 0 \overset{\to}{\times}_\mE \mD. $$ 

\end{theorem}

\begin{proposition}\label{riol}

Let $\n \geq 0$ and
\begin{equation}\label{tria}
\begin{xy}
\xymatrix{
\mC  \ar[d]_\alpha \ar[r]^\phi
& \mD \ar[d]^\beta
\\ 
\mE \ar[r]^\psi & \mF}
\end{xy}\end{equation}
a commutative square of $\infty$-categories such that $\psi$ is an $\n$+1-equivalence. 

\begin{enumerate}
\item The functor $\phi$ is an $\n$-equivalence if for every $Z \in \mE$ the following induced functor on oriented left fibers is an $\n$-equivalence:
\begin{equation}\label{refs} \{Z\} \overset{\to}{\times}_\mE \mC \to \{\psi(Z)\}\overset{\to}{\times}_\mF \mD. \end{equation} 

\item The functor $\phi$ is an $\n$-equivalence if for every $Z \in \mE$ the following induced functor on oriented right fibers is an $\n$-equivalence:
\begin{equation}\label{refs2}\mC \overset{\to}{\times}_\mE \{Z\} \to \mD \overset{\to}{\times}_\mE \{\psi(Z)\}. \end{equation} 

\end{enumerate}

\end{proposition}

\begin{proof}
(1) is equivalent to (2) by taking the dual $(-)^{\co\op}.$	
	
We prove (1) by induction on $\n \geq 0.$ We start with $\n=0.$
We would like to see that $\phi$ is a $0$-equivalence, i.e. induces a bijection on equivalence classes. We first prove  that $\phi$ is essentially surjective: 
let $Y \in \mD$. Then $\beta(Y) \simeq \psi(Z)$ for some $Z \in \mE.$
Hence $(Y, \psi(Z)\to \beta(Y)) \in \{\psi(Z)\}\overset{\to}{\times}_\mF \mD $. So by assumption there is an object $(X, \sigma:  Z \to \alpha(X)) \in \{Z\}\overset{\to}{\times}_\mE \mC$ such that
$\phi(X)\simeq Y$ and such that $\psi(\sigma)$ factors as $$ \psi(Z) \simeq \beta(Y) \simeq \beta(\phi(X)) \simeq \psi(\alpha(X)).$$
So $\phi$ is essentially surjective.
We prove next that $\phi$ is essentially injective: let $Y,Z \in \mC$ and $\kappa: \phi(Y)\simeq \phi(Z)$ an equivalence. 
Since $\psi$ is a 1-equivalence, the equivalence $$\psi(\alpha(Y)) \simeq \beta(\phi(Y))\simeq \beta(\phi(Z)) \simeq \psi(\alpha(Z))$$ is the image of an equivalence $\alpha(Y) \simeq \alpha(Z).$

The objects $ (\phi(Y),\psi( \alpha(Y)) \simeq \beta(\phi(Y)))$ and $ (\phi(Z), \psi(\alpha(Y)) \simeq \beta(\phi(Y))\simeq \beta(\phi(Z)) )$ are equivalent in $ \{\psi(\alpha(Y))\}\overset{\to}{\times}_\mE \mD $ via $\kappa.$
The first object is the image of $ (Y, \id: \alpha(Y) \to \alpha(Y))$ under (\ref{refs}) and the second object is the image of $ (Z,  \alpha(Y) \simeq \alpha(Z))$ under (\ref{refs}).
So by assumption there is an equivalence between $ (Y, \id: \alpha(Y) \to \alpha(Y))$ and $ (Z,  \alpha(Y) \simeq \alpha(Z))$ in $\{\alpha(Y)\}\overset{\to}{\times}_\mE \mD$.
Thus $Y \simeq Z.$ 
So $\phi$ is a $0$-equivalence.

We prove next that $\phi$ is an $\n+1$-equivalence assuming  the statement for $n$
and that for every $Z \in \mE$ the functor (\ref{refs}) is an $\n+1$-equivalence and that $\psi$ is an $n+2$-equivalence.
By the case $\n=0$ the functor $\phi$ is a $0$-equivalence.
Hence $\phi$ is an $\n+1$-equivalence if for every $X,Y \in \mC$ 
the induced functor $$\Mor_{\mC}(X,Y)\to \Mor_{\mD}(\phi(X),\phi(Y)) $$ is an $\n$-equivalence.
By induction hypothesis this holds if for every morphism $\sigma: \alpha(X) \to \alpha(Y)$ in $\mE$ the induced functor on oriented left fibers $$
\{\sigma\} \overset{\to}{\times}_{\Mor_{\mE}(\alpha(X),\alpha(Y))} \Mor_{\mC}(X,Y)
\to \{\sigma\} \overset{\to}{\times}_{\Mor_{\mF}(\psi\alpha(X),\psi\alpha(Y))}\Mor_{\mD}(\phi(X),\phi(Y)) $$ is an $\n$-equivalence.
By \cref{hom} the latter functor identifies with the induced functor $$\Mor_{ \mC \overset{\to}{\times}_\mE \{\alpha(X)\}}((X, \sigma), (Y, \id)) \to \Mor_{\mD \overset{\to}{\times}_\mF \{\psi\alpha(X)\}}((\phi(X), \sigma), (\phi(Y), \id)),$$ 
which is an $\n$-equivalence since the functor (\ref{refs}) is an $\n$+1-equivalence by assumption.

\end{proof}

We prove the following refinement of \cref{riol}:

\begin{proposition}\label{riok}
Let $ \n \geq 0.$
Consider a commutative triangle \ref{tria} of $\infty$-categories and let $Z \in \mE$. If the induced functor on oriented left fibers 
$$ \{Z\} \overset{\to}{\times}_\mE \phi: \{Z\} \overset{\to}{\times}_\mE \mC \to \{Z\}\overset{\to}{\times}_\mE \mD $$ 
induces on morphism $\infty$-categories $\n$-equivalences,
for every $Y \in \mC$ lying over $Z$ and $X \in \mC$ the following induced functor is an $\n$-equivalence:
$$ \Mor_{\mC}(X,Y)\to \Mor_{\mD}(\phi(X),\phi(Y)).$$

\end{proposition}

\begin{proof}
Let $Y \in \mC$ lying over $Z$ and $X \in \mC$ and $\sigma: \alpha(X) \to \alpha(Y)\simeq Z$ a morphism in $\mE$.
By assumption the induced functor $$\Mor_{(\{Z\} \overset{\overline{\to}}{\prod}_\mE \mD)}((X, \sigma), (Y, \id)) \to \Mor_{(\{Z\} \overset{\overline{\to}}{\prod}_\mE \mD)}((\phi(X), \sigma), (\phi(Y), \id)) $$ is an $\n$-equivalence.
By \cref{hom} the latter functor identifies with the functor
$$\{\sigma\} \overset{\overline{\to}}{\prod}_{\Mor_{\mE}(\alpha(X),\alpha(Y))} \Mor_{\mC}(X,Y) \to \{ \sigma\}\overset{\overline{\to}}{\prod}_{\Mor_{\mE}(\alpha(X),\alpha(Y))} \Mor_{\mD}(\phi(X),\phi(Y)).$$
Hence by \cref{riol} the functor $ \Mor_{\mC}(X,Y)\to \Mor_{\mD}(\phi(X),\phi(Y))$ is an $\n$-equivalence.

\end{proof}

\begin{corollary}\label{opost}
	
Let $\n \geq 0$. Consider a commutative triangle (\ref{tria}) of $\infty$-categories with distinguished object.
If the induced functor on oriented left fibers $$ 0 \overset{\to}{\times}_\mE \phi: 0 \overset{\to}{\times}_\mE \mC \to 0 \overset{\to}{\times}_\mE \mD $$ is an $\n+1$-equivalence, the following induced functor on endomorphism $\infty$-categories is an $\n$-equivalence: $$\Omega(\phi): \Omega(\mC) \to \Omega(\mD).$$ 

\end{corollary}

\begin{proof}[Proof of \cref{Fibb}]

For every $T \in \mC$ the commutative triangle of the statement induces a commutative triangle of spectra:
$$\begin{xy}
\xymatrix{
\R\Mor_\mC(T,A) \ar[rd]_{\alpha_*} \ar[rr]^{\phi_*}
&& \R\Mor_\mC(T,B) \ar[ld]^{\beta_*}
\\ 
& \R\Mor_\mC(T,C).
}
\end{xy}$$
By \cref{weiloc} the reduced oriented functor $\R\Mor_\mC(T,-):\mC \to {\Cat\Sp}$
preserves oriented fibers. Hence the induced map of categorical spectra $\R\Mor_\mC(T,0 \overset{\to}{\times}_C \phi)$ identifies with the map $ 0 \overset{\to}{\times}_{ \R\Mor_\mC(T,C)} \phi_*.$
So by the enriched Yoneda lemma (\cref{enryo}) we can assume that $\mC={\Cat\Sp}.$
In this case the statement follows from \cref{opost} and  \cref{homt}.

\end{proof}

\cref{laxfib} and \cref{Fibb} imply the following:

\begin{corollary}\label{detect} Let $\mC$ be a reduced spectral bioriented category that admits oriented fibers.
The antioriented functor $0 \overset{\to}{\times}: \mC^{\bD^1} \to \mC$ detects weighted colimits.

\end{corollary}

\begin{corollary} \label{homolog}
We consider the following commutative diagram of categorical spectra of whose top and bottom sequences are oriented exact:
$$	
\begin{tikzcd}
... \ar{r}{\gamma_{\n-1}} & A_\n \ar{r}{\alpha_\n} \ar{d}{\psi_\n^A} & B_\n \ar{d}{\psi^B_\n} \ar{r}{\beta_\n} & C_\n \ar{d}{\psi^C_\n}  \ar{r}{\gamma_\n} & A_{\n+1} \ar{d}{\psi^A_{\n+1}} \ar{r}{\alpha_{\n+1}} & B_{\n+1} \ar{d}{\psi^B_{\n+1}} \ar{r}{\beta_{\n+1}} & ... \\ 
... \ar{r}{\gamma'_{\n-1}} & A'_\n \ar{r}{\alpha'_\n} & B'_\n \ar{r}{\beta'_\n} & C'_\n \ar{r}{\gamma'_\n} & A'_{\n+1} \ar{r}{\alpha'_{\n+1}} & B'_{\n+1} \ar{r}{\beta'_{\n+1}} & ...
\end{tikzcd}
$$	
If any two of $ \psi^A_*, \psi^B_*, \psi^C_*$ are equivalences, then so is the third.

\end{corollary}

\begin{proof}
If $\psi^A_*, \psi^B_*$ are equivalences, then trivially also $\psi^C_*$ is an equivalence by functoriality of taking oriented cofibers.
If $\psi^B_*, \psi^C_*$ are equivalences, then $\psi^A_*$ is an equivalence by Theorem \ref{Fibb}.
If $\psi^A_*, \psi^C_*$ are equivalences, then $\psi^B_*$ is an equivalence by the dual statement to Theorem \ref{Fibb}.

\end{proof}

\section{Categorical homology}

\subsection{A categorical Hurewicz theorem}

In this subsection we prove a categorical Hurewicz theorem (\cref{H}).

\begin{definition}Let $\mC \in \infty\Cat.$
The poset of categorical components of $\mC$ is the set $\tau(\mC)$ of equivalence classes of objects of $\mC$ modulo the equivalence relation $A \sim B$
if and only if there are morphisms $A \to B$ and $B \to A$ in $\mC.$
We endow $\tau(\mC)$ with the partial order $A \leq B$
if and only if there is a morphism $A \to B$ in $\mC.$
	
\end{definition}

\begin{remark}
The canonical embedding $\Poset \subset \infty\Cat$ 
of the full subcategory of posets admits a left adjoint sending an $\infty$-category $\mC$ to $\tau(\mC)$.
	
\end{remark}

\begin{lemma}\label{deoo}
	
The endomorphisms functor $$\Omega: \Mon_{\bE_\infty}(\infty\Cat)^\cop \to \Mon_{\bE_\infty}(\infty\Cat) $$ admits a left adjoint reduced bioriented functor
$$B: \Mon_{\bE_\infty}(\infty\Cat) \to \Mon_{\bE_\infty}(\infty\Cat)^\cop$$
that preserves finite products and induces an equivalence to the full subcategory of connected $\infty$-categories with distinguished object.
The underlying functor of $B$ is induced by the finite products preserving functor
$B: \Mon(\infty\Cat) \to \infty\Cat $ of \cref{hhnl}.
	
\end{lemma}

\begin{proof}
	
By \cref{hhnl} there is an adjunction \begin{equation}\label{eqioy}
B: \Mon(\infty\Cat) \rightleftarrows \infty\Cat_*, \end{equation} whose right adjoint lifts the endomorphisms functor $\Omega: \infty\Cat_* \to \infty\Cat_*$ along the forgetful functor and whose left adjoint is an embedding that preserves finite products. Moreover the essential image of the left adjoint precisely consists of the connected $\infty$-categories with distinguished object.
Thus the adjunction (\ref{eqioy}) gives rise to an adjunction
$$ B: \Mon_{\bE_\infty}(\infty\Cat) \simeq \Mon_{\bE_\infty}(\Mon(\infty\Cat)) \rightleftarrows\Mon_{\bE_\infty}(\infty\Cat_*) \simeq \Mon_{\bE_\infty}(\infty\Cat)$$
whose left adjoint is an embedding and preserves finite products.
The right adjoint underlies the endomorphisms bioriented functor $$ \Omega: \Mon_{\bE_\infty}(\infty\Cat)^\cop \to \Mon_{\bE_\infty}(\infty\Cat)$$ of $ \Mon_{\bE_\infty}(\infty\Cat), $ which admits a left adjoint reduced bioriented functor that by uniqueness of left adjoints lifts $B$ and therefore is a reduced bioriented embedding.
 
\end{proof}

\begin{lemma}\label{hujt} Let $\n,\m \geq 0$ and $X \in \Mon_{\bE_\infty}(\infty\Cat)$ be $\n$-connected 
and $Y \in \Mon_{\bE_\infty}(\infty\Cat) $ be $\m$-connected.  Then $X \wedge Y$ is $\n+\m+1$-connected.
	
\end{lemma}

\begin{proof}Let $\infty\Cat^{\leq \n} \subset \infty\Cat$ be the full subcategory of $\n$-connected $\infty$-categories.
The embedding $\infty\Cat^{\leq \n} \subset \infty\Cat$ preserves finite products and admits a right adjoint and so induces a left adjoint embedding $$ \Mon_{\bE_\infty}(\infty\Cat^{\leq \n}) \subset  \Mon_{\bE_\infty}(\infty\Cat)$$ whose essential image precisely consists of the symmetric monoidal $\infty$-categories whose underlying $\infty$-category is $\n$-connected.
The category $ \Mon_{\bE_\infty}(\infty\Cat^{\leq \n})$ is generated under small colimits by the essential image of the free functor $ \infty\Cat^{\leq \n}_* \to  \Mon_{\bE_\infty}(\infty\Cat^{\leq \n}).$
Hence the full subcategory $$ \Mon_{\bE_\infty}(\infty\Cat^{\leq \n}) \subset \Mon_{\bE_\infty}(\infty\Cat)$$ is generated under small colimits by the essential image of the restricted free functor $$ \infty\Cat^{\leq \n}_* \subset \infty\Cat_* \to  \Mon_{\bE_\infty}(\infty\Cat).$$
Since the free functor $ \infty\Cat_* \to  \Mon_{\bE_\infty}(\infty\Cat)$ is monoidal, the result follows from \cref{conesto}.
\end{proof}

\begin{theorem}\label{H}
Let $\n \geq 0$ and $\mC \in \infty\Cat_*$ be $\n$-connected and $R$ a rig.

\begin{enumerate}

\item The $\infty$-category $H(\mC;R)$ is $\n$-connected.

\item The functor $$\Omega^{\n+1}(\mC) \to \Omega^{\n+1}(H(\mC;R))$$ induces on categorical components
the universal map to the free left $R$-module in $\C\mon(\Poset)$ (on the abelianization).

\end{enumerate}

\end{theorem}

\begin{proof}

For every $\n \geq 0$ let $ \infty\Cat^{\leq \n} \subset \infty\Cat $ be the full subcategory of $\n$-connected $\infty$-categories. 
We prove first that for every $\n \geq 0$ the free functor $$\infty\Cat_* \to {R\mathrm{-}\Mod(\Mon_{\bE_\infty}(\infty\Cat))}$$ sends $\n$-connected $\infty$-categories with distinguished object to left $R$-modules whose underlying $\infty$-category is $\n$-connected. 
Since the full subcategory $ \infty\Cat^{\leq \n} \subset \infty\Cat $ is closed under finite products,
there is an induced embedding 
$$ \Mon_{\bE_\infty}(\infty\Cat^{\leq \n}) \subset \Mon_{\bE_\infty}(\infty\Cat)$$ whose essential image precisely consists of the symmetric monoidal $\infty$-categories whose underlying $\infty$-category is $\n$-connected. By \cref{hujt} the full subcategory $$ \Mon_{\bE_\infty}(\infty\Cat^{\leq \n}) \subset  \Mon_{\bE_\infty}(\infty\Cat)$$
is closed under the monoidal structure induced by the Gray-tensor product.
So there is an induced embedding 
$$ {R\mathrm{-}\Mod(\Mon_{\bE_\infty}(\infty\Cat^{\leq \n}))} \subset {R\mathrm{-}\Mod(\Mon_{\bE_\infty}(\infty\Cat))}$$ whose essential image precisely consists of the $R$-modules whose underlying $\infty$-category is $\n$-connected. 

By \cref{deloop} there is an adjunction
\begin{equation}\label{adki} B^{\n+1}: \Mon_{\bE_{\n+1}}(\infty\Cat)  \rightleftarrows \infty\Cat_*: \Omega^{\n+1}\end{equation}
whose left adjoint preserves finite products and induces an equivalence to $ \infty\Cat_*^{\leq \n} \subset \infty\Cat_* .$
Hence we obtain an induced adjunction

\begin{equation}\label{ulurti}
B^{\n+1}: \Mon_{\bE_\infty}(\infty\Cat) \simeq \Mon_{\bE_\infty}(\Mon_{\bE_{\n+1}}(\infty\Cat)) \rightleftarrows \Mon_{\bE_\infty}(\infty\Cat_*) \simeq \Mon_{\bE_\infty}(\infty\Cat): \Omega^{\n+1}
\end{equation}
whose left adjoint is an embedding, and which by \cref{deoo} underlies a reduced antioriented adjunction.
Moreover the left adjoint $B^{\n+1}$ is a linear left $\Mon_{\bE_\infty}(\infty\Cat)$-enriched functor that tensors from the right with the free reduced $\bE_\infty$-algebra on $S^{\n+1}.$
Hence adjunction (\ref{ulurti}) gives rise to an adjunction on left $R$-modules
\begin{equation}\label{adio}
B^{\n+1}: {R\mathrm{-}\Mod(\Mon_{\bE_\infty}(\infty\Cat))} \rightleftarrows  {R\mathrm{-}\Mod(\Mon_{\bE_\infty}(\infty\Cat))}  : \Omega^{\n+1}\end{equation}
whose left adjoint induces an equivalence to the full subcategory $$ {R\mathrm{-}\Mod(\Mon_{\bE_\infty}(\infty\Cat^{\leq \n}))}$$ and which covers adjunction (\ref{adki}).
Hence there is a commutative square of right adjoints
$$\begin{xy}
\xymatrix{
{R\mathrm{-}\Mod(\Mon_{\bE_\infty}(\infty\Cat))} \ar[d] \ar[rr]^{B^{\n+1} \circ \Omega^{\n+1}}
&& {R\mathrm{-}\Mod(\Mon_{\bE_\infty}(\infty\Cat ^{\leq \n}))} \ar[d]
\\ 
\infty\Cat_*  \ar[rr]^{B^{\n+1} \circ \Omega^{\n+1}} && \infty\Cat^{\leq \n}_*.
}
\end{xy}$$	

Thus there is a commutative square of left adjoints 
$$\begin{xy}
\xymatrix{
\infty\Cat^{\leq \n}_* \ar[d] \ar[rr]
&& \infty\Cat_* \ar[d]^{H(-;R)}
\\ 
{R\mathrm{-}\Mod(\Mon_{\bE_\infty}(\infty\Cat^{\leq \n}))}  \ar[rr] && {R\mathrm{-}\Mod(\Mon_{\bE_\infty}(\infty\Cat))}.
}
\end{xy}$$	

We prove the second part. By the first part for every $\n$-connected $\mC \in \infty\Cat_*$ the functor $$ \mC \to H(\mC;R) $$
exhibits the target as free on $\mC$ in the category $R\mathrm{-}\Mod(\Mon_{\bE_\infty}(\infty\Cat^{\leq \n})).$
Using adjunction (\ref{adio}) the functor $$ \Omega^{\n+1}(\mC) \to \Omega^{\n+1}(H(\mC;R)) $$
exhibits the target as free $R$-modules in $\Mon_{\bE_\infty}(\infty\Cat) $ on $  \Omega^{\n+1}(\mC) \in \Mon_{\bE_{\n+1}}(\infty\Cat) $.
The embedding $\Poset \subset \infty\Cat$ admits a left adjoint $\tau$ that preserves finite products.
The adjunction $$ \tau: \infty\Cat \rightleftarrows \Poset$$ induces an adjunction
$$R\mathrm{-}\Mod(\Mon_{\bE_\infty}(\infty\Cat)) \rightleftarrows {R\mathrm{-}\Mod(\Mon_{\bE_\infty}(\Poset))}$$
that covers the former adjunction. Since the right adjoints agree via the forgetful functors,
the left adjoint preserves the free functors. Hence the functor 
$$ \tau(\Omega^{\n+1}(\mC))) \to \tau(\Omega^{\n+1}(H(\mC;R)))) $$
exhibits the target as free $R$-module in $\Mon_{\bE_\infty}(\Set) $ on $ \tau(\Omega^{\n+1}(\mC))) \in \Mon_{\bE_{\n+1}}(\Set) $.

\end{proof}

\begin{corollary} Let $\mC \in \infty\Cat_*$ be connected.
	
\begin{enumerate}
\item Then $H(\mC)$ is connected and the functor
$$\Omega(\mC) \to \Omega(H(\mC)) $$ induces on weak equivalence classes the universal map to the abelianization.
		
\item Let $ n > 0$ and $\mC$ be $\n$-connected. Then $H(\mC)$ is $\n$-connected
and the functor $$\Omega^{\n+1}(\mC) \to \Omega^{\n+1}(H(\mC))$$ induces an isomorphism on weak equivalence classes.
\end{enumerate}
	
\end{corollary}

\subsection{A categorical Eilenberg-Zilber theorem}

Although the Gray smash product is not symmetric, we show that 
$B^\infty(R)$ is central for the Gray-smash product for every rig $R$, and that the resulting functor 
$$ Z \mapsto H(Z;R) = H(R) \wedge Z \simeq Z \wedge H(R)$$
is monoidal (\cref{liftos}). This is a categorical Eilenberg--Zilber theorem for categorical $R$-homology.


\begin{theorem}\label{bimo}

Let $R$ be a rig space. There is a monoidal functor $$\Cat\Sp \to {_{B^\infty(R)}\Mod_{B^\infty(R)}(\Cat\Sp)}, Z \mapsto B^\infty(R) \wedge Z \simeq Z \wedge B^\infty(R).$$

\end{theorem}

We obtain the following immediate corollary:

\begin{corollary}\label{bimo2} 
Let $R$ be a rig space. There is a lax monoidal functor 
$$ \infty\Cat_* \to {_{R}\Mod_{R}(\infty\Cat_*)}, Z \mapsto \Omega^\infty(B^\infty(R) \wedge Z) \simeq \Omega^\infty(Z \wedge B^\infty(R)).$$	
	
\end{corollary} 

\begin{corollary}\label{liftos} Let $R$ be a rig. Categorical $R$-homology $ \infty\Cat_* \to \infty\Cat_*$ lifts to a left adjoint monoidal functor 
$$ \infty\Cat_* \to {_{H(R)}\Mod_{H(R)}(\Cat\Sp)}.$$
\end{corollary}

\cref{bimo} follows immediately from \cref{copli} and \cref{ohas}.

\begin{lemma}\label{leull}

Let $\mV, \mW$ be presentably symmetric monoidal categories, $\phi: \mV \to \mW$ a lax symmetric monoidal functor and $A$ an associative algebra in $\mV.$
The image $\phi(\tu) \to \phi(A)$ under $\phi$ of the unique map of associative algebras $\tu \to A$ in $\mV$ factors in $\Alg(\mW)$ as $$\phi(\tu) \to \L\Mor_{{_{\phi(A)}\Mod_{\phi(A)}(\mW)}}(\phi(A),\phi(A))\to \phi(A).$$

\end{lemma}

\begin{proof}
The $\mV$-linear forgetful functor ${_{A}\Mod_A(\mV)} \to {\Mod_{A}(\mV)}$
induces a map in $\Alg(\mV)$: $$\L\Mor_{{_A\Mod_{A}(\mV)}}(A,A) \to \L\Mor_{\Mod_{A}(\mV)}(A,A) \simeq A.$$ 	
The lax symmetric monoidal functor $\phi$ induces a commutative square of left tensored categories and left linear functors:

$$\begin{xy}
\xymatrix{
{_A\Mod_{A}(\mV)} \ar[d] \ar[r]
& {_{\phi(A)}\Mod_{\phi(A)}(\mW)} \ar[d]
\\ 
{\Mod_{A}(\mV)} \ar[r] & {\Mod_{\phi(A)}(\mW)}.
}
\end{xy}$$
The latter induces a commutative square of associative algebras in $\mW$, which implies the claim:
$$\begin{xy}
\xymatrix{
\phi(\L\Mor_{{_A\Mod_{A}(\mV)}}(A,A)) \ar[d] \ar[r]
& \L\Mor_{{_{\phi(A)}\Mod_{\phi(A)}(\mW)}}(\phi(A),\phi(A)) \ar[d]
\\ 
\phi(\L\Mor_{{\Mod_{A}(\mV)}}(A,A)) \ar[r] \ar[d]^\simeq & \L\Mor_{{\Mod_{\phi(A)}(\mW)}}(\phi(A),\phi(A))\ar[d]^\simeq
\\ 
\phi(A) \ar[r]^= & \phi(A).
}
\end{xy}$$

\end{proof}

\begin{corollary}\label{cort}

Let $\mC$ be a presentably symmetric monoidal category and $\mA$ an associative algebra in the monoidal category $_\mC\Mod(\Pr^L).$	
The canonical left adjoint monoidal functor $\mC \to \mA$, the unit of $\mA$, factors as left adjoint monoidal functors $$\mC \to {{\mA \ot \mA^\rev}\mathrm{-}\L\Fun(\mA,\mA)}\to \mA.$$

\end{corollary} 

\begin{proof}

We apply \cref{leull}, where $\phi: \mV \to \mW$ is the lax symmetric monoidal forgetful functor $_\mC\Mod(\Pr^L)\to \Pr^L$ and $A$ is $\mA.$

\end{proof}

\begin{corollary}\label{corril}

Let $\mB$ be a presentably monoidal category.
The left adjoint monoidal functor $$\Mon_{\bE_\infty}(\infty\Grp) \to \Mon_{\bE_\infty}(\mB)$$
induced by the unique left adjoint monoidal functor $\infty\Grp \to \mB$
factors as left adjoint monoidal functors $$\Mon_{\bE_\infty}(\infty\Grp) \to {{\Mon_{\bE_\infty}(\mB) \ot \Mon_{\bE_\infty}(\mB)^\rev}\mathrm{-}\L\Fun(\Mon_{\bE_\infty}(\mB), \Mon_{\bE_\infty}(\mB))} \to \Mon_{\bE_\infty}(\mB).$$

\end{corollary} 

\begin{proof}
By \cite[Theorem 4.6.]{gepner2016universality} the functor $\Mon_{\bE_\infty}: \Pr^L \to \Pr^L$ is lax symmetric monoidal and so lifts to a lax symmetric monoidal functor 
$$\Pr^L \simeq {_{\infty\Grp}\Mod(\Pr^L)} \to {_{\Mon_{\bE_\infty}(\infty\Grp)}\Mod(\Pr^L)}$$
that sends $\mB$ to an associative algebra structure on $\Mon_{\bE_\infty}(\mB)$ in $_{\Mon_{\bE_\infty}(\infty\Grp)}\Mod(\Pr^L)$.
We apply \cref{cort} to $\mC:=\Mon_{\bE_\infty}(\infty\Grp)$ and $\mA:=\Mon_{\bE_\infty}(\mB)$.

\end{proof}

\begin{corollary}\label{corrill}

The identity factors as left adjoint monoidal functors $$\Mon_{\bE_\infty}(\infty\Grp) \to {{\Mon_{\bE_\infty}(\infty\Grp) \ot \Mon_{\bE_\infty}(\infty\Grp)^\rev}\mathrm{-}\L\Fun(\Mon_{\bE_\infty}(\infty\Grp), \Mon_{\bE_\infty}(\infty\Grp))} $$$$ \to \Mon_{\bE_\infty}(\infty\Grp).$$

\end{corollary} 

\begin{lemma}\label{corpy} 
There is a monoidal functor $$ {{\wedge\L\Fun\wedge}(\Mon_{\bE_\infty}(\infty\Cat), \Mon_{\bE_\infty}(\infty\Cat))} \to {{\wedge\L\Fun\wedge}(\Cat\Sp, \Cat\Sp)}.$$

\end{lemma}

\begin{proof}

By \cref{hhoo} there is a canonical equivalence of reduced bioriented categories $$ \Sp(\overline{\Sp}(\Mon_{\bE_\infty}(\infty\Cat))) \simeq \Mon_{\bE_\infty}(\Sp(\overline{\Sp}(\infty\Cat))) \simeq \Mon_{\bE_\infty}(\Cat\Sp) \simeq \Cat\Sp.$$ 

We apply \cref{corp}. This implies the result.

\end{proof}	

\begin{proposition}\label{corrilll}

The canonical left adjoint monoidal embedding $ H: \Mon_{\bE_\infty}(\infty\Grp) \subset \Cat\Sp $ factors as left adjoint monoidal functors $$\Mon_{\bE_\infty}(\infty\Grp) \to {\Cat\Sp \ot \Cat\Sp^\rev\mathrm{-}\L\Fun(\Cat\Sp, \Cat\Sp)} \to \Cat\Sp.$$

\end{proposition} 

\begin{proof}

By \cref{corrill} the identity factors as left adjoint monoidal functors $$\Mon_{\bE_\infty}(\infty\Grp) \to {{\Mon_{\bE_\infty}(\infty\Grp)} \ot {\Mon_{\bE_\infty}(\infty\Grp)^\rev} \mathrm{-}\L\Fun(\Mon_{\bE_\infty}(\infty\Grp), \Mon_{\bE_\infty}(\infty\Grp))} $$$$ \to \Mon_{\bE_\infty}(\infty\Grp).$$

The canonical left adjoint monoidal embedding $\Mon_{\bE_\infty}(\infty\Grp) \subset \Mon_{\bE_\infty}(\infty\Cat)$
induces a functor $${_{\Mon_{\bE_\infty}(\infty\Grp) } \Mod_{\Mon_{\bE_\infty}(\infty\Cat) }} \to {_{\Mon_{\bE_\infty}(\infty\Cat) } \Mod_{\Mon_{\bE_\infty}(\infty\Cat) }} ,$$ which gives rise to a monoidal functor $$ {{\Mon_{\bE_\infty}(\infty\Grp)} \ot {\Mon_{\bE_\infty}(\infty\Grp)^\rev} \mathrm{-}\L\Fun(\Mon_{\bE_\infty}(\infty\Grp), \Mon_{\bE_\infty}(\infty\Grp))} \to $$$$ {{\Mon_{\bE_\infty}(\infty\Cat)} \ot {\Mon_{\bE_\infty}(\infty\Cat)^\rev} \mathrm{-}\L\Fun(\Mon_{\bE_\infty}(\infty\Cat), \Mon_{\bE_\infty}(\infty\Cat))}.$$
\cref{corpy} gives a monoidal functor
$${{\Mon_{\bE_\infty}(\infty\Cat)} \ot {\Mon_{\bE_\infty}(\infty\Cat)^\rev} \mathrm{-}\L\Fun(\Mon_{\bE_\infty}(\infty\Cat), \Mon_{\bE_\infty}(\infty\Cat))} \to $$$$ {\Cat\Sp \ot \Cat\Sp^\rev\mathrm{-}\L\Fun(\Cat\Sp, \Cat\Sp)}.$$

\end{proof}

Passing to associative algebras we obtain the following:

\begin{corollary}\label{copli}
The embedding $ B^\infty: \Rig(\infty\Grp)= \Alg(\Mon_{\bE_\infty}(\infty\Grp)) \subset \Alg(\Cat\Sp) $ factors as left adjoints $$\Rig(\infty\Grp) \to \Alg({\Cat\Sp \ot \Cat\Sp^\rev\mathrm{-}\L\Fun(\Cat\Sp, \Cat\Sp)}) \to \Alg(\Cat\Sp).$$

\end{corollary}

\begin{proposition}\label{ohas}

Let $\mV$ be a presentably monoidal $\infty$-category and $A$ an associative algebra in the monoidal category $\mV \ot \mV^\rev\mathrm{-}\L\Fun(\mV,\mV).$	
The following functor is monoidal $$\mV \to {_A\Mod_{A}(\mV)}, \ V \mapsto A \ot V \simeq V \ot A.$$

\end{proposition}

\begin{proof}

A monad on $\mV$ in the 2-category $\Mod_\mV(\widehat{\Cat})$
is an associative algebra in the monoidal category $\mV^\rev\mathrm{-}\L\Fun(\mV,\mV)$.

The canonical $\mV,\mV$-biaction on $\mV$ corresponds to a left $\mV$-action
on $\mV \in \Mod_\mV(\widehat{\Cat})$ and so to a monoidal functor
$ \mV \to \mV^\rev\mathrm{-}\L\Fun(\mV,\mV),$ 
which is an equivalence by \cite[Remark 4.10. (3)]{heine2017monadicity}.
By \cite[Corollary 5.35. (2), Remark 4.10. (3)]{heine2017monadicity} the right $\mV$-linear functor ${_A\Mod}(\mV)\to \mV$ is an Eilenberg-Moore object of $A$ in $\Mod_\mV(\widehat{\Cat})$ in the sense of \cite[Definition 5.2.]{heine2017monadicity}. 

A monad on $\mV$ in the 2-category $_{\mV}\Mod_\mV(\widehat{\Cat})$
is an associative algebra in the monoidal category $\mV \ot \mV^\rev\mathrm{-}\L\Fun(\mV,\mV)$.
By \cite[Corollary 5.27.]{heine2017monadicity} the monad $A$ on $\mV$ in ${_{\mV}\Mod_\mV}(\widehat{\Cat})$ admits an Eilenberg-Moore object $\mX \to \mV$ that lies over the Eilenberg-Moore object ${_A\Mod}(\mV)\to \mV $ of $A$ in $\Mod_\mV(\widehat{\Cat})$.

By \cite[Remark 4.83.]{heine2024bienrichedinftycategories} there is a canonical equivalence $${_{\mV}\Mod_\mV}(\widehat{\Cat}) \simeq {_\mV\Mod(\Mod_\mV(\widehat{\Cat}))}.$$ Under this equivalence the $\mV,\mV$-bitensored category $\mX$ corresponds to a left $\mV$-action on ${_A\Mod}(\mV)$ with respect to the left $ \widehat{\Cat}$-action on $\Mod_\mV(\widehat{\Cat}).$ The latter corresponds to a monoidal functor $$\mV \to \mV^\rev\mathrm{-}\L\Fun({_A\Mod}(\mV), {{_A\Mod}(\mV)}) \simeq {{_A}\Mod_A(\mV)}$$
sending $V$ to $A \ot V \simeq V \ot A$, where the last equivalence is \cite[Theorem 4.8.4.1.]{lurie.higheralgebra}, \cite[Remark 4.83.]{heine2024bienrichedinftycategories}.

\end{proof}

\section{An equivalence between stable and spectral oriented categories}

\begin{definition}\emph{}

\begin{itemize}

\item A spectral oriented category is quasi-stable if it is reduced and admits suspensions and endomorphisms.

\item A spectral oriented category is stable if it is quasi-stable and admits oriented fibers and oriented cofibers.

\item An antispectral antioriented category is (quasi-)antistable if the spectral oriented category $\mC^\co$ is (quasi-)stable.

\item A spectral bioriented category is (quasi-)stable if it is reduced and its underlying spectral oriented category is (quasi-)stable.		
	
\item An antispectral bioriented category is (quasi-)antistable if it is reduced and its underlying antispectral antioriented category is (quasi-)antistable.	

\item A bispectral bioriented category is (quasi-)bistable if it is reduced, (quasi-)stable and (quasi-)antistable.

\end{itemize}

\end{definition}


    

\begin{notation}\emph{}
	
\begin{itemize}

\item Let ${\Cat_{\mathrm{qst}}} \barwedge \subset {\Cat\barwedge}$ be the full subcategory of quasi-stable spectral oriented categories.

\item Let $\barwedge\Cat_\mathrm{qanst} \subset \barwedge\Cat$ be the full subcategory of quasi-antistable antispectral antioriented categories.

\item Let $\wedge\Cat_{\mathrm{qst}}\barwedge \subset \wedge\Cat\barwedge$ be the full subcategory of quasi-stable spectral bioriented categories.

\item Let ${\barwedge\Cat_\mathrm{qanst}}{\wedge} \subset \barwedge\Cat\wedge$ be the full subcategory of quasi-antistable antispectral bioriented categories.
		
\item Let $\barwedge\Cat_{\mathrm{qbist}}\barwedge  \subset \barwedge\Cat\barwedge$ be the full subcategory of quasi-bistable bispectral bioriented categories.

\end{itemize}

\end{notation}

\begin{notation}\emph{}
    
 \begin{itemize}
 
 \item Let $\Cat_{\mathrm{qst}}\wedge \subset {\Cat\wedge}$ be the subcategory of quasi-stable oriented categories and oriented functors preserving suspensions.
	
\item Let ${\wedge}\Cat_\mathrm{qanst} \subset \wedge\Cat$ be the subcategory of quasi-antistable antioriented categories and antioriented functors preserving antisuspensions.
	
\item Let ${\wedge\Cat_\mathrm{qst}\wedge} \subset \wedge\Cat\wedge$ be the subcategory of quasi-stable bioriented categories and bioriented functors preserving suspensions.
	
\item Let $\wedge\Cat_{\mathrm{qanst}}\wedge \subset \wedge\Cat\wedge$ be the subcategory of quasi-antistable bioriented categories and bioriented functors preserving antisuspensions.
	
\item Let ${\wedge}\Cat_{\mathrm{qbist}}\wedge \subset \wedge\Cat\wedge$ be the subcategory of quasi-bistable bioriented categories and bioriented functors preserving suspensions and antisuspensions.

\end{itemize}   
    
\end{notation}

\begin{theorem}\label{specenr}\emph{}

\begin{enumerate}
\item The following forgetful 2-functor is an equivalence:
$$\Cat_{\mathrm{qst}}\barwedge \to {\Cat_{\mathrm{qst}}}\wedge, $$	
\item The following forgetful 2-functor is an equivalence: $${\barwedge \Cat_{\mathrm{qanst}}} \to {\wedge \Cat_{\mathrm{qanst}}}. $$

\item The following forgetful 2-functor is an equivalence:
 $$\wedge \Cat_{\mathrm{qst}}\barwedge \to \wedge \Cat_{\mathrm{qst}}\wedge, $$
\item The following forgetful 2-functor is an equivalence: $${\barwedge \Cat_{\mathrm{qanst}} \wedge} \to {{\wedge \Cat_{\mathrm{qanst}} \wedge}}, $$
\item The following forgetful 2-functor is an equivalence: $$ {\barwedge \Cat_{\mathrm{qbist}}\barwedge}\to {{\wedge \Cat_{\mathrm{qbist}} \wedge}}.$$

\end{enumerate}

Under these equivalences stability, antistability, bistability corresponds to each other.

\end{theorem}

\begin{proof}
The proofs of (1) and (2) are similar. We prove (2) and after that we prove (3), (4), (5). The forgetful functor of (2) is conservative. So it suffices to construct a fully faithful left adjoint.
Let $\mC$ be a small quasi-antistable antioriented category and $\mD$ a small reduced antispectral antioriented category.
We prove that the induced functor $$\rho:  \barwedge\overline{\Exc}(\bar{\mC}, \mD) \to \wedge\overline{\Exc}(\mC, \mD) $$ is an equivalence. The antispectral Yoneda embedding \cite[Notation 5.36., Corollary 5.14., Lemma 2.55.]{heine2024bienrichedinftycategories} is an antispectral antioriented embedding $\mD \hookrightarrow \mE$ into a presentable antispectral antioriented category that preserves antiendomorphisms.
Thus $\rho$ is the pullback of the similarly defined functor for $\mD$ replaced by $\mE,$ which we prove to be an equivalence.
There is a commutative square 
$$\begin{xy}
\xymatrix{
\barwedge\L\Fun({\Fun\wedge}(\bar{\mC}^\circ, \Cat\Sp), \mE) \ar[d] \ar[r]
&  \wedge\L\Fun({\Exc\wedge}(\mC^\circ, \infty\Cat_*), \mE) \ar[d]
\\ 
\barwedge\overline{\Exc}(\bar{\mC}, \mE) \ar[r]  & \wedge\overline{\Exc}(\mC, \mE) .
}
\end{xy}$$	
The left vertical functor is an equivalence by the universal property of enriched presheaves \cite[Theorem 5.49.]{heine2024bienrichedinftycategories} and because $\barwedge\overline{\Exc}(\bar{\mC}, \mE) = \barwedge\Fun(\bar{\mC}, \mE)$ by \cref{preserv}.
The right vertical functor is an equivalence by \cref{laxoplax} and \cref{unistab}. 
So it remains to see that the top functor is an equivalence.
This is equivalent to say that the antispectral antioriented embedding $ \bar{\mC} \subset {\Exc\wedge}(\mC^\circ, \infty\Cat_*)$
induces an equivalence $\theta: {\Fun\barwedge}(\bar{\mC}^\circ, \Cat\Sp) \to {\Exc\wedge}(\mC^\circ, \infty\Cat_*)$ of antispectral antioriented categories.
By \cite[Theorem 5.22.]{heine2024bienrichedinftycategories} and \cref{excloci} the antioriented category $ {\Exc\wedge}(\mC^\circ, \infty\Cat_*)$ is generated under small colimits by left tensors of representables. Hence also the canonical refinement to an antispectral antioriented category is generated under small colimits by left tensors of representables.
\cite[Corollary 4.54.]{heine2024bienrichedinftycategories} guarantees that $\theta$ is an equivalence if for every $X \in \mC$ the antispectral antioriented functor
$\gamma:  {\Exc\wedge}(\mC^\circ, \infty\Cat_*) \to \Cat\Sp$ corepresented by $X$ preserves small colimits and left tensors. By \cref{spgen} the reduced bioriented category $\Cat\Sp$ is generated under small colimits and left tensors by antiendomorphisms of the sphere spectrum. 
Since source and target of $\gamma$ are antistable and presentable and $\gamma$ preserves antiendomorphisms, $\gamma$ preserves small colimits and left tensors if the underlying reduced antioriented functor $\beta$ of $\gamma$ preserves small colimits and left tensors.
Antistability of the source of $\beta$ implies that $\beta$ is antiexcisive and for every $\n \geq 0$ the $\n$-th factor of $\beta$ denoted by $$\beta_\n: {\Exc\wedge}(\mC^\circ, \infty\Cat_*) \to \infty\Cat_*$$ is corepresented as a reduced antioriented functor by $\bar{\Omega}^\n(\L\Mor_\mC(-,X))$. The structure equivalence $$\L\Mor_{  {\Exc\wedge}(\mC^\circ, \infty\Cat_*) }(\bar{\Omega}^\n(\L\Mor_\mC(-,X)),-)  \to \bar{\Omega}\L\Mor_{  {\Exc\wedge}(\mC^\circ, \infty\Cat_*) }(\bar{\Omega}^{\n+1}(\L\Mor_\mC(-,X)),-) $$
corresponds by the enriched Yoneda lemma (\cref{enryo}) to the object of $$  \bar{\Omega}\L\Mor_{{\Exc\wedge}(\mC^\circ, \infty\Cat_*) }(\bar{\Omega}^{\n+1}(\L\Mor_\mC(-,X)),\bar{\Omega}^\n(\L\Mor_\mC(-,X))) \simeq$$$$ \L\Mor_{\mC}(\bar{\Omega}^{\n+1}(\L\Mor_\mC(-,X)),\bar{\Omega}^{\n+1}(\L\Mor_\mC(-,X))), $$ which is the identity.
By the enriched Yoneda lemma $\beta_\n$ evaluates at $\bar{\Omega}^\n(X)$
and identifies with the reduced antioriented functor $$\nu_X:  {\Exc\wedge}(\mC^\circ, \infty\Cat_*) \to \Cat\Sp $$ of \cref{forgetu}, which by \cref{locpre2} preserves small colimits and left tensors. 

\vspace{1mm}

The strategy of proof is analogous to the proof of \cref{specenr}.
The forgetful functors of (3), (4), (5) are conservative. So it suffices to construct fully faithful left adjoints. 
Let $\mC$ be a small quasi-stable, quasi-antistable, quasi-bistable bioriented category and $\mD$ a small reduced spectral, antispectral, bispectral bioriented category, respectively.
We prove that the bottom functors in the following commutative squares are equivalences:
$$\begin{xy}
\xymatrix{
{\barwedge\L\Fun\wedge}({\wedge\Fun\wedge}(\bar{\mC}^\circ, \Cat\Sp \boxplus \infty\Cat_*), \mE) \ar[d] \ar[r]
&  {\wedge\L\Fun\wedge}({\wedge\Exc\wedge}(\mC^\circ, \infty\Cat_*\boxplus\infty\Cat_*), \mE) \ar[d]
\\ 
{\barwedge\overline{\Exc}\wedge}(\bar{\mC}, \mE) \ar[r]  & {\wedge\overline{\Exc}\wedge}(\mC, \mE),
}
\end{xy}$$		
$$\begin{xy}
\xymatrix{
{\wedge\L\Fun\barwedge}({\wedge\Fun\wedge}(\bar{\mC}^\circ, \infty\Cat_* \boxplus \Cat\Sp), \mE) \ar[d] \ar[r]
&  {\wedge\L\Fun\wedge}({\wedge\overline{\Exc}\wedge}(\mC^\circ, \infty\Cat_*\boxplus\infty\Cat_*), \mE) \ar[d]
\\ 
{\wedge\Exc\barwedge}(\bar{\mC}, \mE) \ar[r]  & {\wedge\Exc\wedge}(\mC, \mE), 
}
\end{xy}$$	
$$\begin{xy}
\xymatrix{
{\barwedge\L\Fun\barwedge}({\wedge\Fun\wedge}(\bar{\mC}^\circ, \Cat\Sp \boxplus \Cat\Sp), \mE) \ar[d] \ar[r]
&  {\wedge\L\Fun\wedge}({\wedge\Bi\Exc\wedge}(\mC^\circ, \infty\Cat_*\boxplus\infty\Cat_*), \mE) \ar[d]
\\ 
{\barwedge\Bi\Exc\barwedge}(\bar{\mC}, \mE) \ar[r]  & {\wedge\Bi\Exc\wedge}(\mC, \mE) .
}
\end{xy}$$	
The enriched Yoneda embedding \cite[Notation 5.36., Corollary 5.14., Lemma 2.55.]{heine2024bienrichedinftycategories} is a spectral, antispectral, bispectral embedding of $\mD$ into a presentable spectral, antispectral, bispectral bioriented category, respectively, that preserves endomorphisms and antiendomorphisms. So we can replace $\mD$ by a presentable spectral, antispectral, bispectral bioriented category, respectively.
The left vertical functors are equivalences by the universal property of enriched presheaves \cite[Theorem 5.49.]{heine2024bienrichedinftycategories} and \cref{preserv}, the right vertical functors are equivalences by \cref{laxoplax2} and \cref{unistab}. 
So it remains to see that the top functors are equivalences.
This is equivalent to say that the following induced bioriented functors are equivalences:
$$ {\wedge\Fun\wedge}(\bar{\mC}^\circ, \Cat\Sp \boxplus \infty\Cat_*)\to {\wedge\Exc\wedge}(\mC^\circ, \infty\Cat_*\boxplus\infty\Cat_*),
$$
$${\wedge\Fun\wedge}(\bar{\mC}^\circ, \infty\Cat_* \boxplus \Cat\Sp) \to {\wedge\overline{\Exc}\wedge}(\mC^\circ, \infty\Cat_*\boxplus\infty\Cat_*),$$
$${\wedge\Fun\wedge}(\bar{\mC}^\circ, \Cat\Sp \boxplus \Cat\Sp) \to  {\wedge\Bi\Exc\wedge}(\mC^\circ, \infty\Cat_*\boxplus\infty\Cat_*).$$

By \cite[Theorem 5.22.]{heine2024bienrichedinftycategories} and \cref{excloci} the right hand bioriented categories are generated under small colimits by left and right tensors of representables. Hence also their respective refinements to spectral, antispectral, bispectral bioriented categories are generated under small colimits by left and right tensors of representables.
\cite[Corollary 4.54.]{heine2024bienrichedinftycategories} implies that the latter are equivalences if for every $X \in \mC$ the antispectral bioriented functor
\begin{equation}\label{spe01}
{\wedge\overline{\Exc}\wedge}(\mC^\circ, \infty\Cat_*\boxplus\infty\Cat_*) \to \Cat\Sp \boxplus \infty\Cat_*,\end{equation} 
spectral bioriented functor
\begin{equation}\label{spe02}  {\wedge\Exc\wedge}(\mC^\circ, \infty\Cat_*\boxplus\infty\Cat_*) \to \infty\Cat_* \boxplus \Cat\Sp,\end{equation}
bispectral bioriented functor
\begin{equation}\label{spe03}  {\wedge\Bi\Exc\wedge}(\mC^\circ, \infty\Cat_*\boxplus\infty\Cat_*) \to \Cat\Sp  \boxplus \Cat\Sp\end{equation}
corepresented by $X$ preserve small colimits and left and right tensors. 
The latter preserve endomorphisms and antiendomorphisms and are between antistable, stable, bistable bioriented categories, respectively. Thus the underlying reduced bioriented functors \ref{spe01}, \ref{spe02}, \ref{spe03} are antiexcisive, excisive, biexcisive, respectively. 
The underlying reduced bioriented functors of \ref{spe01}, \ref{spe02}, \ref{spe03} lift the reduced bioriented functors
$$  {\wedge\overline{\Exc}\wedge}(\mC^\circ, \infty\Cat_*\boxplus\infty\Cat_*) \to \infty\Cat_* \boxplus \infty\Cat_*,  $$$$ {\wedge\Exc\wedge}(\mC^\circ, \infty\Cat_*\boxplus\infty\Cat_*) \to \infty\Cat_* \boxplus \infty\Cat_*,$$
$$ {\wedge\Bi\Exc\wedge}(\mC^\circ, \infty\Cat_*\boxplus\infty\Cat_*) \to \infty\Cat_*  \boxplus \infty\Cat_*$$ corepresented by $X$, which preserve small colimits and left and right tensors by \cite[Theorem 3.50.]{heine2024bienrichedinftycategories}.
\cref{collif} implies that the underlying reduced bioriented functors \ref{spe01}, \ref{spe02}, \ref{spe03} preserve small colimits and left and right tensors.
Since the reduced bioriented category $\Cat\Sp$ is generated under small colimits and left tensors by endomorphisms of the categorical sphere spectrum (\cref{spgen}), also the reduced bioriented functors \ref{spe01}, \ref{spe02}, \ref{spe03} preserve small colimits and left and right tensors.

\end{proof}

\bibliographystyle{plain}
\bibliography{ma}

\end{document}